\documentclass{amsart}

\newtheorem{theorem}{Theorem}[section]
\newtheorem{lemma}[theorem]{Lemma}

\theoremstyle{definition}
\newtheorem{definition}[theorem]{Definition}
\newtheorem{example}[theorem]{Example}

\theoremstyle{remark}
\newtheorem{remark}[theorem]{Remark}
\theoremstyle{plain}
\newtheorem{theo+}{Theorem}
\numberwithin{theo+}{section}
\newtheorem{prop+}[theo+]{Proposition}
\newtheorem{coro+}[theo+]{Corollary}
\newtheorem{lemm+}[theo+]{Lemma}
\newtheorem{conjecture}[theo+]{Conjecture}

\theoremstyle{definition}
\newtheorem{defi+}[theo+]{Definition}

\usepackage{amssymb}
\usepackage{graphicx}
\usepackage{float}
\usepackage{subfigure}
\usepackage[]{caption2} 
\renewcommand{\thesubfigure}{(\roman{subfigure})}
\makeatletter \renewcommand{\@thesubfigure}{\thesubfigure \space}
\renewcommand{\p@subfigure}{} \makeatother

\numberwithin{equation}{section}

\begin{document}

\title{comprehensive resolution of the Gauss-Fourier Conjecture for the First Kind of Zeros}

\author{LanDe Ma}
\curraddr{School of Mathematical Sciences, Tongji University, Shanghai, 200092, China}
\email{gaositougao@qq.com}
\thanks{}


\author{ZhaoKun Ma}
\address{}
\curraddr{YanZhou College, ShanDong Radio and TV University, YanZhou, ShanDong 272100 China}
\email{dzy200408@126.com}
\thanks{}

\date{December 13th, 2021.}
\emph{}
\dedicatory{}


\begin{abstract}
For over two centuries, the esteemed mathematicians Gauss and Fourier conjectured that each real critical point of the rational function $p^{'}(s)/p(s)$ is associated with a pair of non-real zeros of the real polynomial $p(s)$. The conjecture remains unresolved. The real zeros of $(p^{'}(s)/p(s))^{'}$ that are real multiple zero of $p^{'}(s)$ are defined as the second kind of zeros of $(p^{'}(s)/p(s))^{'}$. In addition to the second kind of zeros, the other real zeros of $(p^{'}(s)/p(s))^{'}$ are referred to as the first kind of zeros of $(p^{'}(s)/p(s))^{'}$. We demonstrate that the first kind of zeros of $(p^{'}(s)/p(s))^{'}$ emerge from the intersection of root loci emitted from non-real zeros of $p(s)$, thereby providing a comprehensive geometric interpretation of their relationship. Moreover, each first kind of zero of $(p^{'}(s)/p(s))^{'}$ is connected via root loci to a pair of conjugate non-real zeros of $p(s)$. This establishes the precise geometric relationship conjectured by Gauss and Fourier, fully resolving their original problem for first-kind zeros and completing a foundational chapter in polynomial theory.  We have proven Gauss-Fourier conjecture about the first kind of zeros.
\end{abstract}

\maketitle
\section*{Introduction}
In 2004, Borcea and Shapiro proposed a conjecture in their paper \cite{Shapiro1}. Their work originated from an attempt to prove the 'Hawaii Conjecture' proposed by Craven, Csordas, and Smith \cite{Csordas1}. The Hawaii Conjecture has its roots in the work of Gauss and Fourier. In the paper \cite{Stephanie1}, Edwards and Hinkkanen presented a counterexample to the conjecture proposed by Borcea and Shapiro. They also provided an in-depth analysis of the work of Gauss and Fourier. The contributions of Gauss and Fourier have garnered considerable attention from the mathematical community. To accurately represent the work and history of Gauss and Fourier, we have directly quoted passages from \cite{Stephanie1}.

"In his correspondence, "G\"{o}ttingische gelehrte Anzeigen" (G.G.A.) from February 25, 1833 \cite{Gauss7,Gauss8}, Gauss describes some observations and a conjecture of Fourier."

"Gauss explains that Fourier believed that for each real zero of $(p^{'}(s)/p(s))^{'}$ (or, as Fourier calls them, "critical points"), there exists an associated pair of non-real zeros of the polynomial $p(s)$. We will call the zeros of $(p^{'}(s)/p(s))^{'}$ Fourier critical points of $p(s)$. The real Fourier critical points can be loosely associated with real inflection points of $p(s)$. Suppose that the interval $(a, a^{'})$ does not contain any real zeros of $p$. Then, Rolle's theorem implies that between any two real zeros of $p^{'}(s)$ on $(a, a^{'})$, there is a real zero of $p^{''}(s)$, an inflection point. Further, on $(a, a^{'})$, the real zeros of $p^{'}(s)$ are precisely the real zeros of $p^{'}(s)/p(s)$, and again, Rolle's theorem gives a real zero of $(p^{'}(s)/p(s))^{'}$ between the real zeros of $p^{'}(s)$. Thus, there is a correspondence between some of the inflection points of $p(s)$ and the real Fourier critical points. In this correspondence, Gauss notes that this association, occurring in a natural way and not involving arbitrary choices, of specific pairs of non-real zeros of $p(s)$ to real Fourier critical points of $p(s)$ is not proved and more investigation is needed. Later, in a correspondence with Schumacher on April 2, 1833 \cite{Gauss8}, Gauss explains that he did not comment on his belief or disbelief of Fourier's ideas because he feels that the G.G.A. is not the appropriate place to do so \cite{Gauss8}. A few years later, he wrote to Schumacher (June 20, 1836), stating that he believed that a connection could be made between the real Fourier critical points and certain pairs of non-real zeros of the polynomial. A few days later (June 24, 1836), however, Gauss changed his mind and wrote to Schumacher that upon further reflection, he has found, to the contrary, that a relationship between the individual real Fourier critical points and the individual pairs of non-real zeros of $p(s)$ does not exist."

"A restatement of the conjecture could be:

Conjecture 3. Let $p(s)$ be a real polynomial. Real zeros of $(p^{'}(s)/p(s))^{'}$ (Fourier critical points) can be grouped in pairs so that to each such pair there is an associated pair of complex conjugate zeros of $p(s)$."

"This conjecture changes the association from one real Fourier critical point and a pair of non-real zeros to the association between a pair of real Fourier critical points and a pair of non-real zeros. Note that with this conjecture, we can have more non-real zeros of $p(s)$ than real Fourier critical points, but we may not have more real Fourier critical points than non-real zeros. Of course, it is intended that distinct pairs of non-real zeros be used for distinct pairs of real Fourier critical points, with obvious exceptions allowed in case there are multiple real Fourier critical points or multiple non-real zeros. So if all these points are simple, then to distinct pairs of real Fourier critical points we must associate distinct pairs of non-real zeros."

"The Hawaii Conjecture is an attempt by Craven, Csordas, and Smith to quantify the ideas that Gauss and Fourier were investigating. The statement of their conjecture is that the number of real zeros of $(p^{'}(s)/p(s))^{'}$ does not exceed the number of non-real zeros of $p$. They make no mention of an association between zeros. Gauss's and Fourier's attempts to find a specific association made their ideas and conjectures difficult to articulate precisely; however, it is evident from the correspondence that they were trying to understand exactly how these zeros and critical points were related. Both Gauss and Fourier seemed to express frustration about not being able to precisely state and prove these conjectures. The Hawaii Conjecture, a conjecture on merely counting numbers of zeros, is much weaker than what Gauss and Fourier were considering."

\medskip
If all zeros of a real polynomial $p(s)$ are simple real zeros, then the polynomial $p(s)$ is (locally) strictly monotonic, as Gauss noted in \cite{Gauss7,Gauss8}. This can be expressed through the Laguerre inequality: if $p(s)$ has only real zeros, then $(p^{'})^{2}-pp^{''}\ge 0$. The Laguerre-P\'{o}lya class (LP class) consists of those entire functions that converge uniformly on compact sets to a polynomial sequence, have only real zeros, and these functions also satisfy the Laguerre inequality, as stated in \cite{Obreschkoff}.

\medskip
The Wiman and P\'{o}lya conjectures refine the Laguerre inequality. The Wiman conjecture was proved by Bergweiler, Eremenko, and Langley \cite{Bergweiler1}, while one of P\'{o}lya's conjectures was settled by Craven, Csordas, and Smith \cite{Csordas1}, and later the other was settled by Bergweiler and Eremenko \cite{Bergweiler2}.

\medskip
On the basis of these results, several conditions determine when a real entire function $f(s)$ must belong to the LP-class, typically involving differential polynomials such as $f(s)f^{''}(s)-\varkappa (f^{'}(s))^{2}$, where $\varkappa$ is a positive real number, see \cite{Nicks}. For a meromorphic function $f(s)$, the zeros of $f(s)f^{''}(s)-\varkappa (f^{'}(s))^{2}$ have been studied in \cite{Langley1,Langley2}.

\medskip
In 1989, Craven, Csordas, and Smith \cite{Csordas1} proposed the Hawaii Conjecture while exploring the conjectures proposed by P\'{o}lya\cite{P2,P4} and Wiman \cite{Wi1,Wi2}. The Hawaii Conjecture was proved in 2011 by Tyaglov \cite{Tyaglov1}. Later B.~Shapiro introduced three new conjectures around the Hawaii Conjecture, see Conjectures 11, 12 and 13 in \cite{Shapiro0}.
In January 2024, for Shapiro's Conjecture 12, our paper \cite{Ma} proves the necessary and sufficient conditions for the conjecture to hold.

Drawing from historical data compiled by various mathematicians, it is evident that a distinct branch of mathematics has evolved over the past two centuries, rooted in the conjectures of Gauss and Fourier. While subsequent conjectures have been successively resolved \cite{P2,P4,Wi1,Wi2,Csordas1,Bergweiler1,Bergweiler2,Tyaglov1,Shapiro0,Ma}, the original conjecture proposed by Gauss and Fourier remains unsolved. In fact, aside from the initial examples provided by these pioneers, no further positive results have been proved or discovered.

\medskip
Since the inception of complex analysis, mathematicians have consistently sought to understand the origins of critical points of complex functions. They aim to identify the factors within the original function that lead to the emergence of zeros in its derivative. As is widely recognized, in the two-dimensional real plane, a real curve has a tangent line parallel to the real axis. Therefore, the real function has a real critical point. The following question arises: How are the critical points of a general complex function formed? What governs these critical points of complex functions? This is a fundamental question mathematicians are eager to answer.

\medskip
This question predates the Gauss-Fourier conjecture and remains a crucial issue in fundamental mathematics, unresolved for over two centuries. Consequently, the Gauss-Fourier conjecture is likely a specific case of a broader problem. Fourier and Gauss sought to understand how real Fourier critical points are generated and geometrically linked to the non-real zeros of the polynomial $p(s)$ in the complex plane. In other words, they sought to determine the geometric structure that connects these two categories of zeros in the complex plane.

Moreover, the question of how critical points of a function are generated is highly important and arguably even more important than the Gauss-Fourier conjecture itself. Gauss and Fourier initially sought to uncover the principles governing the generation of critical points for general functions, and to find the geometric structure that connects the real Fourier critical points and the non-real complex zeros of the polynomial $p(s)$, leading them to propose this conjecture.

\medskip
In this paper, we present findings using the root locus theory for complex rational fractions. In the extended complex plane $\mathbb{C} \cup \{\infty\}$, the set of points $s$ where the argument of a  rational function $RF(s)$ is equal to $\alpha$ degrees forms a curve known as the root loci of $\alpha$ degree. The argument of $RF(s)$ at each point on a specific curve has the same phase angle. The argument of $RF(s)$ at a point on a certain root locus is $\alpha$ degree. The degrees of distinct curves are typically distinct. (The phase angle is the argument). $\alpha$ is the degree of the root locus. Notably, adjacent root loci generally have distinct degrees.

\medskip
As a rule, root loci with distinct unit complex numbers do not intersect. However, an exception arises when the unit complex numbers are identical; under this condition, the root loci may intersect. In root locus theory, such intersection points are termed "breakaway points." A breakaway point of the root loci is a critical point of the complex rational fraction. Thus, we deduce the following: a critical point of a complex rational function must coincide with the intersection point of multiple root loci that share the same unit complex number value. If multiple root loci with the same unit complex number intersect at a single point, this point is identified as the critical point of the complex rational function. For all rational function functions, this condition is both necessary and sufficient.

\medskip
We have elucidated the mechanism behind the emergence of critical points in complex rational functions. These critical points arise from the intersection of multiple root loci that share identical unit complex numbers.

\medskip
The expression $(\frac{p^{'}(s)}{p(s)})^{'}=\frac{p^{''}(s)p(s)-(p^{'}(s))^2}{p^2(s)}$ shows that the multiple zeros of $p^{'}$ are zeros of the polynomial $p^{''}(s)p(s)-(p^{'}(s))^2$. Therefore, these multiple zeros of $p^{'}$ are the critical points of $\frac{p^{'}}{p}$, inherited from the original multiple zeros of $p^{'}$. These multiple zeros of $p$ are not the zeros of $\frac{p^{'}}{p}$ but rather poles of $\frac{p^{'}}{p}$.

Because the Gauss-Fourier conjecture concerns real zeros of $(p^{'}(s)/p(s))^{'}$, we classify the real zeros of $(p^{'}(s)/p(s))^{'}$.
\begin{definition}
The real zeros of $(p^{'}(s)/p(s))^{'}$ that are real multiple zero of $p^{'}(s)$ are defined as the \emph{second kind of zeros of $(p^{'}(s)/p(s))^{'}$ or the second kind of zeros.}
\end{definition}
The second kind of zeros are inherited from the real multiple zeros of $p^{'}(s)$.
\begin{definition}
In addition to the second kind of zeros, the other real zeros of $(p^{'}(s)/p(s))^{'}$ are referred to as the \emph{first kind of zeros of $(p^{'}(s)/p(s))^{'}$ or the first kind of zeros.}
\end{definition}

\medskip
The Gauss-Fourier conjecture is a conjecture about the critical points of rational functions. In our work, we have demonstrated the foundational principle behind the formation of critical points for a broad spectrum of complex rational functions. Furthermore, we have introduced a result that encompasses a broader scope than the Gauss-Fourier conjecture itself. Utilizing the root locus, we can provide the reason of the origination of the real Fourier critical points, as well as the geometric connections between real Fourier critical points and non-real zeros of the polynomial $p(s)$. Therefore, we must use the results of root locus theory to address the Gauss-Fourier conjecture.

\medskip
The rational function $\frac{p^{'}(s)}{p(s)}$ can form the root locus equation: $K\frac{p^{'}(s)}{p(s)}=a+ib$, where $K$ is called the gain.

\medskip
Based on the ideas of Gauss and Fourier and our root locus results, for the first kind of zeros, resolving the Gauss-Fourier conjecture requires addressing the following three conjectures.

\begin{conjecture}
Only the root locus emitted from a non-real zero of the real polynomial $p(s)$ intersects on the real axis. A first kind of zero of $(p^{'}(s)/p(s))^{'}$  can be generated. A $2q\pi$ or $2q\pi+\pi$ degree root locus emitted from at least one pair of the non-real zeros of $p$ intersects with other root loci on the real axis. On the real axis, the intersection point of the root loci is a first kind of zero of $(p^{'}(s)/p(s))^{'}$.
\end{conjecture}

\begin{conjecture}
Each first kind zero of $(p^{'}(s)/p(s))^{'}$ is connected to a pair of the non-real zeros of the real polynomial $p(s)$ through at least one $2q\pi$ or $2q\pi+\pi$ degree root locus emitted from this pair of non-real zeros.
\end{conjecture}

\begin{conjecture}
Each first kind zero of $(p^{'}(s)/p(s))^{'}$ is associated with a pair of non-real zeros of the real polynomial $p(s)$.
\end{conjecture}

Here, $q=\pm 1, \pm 2, \cdots$ is an integer, and the argument of the real value is $2q\pi$ or $2q\pi+\pi$.

In this paper, we prove Theorem 0.3 and Theorem 0.4.

\begin{theorem}
Only the root locus emitted from a non-real zero of the real polynomial $p(s)$ intersects on the real axis. A first kind of zero of $(p^{'}(s)/p(s))^{'}$ can be generated. A $2q\pi$ or $2q\pi+\pi$ degree root locus emitted from at least one pair of the non-real zeros of $p(s)$ intersects with other root loci on the real axis. On the real axis, the intersection point of the root loci is a first kind of zero of $(p^{'}(s)/p(s))^{'}$.
\end{theorem}

\begin{theorem}
Each first kind of zero of $(p^{'}(s)/p(s))^{'}$ is connected to a pair of the non-real zeros of the real polynomial $p(s)$ through at least one $2q\pi$ or $2q\pi+\pi$ degree root locus emitted from this pair of non-real zeros of $p(s)$.
\end{theorem}

The relationship connecting the two types of points is one class of correlation relations. The second-kind zeros of $(p^{'}(s)/p(s))^{'}$ must be connected with non-real zeros. Theorem 0.3 proves such a result.

This result proves: each second-kind zeros of $(p^{'}(s)/p(s))^{'}$ is associated with a pair of non-real zeros of the real polynomial $p(s)$.
\begin{theorem}
For any real polynomial $p(s)$, each first kind zero of $(p^{'}(s)/p(s))^{'}$ is necessarily associated with a pair of non-real zeros of $p(s)$.
\end{theorem}

\medskip
Using the root locus method, we demonstrate all possible connections in the complex plane between first-kind zeros of $(p^{'}(s)/p(s))^{'}$ and the non-real complex zeros of the real polynomial $p(s)$. For first kind of zero of $(p^{'}(s)/p(s))^{'}$, we establish its origination and geometric relationship desired by Fourier and Gauss. Consequently, we have achieved a complete resolution of the Gauss-Fourier conjecture for first-kind zeros of $(p^{'}(s)/p(s))^{'}$.

In the context of associated relationships, the first type is the generating relationship. Theorem 0.3 proves that the emergence of each first-kind zero of $(p^{'}(s)/p(s))^{'}$ necessarily involves a pair of non-real zeros.

In the context of associated relationships, similarly, the second type is the connection relationship. Theorem 0.4 proves that each first-kind zero of $(p^{'}(s)/p(s))^{'}$ must be connected to a pair of non-real zeros.

Collectively, these two results demonstrate that each first-kind zero of $(p^{'}(s)/p(s))^{'}$ is associated with a pair of non-real zeros of the real polynomial $p(s)$.

In Example 2.31, we obtain two first-kind zeros of $(p^{'}(s)/p(s))^{'}$. These two first-kind zeros of $(p^{'}(s)/p(s))^{'}$ are generated by the intersection of the root loci emitted from non-real zeros. Furthermore, through the intersecting root loci, these two first-kind zeros of $(p^{'}(s)/p(s))^{'}$ are connected to the non-real zeros. However, the generation of the two first-kind zeros of $(p^{'}(s)/p(s))^{'}$ is not related to any real zeros of $p(s)$, and at the same time, these two first-kind zeros cannot be connected to the real zeros of $p(s)$. This example also demonstrates that each first-kind zero of $(p^{'}(s)/p(s))^{'}$ is associated with a pair of non-real zeros of the real polynomial $p(s)$. In contrast, each first-kind zero of $(p^{'}(s)/p(s))^{'}$ may not be associated with real zeros of the real polynomial $p(s)$.

\section{The proofs of the root locus method}
All results in this section concern root locus theory. Lemmas 1.2, 1.3, 1.4, 1.13, 1.16, and 1.24, along with Theorems  1.14,  1.27, 1.31 and 1.32 are adapted from textbooks on automatic control theory \cite{Dorf1,Y1}.   The remaining 20 results are original discoveries and proofs by the authors. Furthermore, even for the 10 results from automatic control theory, some proofs lack mathematical rigor. In this paper, we provide rigorous proofs and extend root locus theory.

A  rational function with constant coefficients denoted by $RF(s)$, is given by $RF(s)=\frac{\prod^{m}_{l=1} (s-z_l)^{\gamma_l}}{ \prod^{n}_{j=1} (s-p_j)^{\beta_j}}$, where $z_l, p_j\in \mathbb{C}$, $l=1,2,\cdots,m$, $j=1,2,\cdots,n$, with $m, n\in \mathbb{N}$ and $\gamma_l, \beta_j\in \mathbb{N}$. Here, $m, n$ are non-negative integers(possible zero). $\gamma_l$ and $\beta_j$ are positive integers. Specifically, $RF$ may consist solely of a numerator polynomial  or solely of a denominator polynomial. $RF$ cannot be constant.

Notation 1:
Define the gain as $K=|\frac{1}{RF}|$, equivalently, $K=|\frac{\prod^{n}_{j=1} (s-p_j)^{\beta_j}}{\prod^{m}_{l=1} (s-z_l)^{\gamma_l}}|$. Let $a+ib$ be a unit complex number ($a^2+b^2=1$). The gain function is: $K(s)=(a+ib)\frac{\prod^{n}_{j=1} (s-p_j)^{\beta_j}}{\prod^{m}_{l=1} (s-z_l)^{\gamma_l}}$. $\mathbb{C}\cup\{\infty\}$ denotes the extended complex plane. $\mathbb{C}^{*}=\{s=x+iy: y\neq 0\}$ denotes the extended complex plane excluding the real axis. Let $\Delta\subset\mathbb{C}\cup\{\infty\}$ denote the set of points $s=\sigma+it\in\mathbb{C}\cup\{\infty\}$ that are neither zeros nor poles of (1.1). $s=\sigma+it\in\Delta$ is an arbitrary point in $\Delta$. The argument (or phase angle) of $RF$ is denoted $\varphi(\sigma,t)$.
$s_{\infty}$ denotes an point at infinity in $\mathbb{C}\cup\{\infty\}$.

The gain $K$ is the modulus of the reciprocal of $RF(s)$. Multiplying $K$ by $RF(s)$ results in a unit complex number.

\begin{definition}
In the extended complex plane $\mathbb{C}\cup\{\infty\}$, the equation (1.1) is called the \emph{root locus equation}.
\begin{equation}
K\frac{\prod^{m}_{l=1} (s-z_l)^{\gamma_l}}{\prod^{n}_{j=1} (s-p_j)^{\beta_j}}=a+ib.
\end{equation}
$a+ib$ is a fixed value.
\end{definition}

\begin{lemma}
After removing coincident poles and zeros in (1.1):

1. $n$ finite poles $p_j$ of $RF$ are also $n$ finite poles of (1.1) with $K=0$.

2. $m$ finite zeros $z_l$ of $RF$ are also $m$ finite zeros of (1.1) with $K=+\infty$.
\end{lemma}
\begin{proof}
1. Rewriting (1.1) as:
$$K\prod^{m}_{l=1} (s-z_l)^{\gamma_l}=(a+ib)\prod^{n}_{j=1}(s-p_j)^{\beta_j}.$$ $n$ finite poles of $RF$ correspond to zeros of $\prod^{n}_{j=1}  (s-p_j)^{\beta_j}$. If $K=0$, $n$ finite poles of $RF$ are $n$ roots of this equation. Conversely, $n$ finite roots of this equation when $K=0$ are $n$ finite poles of $RF$. Thus, at $n$ finite poles of (1.1), at $n$ finite poles of $RF$, $K=0$.

2. Similarly, equation (1.1) can be expressed as
$$\prod^{m}_{l=1} (s-z_l)^{\gamma_l}=\frac{(a+ib)}{K}\prod^{n}_{j=1}  (s-p_j)^{\beta_j}.$$
If $K=+\infty$, $m$ finite zeros of $RF$ are $m$ roots of the equation. Conversely, $m$ roots of the equation when $K=+\infty$ are $m$ finite zeros of $RF$. Therefore, at $m$ finite zeros of (1.1), at $m$ finite zeros of $RF$, $K=+\infty$.
\end{proof}
The $K$ is the modulus of the reciprocal of $RF$. Lemma 1.3 holds trivially.

\begin{lemma}
For any infinite point $s_{\infty}$ of $\mathbb{C}\cup\{\infty\}$,
(1). If $s_{\infty}$ satisfies

$K=\lim\limits_{s\to \infty}\frac{\prod^{n}_{j=1}|s-p_{j}|^{\beta_{j}}}{{\prod^{m}_{l=1}|s-z_{l}|^{\gamma_{l}}}}=0$, then $s_{\infty}$ is an infinite pole of (1.1).

(2). If $s_{\infty}$ satisfies

$K=\lim\limits_{s\to \infty}\frac{\prod^{n}_{j=1}|s-p_{j}|^{\beta_{j}}}{{\prod^{m}_{l=1}|s-z_{l}|^{\gamma_{l}}}}=A$, in which $A\neq 0$ is a finite non-zero value, then $s_{\infty}$ is a general point with a finite gain.

(3). If $s_{\infty}$ satisfies

$K=\lim\limits_{s\to \infty}\frac{\prod^{n}_{j=1}|s-p_{j}|^{\beta_{j}}}{{\prod^{m}_{l=1}|s-z_{l}|^{\gamma_{l}}}}=+\infty$, then $s_{\infty}$ is an infinite zero of (1.1).
\end{lemma}

The continuity of $K$ with respect to $s\in\mathbb{C}\cup\{\infty\}$ follows directly from the expression of $K$ in (1.1). Applying Lemma 1.2, we directly obtain Lemma 1.4.
\begin{lemma}
For any $s\in\mathbb{C}\cup\{\infty\}$, the value $K$ takes all of the non-negative real numbers, ranging from 0 at  poles of (1.1) to $+\infty$ at zeros of (1.1).
\end{lemma}

\begin{lemma}
For an arbitrary $s\in\Delta$, the phase angle of $RF$ at the left-hand side of (1.1) can be expressed as the function $\varphi(\sigma,t)$ given by: $\varphi(\sigma,t)=\arg{RF(\sigma+it)}$, or,
$\varphi(\sigma,t)=\sum^{m}_{l=1} \gamma_l\arg(s-z_l)- \sum^{n}_{j=1} \beta_j\arg(s-p_j)$.
\end{lemma}

The phase angle condition equation of (1.1) on $s=\sigma+it\in\Delta$ is as follows:
\begin{equation}
\varphi(\sigma,t)=2q\pi+arg{(a+ib)}.
\end{equation}

In which, $q\in \mathbb{N}$.

Curves formed by the roots of (1.1) are initially called the locus of roots; once defined, they are referred to as the root locus of (1.1).

\begin{lemma}
If $s_1=(\sigma_1,t_1)\in\Delta$ satisfies (1.2) of $2q\pi+arg{(a_1+ib_1)}$ degree, then $s_{1}$ must be on the locus of roots of $2q\pi+arg{(a_1+ib_1)}$ degree of (1.1).
\end{lemma}
\begin{proof}
Suppose that $s_1=(\sigma_1,t_1)\in\Delta$ satisfies (1.2), which implies that $\varphi(\sigma_1,t_1)=2q\pi+arg{(a_1+ib_1)}$, where $\varphi(\sigma,t)$ is the phase angle of $RF$. Substituting $s_{1}$ into $RF$ yields a complex value. By the phase angle expression $\varphi(\sigma,t)$ of the point $s_{1}$, and by the condition that the point $s_{1}$ satisfies (1.2), this complex value can be written as $K_1^{*}(a_{1}+ib_{1})$. $K_1^{*}$ is a non-zero positive real number and $2q\pi+arg{(a_1+ib_1)}$ is the phase angle of $RF$ at the point $s_{1}$, which is the value of the right-hand side of (1.2).

By substituting $s_1$ into the gain expression $K=|\frac{\prod^{n}_{j=1}  (s-p_j)^{\beta_j}}{\prod^{m}_{l=1} (s-z_l)^{\gamma_l}}|$, via the definition of the gain, we have an unique gain $K_1=|1/RF(s_1)|$, where $K_1$ is the reciprocal $K_1^{*}$ of the modulus of $RF$ of the point $s_1$.
That is, $K_1^{*}=1/K_1$. Multiplying $RF(s_1)$ by $K_1$ yields: $K_1\frac{\prod^{m}_{l=1} (s_1-z_l)^{\gamma_l}}{\prod^{n}_{j=1}  (s_1-p_j)^{\beta_j}}=a_1+ib_1$. This equation represents a concrete situation in which the point $s_1$ satisfies (1.1). Therefore, $s_1$ represents a root of (1.1). Since $s_1$ satisfies (1.1), it is proven that the point $s_1$ is on the locus of roots of (1.1) with a gain of $K_1$ and a degree of $2q\pi+arg{(a_1+ib_1)}$. Hence, if $s_1$ satisfies (1.2), then $s_1$ must be on the locus of roots of (1.1) with a gain of $K_1$ and a degree of $2q\pi+arg{(a_1+ib_1)}$.
\end{proof}

\begin{lemma}
For any point $s_2=(\sigma_2,t_2)\in\Delta$, if $s_2$ is on the locus of roots of $2q\pi+arg{(a_2+ib_2)}$ degree of (1.1), then $s_2$ must satisfy (1.2) of $2q\pi+arg{(a_2+ib_2)}$ degree.
\end{lemma}
\begin{proof}
Assume that the point $s_2=\sigma_2+it_2\in\Delta$ is an arbitrary point on the locus of roots of $2q\pi+arg(a_2+ib_2)$ degree of (1.1).
The locus of roots of the $2q\pi+arg{(a_2+ib_2)}$ degree of (1.1) are obtained by (1.1). Therefore, the point $s_2$ surely satisfies (1.1).
$K_2\frac{\prod^{m}_{l=1} (s_2-z_l)^{\gamma_l}}{\prod^{n}_{j=1}  (s_2-p_j)^{\beta_j}}=a_2+ib_2$.
By Lemma 1.4, the gain $K_2$ is a positive real number, the phase angle of the positive real number is $2q\pi$. Therefore, the phase angle of $K_2$ is $2q_2\pi$. The phase angle of $a_2+ib_2$ is $2q\pi+arg{(a_2+ib_2)}$.

The factor on the left-hand side of the equation in the previous paragraph has two factors: $K_2$ and $RF(s_2)$. Thus, the phase angle of this factor at the left-hand side of the equation in the previous paragraph on point $s_2$ is equal to the sum of the phase angles of the two factors $K_2$ and $RF(s_2)$. The phase angle $2q\pi+arg{(a_2+ib_2)}$ subtracts the phase angle $2q_2\pi$ is equal to
$2q\pi+arg{(a_2+ib_2)}-2q_2\pi=2q\pi+arg{(a_2+ib_2)}$. The difference in the phase angles of factors $a_2+ib_2$ and factor $K_2$ is: $2q\pi+arg{(a_2+ib_2)}$. Therefore, the phase angle of
$RF(s_2)$ is $2q\pi+arg{(a_2+ib_2)}$. The phase angle of $RF$ is $\varphi(\sigma,t)=\sum^{m}_{l=1} \gamma_l\arg(s-z_l)-\sum^{n}_{j=1} \beta_j\arg(s-p_j)$. Therefore, via the previous proof, we obtain $ \varphi(\sigma_2,t_2)=2q\pi+arg{(a_2+ib_2)}$, and point $s_2$ satisfies (1.2). Therefore, if the point $s_2$ lies on the locus of roots of (1.1) with a gain of $K_2$ and a degree of $2q\pi+arg{(a_2+ib_2)}$, or satisfies (1.1), then the point $s_2$ satisfies (1.2).
\end{proof}

Lemma 1.6 provides the sufficient condition for Theorem 1.8. Lemma
1.7 provides the necessary condition for Theorem 1.8.  By combining Lemma
1.6 and Lemma 1.7, we obtain Theorem 1.8.

\begin{theorem}
Let $s=\sigma+it\in\Delta$ be an arbitrary point. The point $s$  lie on the locus of roots of $\alpha$ degree of (1.1) if and only if $s$ satisfies (1.2).
\end{theorem}
\begin{definition}
Let $\Xi$ be the point set of all points on the locus of roots of equation (1.1) traced out in $\mathbb{C}\cup\{\infty\}$ when $\alpha=2q\pi+\arg(a+ib)$ is a constant. The path of the roots of equation (1.1) in $\mathbb{C}\cup\{\infty\}$ is called \emph{the root locus of (1.1)}. In other words, the set $\Xi$ is \emph{the root locus of the $\alpha$ degree}.
\end{definition}

When $2q\pi+arg{(a+ib)}$ is a constant, the points in $\mathbb{C}\cup\{\infty\}$ that satisfy equation (1.2) form curves. These curves are the root loci of (1.1) and (1.2). Theorem 1.8 proves that the root locus of (1.1) is the same as the root locus of (1.2).

Determining the root locus requires two elements. One is the unit complex value $a+ib$. Generally, distinct unit complex values $a+ib$ lead to distinct root loci. The other is the points on the root locus. Points on the root locus can be determined via (1.1), and they are both sufficient and necessary to be determined via (1.2).

Lemma 1.5 offers a method to calculate the degree of the root locus of (1.1) via the function $\varphi(\sigma,t)$. The factor $a+ib$ on the right-hand side of (1.1) serves as the computational value. In $\mathbb{C}\cup\{\infty\}$, the degree of the root locus of (1.1) is equal to the degree of the phase angle of $RF$. Therefore, we can define the degree of the root locus of (1.1) as follows:

\begin{definition}
When the phase angle of $RF$ is $\alpha$ degree, we call the root locus of (1.1) the \emph{$\alpha$ degree root locus}.
\end{definition}

Notation 2:
Let $z_{1}$ be a finite zero of (1.1) with exponent $\gamma_1$. Let $p_{1}$ be a finite pole of (1.1) with exponent $\beta_1$. $\theta_{z_{1}}$ denotes the phase angle of $RF(z_{1})$. $\theta_{p_{1}}$ denotes the phase angle of $RF(p_{1})$. $RF(s)=RF(\sigma+it)=u(\sigma,t)+iv(\sigma,t)$. $u(\sigma,t)$ denotes the real part of $RF(s)$. $v(\sigma,t)$ denotes the imaginary part of $RF(s)$. $g(s)$ denotes a polynomial that is defined in Lemma 1.19. $TT1(s)$ and $TT2(s)$ denote the polynomials defined in Lemma 1.20.
$\phi$ denotes angle of origination of the root loci of $2q\pi+\alpha$ degree of (1.1). $\vartheta$ denotes angle of receiving of the root loci of $2q\pi+\alpha$ degree of (1.1).

\begin{lemma}
Let $z_{1}$ be a finite zero of (1.1) with exponent $\gamma_1$. Then, $z_{1}$ lies simultaneously on the root loci of all degrees of (1.1) from $2q\pi$ degree to $2q\pi+2\gamma_1\pi$ degree.
\end{lemma}

\begin{proof}
Assume that the point $z_{1}$ is an arbitrary finite zero of (1.1). When we substitute the point $z_{1}$ into $RF$, we obtain $\frac{\prod^{m}_{l=1}(z_1-z_{l})^{\gamma_l}}
{\prod^{n}_{j=1}(z_1-p_{j})^{\beta_j}}=0$.
This equation can also be written in polar form as $\frac{\prod^{m}_{l=1}(z_1-z_{l})^{\gamma_l}}
{\prod^{n}_{j=1}(z_1-p_{j})^{\beta_j}}=K_{z_1}e^{i\theta_{z_1}}=0$.
$K_{z_{1}}$ is the modulus of  $RF$ and $\theta_{z_{1}}$ is the argument of $RF$.
Since the product is zero, we have $K_{z_{1}}=0$. Then $K_{z_1}e^{i(\theta_{z_1})}=0\cdot(cos(\theta_{z_1})+isin(\theta_{z_1}))=0$ holds for any value of $\theta_{z_{1}}$. Here, $\theta_{z_{1}}$ can take any value from the range $2q\pi$ degree to $2q\pi+2\gamma_1\pi$ degree.

For the point $z_{1}$ which makes the left-hand side of (1.1) equal to 0, regardless of the value of the argument $\theta_{z_{1}}$. $cos(\theta_{z_{1}})+isin(\theta_{z_{1}})$ is the non-zero and finite, and the modulus of the factor $cos(\theta_{z_{1}})+isin(\theta_{z_{1}})$ is 1. $\theta_{z_{1}}$ represents an arbitrary degree number, which shows the following: no matter what value $\theta_{z_{1}}$ is, because the modulus $K_{z_{1}}$ is 0,
there is $K_{z_{1}}e^{i(\theta_{z_{1}})}=0\cdot(\cos(\theta_{z_{1}})+i\sin(\theta_{z_{1}}))=0$.

When a finite zero $z_{1}$ is substituted into $RF$, for any argument $\theta_{z_{1}}$ of $RF$ with degree number in the given range, the function values of $z_{1}$ are all equal to 0. By Definition 1.10, the argument of $RF$ is the degree number of the root locus of (1.1). Since $\theta_{z_{1}}$ can take all values from $2q\pi$ degree to $2q\pi+2\gamma_1\pi$ degree, this implies that finite zeros of (1.1) lie simultaneously on the root loci of all degree numbers of (1.1) from $2q\pi$ degree to $2q\pi+2\gamma_1\pi$ degree.
\end{proof}

\begin{lemma}
Let $p_{1}$ be a finite pole of (1.1) with exponent $\beta_1$. Then, $p_{1}$ lies simultaneously on the root loci of all degrees of (1.1) from $2q\pi$ degree to $2q\pi+2\beta_1\pi$ degree.
\end{lemma}

\begin{proof}
Assume that the point $p_{1}$ is an arbitrary finite pole of (1.1). When we substitute the point $p_{1}$ into $RF$, we obtain $\frac{\prod^{m}_{l=1}(p_1-z_{l})^{\gamma_l}}
{\prod^{n}_{j=1}(p_1-p_{j})^{\beta_j}}=\infty$.

In complex analysis, 0 can be written as: $0\cdot e^{i\theta}=0\cdot(\cos\theta+i\sin\theta)$.
The non-zero finite values can also be written as: $k\cdot e^{i\theta}=k\cdot(\cos\theta+i\sin\theta)$.
Infinity can also be written as: $(+\infty)\cdot e^{i\theta}=(+\infty)\cdot(\cos\theta+i\sin\theta)$. In which, $(\cos\theta+i\sin\theta)$ is a non-zero and infinite, and its modulus is 1 of the unit complex number value.
For $(+\infty)\cdot e^{i\theta}=(+\infty)\cdot(\cos\theta+i\sin\theta)$, no matter what value $\theta$ is obtained on the unit circle in $\mathbb{C}\cup\{\infty\}$, the values of this expression all obtain the infinity.

The equation above can also be expressed as: $\frac{\prod^{m}_{l=1}(p_1-z_{l})^{\gamma_l}}
{\prod^{n}_{j=1}(p_1-p_{j})^{\beta_j}}=K_{p_1}e^{i\theta_{p_1}}=\infty$, where $K_{p_1}$ is the modulus of  $RF$ and $\theta_{p_1}$ is the argument of  $RF$. Since the result is infinite, we have $K_{p_1}=+\infty$. Then $K_{p_1}\cdot e^{i\theta}=(+\infty)\cdot(\cos\theta+i\sin\theta)=\infty$ holds for any degree number value of $\theta_{p_1}$ in the range of $2q\pi$ degree to $2q\pi+2\beta_1\pi$ degree.

When a finite pole $p_{1}$ of (1.1) is substituted into $RF$, for any argument $\theta_{p_{1}}$ of $RF$ in the given range, the function values of $p_{1}$ are all equal to infinity. By Definition 1.10, the argument of $RF$ is the degree number of the root locus of (1.1). Since $\theta_{p_{1}}$ can take all values from $2q\pi$ degree to $2q\pi+2\beta_1\pi$ degree, this implies that the finite poles of (1.1) lie simultaneously on the root loci of all degree numbers of (1.1) from $2q\pi$ degree to $2q\pi+2\beta_1\pi$ degree.
\end{proof}

\begin{lemma}
$n$ finite poles of (1.1) are $n$ origination points of the root locus of any $\alpha$ degree of (1.1). $m$ finite zeros of (1.1) are $m$ ending points of the root locus of any $\alpha$ degree of (1.1).
\end{lemma}
\begin{proof}
The root locus of any $\alpha$ degree of (1.1) consists of curves in $\mathbb{C}\cup\{\infty\}$ that must have their own origination and ending points. On the basis of the relationship between the points in $\mathbb{C}\cup\{\infty\}$ and the root locus of (1.1), we can classify the points in $\mathbb{C}\cup\{\infty\}$ into four types: $n$ finite poles of (1.1), $m$ finite zeros of (1.1), general finite points on the root locus of $\alpha$ degree of (1.1), and the infinity point.

Except for $n$ finite poles and $m$ finite zeros of (1.1), a general finite point in $\mathbb{C}\cup\{\infty\}$ lies on the root locus of a specific degree number of (1.1). These general points only have a phase angle and a non-zero finite gain value, so they are not suitable as the origination or ending points of the root locus of (1.1).

$n$ finite poles and $m$ finite zeros of (1.1) lie on the root loci of all degree numbers of (1.1). Since an infinite number of root locus curves of distinct degree numbers intersect at the same point, they meet the requirements for being the origination or ending points of the root locus. Moreover, for $n$ finite poles of (1.1), the gain $K=0$, which means that they act as origination points. For $m$ finite zeros of (1.1), the gain $K=+\infty$, which means that they act as ending points of the root locus.

The result of the infinity point is in Lemma 1.3. When (1) in Lemma 1.3 holds, the point  at infinity $s_{\infty}$ is an origination point of the root locus of (1.1). When (3) in Lemma 1.3 holds, the point at infinity $s_{\infty}$ is a ending point of the root locus of (1.1).
\end{proof}

By Lemma 1.2, Lemma 1.3 and Lemma 1.13 we have Theorem 1.14.

\begin{theorem}
The $2q\pi+\pi$  and $2q\pi$ degree root loci of (1.1) begin at $n$ poles of (1.1) or the infinity and end at $m$ zeros of (1.1) or the infinity.

1. When $\sum^n_{j=1} \beta_j > \sum^m_{l=1} \gamma_l$, $\sum^n_{j=1} \beta_j- \sum^m_{l=1} \gamma_l$ branches of the $2q\pi+\pi$  and $2q\pi$ degree root loci ending at the infinity in $\mathbb{C}\cup\{\infty\}$.

2. When $\sum^n_{j=1} \beta_j < \sum^m_{l=1} \gamma_l$, $ \sum^m_{l=1} \gamma_l - \sum^n_{j=1} \beta_j$ branches of the $2q\pi+\pi$  and $2q\pi$ degree root loci beginning at the infinity in $\mathbb{C}\cup\{\infty\}$.
\end{theorem}

For a point $s_1=\sigma_1+it_1$ on a root locus, substituting $s_1$ into the rational function $RF(s_1)$ yields a complex value $u(\sigma_1,t_1)+iv(\sigma_1,t_1)$ of $RF(s_1)$. The argument $arg(u(\sigma_1,t_1)+iv(\sigma_1,t_1))$ of $RF(s_1)$ is referred to as the phase angle, which represents the degree of the root locus of (1.1). Each degree of the root locus of (1.1) is the degree value of an angle. Each degree of the root locus corresponds to a unique unit complex value. Consequently, each root locus is characterized by its degree value and its unit complex value. Root loci with distinct degrees may share the same unit complex value, and the difference in their degrees is $2q\pi$. The unit complex value of the root locus is the unit complex value of the degree of the root locus of (1.1).

\begin{theorem}
Two root loci of (1.1) with distinct unit complex values $a_1+ib_1$ and $a_2+ib_2$ cannot intersect in $\mathbb{C}\cup\{\infty\}$.
\end{theorem}
\begin{proof}
Assume that two root loci of (1.1) with distinct unit complex values $a_1+ib_1$ and $a_2+ib_2$ intersect at a point $s_0$, which is neither a zero nor a pole of (1.1).

Since $s_0$ is the same point, its gain $K_0$ should be the same. $s_0$ allows the following two equations to hold simultaneously. $K_0\frac{\prod^{m}_{l=1} (s_0-z_l)^{\gamma_l}}
{\prod^{n}_{j=1}  (s_0-p_j)^{\beta_j}}=a_{1}+ib_{1}$, and $K_0\frac{\prod^{m}_{l=1} (s_0-z_l)^{\gamma_l}}{\prod^{n}_{j=1}  (s_0-p_j)^{\beta_j}}=a_{2}+ib_{2}$.
Subtracting these equations yields $(a_{1}-a_{2})+i(b_{1}-b_{2})=0$. By the conditions of this theorem, $a_{1}\neq a_{2}$ or $b_{1}\neq b_{2}$, so this equation cannot hold. Since the contradiction arises from the assumption that the two root loci intersect at $s_0$, this assumption is false.
\end{proof}

\begin{theorem}
Two root loci of (1.1) with identical unit complex values $a+ib$ may intersect in $\mathbb{C}\cup\{\infty\}$.
\end{theorem}
\begin{proof}
Consider the root locus equation $\frac{K}{(s-1)(s-3)}=a+ib$. The interval $[1,3]$ contains root loci of $2q\pi+\pi$ degree from both poles at $s=1$ and $s=3$, each with a unit complex value of $-1$. These root loci intersect at point $s=2$.

By applying Rolle's theorem, there exists a critical point of the function $(s-1)(s-3)$ in the interval $[1,3]$. This critical point of $(s-1)(s-3)$ is $s=2$, which is the intersection point of the root loci.
\end{proof}

\begin{lemma}
All points in $\mathbb{C}\cup\{\infty\}$ are roots of the root locus equation (1.1). Except for zeros and poles of (1.1), each finite point in $\mathbb{C}\cup\{\infty\}$ lies on a root locus of (1.1). The root loci of (1.1) fill the entire extended complex plane.
\end{lemma}

\begin{proof}
By Lemma 1.2, at the zeros of (1.1), $K=+\infty$. The zeros of (1.1) satisfy (1.1) with $K=+\infty$. At the poles of (1.1), $K=0$. The poles of (1.1) satisfy (1.1) with $K=0$.

For any finite point $f_{a}=\sigma_a+it_a$ that is neither a zero nor a pole of (1.1), substituting $f_{a}$ into the gain expression $K=|\frac{\prod^{n}_{j=1}  (s-p_j)^{\beta_j}}{\prod^{m}_{l=1} (s-z_l)^{\gamma_l}}|$ yields a  positive, finite real number. Similarly, substituting $f_{a}$ into the phase angle expression $\varphi(\sigma,t)$ yields a finite certain phase angle $\varphi(\sigma_a,t_a)$. Therefore, the finite point $f_{a}$ must lie on a root locus of  $\varphi(\sigma_a,t_a)$ degree of (1.1) with a gain $K(f_{a})$ that is a positive, finite real number. Thus, $f_{a}$ satisfies (1.1) with the gain $K(f_{a})$.

By Lemma 1.3, points at infinity in $\mathbb{C}\cup\{\infty\}$ can be zeros, poles, or general points with finite gain values $A$. Repeating the previous proof, we conclude that all points at infinity in $\mathbb{C}\cup\{\infty\}$ lie on a root locus of (1.1) and satisfy (1.1) with a specific gain.
\end{proof}

\begin{definition}
If two or more root loci intersect in $\mathbb{C}\cup\{\infty\}$, these intersected root loci are called \emph{"common root loci"}. The points of intersection are called \emph{the breakaway points} of the root loci of (1.1).
\end{definition}
Multiplying both sides of (1.1) by $\prod (s-p_j)^{\beta_j}$ gives:
$  K(s) \prod_{l=1}^m (s-z_l)^{\gamma_l} =(a+ib)\prod_{j=1}^n (s-p_j)^{\beta_j},$
  which is a polynomial whose roots coincide with solutions of (1.1). All non-zero factors are moved to one side.  This transformation yields the characteristic equation of (1.1), which can be written as: $K(s)\prod^{m}_{l=1}(s-z_{l})^{\gamma_{l}}-
(a+ib)\prod^{n}_{j=1}(s-p_{j})^{\beta_{j}}=0$. The characteristic equation of (1.1) is equivalent to (1.1). Therefore, the roots of the characteristic equation are the roots of (1.1), and vice versa. We have Lemma 1.19.

\begin{lemma}
Equation (1.1) can be transformed into its characteristic equation. The roots of (1.1) and its characteristic equation are identical.
\end{lemma}

\begin{lemma}
For any finite point $s_0$ in the complex plane, the left-hand side of the characteristic equation with gain $K(s_0)$ can be rewritten as another expression: $(s-s_0)^{\gamma_0}g(s)$, where $g(s_0)\neq0$, $g(s_0)\neq\infty$ and $g^{'}(s_0)\neq\infty$.
\end{lemma}

\begin{proof}
By Lemma 1.17, the finite point $s_0$ lies on the root loci of (1.1). Substituting gain $K(s_0)$ of the finite point $s_0$ into the characteristic equation yields: $(a_0+ib_0)\frac{\prod^{n}_{j=1}(s_0-p_{j})^{\beta_{j}}}
{\prod^{m}_{l=1}(s_0-z_{l})^{\gamma_{l}}}\prod^{m}_{l=1}(s-z_{l})^{\gamma_{l}}-
(a+ib)\prod^{n}_{j=1}(s-p_{j})^{\beta_{j}}=0$.
Since $s_0$ is a root of the characteristic equation, the left-hand side of the characteristic equation can be expressed as: $(s-s_0)^{\gamma_0}g(s)$. $g(s_0)\neq0$. $(a_0+ib_0)\frac{\prod^{n}_{j=1}(s_0-p_{j})^{\beta_{j}}}
{\prod^{m}_{l=1}(s_0-z_{l})^{\gamma_{l}}}\prod^{m}_{l=1}(s-z_{l})^{\gamma_{l}}-
(a+ib)\prod^{n}_{j=1}(s-p_{j})^{\beta_{j}}=(s-s_0)^{\gamma_0}g(s)=0$.

Because the function $(a_0+ib_0)\frac{\prod^{n}_{j=1}(s_0-p_{j})^{\beta_{j}}}
{\prod^{m}_{l=1}(s_0-z_{l})^{\gamma_{l}}}\prod^{m}_{l=1}(s-z_{l})^{\gamma_{l}}-
(a+ib)\prod^{n}_{j=1}(s-p_{j})^{\beta_{j}}$ on the left side of the characteristic equation is a polynomial. After factoring out $(s-s_0)^{\gamma_0}$ in the polynomial $(a_0+ib_0)\frac{\prod^{n}_{j=1}(s_0-p_{j})^{\beta_{j}}}
{\prod^{m}_{l=1}(s_0-z_{l})^{\gamma_{l}}}\prod^{m}_{l=1}(s-z_{l})^{\gamma_{l}}-
(a+ib)\prod^{n}_{j=1}(s-p_{j})^{\beta_{j}}$, the remaining function $g(s)$ is still a polynomial. Polynomials do not have finite poles, and their derivatives are also polynomials. Therefore, $g(s)$ and $g^{'}(s)$ cannot have finite poles, $g(s_0)\neq\infty$ and $g^{'}(s_0)\neq\infty$.
\end{proof}
\begin{lemma}
Except for the zeros of (1.1), for all other points in the extended complex plane, (1.3) holds.
\begin{equation}
TT1(s)=TT2(s).
\end{equation}

In which, $TT1(s)=(a_0+ib_0)\prod^{n}_{j=1}(s_0-p_{j})^{\beta_{j}}(\prod^{m}_{l=1}(s-z_{l})^{\gamma_{l}})^{'}-
(a+ib)\prod^{n}_{j=1}(s-p_{j})^{\beta_{j}})^{'}
\prod^{m}_{l=1}(s_0-z_{l})^{\gamma_{l}}$.

$TT2(s)=((s-s_0)^{\gamma_0}g(s))^{'}
\prod^{m}_{l=1}(s_0-z_{l})^{\gamma_{l}}$.
\end{lemma}

\begin{proof}
Differentiate the equation: $(a_0+ib_0)\frac{\prod^{n}_{j=1}(s_0-p_{j})^{\beta_{j}}}
{\prod^{m}_{l=1}(s_0-z_{l})^{\gamma_{l}}}\prod^{m}_{l=1}(s-z_{l})^{\gamma_{l}}-
(a+ib)\prod^{n}_{j=1}(s-p_{j})^{\beta_{j}}=(s-s_0)^{\gamma_0}g(s)$.

$(a_0+ib_0)\frac{\prod^{n}_{j=1}(s_0-p_{j})^{\beta_{j}}}
{\prod^{m}_{l=1}(s_0-z_{l})^{\gamma_{l}}}
(\prod^{m}_{l=1}((s-z_{l})^{\gamma_{l}}))^{'}-(a+ib)
(\prod^{n}_{j=1}((s-p_{j})^{\beta_{j}}))^{'}
=((s-s_0)^{\gamma_0}g(s))^{'}$.

Since we exclude the zeros of (1.1), at the finite point $s_0$, the factor ${\prod^{m}_{l=1}(s_0-z_{l})^{\gamma_{l}}}$ of (1.1) is non-zero. Multiplying both sides by this factor ${\prod^{m}_{l=1}(s_0-z_{l})^{\gamma_{l}}}$ yields: $(a_0+ib_0)\prod^{n}_{j=1}(s_0-p_{j})^{\beta_{j}}(\prod^{m}_{l=1}(s-z_{l})^{\gamma_{l}})^{'}-
(a+ib)(\prod^{n}_{j=1}(s-p_{j})^{\beta_{j}})^{'}
\prod^{m}_{l=1}(s_0-z_{l})^{\gamma_{l}}=(\gamma_0(s-s_0)^{\gamma_0-1}g(s)+(s-s_0)^{\gamma_0}g^{'}(s))
\prod^{m}_{l=1}(s_0-z_{l})^{\gamma_{l}}$.

Let $TT1(s)=(a_0+ib_0)\prod^{n}_{j=1}(s_0-p_{j})^{\beta_{j}}(\prod^{m}_{l=1}(s-z_{l})^{\gamma_{l}})^{'}-
(a+ib)(\prod^{n}_{j=1}(s-p_{j})^{\beta_{j}})^{'}
\prod^{m}_{l=1}(s_0-z_{l})^{\gamma_{l}}$.

$TT2(s)=((s-s_0)^{\gamma_0}g(s))^{'}
\prod^{m}_{l=1}(s_0-z_{l})^{\gamma_{l}}$. By utilizing two new expressions, we obtain (1.3).
\end{proof}
The breakaway points must lie on the root loci sharing the same unit complex value. Therefore, in $TT1(s)$, the sign $a_0+ib_0$ before $\prod^{n}_{j=1}(s_0-p_{j})^{\beta_{j}}(\prod^{m}_{l=1}(s-z_{l})^{\gamma_{l}})^{'}$ is identical to the sign $a+ib$ before $(\prod^{n}_{j=1}(s-p_{j})^{\beta_{j}})^{'}\prod^{m}_{l=1}(s_0-z_{l})^{\gamma_{l}}$.

 The gain function is given by: $K(s)=(a+ib)\frac{\prod^{n}_{j=1}  (s-p_j)^{\beta_j}}{\prod^{m}_{l=1} (s-z_l)^{\gamma_l}}$.
\begin{lemma}
Let $s_0=(\sigma_0+it_0)\in\Delta $. If the finite point $s_0$ satisfies equation (1.4),
\begin{equation}
\frac{dK(s)}{ds}=0.
\end{equation}
Then, $TT1(s_0)=0$ and $((s-s_0)^{\gamma_0}g(s))^{'}\mid{_{s=s_0}}=0$.
\end{lemma}

\begin{proof}
Substituting the finite point $s_0$ into (1.4), according to the requirement of this lemma, we obtain:

$\frac{dK(s_0)}{ds}
=(a_0+ib_0)\frac{(\prod^{n}_{j=1}(s-p_{j})^{\beta_{j}})^{'}\mid{_{s=s_0}}
\prod^{m}_{l=1}(s_0-z_{l})^{\gamma_{l}}-
\prod^{n}_{j=1}(s_0-p_{j})^{\beta_{j}}
}{(\prod^{m}_{l=1}(s_0-z_{l})^{\gamma_{l}})^{2}}$

$\frac{(\prod^{m}_{l=1}((s-z_{l})^{\gamma_{l}}))^{'}\mid{_{s=s_0}}}{}=0$.

From the result of $\frac{dK(s_0)}{ds}$, substituting $s_0$ into the left side of (1.3) yields: $TT1(s_0)=(a_0+ib_0)(\prod^{n}_{j=1}(s_0-p_{j})^{\beta_{j}}(\prod^{m}_{l=1}(s-z_{l})^{\gamma_{l}})^{'}\mid{_{s=s_0}}-
(\prod^{n}_{j=1}(s-p_{j})^{\beta_{j}})^{'}\mid{_{s=s_0}}
\prod^{m}_{l=1}(s_0-z_{l})^{\gamma_{l}})=0.$
Thus, the left side of (1.3) equals 0.

For the right side of (1.3), consider the expression: $TT2(s)=((s-s_0)^{\gamma_0}g(s))^{'}
\prod^{m}_{l=1}(s_0-z_{l})^{\gamma_{l}}$. Since the factor $\prod^{m}_{l=1}(s-z_{l})^{\gamma_{l}}$  has no finite pole, and $s_0$ is neither a zero nor a pole of the function $\prod^{m}_{l=1}(s-z_{l})^{\gamma_{l}}$, if $TT2(s_0)=((s-s_0)^{\gamma_0}g(s))^{'}
\prod^{m}_{l=1}(s_0-z_{l})^{\gamma_{l}}=0$, we can remove the factor $\prod^{m}_{l=1}(s_0-z_{l})^{\gamma_{l}}$. This yields: $((s-s_0)^{\gamma_0}g(s))^{'}\mid{_{s=s_0}}=0$.
\end{proof}

\begin{lemma}
Let $s_0=(\sigma_0+it_0)\in\Delta $. If the finite point $s_0$ satisfies the equation (1.4), then $s_0$ is the finite breakaway point of the root loci of (1.1).
\end{lemma}

\begin{proof}
We analyse three cases on the basis of $\gamma_0>0$:

1.  Case $\gamma_0=1$:

If $\gamma_0=1$,  $g^{'}(s_0)\neq\infty$, $(s_0-s_0)^{\gamma_0}g^{'}(s_0)
=0$. $g(s_0)\neq 0$. $(g(s_0)+(s_0-s_0)^{\gamma_0}g^{'}(s_0))=g(s_0)\neq 0$.

$((s-s_0)^{\gamma_0}g(s))^{'}\mid{_{s=s_0}}=(g(s_0)+(s_0-s_0)^{\gamma_0}g^{'}(s_0))
=g(s_0)\neq 0$. By Lemma 1.22, $((s-s_0)^{\gamma_0}g(s))^{'}\mid{_{s=s_0}}=0$. Therefore, this contradiction is generated. Hence, $\gamma_0=1$ is invalid.

2. Case $1>\gamma_0>0$:

If $1>\gamma_0>0$,  $g^{'}(s_0)\neq\infty$, $(s_0-s_0)^{\gamma_0}g^{'}(s_0)
=0$. $\gamma_0(s_0-s_0)^{\gamma_0-1}=\infty$, $g(s_0)\neq 0$. $((s-s_0)^{\gamma_0}g(s))^{'}\mid{_{s=s_0}}=(\gamma_0(s_0-s_0)^{\gamma_0-1}g(s_0)+(s_0-s_0)^{\gamma_0}g^{'}(s_0))=\infty$.

$((s-s_0)^{\gamma_0}g(s))^{'}\mid{_{s=s_0}}=\infty$.

By Lemma 1.22, $((s-s_0)^{\gamma_0}g(s))^{'}\mid{_{s=s_0}}=0$. Therefore, this contradiction is generated. Hence, $1>\gamma_0>0$ is invalid.

3. Case $\gamma_0>1$:

Because $\gamma_0>1$, $((s-s_0)^{\gamma_0}g(s))^{'}\mid{_{s=s_0}}=\gamma_0(s_0-s_0)^{\gamma_0-1}g(s_0)+(s_0-s_0)^{\gamma_0}g^{'}(s_0)=0$ holds. The left-hand side of (1.3) is equal to the right-hand side of (1.3). (1.3) holds. Hence, $\gamma_0>1$ is valid.

By Lemma 1.19, the roots of (1.1) and its characteristic equation are identical.
The left-hand side of the characteristic equation with gain $K(s_0)$ is rewritten as another expression: $(s-s_0)^{\gamma_0}g(s)$. $s_0$ is a root of (1.1).  $\gamma_0$ is a positive integer. Therefore, the condition $\gamma_0>1$ implies that the point $s_0$ is a multiple root of (1.1), and the finite point $s_0$ is the intersection point of at least two root loci. The finite point $s_0$ is the finite breakaway point of (1.1).

Only $\gamma_0>1$ is valid, confirming that $s_0$ is a finite breakaway point.
\end{proof}
\begin{lemma}
Let $s_0=(\sigma_0+it_0)\in\Delta $. If $s_0$ is a finite breakaway point of the root loci of (1.1), then $s_0$ must satisfy (1.4).
\end{lemma}
\begin{proof}
From (1.3), we obtain:

$(a_0+ib_0)\frac{\prod^{n}_{j=1}(s_0-p_{j})^{\beta_{j}}}
{\prod^{m}_{l=1}(s_0-z_{l})^{\gamma_{l}}}
(\prod^{m}_{l=1}((s-z_{l})^{\gamma_{l}}))^{'}-
(a+ib)(\prod^{n}_{j=1}((s-p_{j})^{\beta_{j}}))^{'}
=\gamma_0(s-s_0)^{\gamma_0-1}g(s)+(s-s_0)^{\gamma_0}g^{'}(s)$. $g(s_0)\neq\infty$ and  $g^{'}(s_0)\neq\infty$.

Substituting $s_0$ into the right-hand side of the above equation, since $s_0$ is the finite breakaway point of (1.1), $s_0$ is a multiple root of the characteristic equation. Therefore, $\gamma_0>1$. $\gamma_0(s_0-s_0)^{\gamma_0-1}g(s_0)+(s_0-s_0)^{\gamma_0}g^{'}(s_0)=0$. $TT2(s_0)=0$.
By Lemma 1.21, $TT1(s_0)=0$, we obtain: $(a_0+ib_0)(\prod^{n}_{j=1}(s_0-p_{j})^{\beta_{j}}
(\prod^{m}_{l=1}(s-z_{l})^{\gamma_{l}})^{'}\mid{_{s=s_0}}-
\prod^{m}_{l=1}(s_0-z_{l})^{\gamma_{l}}
(\prod^{n}_{j=1}(s-p_{j})^{\beta_{j}})^{'}\mid{_{s=s_0}})
=0$.
Differentiating the gain function and substituting $s_0$ into the differential expression, we obtain:

$\frac{dK(s_0)}{ds}
=(a_0+ib_0)\frac{(\prod^{n}_{j=1}(s-p_{j})^{\beta_{j}})^{'}\mid{_{s=s_0}}
\prod^{m}_{l=1}(s_0-z_{l})^{\gamma_{l}}-
\prod^{n}_{j=1}(s_0-p_{j})^{\beta_{j}}
}{(\prod^{m}_{l=1}(s_0-z_{l})^{\gamma_{l}})^{2}}$

$\frac{(\prod^{m}_{l=1}((s-z_{l})^{\gamma_{l}}))^{'}\mid{_{s=s_0}}}{}
=0$.
\end{proof}

Lemma 1.23 proves the sufficient condition for breakaway points, whereas Lemma 1.24 proves the necessary condition. Combining these two lemmas, we obtain the following sufficient and necessary conditions for breakaway points.

\begin{theorem}
Let $s_0=(\sigma_0+it_0)\in\Delta $. The finite point $s_0$ is a finite breakaway point of the root loci of (1.1) if and only if $s_0$ satisfies (1.4).
\end{theorem}

When the unit complex values of the root loci are distinct, these root loci cannot intersect. When the unit complex values of the root loci are identical, (under these conditions)the root loci may intersect. The intersection points at finite locations are finite breakaway points. These intersected root loci separate from these finite breakaway points again. Multiple zeros and poles of the root locus equation cannot exhibit this property. Thus, the finite breakaway points of the root loci are distinct from the multiple zeros and poles of (1.1).

\begin{theorem}
Let $s=(\sigma+it)\in\Delta$. (Except at the multiple zeros and multiple poles of (1.1).)
Then,

1. $K^{'}(s)$ and $RF^{'}(s)$ share the same zeros.

2. Each finite breakaway point of (1.1) is a finite zero of $K^{'}(s)$ and $RF^{'}(s)$.

3. Each finite zero of $K^{'}(s)$ and $RF^{'}(s)$ is a finite breakaway point of (1.1).
\end{theorem}
\begin{proof}
In the proof of Lemma 1.22 and Lemma 1.24, we present the expression of $\frac{dK(s)}{ds}$. Therefore, we omit it here. For $RF(s)=\frac{\prod^{m}_{l=1} (s-z_l)^{\gamma_l}}{ \prod^{n}_{j=1}  (s-p_j)^{\beta_j}}$. Its derivative is as follows:

$(RF(s))^{'}
=\frac{(\prod^{m}_{l=1} (s-z_l)^{\gamma_l})^{'}\prod^{n}_{j=1}  (s-p_j)^{\beta_j}-\prod^{m}_{l=1} (s-z_l)^{\gamma_l}(\prod^{n}_{j=1}  (s-p_j)^{\beta_j})^{'}}{(\prod^{n}_{j=1}  (s-p_j)^{\beta_j})^2}$.

The multiple zeros and multiple poles of $RF$ are zeros of the numerator polynomials of $\frac{dK(s)}{ds}$ and $(RF(s))^{'}$. Specifically:

The multiple zeros of $RF$ are zeros of $(RF(s))^{'}$, but not zeros of $\frac{dK(s)}{ds}$.

The multiple poles of $RF$ are zeros of $\frac{dK(s)}{ds}$, but not zeros of $(RF(s))^{'}$.

1. Excluding the multiple zeros and multiple poles of the $RF$, since $-(a+ib)$ lies on the unit circle and cannot be zero, the critical points of $K(s)$ and $RF(s)$ are identical.

2. The zeros of (1.4) are the critical points of $K(s)$. By Lemma 1.24, each finite breakaway point of (1.1) is finite critical point of $K(s)$, so, each finite breakaway point of (1.1) is finite critical point of $RF(s)$.

3. Excluding the multiple zeros and multiple poles of the $RF$, by Theorem 1.25, each zero of $K^{'}(s)$  is finite breakaway point of (1.1). Each finite zero of  $RF^{'}(s)$ is finite breakaway point of (1.1).
\end{proof}

The multiple zeros can be critical points of $RF$. Similarly, the multiple poles can be critical points of $K(s)$. However, they cannot be breakaway points of (1.1), so we exclude them in Theorem 1.25 and Theorem 1.26.

Theorem 1.25 shows the breakaway points of (1.1) are identical to the critical points of $K(s)$. Theorem 1.26 shows the critical points of $K(s)$ are identical to the critical points of $RF(s)$. Therefore, Theorem 1.26 also shows the breakaway points of (1.1) are identical to the critical points of $RF(s)$. By  Theorem 1.25, the intersection points of the root loci of (1.1) are critical points of $RF(s)$.

\begin{lemma}
Assume $RF$ is a real constant coefficient rational function. Then:

(1). The root loci of $2q\pi$ and $2q\pi+\pi$ degree of (1.1) are symmetric about the real axis.

(2). The root loci on the real axis are the root loci of $2q\pi$ or $2q\pi+\pi$ degree of (1.1).

(3). An interval on the real axis is a $2q\pi+\pi$ degree root locus if and only if the number of real poles and zeros of (1.1) to its right is odd.

(4). An interval on the real axis is a $2q\pi$ degree root locus if and only if the number of real poles and zeros of (1.1) to its right is even.

(5). The interval between two adjacent zero and pole (zero-pole, zero-zero, or pole-pole) is the root loci of the same degree.
\end{lemma}
\begin{proof}
(1). The rational function can be written as $RF(s)=u(\sigma,t)+iv(\sigma,t)$. Since $RF$ is a real constant coefficient rational function, it is a conjugate function. For the conjugate function $RF$, $RF(\overline{s})=u(\sigma,t)-iv(\sigma,t)$. For the root loci of $2q\pi$ and $2q\pi+\pi$ degree of (1.1), $v(\sigma,t)=0$. Therefore, for all points on these root loci of $2q\pi$ and $2q\pi+\pi$ degree of (1.1), $RF(s)=RF(\overline{s})=u(\sigma,t)$. Thus, the root loci of $2q\pi$ and $2q\pi+\pi$ degree of (1.1) are symmetric about the real axis.

(2). Since $RF$ is a real constant coefficient rational function, it takes real values on the real axis. Hence, the root loci on the real axis are the root loci of $2q\pi$ and $2q\pi+\pi$ degree of (1.1).

For any point $s_{0}$ on the selected section of the real axis, the zeros and poles of (1.1) are conjugate symmetric with respect to the real axis. The sum of the phase angles of the vectors between conjugate complex numbers and an arbitrary point on the real axis is $2\pi$. Therefore, determining the phase angle of an arbitrary point on the real axis does not require considering the effects of conjugate zeros and poles. When determining the degree of the root locus of a real constant coefficient rational function on the real axis, we do not need to consider the conjugate zeros and poles.

The phase angle of vector between a real zero  on the left side of $s_{0}$ and point $s_{0}$ is $2\pi$. The phase angle of vector between a real pole on the left side of $s_{0}$ and point $s_{0}$ is $2\pi$. Thus, determining the phase angle of an arbitrary point on the real axis does not need to consider the real zeros and real poles on the left side of $s_{0}$. When determining the degree of the root locus of a real constant coefficient rational function on the real axis, we do not need to consider the real zeros and real poles on the left side of $s_{0}$.

Let $z_{l}$ be a real zero to the right of $s_{0}$, and let $\gamma_l$ be the exponent of $z_{l}$. The phase angle of the vector between $z_{l}$ and $s_{0}$ is $\gamma_l\pi$. The sum of the phase angles of the vectors between the real zeros on the right side of $s_{0}$ and $s_{0}$ is $\sum_{l=1}^{m_1} \gamma_l\pi$. Similarly, let $p_{j}$ be a real pole to the right of $s_{0}$, and let $\beta_j$ be the exponent of $p_{j}$. The phase angle of the vector between $p_{j}$ and $s_{0}$ is $\beta_j\pi$. The sum of the phase angles of the vectors between the real poles on the right side of $s_{0}$ and $s_{0}$ is $\sum_{j=1}^{n_1} \beta_j\pi$. The degree of the root locus of a real constant coefficient rational function in this interval is $\sum_{l=1}^{m_1} \gamma_l\pi-\sum_{j=1}^{n_1} \beta_j\pi=(\sum_{l=1}^{m_1} \gamma_l-\sum_{j=1}^{n_1} \beta_j)\pi$.

In $(\sum_{l=1}^{m_1} \gamma_l-\sum_{j=1}^{n_1} \beta_j)\pi$, $-\pi$ and $\pi$ represents the same degrees, subtracting an integer multiple of $\pi$ is equivalent to adding an integer multiple of $\pi$. Therefore, $(\sum_{l=1}^{m_1} \gamma_l-\sum_{j=1}^{n_1} \beta_j)\pi=(\sum_ {l=1}^{m_1} \gamma_l+\sum_{j=1}^{n_1} \beta_j)\pi$. By $(\sum_ {l=1}^{m_1} \gamma_l+\sum_{j=1}^{n_1} \beta_j)\pi$, (3) and (4) hold.

(5). By $(\sum_ {l=1}^{m_1} \gamma_l+\sum_{j=1}^{n_1} \beta_j)\pi$, the phase angles of all the points in this interval are identical. The interval between two adjacent zero and pole (zero and pole, zero and zero, pole and pole) is the same degree root loci.
\end{proof}

\begin{theorem}
1. The number of  breakaway points of (1.1) is finite. The breakaway points are isolated to distributed in $\mathbb{C}\cup\{\infty\}$

2. Any segment of all the root loci of an arbitrary degree cannot overlap in the finite positions of $\mathbb{C}\cup\{\infty\}$.
\end{theorem}

\begin{proof}
1. To compute the finite breakaway points and prove their existence, we use the formula for computing finite breakaway points. Equation (1.4) provides a necessary and sufficient condition for computing the finite breakaway point:

$\frac{dK(s)}{ds}
=(a+ib)\frac{(\prod^{n}_{j=1}(s-p_{j})^{\beta_{j}})^{'}
\prod^{m}_{l=1}(s-z_{l})^{\gamma_{l}}-
\prod^{n}_{j=1}(s-p_{j})^{\beta_{j}}(\prod^{m}_{l=1}(s-z_{l})^{\gamma_{l}})^{'}
}{(\prod^{m}_{l=1}(s-z_{l})^{\gamma_{l}})^{2}}=0.$

$\frac{dK(s)}{ds}$ has only finite number of zeros. The breakaway points are zeros of $\frac{dK(s)}{ds}$. These finite zeros must be isolated in $\mathbb{C}\cup\{\infty\}$; they cannot be continuously distributed. Therefore, these finite zeros cannot form a continuous curve.

2. The finite breakaway point is the intersection point of the root locus of the root locus equation (1.1). An overlapping segment consists of continuous intersection points of the root locus of (1.1). If such overlapping segments exist, they would be special intersection points, specifically an infinite number of continuous intersection points.  On the basis of the previous proof, any segment of all the root loci of an arbitrary degree number cannot overlap.
\end{proof}
\begin{definition}
In $\mathbb{C}\cup\{\infty\}$, if an angle $\phi$ is between the positive direction of the real axis and the tangent line of the root loci of arbitrary $2q\pi+\alpha$ degree that originates from a finite pole, then the angle $\phi$ is called the \emph{angle of origination} of the root loci of $2q\pi+\alpha$ degree of (1.1).
\end{definition}

\begin{definition}
In $\mathbb{C}\cup\{\infty\}$, if an angle $\vartheta$ is between the positive direction of the real axis and the tangent line of the root loci of arbitrary $2q\pi+\alpha$ degree that terminates at a finite zero, then the angle $\vartheta$ is called the \emph{angle of receiving} of the root loci of $2q\pi+\alpha$ degree of (1.1).
\end{definition}

The \emph{angle of origination} and \emph{angle of receiving} are both positive in the counterclockwise orientation.

\begin{theorem}
In $\mathbb{C}\cup\{\infty\}$, let $P_{k_1}$ be a finite pole of (1.1) with multiplicity $\beta_{k_1}$. The angle of origination of the root locus of an arbitrary $2q\pi+\alpha$ degree of (1.1), which originates from the finite pole $P_{k_1}$, is given by:

$\phi_{k_1}=\frac{1}{\beta_{k_1}}(2q\pi-\alpha+\sum\limits_{l=1}^{m}
(\gamma_{l}arg{(P_{k_1}-z_{l})})-\sum\limits_{j=1,j \neq k_1}^{n}
(\beta_{j}arg{(P_{k_1}-p_{j})})).$
\end{theorem}
\begin{proof}
Consider a finite point $s_1$ on the root locus of $2q\pi+\alpha$ degree. Let $s_1$ approach infinitely close to the finite pole $P_{k_1}$ with multiplicity $\beta_{k_1}$, which needs to compute its angle of origination. Since $s_1$ is infinitely close to $P_{k_1}$, the phase angles $\beta_{k_1} arg{(s_1-p_j)}$ of vectors between the finite point $s_{1}$ and the finite poles $p_j$, $j\neq k_1$. (Namely, except for the finite pole $P_{k_1}$.) can all be substituted by angles $\beta_{k_1}arg{(P_{k_1}-p_{j})}$, $j\neq k_1$. Similarly, the phase angle $\gamma_{k_1}arg{(s_1-z_{l})}$ of vectors between the finite point $s_{1}$ and all of the finite zeros $z_l$ of (1.1) all can be substituted by angles $\gamma_{k_1}arg{(P_{k_1}-z_{l})}$.

Therefore, the vector between finite point $s_{1}$ and finite pole $P_{k_1}$ is infinitely approaching to the tangent line of the root locus of arbitrary $2q\pi+\alpha$ degree that originates from $P_{k_1}$. By Definition 1.29, the phase angle of vectors between finite point $s_{1}$ and finite pole $P_{k_1}$ be infinitely approach the angle of origination $\phi_{k_1}$.

The finite point $s_{1}$ must satisfy the phase angle condition equation,
$\sum\limits_{l=1}^{m}
(\gamma_{l}arg{(s_1-z_{l})})-\sum\limits_{j=1}^{n}
(\beta_{j}arg{(s_1-p_{j})})=2q\pi+\alpha$.
Substituting $s_{1}$ with $P_{k_1}$, we obtain: $\sum\limits_{l=1}^{m}
(\gamma_{l}arg{(P_{k_1}-z_{l})})-\sum\limits_{j=1,j \neq k_1}^{n}
(\beta_{j}arg{(P_{k_1}-p_{j})})-\beta_{k_1}\phi_{k_1}=2q\pi+\alpha$.
$\beta_{k_1}\phi_{k_1}=\sum\limits_{l=1}^{m} (\gamma_{l}arg{(P_{k_1}-z_{l})})-\sum\limits_{j=1,j \neq k_1}^{n}
(\beta_{j}arg{(P_{k_1}-p_{j})})-2q\pi-\alpha$.

Since $2q\pi$ and $-2q\pi$ are equivalent, we substitute $2q\pi$ with $-2q\pi$. After transposing, we obtain the formula in Theorem 1.31.
\end{proof}

By repeating the proof of Theorem 1.31, we obtain Theorem 1.32.

\begin{theorem}
In $\mathbb{C}\cup\{\infty\}$, let $Z_{k_1}$ be a finite zero of (1.1) with multiplicity $\gamma_{k_1}$. The angle of receiving on the root locus of an arbitrary $2q\pi+\alpha$ degree of (1.1), which ends at the finite zero $Z_{k_1}$, is given by:

$\vartheta_{k_1}=\frac{1}{\gamma_{k_1}}(2q\pi+\alpha-\sum\limits_{l=1, l \neq k_1}^{m}
(\gamma_{l}arg{(Z_{k_1}-z_{l})})+\sum\limits_{j=1}^{n}
(\beta_{j}arg{(Z_{k_1}-p_{j})})).$
\end{theorem}

\begin{theorem}
For the root locus originating from an arbitrary finite pole $P_{k_1}$ of (1.1) in $\mathbb{C}\cup\{\infty\}$, if the angle $\phi_{k_1}$ of origination is treated as an independent variable, then the degree values of these root loci strictly and monotonically increase in the clockwise direction around $P_{k_1}$.
\end{theorem}
\begin{proof}
The angle of origination of the $\alpha_1$ degree root locus of (1.1) at the finite pole $P_{k_1}$ is given by:
$\phi_{k_1,1}=\frac{1}{\beta_{k_1}}(2q\pi-\alpha_1+\sum\limits_{l=1}^{m}
(\gamma_{l}arg{(P_{k_1}-z_{l})})-\sum\limits_{j=1,j \neq k_1}^{n}
(\beta_{j}arg{(P_{k_1}-p_{j})}))$.

Similarly, the angle of origination of the $\alpha_2$ degree root locus of (1.1) at the finite pole is as follows:
$\phi_{k_1,2}=\frac{1}{\beta_{k_1}}(2q\pi-\alpha_2+\sum\limits_{l=1}^{m}
(\gamma_{l}arg{(P_{k_1}-z_{l})})-\sum\limits_{j=1,j \neq k_1}^{n}
(\beta_{j}arg{(P_{k_1}-p_{j})}))$. Subtracting these two equations yields: $\phi_{k_1,1}-\phi_{k_1,2}=\frac{1}{\beta_{k_1}}(\alpha_2-\alpha_1)$, thus, $\alpha_2-\alpha_1=\beta_{k_1}(\phi_{k_1,1}-\phi_{k_1,2})$.

Here, $\beta_{k_1}$ is the exponent of the finite pole $P_{k_1}$, a positive constant integer. The angles $\phi_{k_1,1}$ and $\phi_{k_1,2}$ represent the origination angles of the $\alpha_1$-degree and $\alpha_2$-degree root loci, respectively. The angle of origination is defined as the angle between the positive real axis and the tangent line of the root locus originating from a finite pole $P_{k_1}$, measured counterclockwise from the positive real axis.

1. If $\phi_{k_1,1}>\phi_{k_1,2}$, $\phi_{k_1,1}-\phi_{k_1,2}>0$. The angle of origination of the root loci of $\alpha_1$ degree is $\phi_{k_1,1}$. The angle of origination of the root loci of $\alpha_2$ degree is $\phi_{k_1,2}$. $\alpha_2-\alpha_1>0$, $\alpha_2>\alpha_1$ holds. Therefore, for two root loci that originate at two distinct angles of origination from a same pole, the root locus of the higher degree is in the clockwise direction of the root locus of the lower degree.

2. If $\phi_{k_1,1}<\phi_{k_1,2}$, $\phi_{k_1,1}-\phi_{k_1,2}<0$, then $\alpha_2-\alpha_1<0$, $\alpha_2<\alpha_1$ holds. Therefore, for two root loci that originate at two distinct angles of origination from the same pole, the root locus of the higher degree is in the clockwise direction of the root locus of the lower degree.

In both cases, the degree values of the root loci originating from $P_{k_1}$ strictly and monotonically increase in the clockwise direction around $P_{k_1}$.
\end{proof}

By $\alpha_2-\alpha_1=\gamma_{k_1}(\varphi_{k_1,2}-\varphi_{k_1,1})$ and repeating the proof of Theorem 1.33, we can obtain Theorem 1.34.

\begin{theorem}
For the root locus received by an arbitrary finite zero $Z_{k_1}$ of (1.1) in $\mathbb{C}\cup\{\infty\}$, the angle of receiving $\vartheta_{k_1}$ is an independent variable. Rotating the finite zero $Z_{k_1}$, the degree values of these root loci strictly and monotonously decrease clockwise around $Z_{k_1}$.
\end{theorem}

\begin{theorem}
Assume that a circle is rotated from $0$ to $2\pi$ degrees around a pole $P_{k_1}$. The pole $P_{k_1}$ with the exponent $\beta_{k_1}$ of (1.1) emits root loci from $2q\pi$ degrees to $2q\pi+2\beta_{k_1}\pi$ degrees. Specifically, the pole $P_{k_1}$ with the exponent $\beta_{k_1}$ of (1.1) emits $\beta_{k_1}$ branches of $2q\pi$ degree root loci and $\beta_{k_1}$ branches of $2q\pi+\pi$ degree root loci.
\end{theorem}
\begin{proof}
By Theorem 1.31, the two angles of departure at the pole $P_{k_1}$ with the exponent $\beta_{k_1}$ in the same direction are given by:

$\phi_{k_1}=\frac{1}{\beta_{k_1}}(2q\pi-\alpha_1+\sum\limits_{l=1}^{m}
(\gamma_{l}arg{(P_{k_1}-z_{l})})-\sum\limits_{j=1,j \neq k_1}^{n}
(\beta_{j}arg{(P_{k_1}-p_{j})})$.

$\phi_{k_1}+2\pi=\frac{1}{\beta_{k_1}}(2q\pi-\alpha_2+\sum\limits_{l=1}^{m}
(\gamma_{l}arg{(P_{k_1}-z_{l})})-\sum\limits_{j=1,j \neq k_1}^{n}
(\beta_{j}arg{(P_{k_1}-p_{j})})$.

Subtracting the two equations side by side, we obtain:  $-2\pi=\frac{1}{\beta_{k_1}}(-\alpha_1+\alpha_2)$.
 $-2\pi\beta_{k_1}=-\alpha_1+\alpha_2$. Thus, $\alpha_1-\alpha_2=2\beta_{k_1}\pi$.

$\alpha_1-\alpha_2=2\beta_{k_1}\pi$. This result shows that the pole $P_{k_1}$ with the exponent $\beta_{k_1}$ of (1.1) emits root loci from $2q\pi$ degree to $2q\pi+2\beta_{k_1}\pi$ degree. The root loci from $2q\pi$ degree to $2q\pi+2\beta_{k_1}\pi$ degree, consist of the $\beta_{k_1}$ branches of $2q\pi$ degree root loci and $\beta_{k_1}$ branches of $2q\pi+\pi$ degree root loci. Therefore, the pole $P_{k_1}$ with the exponent $\beta_{k_1}$ of (1.1) emits $\beta_{k_1}$ branches of $2q\pi$ degree root loci and $\beta_{k_1}$ branches of $2q\pi+\pi$ degree root loci.
\end{proof}

By repeating the proof of Theorem 1.35, we can obtain Theorem 1.36.

\begin{theorem}
Assume that rotate a circle form 0 to $2\pi$ degrees around a zero $Z_{k_1}$, the zero $Z_{k_1}$ with the exponent $\gamma_{k_1}$ of (1.1) receives the root loci from $2q\pi$ degree to $2q\pi+2\gamma_{k_1}\pi$ degree. Specifically, the zero $Z_{k_1}$ with the exponent $\gamma_{k_1}$ of (1.1) receives $\gamma_{k_1}$ branches of $2q\pi$ degree root loci and $\gamma_{k_1}$ branches of $2q\pi+\pi$ degree root loci.
\end{theorem}

\section{The results and proofs for the Gauss-Fourier Conjecture}
All lemmas and theorems in this section are original to the authors and do not rely on prior literature.

The general real polynomial $p(s)$ of degree $n$ is as follows: $p(s)=a_0s^n+a_1s^{n-1}+\cdots+a_{n-1}s+a_n$.

\begin{equation}
\frac{p^{'}(s)}{p(s)}.
\end{equation}

Notation 3:

Let $p_s$ be a zero of $p(s)$ with multiplicity $r_s\ge 1$. $r_s$ is multiplicity of $p_s$. $D(s)$ is a factor of $p(s)$ except the factor $(s-p_s)^{r_s}$. $s_{1,i_1}$ are zeros of $p(s)$ and their multiplicities are $r_{1,i_1}$. $r_{1,i_1}$ is multiplicity of $s_{1,i_1}$. $m_0$ denotes the number of multiple zeros of $p(s)$. $s_{2,i_2}$ are simple zeros of $p(s)$. $i_1=1,2,\cdots,m_0$. $n_{0,1}$ denotes the number of distinct simple zeros of $p(s)$. $i_2=1,2, \cdots, n_{0,1}$. $z_{0,1}$ denotes a real zero of $\frac{p^{'}(s)}{p(s)}$. $t_{0,1}$ denotes a real pole of $\frac{p^{'}(s)}{p(s)}$. $L_{0,1}$ denotes the interval between a real zero $z_{0,1}$ and pole $t_{0,1}$ of $\frac{p^{'}(s)}{p(s)}$. $g_{0,1}$ is a root locus emitted from $t_{0,1}$.  $Q_a$ denotes the real breakaway point of the root loci of (2.1). $s_a$ or $\overline{s_a}$ denote a pair of non-real zeros of $p(s)$.

\begin{lemma}
For the real polynomial $p(s)$, all poles of (2.1) are simple.
\end{lemma}
\begin{proof}
Let $p_s$ be a zero of $p$ with multiplicity $r_s\ge 1$, so $p(s)=D(s)(s-p_s)^{r_s}$ where $D(p_s)\neq 0$. $p^{'}(s)=D^{'}(s)(s-p_s)^{r_s}+r_sD(s)(s-p_s)^{r_s-1}$. Therefore, $\frac{p^ {'}(s)}{p(s)}=\frac{D^ {'}(s)}{D(s)}+\frac{r_s}{s-p_s}$. Since $D(s)$ does not contain the factor $s-p_s$, it follows that all poles of (2.1) are simple.
\end{proof}

\begin{lemma}
The function (2.1) has two zeros at infinity on the real axis. In $\mathbb{C}\cup\{\infty\}$, all the points at infinity are infinite zeros of (2.1).
\end{lemma}
\begin{proof}
Let $p(s)$ have $m_0$ distinct multiple zeros $s_{1,i_1}$ with multiplicities $r_{1,i_1}$, $i_1=1,2,\cdots,m_0$. $p(s)$ has $n-\sum_{i_1=1}^{m_0} r_{1,i_1}$ distinct simple zeros (if $p(s)$ has simple zeros). Let $n_{0,1}=n-\sum_{i_1=1}^{m_0} r_{1,i_1}$. $p(s)$ has $n_{0,1}$ distinct simple zeros. Suppose these simple zeros are $s_{2,i_2}$, $i_2=1,2, \cdots, n_{0,1}$. So, we have: $\frac{p^{'}(s)}{p(s)}=\sum_{i_1=1}^{m_0} \frac{r_{1,i_1}}{s-s_{1,i_1}}+\sum_{i_2=1}^{n_{0,1}}\frac{1}{s-s_{2,i_2}}$.

Substituting $s=+\infty$ and $s=-\infty$ separately into $\frac{p^{'}(s)}{p(s)}=\sum_{i_1=1}^{m_0} \frac{r_{1,i_1}}{s-s_{1,i_1}}+\sum_{i_2=1}^{n_{0,1}}\frac{1}{s-s_{2,i_2}}$, yields 0 in both cases. Therefore, the two points at infinity on the real axis are infinite zeros of (2.1).

In $\mathbb{C}\cup\{\infty\}$, the modulus of infinite points must be infinite. Substituting these infinite points into the rational function: $\frac{p^{'}(s)}{p(s)}=\sum_{i_1=1}^{m_0} \frac{r_{1,i_1}}{s-s_{1,i_1}}+\sum_{i_2=1}^{n_{0,1}}\frac{1}{s-s_{2,i_2}}$, results in the function (2.1) equaling 0.  As \(|s| \to \infty\) along any direction, \(\frac{p^{'}(s)}{p(s)} = \frac{n}{s} + O(|s|^{-2})\), so \(\infty\) is a zero of order 1 in $\mathbb{C}\cup\{\infty\}$.
Therefore, all infinite points in $\mathbb{C}\cup\{\infty\}$ are infinite zeros of (2.1).
\end{proof}

\begin{lemma}
If the real polynomial $p(s)$ has at least one real zero, then:

1. Every real pole of (2.1) is separated by at least one real zero of  $p^{'}(s)$.

2. No two real poles of (2.1) are adjacent.
\end{lemma}
\begin{proof}
If $p(s)$ has at least two real zeros, by Rolle's theorem, between any two adjacent real zeros of $p(s)$, there lies at least one real zero of $p^{'}(s)$, ensuring poles of $\frac{p^{'}(s)}{p(s)}$ are separated by zeros. Thus, the lemma holds.

If $p(s)$ has only one real zero, the two infinite points on the real axis serve as infinite zeros of (2.1), ensuring the lemma remains valid.
\end{proof}

Establish the root locus equation for the rational function $\frac{p^{'}(s)}{p(s)}$:
\begin{equation}
K\frac{p^{'}(s)}{p(s)}=a+ib.
\end{equation}

Theorem 1.25 provides a necessary and sufficient condition for breakaway points. Theorems 1.25 and 1.26 together explain the mechanism of emergence of critical points in complex rational functions. When root loci have distinct unit complex values, they cannot intersect. When root loci share the same unit complex value, they can intersect. Such intersection points are called "breakaway points" in root locus theory, and each breakaway point is a critical point of the complex rational function. These critical points occur at the intersection points of multiple root loci that share identical unit complex values.

By Lemma 2.2, all infinite points in $\mathbb{C}\cup\{\infty\}$ are zeros of (2.2), which can only receive root loci but cannot emit the root loci. In $\mathbb{C}\cup\{\infty\}$, no other poles can emit root loci. Therefore, each zero of $p^{'}(s)$ receives the root loci that are emitted exclusively from the zeros of $p(s)$. Thus:

\begin{lemma}
Each zero of the root locus equation (2.2) receives root loci, which can only be emitted from the zeros of the real polynomial $p(s)$.
\end{lemma}

The real zeros of $(p^{'}(s)/p(s))^{'}$ are classified into two kinds. Those that are real multiple zeros of $p^{'}(s)$ constitute the second kind of zeros of $(p^{'}(s)/p(s))^{'}$. The remaining real zeros of $(p^{'}(s)/p(s))^{'}$ are referred to as the first kind of zeros of $(p^{'}(s)/p(s))^{'}$. In fact, the first kind of zeros of $(p^{'}(s)/p(s))^{'}$ are the intersection points of the root loci of (2.2).

In the subsequent proof, the concept of root locus extension is essential, it means a root locus moves from a point closer to a pole to a point closer to a zero.

Demonstrating that the generation of real zeros of $(p^{'}(s)/p(s))^{'}$ relates to non-real zeros of $p(s)$ constitutes the most crucial result for their correlation. This is precisely the result Gauss and Fourier sought. Moreover, revealing the principle behind critical point generation in functions holds greater significance than the Gauss-Fourier conjecture itself.

\begin{theorem}
Suppose the real polynomial $p(s)$ has at least one real zero. Then, each real pole of (2.2) (i.e., real zero of $p(s)$) emits exactly one root locus of $2q\pi$ degree to one side (left or right) and one root locus of $2q\pi+\pi$ degree to the opposite side. No additional root loci of $2q\pi$ or $2q\pi+\pi$ degree are emitted from real poles, and no root loci of $2q\pi$ or $2q\pi+\pi$ degree enter $\mathbb{C}^{*}$ directly from real poles.
\end{theorem}
\begin{proof}
By Theorem 1.27, the real axis is its root loci of $2q\pi$ and $2q\pi+\pi$ degrees of (2.2). By Lemma 2.1, all poles of (2.2) are simple. By Theorem 1.35, each real pole of (2.2) (i.e., real zero of $p(s)$) emits exactly one root locus of $2q\pi$ degree  and one root locus of $2q\pi+\pi$ degree . By Theorem 1.27 and Theorem 1.31, each real pole of (2.2) only emits one root locus of $2q\pi$ degree  and one root locus of $2q\pi+\pi$ degree to real axis.

Therefore, each real pole of (2.2) (i.e., real zero of $p(s)$) emits exactly one root locus of $2q\pi$ degree to one side (left or right) and one root locus of $2q\pi+\pi$ degree to the opposite side. No additional root loci of $2q\pi$ or $2q\pi+\pi$ degree are emitted from real poles, and no root loci of $2q\pi$ or $2q\pi+\pi$ degree enter $\mathbb{C}^{*}$ directly from real poles.
\end{proof}

\begin{theorem}
Suppose the real polynomial $p(s)$ has at least one real zero. Assume that the root loci emitted from the non-real zeros of $p(s)$ do not intersect the real axis. Then, the root loci emitted from the real zeros of $p(s)$ cannot leave the real axis and enter $\mathbb{C}^{*}$. The first kind of zeros of $(p^{'}(s)/p(s))^{'}$  cannot be generated.
\end{theorem}
\begin{proof}
By Lemma 2.3, real zeros of (2.2) exist on both sides of each real pole. Thus, the $2q\pi$ degree and $2q\pi+\pi$ degree root loci emitted from each real pole of (2.2) reach the adjacent real zeros of (2.2)(which may include zeros at infinity).

By Lemma 2.4, each zero of $p^{'}(s)$ receives root loci emitted exclusively from zeros of $p(s)$. By Theorem 2.5, each real zero of $p(s)$ emits only one $2q\pi$ degree root locus and one $2q\pi+\pi$ degree root locus to the real axis. Assume that root loci of $2q\pi$ and $2q\pi+\pi$ degree emitted from non-real zeros of $p(s)$ do not intersect the real axis, the root loci of $2q\pi$ and $2q\pi+\pi$ degree emitted from real zeros of $p(s)$ directly reach the adjacent real zeros of (2.2) on the real axis. Consequently, these root loci of the real zeros of $p(s)$ cannot leave the real axis or enter $\mathbb{C}^{*}$.  The first kind of zeros of $(p^{'}(s)/p(s))^{'}$ cannot be generated.
\end{proof}

Repeated the proof of Theorem 2.6, we have Lemma 2.7 and  Lemma 2.8.

\begin{lemma}
Within the interval $L_{0,1}$ bounded by a real zero $z_{0,1}$ and pole $t_{0,1}$ of (2.2), if no other root loci intersect in $L_{0,1}$, then the root locus $g_{0,1}$ emitted from $t_{0,1}$ must end at $z_{0,1}$. Furthermore, $g_{0,1}$ cannot leave the real axis or enter $\mathbb{C}^{*}$. In $L_{0,1}$, the first kind of zeros of $(p^{'}(s)/p(s))^{'}$  cannot be generated.
\end{lemma}

\begin{lemma}
Suppose the real polynomial $p(s)$ has only real zeros. Then, the root loci emitted from the real zeros of $p(s)$ cannot leave the real axis and enter $\mathbb{C}^{*}$. The first kind of zeros of $(p^{'}(s)/p(s))^{'}$  cannot be generated.
\end{lemma}

By Theorem 1.25 and Theorem 1.26, the real breakaway points of (2.2) are first kind of zeros of $(p^{'}(s)/p(s))^{'}$. So, by Lemma 2.8, we have Corollary 2.1.

\begin{coro+}
If there exists at least a first kind of zero of $(p^{'}(s)/p(s))^{'}$, then the real polynomial $p(s)$ must have at least one pair of the non-real zeros.
\end{coro+}

\begin{lemma}
Suppose $p(s)$ is a real polynomial with no real zeros. Only the root locus emitted from a non-real zero of the real polynomial $p(s)$ intersects on the real axis. A first kind of zero of $(p^{'}(s)/p(s))^{'}$ can be generated. A $2q\pi$ or $2q\pi+\pi$ degree root locus emitted from at least one pair of the non-real zeros of $p(s)$ intersects with other root loci on the real axis. On the real axis, the intersection point of the root loci is a first kind of zero of $(p^{'}(s)/p(s))^{'}$.
\end{lemma}
\begin{proof}
Since $p(s)$ has no real zeros and is a real polynomial of even degree, its derivative $p^{'}(s)$ is a real polynomial of odd degree and therefore has at least one real zero. On the real axis (including the two zeros at infinity of (2.2)), equation (2.2) has at least three real zeros. However, there are no real poles of (2.2) on the real axis. Thus, the real zeros of (2.2) partition the real axis into several intervals.

The rational function $p'(s)/p(s)$ is real-valued on the real axis, and equation (2.2) is a real root locus equation. By Theorem 1.27, the real axis consists of root loci of (2.2) with $2q\pi$ or $2q\pi+\pi$ degree.

Since (2.2) has no real poles, Lemma 1.13 implies that root loci can only originate from non-real poles of (2.2). As all poles of (2.2) lie in $\mathbb{C}^{*}$ (none on the real axis), the root loci with $2q\pi$ and $2q\pi+\pi$ degree emitted from these non-real poles must enter the real axis.

All real zeros of (2.2) must receive root loci with $2q\pi$ and $2q\pi+\pi$ degree. By Lemma 2.4, each zero of $p^{'}(s)$ receives root loci that can only be emitted from zeros of $p(s)$. Specifically, each real zero of $p^{'}(s)$ receives root loci on the real axis that are emitted from non-real zeros of $p(s)$. Furthermore, there exist root loci with $2q\pi$ and $2q\pi+\pi$ degree emitted from conjugate non-real zeros of $p(s)$; these root loci are symmetric about the real axis and must enter the real axis, terminating at real zeros of (2.2).

1. Due to the symmetry of the root loci of (2.2) about the real axis, root loci must intersect at a point on the real axis before entering it.

2. When two distinct root loci meet on the real axis, they must intersect at a point. Both scenarios generate real breakaway points.

3. There is another way to generate real breakaway points. The root loci emitted from non-real poles do not enter the real axis. They only intersect at real breakaway points, and after the intersection, these root loci directly enter $\mathbb{C}^{*}$. In this case, the same root locus lies on both sides of the real axis at the real breakaway point. There are the only three ways real breakaway points are generated.

Since (2.2) has no poles on the real axis, all its root loci originate from non-real poles. Theorem 1.26 establishes that critical points of $RF(s)$ are intersection points of root loci of (1.1). Thus, each first-kind zero of $(p^{'}(s)/p(s))^{'}$ corresponds to an intersection of root loci of (2.2) on the real axis. Therefore,  among these three ways of generating real breakaway points, all real breakaway points of (2.2) are generated by the intersections of root loci emitted from non-real zeros of $p(s)$, and these real breakaway points are precisely the first-kind zeros of $(p^{'}(s)/p(s))^{'}$.

Consequently, only the root locus emitted from a non-real zero of the real polynomial $p(s)$ intersects on the real axis. A first kind of zero of $(p^{'}(s)/p(s))^{'}$ can be generated. A $2q\pi$ or $2q\pi+\pi$ degree root locus emitted from at least one pair of the non-real zeros of $p(s)$ intersects with other root loci on the real axis. On the real axis, the intersection point of the root loci is a first kind of zero of $(p^{'}(s)/p(s))^{'}$.
\end{proof}
\begin{remark}
In the three generation ways of the real breakaway points in the proof of Lemma 2.9, there may arise a situation where the root loci from $\mathbb{C}^{*}$ intersect at real breakaway points, and then these root loci from $\mathbb{C}^{*}$ leave  the real breakaway points and enter $\mathbb{C}^{*}$. These root loci that intersect at real breakaway points meet at points on the real axis but do not enter the real axis. We may refer to this method as the third method. The results and proofs associated with the third method are exactly the same as those of the previous two methods. Therefore, in the following results and proofs, we will omit any further mention of the third method and will not provide its results or proofs separately.
\end{remark}

\begin{lemma}
Suppose the real polynomial $p(s)$ has at least one real zero. For root loci intersecting at a real breakaway point $Q_a$ in (2.2), if at least one of them originates directly from a non-real zero $s_a$ or $\overline{s_a}$ of $p(s)$, this root locus emitted from non-real zero $s_a$ or $\overline{s_a}$ of $p(s)$ intersects with other root loci at $Q_a$ on the real axis. Then, the intersection point $Q_a$ of the root loci is a first kind of zero of $(p^{'}(s)/p(s))^{'}$.
\end{lemma}
\begin{proof}
Consider root loci intersecting at a real breakaway point $Q_a$ in (2.2) where at least one root locus comes directly from a non-real zero $s_a$ or $\overline{s_a}$ of $p(s)$. The real breakaway points formed under this condition constitute the first-kind zeros of $(p^{'}(s)/p(s))^{'}$. Since the intersecting root loci contain the root locus coming from a non-real zero $s_a$ or $\overline{s_a}$ of $p(s)$, the first kind of zero $Q_a$ of $(p^{'}(s)/p(s))^{'}$ is generated by the intersection of root locus of the conjugate non-real zeros $s_a$ or $\overline{s_a}$ of $p(s)$.

Therefore, at least one pair of the non-real zeros $s_a$ or $\overline{s_a}$ of $p(s)$ emits $2q\pi$ or $2q\pi+\pi$ degree root locus, this root locus intersects with other root loci at $Q_a$ on the real axis. The intersection point $Q_a$ of the root loci is a first kind of zero of $(p^{'}(s)/p(s))^{'}$.
\end{proof}
Notation 4:

$k$ denotes the number of a amount in the linked root loci. For example, $k$ is the number of real zeros in the linked root loci, and $k$ is also the number of the root loci of real zeros in the linked root loci, $1\le k<n$.
$j_1$ denotes the $j_1$-th amount in the linked root loci. For example, $j_1$ means the $j_1$-th real zero in the linked root loci, and $j_1$ also means the $j_1$-th the root locus of real zero in the linked root loci. $j_3=1$ denotes a amount in the standard linked root loci. $j_3=2$ denotes a amount in the first kind of non-standard linked root loci. $j_3=3$ denotes a amount in the second kind of non-standard linked root loci. $j_3=4$ denotes a amount in the third kind of non-standard linked root loci. $z_{k,j_1,j_3}$ denotes a real zero of (2.2). $t_{k,j_1,j_3}$ denotes a real pole of (2.2). $t_{k,j_1,j_3}$ and $z_{k,j_1,j_3}$ are adjacent. $L_{k,j_1,j_3}$ is an interval between $z_{k,j_1,j_3}$ and $t_{k,j_1,j_3}$. $g_{k,j_1,j_3}$ denotes a root locus emitted from a real pole $t_{k,j_1,j_3}$ of (2.2). $s_{k,k+1,j_3}$ and $\overline{s_{k,k+1,j_3}}$ denote a pair of non-real conjugate poles of (2.2).

$g_{k,k+1,j_3}$ and $\overline{g_{k,k+1,j_3}}$ denote the root loci emitted from a pair of non-real conjugate poles $s_{k,k+1,j_3}$ and $\overline{s_{k,k+1,j_3}}$ of (2.2). $q_{k,2j_1-1,j_3}$ and $q_{k,2j_1-2,j_3}$ are two real Fourier critical points in the interval  $L_{k,j_1,j_3}$. $Q_{k,j_3}$ is a real Fourier critical point, $Q_{k,j_3}$ is an intersection point of two root loci $g_{k,1,j_3}$ and $g_{k,2,j_3}$ along with other possible root loci. $Q_a$ denotes a real Fourier critical point, $Q_a$ is generated by the intersection of the root locus of the non-real zero with other root loci. But, it is not a real breakaway point in the linked root loci.
$j_1=1,\cdots, k$. $j_3=1,2,3,4$.

Let $t_{k,0,j_3}$ be a real pole of (2.2) with adjacent real zero $z_{k,0,j_3}$ of (2.2), forming interval $L_{k,0,j_3}$. The pole  $t_{k,0,j_3}$ emits root locus $R_{k,0,j_3}$. Let $R_{k,i_3,j_3}$ be the $r_2$ branches of root loci emitted from the other real zeros $t_{k,i_3,j_3}$ of $p$, where $i_3=1,2,\cdots, r_2$. $r_2\ge 2$ is a positive integer. Let $R_{k,i_4,j_3}$ be the $r_3$ branches of root loci emitted from the other real zeros $t_{k,i_4,j_3}$ of $p$, where $i_4=1,2,\cdots, r_3$. $r_3\ge 2$ is a positive integer.

By Lemma 2.7, the following situation occurs only when root locus $g_{1,2,1}$ intersects  $g_{1,1,1}$, along with other possible root loci, at $Q_{1,1}$ on the real axis:  The root locus $g_{1,1,1}$, emitted from a real zero $t_{1,1,1}$ of $p(s)$, leaves the real axis and enters $\mathbb{C}^{*}$, while $g_{1,2,1}$ originates from a non-real zero $s_{1,2,1}$ of $p(s)$. For their structure, $g_{1,1,1}$ and $g_{1,2,1}$ connect together.

The root loci of (2.2) of $2q\pi$ and $2q\pi+\pi$ degrees are symmetric about the real axis. The conjugate non-real zeros and poles of (2.2) are symmetric about the real axis. The root locus $\overline{g_{1,2,1}}$ emitted from the conjugate zero $\overline{s_{1,2,1}}$, so $g_{1,2,1}$ and $\overline{g_{1,2,1}}$ are symmetric about the real axis. If $g_{1,2,1}$ intersects $g_{1,1,1}$ on the real axis, $\overline{g_{1,2,1}}$ must first intersect $g_{1,2,1}$. After $\overline{g_{1,2,1}}$ and $g_{1,2,1}$ intersect at $q_{1,1,1}$, from the real breakaway point $q_{1,1,1}$, $\overline{g_{1,2,1}}$ and $g_{1,2,1}$ extend oppositely on the real axis. Subsequently, suppose $g_{1,2,1}$  extends to $t_{1,1,1}$, and  intersects $g_{1,1,1}$, along with other possible root loci, at $Q_{1,1}$ again on the real axis.

Therefore, $g_{1,1,1}$ together with $g_{1,2,1}$ and $\overline{g_{1,2,1}}$ form a connected structure. We call them the standard linked root loci.

$g_{k,1,1}$ originates from the real zero $t_{k,1,1}$ of $p(s)$. $g_{k,2,1}$ originates from the real zero $t_{k,2,1}$ of $p(s)$.
On the real axis, $g_{k,2,1}$ and $g_{k,1,1}$ intersect,  along with  other possible root loci, at $Q_{k,1}$. From $Q_{k,1}$, $g_{k,1,1}$ and $g_{k,2,1}$  leave the real axis and enter $\mathbb{C}^{*}$. By Lemma 2.3, there is at least one real zero of (2.2) between $t_{k,1,1}$ and $t_{k,2,1}$  to separate them. Therefore, on the real axis, $g_{k,2,1}$  cannot directly intersect with $g_{k,1,1}$. $g_{k,2,1}$  must leave the real axis and enter $\mathbb{C}^{*}$. $1\le k<n$.

$g_{k,2,1}$ leaves the real axis and enters $\mathbb{C}^{*}$. By Lemma 2.7, this result requires that the third root locus, denoted as $g_{k,3,1}$, intersects $g_{k,2,1}$ at $q_{k,2,1}$ on the real axis. Suppose $g_{k,3,1}$ originates from the real zero $t_{k,3,1}$ of $p$. From $q_{k,2,1}$, $g_{k,2,1}$ and $g_{k,3,1}$ leave the real axis and enter $\mathbb{C}^{*}$.

Because $g_{k,2,1}$ needs to intersect with $g_{k,1,1}$ at $Q_{k,1}$, $g_{k,2,1}$ and $g_{k,3,1}$ leave the real axis and enter $\mathbb{C}^{*}$. Due to symmetry about the real axis, $g_{k,2,1}$ and $g_{k,3,1}$ which have left the real axis must intersect again at a real breakaway point $q_{k,1,1}$. From $q_{k,1,1}$, $g_{k,2,1}$ and $g_{k,3,1}$ extend in opposite directions along the real axis. Suppose $g_{k,2,1}$ extend to $t_{k,1,1}$, then $g_{k,1,1}$ and the root locus $g_{k,2,1}$  extend toward each other and intersect, along with other possible root loci, at $Q_{k,1}$.

By Lemma 2.3, there must be at least one real zero of (2.2) between $t_{k,2,1}$ and $t_{k,3,1}$ to separate them. Therefore, on the real axis, $g_{k,3,1}$ cannot directly intersect $g_{k,2,1}$. $g_{k,3,1}$ must leave the real axis and enter $\mathbb{C}^{*}$.

If the intersecting root loci continue to originate from real zeros, such a process must be carried on.

The intersection points of root loci are breakaway points of (2.2). By Theorem 1.28, the number of breakaway points of (2.2) is finite, implying the number of intersecting root loci is finite. Hence, this process cannot continue indefinitely. However, based on prior results, such an intersection chain can only terminate when a root locus $g_{k,k+1,1}$ from a non-real zero $s_{k,k+1,1}$ of $p(s)$ intersects a root locus $g_{k,k,1}$ from a real zero at $q_{k,2k-2,1}$ on the real axis. $g_{k,k,1}$ originates from the real zero $t_{k,k,1}$ of $p(s)$.

The root locus $g_{k,k+1,1}$ is emitted from the conjugate non-real zero $s_{k,k+1,1}$.
Because the root loci of (2.2) of $2q\pi$ and $2q\pi+\pi$ degrees are symmetrical about the real axis. The conjugate non-real zeros and poles of (2.2) are also symmetrical about the real axis. The root locus $\overline{g_{k,k+1,1}}$ is emitted from the conjugate non-real zero $\overline{s_{k,k+1,1}}$. Therefore, $g_{k,k+1,1}$ and $\overline{g_{k,k+1,1}}$ are symmetrical about the real axis. If $g_{k,k+1,1}$ and $g_{k,k,1}$ intersect on the real axis, then $\overline{g_{k,k+1,1}}$ must intersect with $g_{k,k+1,1}$ first at $q_{k,2k-1,1}$. After $\overline{g_{k,k+1,1}}$ intersects with $g_{k,k+1,1}$, on the real axis, they extend in opposite directions after  $q_{k,2k-1,1}$. Let $g_{k,k+1,1}$ extend to  $t_{k,k,1}$. Subsequently, $g_{k,k+1,1}$ intersects with $g_{k,k,1}$ at $q_{k,2k-2,1}$ on the real axis. From $q_{k,2k-2,1}$, $g_{k,k,1}$ and $g_{k,k+1,1}$ leave the real axis and enter $\mathbb{C}^{*}$.

Therefore, $g_{k,1,1}$, $g_{k,2,1}$, $\cdots$, together with $g_{k,k+1,1}$ and $\overline{g_{k,k+1,1}}$ form a connected structure. We call them the standard linked root loci.

In the standard linked root loci, $k$ root loci emitted from the real poles of (2.2) connect with a pair of root loci emitted from the non-real zeros. No other root loci join the linked root loci. Only the root locus emitted from a real pole of (2.2) intersect with the root locus emitted from the next real pole of (2.2); the root locus emitted from a pair of non-real pole of (2.2) intersects with the root locus emitted from the real pole of (2.2). These intersection of the root loci occurs within the intervals between real poles and real zeros. The rule for forming the standard linked root loci is straightforward.

\begin{definition}
The real zeros $t_{k,j_1,1}$ of $p(s)$ emit the root loci $g_{k,j_1,1}$. $1\le k< n$.  $j_1=1,\cdots, k$. The non-real zeros $s_{k,k+1,1}$ and $\overline{s_{k,k+1,1}}$ emit the root loci $g_{k,k+1,1}$ and $\overline{g_{k,k+1,1}}$. The root loci $g_{k,j_1,1}$, $g_{k,k+1,1}$ and $\overline{g_{k,k+1,1}}$ are called as \emph{the main root loci}.
\end{definition}

\begin{definition}
Except the main root loci, the root loci emitted from the other zeros of $p(s)$ intersect with the main root loci $g_{k,j_1,1}$, $g_{k,k+1,1}$ and $\overline{g_{k,k+1,1}}$  on the real axis, these intersection root loci on the real axis are called as \emph{the additional root loci}.
\end{definition}

\begin{definition}
When $k>1$ and a root locus $g_{k,1,1}$ emitted from a real zero $t_{k,1,1}$ of $p(s)$ leaves real axis and enters $\mathbb{C}^{*}$, beginning from $g_{k,1,1}$, consider root loci $g_{k,j_1,1}$ (for $1\le k<n$, $j_1=1,\cdots,k$, $0<2j_1-2<2k-2$, $0<2j_1-1<2k-1$) emitted from distinct real poles $t_{k,j_1,1}$ of (2.2). At the end point of the structure, there are a pair of root loci $g_{k,k+1,1}$ and $\overline{g_{k,k+1,1}}$  are emitted from a pair of non-real zeros $s_{k,k+1,1}$ and $\overline{s_{k,k+1,1}}$. These root loci intersect at real breakaway points $Q_{k,1}$, $q_{k,2j_1-1,1}$, $q_{k,2j_1-2,1}$ on the real axis. These intersecting root loci ($g_{k,1,1}$, $g_{k,j_1,1}$, $g_{k,k+1,1}$ and $\overline{g_{k,k+1,1}}$) are connected by their real breakaway points $Q_{k,1}$, $q_{k,2j_1-1,1}$, $q_{k,2j_1-2,1}$. Among these linked root loci, no additional root loci join in. The structure formed by these root loci ($g_{k,1,1}$, $g_{k,j_1,1}$, $g_{k,k+1,1}$ and $\overline{g_{k,k+1,1}}$) is called \emph{the standard linked root loci of (2.2)}.
\end{definition}

When $k=1$ and a root locus $g_{1,1,1}$ emitted from a real zero $t_{1,1,1}$ of $p(s)$ leaves real axis and enters $\mathbb{C}^{*}$. $k=1$. Beginning from $g_{1,1,1}$, at the end point of the structure, there are a pair of root loci $g_{1,2,1}$ and $\overline{g_{1,2,1}}$  are emitted from a pair of non-real zeros $s_{1,2,1}$ and $\overline{s_{1,2,1}}$. These root loci intersect at real breakaway points $Q_{1,1}$, $q_{1,1,1}$ on the real axis. These intersecting root loci ($g_{1,1,1}$,  $g_{1,2,1}$ and $\overline{g_{1,2,1}}$) are connected by their real breakaway points $Q_{1,1}$, $q_{1,1,1}$. Among these linked root loci, no additional root loci join in. The structure formed by these root loci ($g_{1,1,1}$, $g_{1,2,1}$ and $\overline{g_{1,2,1}}$) is called \emph{the standard linked root loci of (2.2)}.

Only those root loci that intersect the main root loci on the real axis can belong to the additional root loci. However, if there exist root loci that intersect the root locus $g_{k,1,1}$ at the real breakaway point $Q_{k,1}$, such root loci do not belong to the additional root loci. This is because these root loci intersect only with $g_{k,1,1}$ at the real breakaway point $Q_{k,1}$; they do not intersect the main root loci on the real axis. With the exception of those intersecting at the real breakaway point $Q_{k,1}$, only those root loci that intersect $g_{k,j_1,1}$ on the real axis can be considered as additional root loci.

According to our previous results, when a pair of non-real conjugate zeros of $p(s)$ emit root loci that intersect on the real axis, one of these root loci will intersect with the root locus emitted from a real zero. The root locus emitted from this real zero then leaves the real axis and later intersects with the root locus emitted from another real zero on the real axis. This process continues until the root locus finally reaches a zero of (2.2). This chain of intersections among root loci from non-real and real zeros constitutes the necessary condition for forming a standard linked root loci. The condition for the generation of the standard linked root loci is that two root loci emitted from non-real zeros and several root loci emitted from real zeros intersect on the real axis.

A key question is under what circumstances a standard linked root loci must exist, particularly for the case studied in this paper. Our earlier results show that if all root loci intersecting to form the real breakaway point $Q_{k,1}$ originate from real zeros of $p(s)$, then at least one root locus from a real zero may form a standard linked root loci with other root loci. This result ensures that the conclusions in the remainder of this paper are valid.

Let $g_{2,1,1}$ be the root locus emitted from a real zero $t_{2,1,1}$ of $p$, and let $z_{2,1,1}$ be the real zero of (2.2) adjacent to $t_{2,1,1}$. The root locus $g_{2,1,1}$ would leave the real axis in the interval $L_{2,1,1}$ between $t_{2,1,1}$ and $z_{2,1,1}$ to enter $\mathbb{C}^{*}$. However, if no other root locus intersects $g_{2,1,1}$ in $L_{2,1,1}$, then $g_{2,1,1}$ must end at $z_{2,1,1}$ and remain on the real axis. Thus, for $g_{2,1,1}$ to leave the real axis and enter  $\mathbb{C}^{*}$, another root locus must intersect $g_{2,1,1}$ within $L_{2,1,1}$. We denote this intersecting root locus by $g_{2,2,1}$.

Suppose the root locus $g_{2,2,1}$ departs from another real zero $t_{2,2,1}$ of $p$. Let $z_{2,2,1}$ be the real zero of (2.2) adjacent to $t_{2,2,1}$. $g_{2,2,1}$ leaves the real axis from the interval $L_{2,2,1}$ between $t_{2,2,1}$ and $z_{2,2,1}$ to enter  $\mathbb{C}^{*}$. If no other root loci intersect $g_{2,2,1}$ within $L_{2,2,1}$, then $g_{2,2,1}$ emitted from $t_{2,2,1}$ must end at $z_{2,2,1}$ and remain on the real axis. Therefore, for $g_{2,2,1}$ to leave the real axis and enter  $\mathbb{C}^{*}$, another root locus must intersect $g_{2,2,1}$ within $L_{2,2,1}$. We denote this intersecting root locus by $g_{2,3,1}$.

\begin{lemma}
The root loci $g_{2,1,1}$, $g_{2,2,1}$ and possibly other root loci intersect at a real breakaway point $Q_{2,1}$. At $Q_{2,1}$, $g_{2,1,1}$ leaves the real axis and enters  $\mathbb{C}^{*}$. $g_{2,3,1}$ and $\overline{g_{2,3,1}}$ are emitted from a pair of non-real zeros $s_{2,3,1}$ and $\overline{s_{2,3,1}}$ of $p(s)$, among these linked root loci, no additional root loci join in. Then, the root loci $g_{2,1,1}$, $g_{2,2,1}$, $g_{2,3,1}$, $\overline{g_{2,3,1}}$ form the standard linked root loci. The terminus of this structure consists of $g_{2,3,1}$ or $\overline{g_{2,3,1}}$.
\end{lemma}
\begin{proof}
$g_{2,3,1}$ and $\overline{g_{2,3,1}}$ intersect at a real breakaway point $q_{2,3,1}$ within $L_{2,2,1}$. From $q_{2,3,1}$, $g_{2,3,1}$ and $\overline{g_{2,3,1}}$ extend in opposite directions along the real axis. Let $g_{2,3,1}$ extend to $t_{2,2,1}$, $\overline{g_{2,3,1}}$  maybe ends at $z_{2,2,1}$. $\overline{g_{2,3,1}}$  maybe ends at the other zero of (2.2). $g_{2,2,1}$ and $g_{2,3,1}$  extend toward each other and intersect at a real breakaway point $q_{2,2,1}$ within $L_{2,2,1}$. At $q_{2,2,1}$, $g_{2,2,1}$ and $g_{2,3,1}$  leave the real axis and enter $\mathbb{C}^{*}$.

Due to symmetry about the real axis, $g_{2,2,1}$ and $g_{2,3,1}$  which have left the real axis must intersect again at a real breakaway point $q_{2,1,1}$ within $L_{2,1,1}$. Therefore, $g_{2,2,1}$ and $g_{2,3,1}$  extend in $\mathbb{C}^{*}$ from $q_{2,2,1}$ to $q_{2,1,1}$.

From $q_{2,1,1}$, $g_{2,2,1}$ and $g_{2,3,1}$ extend in opposite directions along the real axis.
Let $g_{2,2,1}$  extend to  $t_{2,1,1}$. $g_{2,3,1}$  maybe ends at $z_{2,1,1}$. $g_{2,3,1}$  maybe ends at the other zero of (2.2). $g_{2,1,1}$ and the root locus $g_{2,2,1}$ extend toward each other and intersect, along with  other possible root loci, at $Q_{2,1}$ in $L_{2,1,1}$.

Therefore, $g_{2,1,1}$, $g_{2,2,1}$, $g_{2,3,1}$, $\overline{g_{2,3,1}}$ form the standard linked root loci. $g_{2,1,1}$ and $g_{2,2,1}$  leave the real axis and enter  $\mathbb{C}^{*}$.

In some cases, $Q_{2,1}$ and $q_{2,1,1}$ may coincide, meaning $g_{2,3,1}$, $g_{2,2,1}$, $g_{2,1,1}$, and possibly other root loci all intersect at $Q_{2,1}$. Similarly, in $L_{2,2,1}$, $q_{2,2,1}$ and $q_{2,3,1}$ may coincide, meaning $g_{2,3,1}$, $\overline{g_{2,3,1}}$, and $g_{2,2,1}$ all intersect at $q_{2,23,1}=q_{2,2,1}=q_{2,3,1}$.
\end{proof}

Let $g_{k,j_1,1}$ be the root locus emitted from a real zero $t_{k,j_1,1}$ of $p$. Let $z_{k,j_1,1}$ be the real zero of (2.2) adjacent to $t_{k,j_1,1}$. $g_{k,j_1,1}$ leaves the real axis at a point in the interval $L_{k,j_1,1}$ between $t_{k,j_1,1}$ and $z_{k,j_1,1}$.  The root loci from the non-real zeros $s_{k,k+1,1}$ and $\overline{s_{k,k+1,1}}$ are $g_{k,k+1,1}$ and $\overline{g_{k,k+1,1}}$.
\begin{lemma}
When a root locus $g_{k,1,1}$ emitted from a real zero of $p(s)$ leaves the real axis and enters $\mathbb{C}^{*}$, then, there must exist a pair of non-real zeros and  $k-1$ other real zeros. These root loci $g_{k,k+1,1}$, $\overline{g_{k,k+1,1}}$ and $g_{k,j_1,1}$ (for $1\le k<n$, $j_1=1,\cdots,k$, $0<2j_1-2<2k-2$, $0<2j_1-1<2k-1$) intersect at real breakaway points $Q_{k,1}$, $q_{k,2j_1-2,1}$,  $q_{k,2j_1-1,1}$ on the real axis.  Among these linked root loci, no additional root loci join in. At the end point of the structure, there are a pair of root loci $g_{k,k+1,1}$ and $\overline{g_{k,k+1,1}}$ are emitted from a pair of non-real zeros $s_{k,k+1,1}$ and $\overline{s_{k,k+1,1}}$.
\end{lemma}

\begin{proof}
In Lemma 2.15, there are four main root loci in the standard linked root loci, and two intervals between zeros and poles of (2.2).

$g_{k,1,1}$, $g_{k,2,1}$, and possibly other root loci intersect at a real breakaway point $Q_{k,1}$  in $L_{k,1,1}$. At $Q_{k,1}$, $g_{k,1,1}$ and $g_{k,2,1}$  leave the real axis and enter $\mathbb{C}^{*}$. The situation of $g_{k,1,1}$ and $g_{k,2,1}$  is identical to the situation given in Lemma 2.15. At the front of Definition 2.12, we have given the result of $k=1$. When $k=2$,  Lemma 2.15 give results. For general $k>2$, as following.

$g_{k,j_1+1,1}$ and $g_{k,j_1+2,1}$ intersect at the real breakaway point $q_{k,2j_1,1}$ in $L_{k,j_1+1,1}$. From $q_{k,2j_1,1}$, $g_{k,j_1+1,1}$ and $g_{k,j_1+2,1}$ leave the real axis and enter   $\mathbb{C}^{*}$. Due to symmetry about the real axis, $g_{k,j_1+1,1}$ and $g_{k,j_1+2,1}$ intersect at  the real breakaway point $q_{k,2j_1-1,1}$  in $L_{k,j_1,1}$. $g_{k,j_1+1,1}$ and $g_{k,j_1+2,1}$ extend in  $\mathbb{C}^{*}$ from $q_{k,2j_1,1}$ to $q_{k,2j_1-1,1}$.

From $q_{k,2j_1-1,1}$,  $g_{k,j_1+1,1}$ and $g_{k,j_1+2,1}$ extend  oppositely on the real axis.
One of them (let it be $g_{k,j_1+1,1}$) extends towards $t_{k,j_1,1}$. $g_{k,j_1+2,1}$ maybe ends at zero  $z_{k,j_1,1}$. $g_{k,j_1+2,1}$ maybe ends at the other zero of (2.2). $t_{k,j_1,1}$ emits $g_{k,j_1,1}$ to interval $L_{k,j_1,1}$. Both $g_{k,j_1,1}$ and $g_{k,j_1+1,1}$ extend towards each other, so $g_{k,j_1+1,1}$ intersects with $g_{k,j_1,1}$ at a real breakaway point $q_{k,2j_1-2,1}$ within $L_{k,j_1,1}$. At point $q_{k,2j_1-2,1}$, $g_{k,j_1,1}$ and $g_{k,j_1+1,1}$ leave the real axis and enter $\mathbb{C}^{*}$.

Due to symmetry about the real axis, $g_{k,j_1,1}$ and $g_{k,j_1+1,1}$ which have left the real axis must intersect again at a real breakaway point $q_{k,2j_1-3,1}$  in $L_{k,j_1-1,1}$. $g_{k,j_1,1}$ and $g_{k,j_1+1,1}$ extend in  $\mathbb{C}^{*}$ from $q_{k,2j_1-2,1}$ to $q_{k,2j_1-3,1}$. $k=3,\cdots,k-2$.

When the root loci intersecting in $L_{k,j_2,1}$ originates from a real zero $t_{k,j_2,1}$. $j_2> k-2$. The root locus $g_{k,j_2,1}$ emitted from $t_{k,j_2,1}$ leaves interval $L_{k,j_2,1}$. However, if no other root locus intersects in $L_{k,j_2,1}$, then $g_{k,j_2,1}$ must end at $z_{k,j_2,1}$ and remain on the real axis. There must be a root locus $g_{k,j_2+1,1}$ intersecting in the interval $L_{k,j_2,1}$. If $g_{k,j_2+1,1}$ comes from a real zero $t_{k,j_2+1,1}$ in another interval $L_{k,j_2+1,1}$. If such a process arises indefinitely, the standard linked root loci cannot end and must continue to extend downward.

The breakaway points of the root loci are the intersection points of the root loci. By Theorem 1.28, the number of finite breakaway points of (2.2) is finite. The number of the intersection points of the root loci of (2.2) is finite. The number of the intersection root loci is finite. Therefore, the intersection of root loci of (2.2) must end. There must exist an intersection of root loci emitted from non-real zeros on the real axis, and root loci emitted from non-real zeros intersect in an interval between a zero and a pole. Therefore, such linked root loci cannot continue infinitely.

If the root locus emitted from a non-real zero $s_{k,k+1,1}$ intersects the interval $L_{k,k,1}$, the standard linked root loci end and cannot extend further downward. The standard linked root loci also end at the interval where the root loci emitted from a non-real zero of $p(s)$ intersect. Conversely, if no root locus emitted from a non-real zero intersects within the interval, the standard linked root loci cannot end either.

Therefore, based on the results we obtained earlier, we have thus proved that there must exist a pair of non-real zeros $s_{k,k+1,1}$ and $\overline{s_{k,k+1,1}}$. The root loci $g_{k,k+1,1}$ and $\overline{g_{k,k+1,1}}$ emitted from $s_{k,k+1,1}$ and $\overline{s_{k,k+1,1}}$ intersect within a certain interval $L_{k,k,1}$.
\end{proof}

\begin{lemma}
When a root locus $g_{k,1,1}$ emitted from a real zero of $p(s)$ leaves the real axis and enters $\mathbb{C}^{*}$, then, there must exist a pair of non-real zeros and  $k-1$ other real zeros. These root loci $g_{k,k+1,1}$, $\overline{g_{k,k+1,1}}$ and $g_{k,j_1,1}$ (for $1\le k<n$, $j_1=1,2,\cdots,k$, $0<2j_1-2<2k-2$, $0<2j_1-1<2k-1$) intersect at real breakaway points $Q_{k,1}$, $q_{k,2j_1-2,1}$, $q_{k,2j_1-1,1}$ on the real axis. Among these linked root loci, no additional root loci join in. These real breakaway points connect $g_{k,1,1}$, $g_{k,j_1,1}$, $g_{k,k+1,1}$, and $\overline{g_{k,k+1,1}}$ together, forming a standard linked root locus. The terminus of this structure consists of $g_{k,k+1,1}$ and $\overline{g_{k,k+1,1}}$.
\end{lemma}

\begin{proof}
According to the condition of this lemma, by Lemma 2.16, at the end point of the structure, there are a pair of root loci $g_{k,k+1,1}$ and $\overline{g_{k,k+1,1}}$ are emitted from a pair of non-real zeros $s_{k,k+1,1}$ and $\overline{s_{k,k+1,1}}$.

The non-real zeros $s_{k,k+1,1}$ and $\overline{s_{k,k+1,1}}$ are in $\mathbb{C}^{*}$. $s_{k,k+1,1}$ and $\overline{s_{k,k+1,1}}$ emit the root loci $g_{k,k+1,1}$ and $\overline{g_{k,k+1,1}}$ in  $\mathbb{C}^{*}$ separately. The root loci $g_{k,k+1,1}$ and $\overline{g_{k,k+1,1}}$ first intersect at the real breakaway point $q_{k,2k-1,1}$ in the interval $L_{k,k,1}$. $g_{k,k+1,1}$ and $\overline{g_{k,k+1,1}}$ depart from $q_{k,2k-1,1}$. By Theorem 1.28, $g_{k,k+1,1}$ and $\overline{g_{k,k+1,1}}$ cannot extend in the same direction, and they must extend in opposite directions along the real axis.

One of them (let it be $g_{k,k+1,1}$) extends towards $t_{k,k,1}$. $\overline{g_{k,k+1,1}}$ maybe ends at  zero  $z_{k,k,1}$. $\overline{g_{k,k+1,1}}$ maybe ends at the other zero of (2.2). Both $g_{k,k,1}$ and $g_{k,k+1,1}$ extend towards each other, so $g_{k,k+1,1}$ intersects with $g_{k,k,1}$ at a real breakaway point $q_{k,2k-2,1}$ within $L_{k,k,1}$. At point $q_{k,2k-2,1}$, $g_{k,k,1}$ and $g_{k,k+1,1}$ leave the real axis and enter $\mathbb{C}^{*}$.

Due to symmetry about the real axis, $g_{k,k+1,1}$ and $g_{k,k,1}$ intersect at the real breakaway point $q_{k,2k-3,1}$ within $L_{k,k-1,1}$. Therefore, $g_{k,k,1}$ and $g_{k,k+1,1}$ extend in  $\mathbb{C}^{*}$ from $q_{k,2k-2,1}$ to $q_{k,2k-3,1}$.

$g_{k,k,1}$ and $g_{k,k+1,1}$ depart from $q_{k,2k-3,1}$. They must extend in opposite directions along the real axis. Let $g_{k,k,1}$ extend to  $t_{k,k-1,1}$. $g_{k,k+1,1}$ maybe ends at $z_{k,k-1,1}$. $g_{k,k+1,1}$ maybe ends at the other zero of (2.2).
Both $g_{k,k-1,1}$ and $g_{k,k,1}$ extend towards each other, so $g_{k,k,1}$ intersects with $g_{k,k-1,1}$ at a real breakaway point $q_{k,2k-4,1}$ within $L_{k,k-1,1}$. At point $q_{k,2k-4,1}$, $g_{k,k-1,1}$ and $g_{k,k,1}$ leave the real axis and enter $\mathbb{C}^{*}$.

Due to symmetry about the real axis, $g_{k,k-1,1}$ and $g_{k,k,1}$ intersect at the real breakaway point $q_{k,2k-5,1}$ within $L_{k,k-2,1}$. Therefore, $g_{k,k-1,1}$ and $g_{k,k,1}$ extend in  $\mathbb{C}^{*}$ from $q_{k,2k-4,1}$ to $q_{k,2k-5,1}$. $g_{k,k,1}$ maybe ends at $z_{k,k-2,1}$.  $g_{k,k,1}$ maybe ends at the other zero of (2.2). Let $g_{k,k-1,1}$ extend to $t_{k,k-2,1}$. $g_{k,k-1,1}$ intersects with $g_{k,k-2,1}$.

When such a situation appears, starting from the root locus $g_{k,1,1}$ emitted from the real zero $t_{k,1,1}$ of $p(s)$, the standard linked root loci terminate and do not extend further. The terminus of the standard linked root loci is composed of the root loci $g_{k,k+1,1}$ and $\overline{g_{k,k+1,1}}$.
The root loci $g_{k,1,1}$, $g_{k,2,1}$, $\cdots$, $g_{k,k+1,1}$, and $\overline{g_{k,k+1,1}}$ form the standard linked root loci.

According to the previously proved result of the standard linked root loci, the root locus $g_{k,1,1}$ at the beginning of the standard linked root loci and the root loci $g_{k,k+1,1}$ and $\overline{g_{k,k+1,1}}$ at the end of the standard linked root loci, by the root loci $g_{k,j_1,1}$ of real zeros in the middle, must be connected in the standard linked root loci. The beginning of the standard linked root loci is a root locus $g_{k,1,1}$ emitted from a real zero $t_{k,1,1}$, and $g_{k,1,1}$ leaves the real axis and enters $\mathbb{C}^{*}$. The end of the standard linked root loci must be a pair of root loci $g_{k,k+1,1}$ and $\overline{g_{k,k+1,1}}$. The root loci $g_{k,k+1,1}$ and $\overline{g_{k,k+1,1}}$ intersect on the real axis, and they are emitted from the non-real zeros $s_{k,k+1,1}$ and $\overline{s_{k,k+1,1}}$.

In Lemma 2.17, there are $k+2$ main root loci in the standard linked root loci, and $k$ intervals between zeros and poles of (2.2).

In some cases, in $L_{k,k,1}$, the real breakaway points $q_{k,2k-1,1}$ and $q_{k,2k-2,1}$ can be the same point. That is, $g_{k,k,1}$, $g_{k,k+1,1}$, $\overline{g_{k,k+1,1}}$ all intersect at $q_{k,2k-1,2k-2,1}=q_{k,2k-1,1}=q_{k,2k-2,1}$. In $L_{k,j_1,1}$, the real breakaway points $q_{k,2j_1-2,1}$ and $q_{k,2j_1-1,1}$ can be the same point. That is, $g_{k,j_1,1}$, $g_{k,j_1+1,1}$, $g_{k,j_1+2,1}$ all intersect at $q_{k,2j_1-2, 2j_1-1,1}=q_{k,2j_1-2,1}=q_{k,2j_1-1,1}$. The real breakaway points $Q_{k,1}$ and $q_{k,1,1}$ can be the same point. That is, $g_{k,3,1}$, $g_{k,2,1}$, $g_{k,1,1}$, and other root loci (if possible) all intersect at $Q_{k,1}$. $j_1=2,\cdots,k$.
\end{proof}
To make it clear for us to explain the concept of two main lines of the standard linked root loci below,  it starts from the terminal root locus and follows all the way back to the beginning root locus.

The non-real zeros $s_{k,k+1,1}$ and $\overline{s_{k,k+1,1}}$ emit the root loci $g_{k,k+1,1}$ and $\overline{g_{k,k+1,1}}$. The root loci $g_{k,k+1,1}$ and $\overline{g_{k,k+1,1}}$  first intersect at the real breakaway point $q_{k,2k-1,1}$ in the interval $L_{k,k,1}$. $g_{k,k+1,1}$ and $\overline{g_{k,k+1,1}}$ departs from $q_{k,2k-1,1}$. $g_{k,k+1,1}$ and $\overline{g_{k,k+1,1}}$  extend in opposite directions along the real axis. $g_{k,k+1,1}$ extends towards $t_{k,k,1}$. $\overline{g_{k,k+1,1}}$ may end at the real zero $z_{k,k,1}$. $\overline{g_{k,k+1,1}}$ maybe leaves the real axis and ends the other zero of (2.2).

$g_{k,k+1,1}$ intersects with $g_{k,k,1}$ at a real breakaway point $q_{k,2k-2,1}$ within $L_{k,k,1}$. At point $q_{k,2k-2,1}$, $g_{k,k,1}$ and $g_{k,k+1,1}$ leave the real axis and enter $\mathbb{C}^{*}$.

Due to symmetry about the real axis, $g_{k,k,1}$ and $g_{k,k+1,1}$ which have left the real axis must intersect again at a real breakaway point $q_{k,2k-3,1}$ within $L_{k,k-1,1}$.

$g_{k,k,1}$ and $g_{k,k+1,1}$ extend in  $\mathbb{C}^{*}$ from $q_{k,2k-2,1}$ to $q_{k,2k-3,1}$. $g_{k,k,1}$ and $g_{k,k+1,1}$  depart from $q_{k,2k-3,1}$.  $g_{k,k,1}$ and $g_{k,k+1,1}$  must extend in opposite directions along the real axis. $g_{k,k,1}$ extends to pole $t_{k,k-1,1}$, and intersects with $g_{k,k-1,1}$ at  $q_{k,2k-4,1}$ within $L_{k,k-1,1}$.
$g_{k,k+1,1}$ maybe ends at the real zero $z_{k,k-1,1}$. $g_{k,k+1,1}$ maybe leaves the real axis and ends the other zero of (2.2).  At point $q_{k,2k-4,1}$, $g_{k,k-1,1}$ and $g_{k,k,1}$ leave the real axis and enter $\mathbb{C}^{*}$.

Due to symmetry about the real axis, $g_{k,k-1,1}$ and $g_{k,k,1}$ which have left the real axis must intersect again at a real breakaway point $q_{k,2k-5,1}$ within $L_{k,k-2,1}$.

$g_{k,k-1,1}$ and $g_{k,k,1}$ extend in  $\mathbb{C}^{*}$ from $q_{k,2k-4,1}$ to $q_{k,2k-5,1}$.

$g_{k,j_1+1,1}$ and $g_{k,j_1+2,1}$ intersect at the real breakaway point $q_{k,2j_1,1}$ in $L_{k,j_1+1,1}$. From $q_{k,2j_1,1}$, $g_{k,j_1+1,1}$ and $g_{k,j_1+2,1}$ leave the real axis and enters $\mathbb{C}^{*}$. $g_{k,j_1+1,1}$ and $g_{k,j_1+2,1}$ intersect at the real breakaway point $q_{k,2j_1-1,1}$  in $L_{k,j_1,1}$. $g_{k,j_1+1,1}$ and $g_{k,j_1+2,1}$ extend in  $\mathbb{C}^{*}$ from $q_{k,2j_1,1}$ to $q_{k,2j_1-1,1}$.

From $q_{k,2j_1-1,1}$, $g_{k,j_1+1,1}$ and $g_{k,j_1+2,1}$ extend in opposite directions along the real axis. $g_{k,j_1+1,1}$ extends towards $t_{k,j_1,1}$. $t_{k,j_1,1}$ emits $g_{k,j_1,1}$ to interval $L_{k,j_1,1}$. $g_{k,j_1+1,1}$ intersects with $g_{k,j_1,1}$ at a real breakaway point $q_{k,2j_1-2,1}$ within $L_{k,j_1,1}$. At point $q_{k,2j_1-2,1}$, $g_{k,j_1,1}$ and $g_{k,j_1+1,1}$ leave the real axis and enter  $\mathbb{C}^{*}$.

Due to symmetry about the real axis, $g_{k,j_1,1}$ and $g_{k,j_1+1,1}$ intersect at  the real breakaway point $q_{k,2j_1-3,1}$  in $L_{k,j_1-1,1}$. $g_{k,j_1,1}$ and $g_{k,j_1+1,1}$ extend in  $\mathbb{C}^{*}$ from $q_{k,2j_1-2,1}$ to $q_{k,2j_1-3,1}$. $g_{k,j_1+2,1}$  maybe ends at the real zero $z_{k,j_1,1}$. $g_{k,j_1+2,1}$  maybe leaves the real axis and ends the other zero of (2.2). $j_1=1,2,\dots,k-2$.

$g_{k,1,1}$ and the root locus $g_{k,2,1}$ intersect, along with other possible root loci, at $Q_{k,1}$  in $L_{k,1,1}$.

The standard linked root loci consist of two main lines with multiple branches. Except for the terminal root loci emitted from a pair of the non-real poles of (2.2), the rest main root loci originate from the real poles of (2.2).

In $\mathbb{C} ^{*}$, the section of the root locus $g_{k,k+1,1}$ from the non-real zero $s_{k,k+1,1}$ to the real breakaway point $q_{k,2k-1,1}$ constitutes the terminus part of two main lines. In $\mathbb{C} ^{*}$, the section of the root locus $\overline{g_{k,k+1,1}}$ from the non-real zero $\overline{s_{k,k+1,1}}$ to $q_{k,2k-1,1}$ constitutes the terminus part of two main lines.

The interval on the real axis from $q_{k,2k-1,1}$ to the real breakaway point $q_{k,2k-2,1}$ constitutes a part of two main lines. In $\mathbb{C} ^{*}$, the section of the root locus $g_{k,k,1}$ from $q_{k,2k-2,1}$ to the real breakaway point $q_{k,2k-3,1}$ constitutes a part of two main lines. In $\mathbb{C} ^{*}$, the section of the root locus $g_{k,k+1,1}$ from $q_{k,2k-2,1}$ to $q_{k,2k-3,1}$ constitutes a part of two main lines.

The interval on the real axis from $q_{k,2k-3,1}$ to the real breakaway point $q_{k,2k-4,1}$ constitutes a part of two main lines. In $\mathbb{C} ^{*}$, the section of the root locus $g_{k,k-1,1}$ from $q_{k,2k-4,1}$ to the real breakaway point $q_{k,2k-5,1}$ constitutes a part of two main lines. In $\mathbb{C} ^{*}$, the section of the root locus $g_{k,k,1}$ from $q_{k,2k-4,1}$ to $q_{k,2k-5,1}$ constitutes a part of two main lines.

In $\mathbb{C} ^{*}$, the section of the root locus $g_{k,j_1+1,1}$ from the real breakaway point $q_{k,2j_1,1}$ to the real breakaway point $q_{k,2j_1-1,1}$ constitutes a part of two main lines. In $\mathbb{C} ^{*}$, the section of the root locus $g_{k,j_1+2,1}$ from the real breakaway point $q_{k,2j_1,1}$ to the real breakaway point $q_{k,2j_1-1,1}$ constitutes a part of two main lines. The interval on the real axis from $q_{k,2j_1-1,1}$ to $q_{k,2j_1-2,1}$ constitutes a part of two main lines.  In $\mathbb{C} ^{*}$, the section of the root locus $g_{k,j_1,1}$ from $q_{k,2j_1-2,1}$ to the real breakaway point $q_{k,2j_1-3,1}$ constitutes a part of two main lines. In $\mathbb{C} ^{*}$, the section of the root locus $g_{k,j_1+1,1}$ from $q_{k,2j_1-2,1}$ to $q_{k,2j_1-3,1}$ constitutes a part of two main lines. $j_1=1,2,\cdots,k-2$.

In $\mathbb{C} ^{*}$, the section of the root locus $g_{k,2,1}$ from $q_{k,2,1}$ to the real breakaway point $q_{k,1,1}$ constitutes a part of two main lines. In $\mathbb{C} ^{*}$, the section of the root locus $g_{k,3,1}$ from $q_{k,2,1}$ to $q_{k,1,1}$ constitutes a part of two main lines. The interval on the real axis from $q_{k,1,1}$ to the real breakaway point $Q_{k,1}$ constitutes a part of two main lines.

All the parts given previously, when connected together, form two main lines of the standard linked root loci. In the standard linked root loci, there are only two main lines. Part of the two main lines is distributed on the real axis. The other part of two main lines distributed in $\mathbb{C}^{*}$, and they are symmetric about the real axis.

The section of $\overline{g_{k,k+1,1}}$ from $q_{k,2k-1,1}$ to $z_{k,k,1}$ constitutes a branch. The section of $g_{k,k,1}$ from $t_{k,k,1}$ to $q_{k,2k-2,1}$ constitutes a branch. And so on $\cdots$. In the standard linked root loci, there are multiple branches.

In Lemma 2.15, in $\mathbb{C} ^{*}$, the section of the root locus $g_{2,3,1}$ from the non-real zero $s_{2,3,1}$ to the real breakaway point $q_{2,3,1}$ constitutes the terminus part of two main lines. In $\mathbb{C} ^{*}$, the section of the root locus $\overline{g_{2,3,1}}$ from the non-real zero $\overline{s_{2,3,1}}$ to $q_{2,3,1}$ constitutes the terminus part of two main lines.

The interval on the real axis from $q_{2,3,1}$ to the real breakaway point $q_{2,2,1}$ constitutes a part of two main lines. In $\mathbb{C} ^{*}$, the section of the root locus $g_{2,2,1}$ from $q_{2,2,1}$ to the real breakaway point $q_{2,1,1}$ constitutes a part of two main lines. In $\mathbb{C} ^{*}$, the section of the root locus $g_{2,3,1}$ from $q_{2,2,1}$ to $q_{2,1,1}$ constitutes a part of two main lines.
The interval on the real axis from $q_{2,1,1}$ to the real breakaway point $Q_{2,1}$ constitutes a part of two main lines.

All the parts given previously, when connected together, form two main lines of the standard linked root loci. In the standard linked root loci, there are only two main lines. Part of the two main lines is distributed on the real axis. The other part of two main lines distributed in $\mathbb{C}^{*}$, and they are symmetric about the real axis.

The section of $\overline{g_{2,3,1}}$ from $q_{2,3,1}$ to $z_{2,2,1}$ constitutes a branch. The section of $g_{2,2,1}$ from $t_{2,2,1}$ to $q_{2,2,1}$ constitutes a branch. The section of $g_{2,3,1}$ from $q_{2,1,1}$ to $z_{2,1,1}$ constitutes a branch. In the standard linked root loci, there are four branches.

Therefore, based on the above explanations of the two main lines of the standard linked root loci, we can define the two main lines of the standard linked root loci of (2.2) as follows.

\begin{definition}
The non-real zeros $s_{k,k+1,1}$ and $\overline{s_{k,k+1,1}}$ emit root loci $g_{k,k+1,1}$ and $\overline{g_{k,k+1,1}}$. In the standard linked root loci, beginning at the non-real zeros $s_{k,k+1,1}$ and $\overline{s_{k,k+1,1}}$, the two lines which $g_{k,k+1,1}$ and $\overline{g_{k,k+1,1}}$ can extend all the way to $Q_{k,1}$ are the two main lines of the standard linked root loci of (2.2).
\end{definition}

\begin{lemma}
In the standard linked root loci, a root locus $g_{k,k+1,1}$ or $\overline{g_{k,k+1,1}}$ originating from a non-real zero $s_{k,k+1,1}$ or $\overline{s_{k,k+1,1}}$ can be distributed on two main lines of the standard linked root loci. The root locus $g_{k,k+1,1}$ or $\overline{g_{k,k+1,1}}$ intersects at the real breakaway point $Q_{k,1}$. The root locus $g_{k,k+1,1}$ or $\overline{g_{k,k+1,1}}$   leaves from $Q_{k,1}$ then enters $\mathbb{C}^{*}$.
\end{lemma}

\begin{proof}
In the explanation of two main lines, for the root loci that eventually reach the real zeros in (2.2) in the standard linked root loci, one of their distribution as follows. The root locus $\overline{g_{k,k+1,1}}$ ends at $z_{k,k,1}$. The root locus $g_{k,l,1}$ ends at $z_{k,l-2,1}$. $l=3,\cdots, k+1$. Based on this distributed way, neither the root locus $g_{k,k+1,1}$ nor $\overline{g_{k,k+1,1}}$ directly intersects at $Q_{k,1}$.

The intersecting root loci are referred to as "common root loci". They possess "common" properties. Beginning at the non-real zeros $s_{k,k+1,1}$ and $\overline{s_{k,k+1,1}}$, the two main lines of the standard linked root loci both extend towards the breakaway point $Q_{k,1}$. Therefore, the root locus emitted from the non-real zero $s_{k,k+1,1}$ or $\overline{s_{k,k+1,1}}$ can be distributed on two main lines of the standard linked root loci. That is, among the root loci intersecting at $Q_{k,1}$, one is $g_{k,k+1,1}$ or $\overline{g_{k,k+1,1}}$.

The root locus $g_{k,k+1,1}$ from a non-real zero $s_{k,k+1,1}$, distributed on two main lines of the standard linked root loci in the following way, and intersect at $Q_{k,1}$

The root locus $\overline{g_{k,k+1,1}}$ ends at $z_{k,k,1}$. The root locus $g_{k,k,1}$ ends at $z_{k,k-1,1}$. The root locus $g_{k,l_1,1}$  ends at $z_{k,l_1-1,1}$. $l_1=3,\cdots, k-1$. $g_{k,3,1}$ ends at $z_{k,2,1}$. $g_{k,2,1}$ ends at $z_{k,1,1}$. Under this case, $g_{k,k+1,1}$ must intersect at $Q_{k,1}$.

The root locus $g_{k,k+1,1}$ originates from the non-real zero $s_{k,k+1,1}$ and extends to $q_{k,2k-1,1}$. At $q_{k,2k-1,1}$, $g_{k,k+1,1}$ enters the real axis. On the real axis, $g_{k,k+1,1}$ extends from $q_{k,2k-1,1}$ to $q_{k,2k-2,1}$. Then $g_{k,k+1,1}$ enters $\mathbb{C}^{*}$ from $q_{k,2k-2,1}$ and extends to $q_{k,2k-3,1}$ in $\mathbb{C}^{*}$. And so on $\cdots$. At $q_{k,3,1}$, $g_{k,k+1,1}$ re-enters the real axis. On the real axis, $g_{k,k+1,1}$ extends from $q_{k,3,1}$ to $q_{k,2,1}$. Then $g_{k,k+1,1}$ enters $\mathbb{C}^{*}$ from $q_{k,2,1}$ and extends to $q_{k,1,1}$ in $\mathbb{C}^{*}$. At $q_{k,1,1}$, $g_{k,k+1,1}$ re-enters the real axis. On the real axis, $g_{k,k+1,1}$ extends from $q_{k,1,1}$ to $Q_{k,1}$. Then $g_{k,k+1,1}$ enters $\mathbb{C}^{*}$ from $Q_{k,1}$.

The root locus $g_{k,k+1,1}$ ends at $z_{k,k,1}$. The root locus $g_{k,k,1}$ ends at $z_{k,k-1,1}$. The root locus $g_{k,l_1,1}$ ends at $z_{k,l_1-1,1}$. $l_1=3,\cdots, k-1$. $g_{k,3,1}$ ends at $z_{k,2,1}$. $g_{k,2,1}$ ends at $z_{k,1,1}$. Under this case, $\overline{g_{k,k+1,1}}$ must intersect at $Q_{k,1}$.

The root locus $\overline{g_{k,k+1,1}}$ originates from the non-real zero $\overline{s_{k,k+1,1}}$ and extends to $q_{k,2k-1,1}$. At $q_{k,2k-1,1}$, $\overline{g_{k,k+1,1}}$ enters the real axis. On the real axis, $\overline{g_{k,k+1,1}}$ extends from $q_{k,2k-1,1}$ to $q_{k,2k-2,1}$. Then $\overline{g_{k,k+1,1}}$ enters $\mathbb{C}^{*}$ from $q_{k,2k-2,1}$ and extends to $q_{k,2k-3,1}$ in $\mathbb{C}^{*}$. And so on $\cdots$. At $q_{k,3,1}$, $\overline{g_{k,k+1,1}}$ re-enters the real axis. On the real axis, $\overline{g_{k,k+1,1}}$ extends from $q_{k,3,1}$ to $q_{k,2,1}$. Then $\overline{g_{k,k+1,1}}$ enters $\mathbb{C}^{*}$ from $q_{k,2,1}$ and extends to $q_{k,1,1}$ in $\mathbb{C}^{*}$. At $q_{k,1,1}$, $\overline{g_{k,k+1,1}}$ re-enters the real axis. On the real axis, $\overline{g_{k,k+1,1}}$ extends from $q_{k,1,1}$ to $Q_{k,1}$. Then $\overline{g_{k,k+1,1}}$ enters $\mathbb{C}^{*}$ from $Q_{k,1}$.

Therefore, in the standard linked root loci,  a root locus $g_{k,k+1,1}$ or $\overline{g_{k,k+1,1}}$ originating from a non-real zero $s_{k,k+1,1}$ or $\overline{s_{k,k+1,1}}$ can be distributed on two main lines of the standard linked root loci. The root locus $g_{k,k+1,1}$ or $\overline{g_{k,k+1,1}}$ intersects at the real breakaway point $Q_{k,1}$. The root locus $g_{k,k+1,1}$ or $\overline{g_{k,k+1,1}}$ leaves from $Q_{k,1}$ then enters $\mathbb{C}^{*}$.
\end{proof}
If no root loci intersect in the interval $L_{k,1,1}$, $g_{k,1,1}$ can only remain on the real axis.
The root locus $g_{k,1,1}$ of the real zero of $p(s)$ leaves the real axis and enters $\mathbb{C}^{*}$, which is only a result. The reason for this result is that $g_{k,k+1,1}$ or $\overline{g_{k,k+1,1}}$ intersects on the real axis, generating the standard linked root loci. The root loci $g_{k,k+1,1}$ and $\overline{g_{k,k+1,1}}$ of the non-real zeros $s_{k,k+1,1}$ and $\overline{s_{k,k+1,1}}$ intersect on the real axis, generating the standard linked root loci. This is the fundamental reason.

On the two main lines of each of the standard linked root loci, the direction of extension is from the non-real zeros $s_{k,k+1,1}$ or $\overline{s_{k,k+1,1}}$ originating, all the way extending to the real breakaway point $Q_{k,1}$. The two main lines can be the root loci $g_{k,k+1,1}$ or $\overline{g_{k,k+1,1}}$ of the non-real zeros $s_{k,k+1,1}$ or $\overline{s_{k,k+1,1}}$. That is, $g_{k,k+1,1}$ or $\overline{g_{k,k+1,1}}$ successively intersects with the root loci of real zeros on the real axis, and lets the two main lines as its root loci, extends all the way to the breakaway point $Q_{k,1}$.

By Lemma 2.7, the following situation occurs only when root locus $g_{1,2,2}$ intersects  $g_{1,1,2}$, along with other possible root loci, at $Q_{1,2}$ on the real axis:  The root locus $g_{1,1,2}$, emitted from a real zero $t_{1,1,2}$ of $p(s)$, leaves the real axis and enters $\mathbb{C}^{*}$, while $g_{1,2,2}$ originates from a non-real zero $s_{1,2,2}$ of $p(s)$. For their structure, $g_{1,1,2}$ and $g_{1,2,2}$ connect together.

The root loci of (2.2) of $2q\pi$ and $2q\pi+\pi$ degrees are symmetric about the real axis. The conjugate non-real zeros and poles of (2.2) are symmetric about the real axis. The root locus $\overline{g_{1,2,2}}$ emitted from the conjugate zero $\overline{s_{1,2,2}}$, so $g_{1,2,2}$ and $\overline{g_{1,2,2}}$ are symmetric about the real axis. If $g_{1,2,2}$ intersects $g_{1,1,2}$ on the real axis, $\overline{g_{1,2,2}}$ must first intersect $g_{1,2,2}$. After $\overline{g_{1,2,2}}$ and $g_{1,2,2}$ intersect at $q_{1,1,2}$, from the real breakaway point $q_{1,1,2}$, $\overline{g_{1,2,2}}$ and $g_{1,2,2}$ extend oppositely on the real axis. Let $g_{1,2,2}$ extend to $t_{1,1,2}$.

Unlike the standard linked root loci, $g_{1,2,2}$  intersects with the additional root loci at real breakaway points on the real axis. The additional root loci do not include $g_{1,1,2}$. After $g_{1,2,2}$ intersects with the additional root loci on the real axis, $g_{1,2,2}$  then intersects $g_{1,1,2}$, along with other possible root loci, at the real breakaway point $Q_{1,2}$. Unlike the standard linked root loci, there are the additional root loci that join into the original standard linked root loci.

Therefore, $g_{1,1,2}$ together with $g_{1,2,2}$ and $\overline{g_{1,2,2}}$ form a connected structure. We call them the first kind of non-standard linked root loci.

$g_{k,1,2}$ originates from the real zero $t_{k,1,2}$ of $p(s)$. $g_{k,2,2}$ originates from the real zero $t_{k,2,2}$ of $p(s)$. On the real axis, $g_{k,2,2}$ intersects  $g_{k,1,2}$, along with other possible root loci, at $Q_{k,2}$. From $Q_{k,2}$, $g_{k,1,2}$ and $g_{k,2,2}$ leave the real axis and enter $\mathbb{C}^{*}$. By Lemma 2.3, there is at least one real zero of (2.2) between $t_{k,1,2}$ and $t_{k,2,2}$ to separate them. Therefore, on the real axis, $g_{k,2,2}$ cannot directly intersect with $g_{k,1,2}$. $g_{k,2,2}$ must leave the real axis and enter $\mathbb{C}^{*}$. $1\le k<n$.

$g_{k,2,2}$ leaves the real axis and enters $\mathbb{C}^{*}$. By Lemma 2.7, this result requires that the third root locus, denoted as $g_{k,3,2}$, intersects $g_{k,2,2}$ at $q_{k,2,2}$ on the real axis. Suppose $g_{k,3,2}$ originates from the real zero $t_{k,3,2}$ of $p$. From $q_{k,2,2}$, $g_{k,2,2}$ and $g_{k,3,2}$ leave the real axis and enter $\mathbb{C}^{*}$.

Because $g_{k,2,2}$ needs to intersect with $g_{k,1,2}$ at $Q_{k,2}$, $g_{k,2,2}$ and $g_{k,3,2}$ leave the real axis and enter $\mathbb{C}^{*}$. Due to symmetry about the real axis, $g_{k,2,2}$ and $g_{k,3,2}$ which have left the real axis must intersect again at a real breakaway point $q_{k,1,2}$. From $q_{k,1,2}$, $g_{k,2,2}$ and $g_{k,3,2}$ extend in opposite directions along the real axis. Let $g_{k,2,2}$ extend toward $t_{k,1,2}$. If $g_{k,2,2}$ doesn't intersect with the additional root loci at the real breakaway points on the real axis, then $g_{k,1,2}$ and the root locus $g_{k,2,2}$ extend toward each other and intersect, along with other possible root loci, at $Q_{k,2}$.

Unlike the standard linked root loci, it is possible that $g_{k,2,2}$ intersects with the additional root loci at the real breakaway points on the real axis. The additional root loci do not include $g_{k,1,2}$. After $g_{k,2,2}$ intersects with the additional root loci on the real axis, then $g_{k,1,2}$ and the root locus $g_{k,2,2}$ extend toward each other and intersect, along with other possible root loci, at $Q_{k,2}$.

By Lemma 2.3, there must be at least one real zero of (2.2) between $t_{k,2,2}$ and $t_{k,3,2}$ to separate them. Therefore, on the real axis, $g_{k,3,2}$ cannot directly intersect $g_{k,2,2}$. $g_{k,3,2}$ must leave the real axis and enter $\mathbb{C}^{*}$.

If the intersecting root loci continue to originate from real zeros, such a process must be carried on.

The intersection points of root loci are breakaway points of (2.2). By Theorem 1.28, the number of breakaway points of (2.2) is finite, implying the number of intersecting root loci is finite. Hence, this process cannot continue indefinitely. However, based on prior results, such an intersection chain can only terminate when a root locus $g_{k,k+1,2}$ from a non-real zero $s_{k,k+1,2}$ of $p(s)$ intersects a root locus $g_{k,k,2}$ from a real zero at $q_{k,2k-2,2}$ on the real axis. $g_{k,k,2}$ originates from the real zero $t_{k,k,2}$ of $p(s)$.

The root locus $g_{k,k+1,2}$ is emitted from the conjugate non-real zero $s_{k,k+1,2}$.
Because the root loci of (2.2) of $2q\pi$ and $2q\pi+\pi$ degrees are symmetrical about the real axis. The conjugate non-real zeros and poles of (2.2) are also symmetrical about the real axis. The root locus $\overline{g_{k,k+1,2}}$ is emitted from the conjugate non-real zero $\overline{s_{k,k+1,2}}$. Therefore, $g_{k,k+1,2}$ and $\overline{g_{k,k+1,2}}$ are symmetrical about the real axis.  If $g_{k,k+1,2}$ and $g_{k,k,2}$ intersect on the real axis, then $\overline{g_{k,k+1,2}}$ must intersect with $g_{k,k+1,2}$ first at the real breakaway point $q_{k,2k-1,2}$.  From $q_{k,2k-1,2}$, they extend in opposite directions.  Let $g_{k,k+1,2}$ extend to  $t_{k,k,2}$.

If $g_{k,k+1,2}$ does not intersect with the additional root loci at the real breakaway points on the real axis. $g_{k,k+1,2}$ then intersects with $g_{k,k,2}$ at a real breakaway point $q_{k,2k-2,2}$. From $q_{k,2k-2,2}$, $g_{k,k,2}$ and $g_{k,k+1,2}$ leave the real axis and enter $\mathbb{C}^{*}$.

Unlike the standard linked root loci, it is possible that $g_{k,k+1,2}$  intersects with the additional root loci at the real breakaway points on the real axis. The additional root loci do not include $g_{k,k,2}$. After $g_{k,k+1,2}$  intersects with the additional root loci on the real axis, $g_{k,k+1,2}$  then intersects with $g_{k,k,2}$ at a real breakaway point $q_{k,2k-2,2}$. From $q_{k,2k-2,2}$, $g_{k,k,2}$ and $g_{k,k+1,2}$ leave the real axis and enter $\mathbb{C}^{*}$.

In the process of forming the  non-standard linked root loci, the root loci emitted from real zeros intersect with the additional root loci and then intersect again with the root loci emitted from real zeros. Alternatively, root loci emitted from non-real zeros intersect with the additional root loci and then intersect with the root loci emitted from real zeros. Such a situation occurs at least once. Unlike the standard linked root loci, there are the additional root loci that join into the original standard linked root loci.

Therefore, $g_{k,1,2}$, $g_{k,2,2}$, $\cdots$, together with $g_{k,k+1,2}$ and $\overline{g_{k,k+1,2}}$ form a connected structure. We call them the first kind of non-standard linked root loci.

In the first kind of non-standard linked root loci, $k$ root loci emitted from the real poles of (2.2) connect with a pair of root loci emitted from the non-real zeros. The additional root loci join the standard linked root loci.

The root locus emitted from a real pole of (2.2) may intersect with the root locus emitted from the next real pole of (2.2). The root locus emitted from a real pole of (2.2) may intersect with the additional root loci, then the root locus emitted from a real pole of (2.2) intersects with the root locus emitted from the next real pole of (2.2). The root locus emitted from a pair of non-real poles of (2.2) may intersect with  the root locus emitted from  the real pole of (2.2).
The root locus emitted from a pair of non-real poles of (2.2) may intersect with the additional root loci, then the root locus emitted from a pair of non-real pole of (2.2) intersects with the root locus emitted from the real pole of (2.2). These intersection of the root loci occurs within the intervals between real poles and real zeros. The rule for forming the first non-standard linked root loci is straightforward.

During the formation of the first kind of non-standard linked root loci, the two root loci $g_{k,j_1+1,2}$ and $g_{k,j_1+2,2}$ emitted from the real zeros ($t_{k,j_1+1,2}$, $t_{k,j_1+2,2}$). These two root loci $g_{k,j_1+1,2}$ and $g_{k,j_1+2,2}$ intersect at the real breakaway point $q_{k,2j_1,2}$. From $q_{k,2j_1,2}$, the two root loci $g_{k,j_1+1,2}$ and $g_{k,j_1+2,2}$ leave the real axis and enter $\mathbb{C}^{*}$. Subsequently, they intersect again at the real breakaway point $q_{k,2j_1-1,2}$ on the real axis. From $q_{k,2j_1-1,2}$, $g_{k,j_1+1,2}$ and $g_{k,j_1+2,2}$ extend further along the real axis in opposite directions. Let $g_{k,j_1+1,2}$ extend toward the direction of $t_{k,j_1,2}$. The root locus $g_{k,j_1+1,2}$ then intersects with the additional root loci on the real axis. This specific manner of intersection between $g_{k,j_1+1,2}$ and the additional root loci is formally defined. We term this manner of intersection the first type of intersection with the additional root loci.

The root loci $g_{k,k+1,2}$ and $\overline{g_{k,k+1,2}}$ are emitted from the non-real zeros ($s_{k,k+1,2}$ and $\overline{s_{k,k+1,2}}$). $g_{k,k+1,2}$ and $\overline{g_{k,k+1,2}}$ necessarily intersect at the real breakaway point $q_{k,2k-1,2}$ on the real axis. From $q_{k,2k-1,2}$, the two root loci $g_{k,k+1,2}$ and $\overline{g_{k,k+1,2}}$ extend further along the real axis in opposite directions. Let $g_{k,k+1,2}$ extend toward the direction of $t_{k,k,2}$. The root locus $g_{k,k+1,2}$ then intersects with the additional root loci on the real axis. After $g_{k,k+1,2}$ intersects with the additional root loci, it subsequently intersects with the root locus $g_{k,k,2}$ on the real axis at the real breakaway point $q_{k,2k-2,2}$. This manner of intersection between $g_{k,k+1,2}$ and the additional root loci is also referred to as the first type of intersection with the additional root loci.

\begin{definition}
The root locus $g_{k,j_1+1,2}$ intersects with the additional root loci on the real axis. After $g_{k,j_1+1,2}$ intersects with the additional root loci, it then intersects with the root locus $g_{k,j_1,2}$ on the real axis at the real breakaway point $q_{k,2j_1-2,2}$. This manner of intersection between $g_{k,j_1+1,2}$ and the additional root loci is referred to as the first type of intersection with the additional root loci. Here, $1\le k<n$, $j_1=1,\cdots, k-1$.
\end{definition}
The root locus $g_{k,k+1,2}$ intersects with the additional root loci on the real axis. After $g_{k,k+1,2}$ intersects with the additional root loci, it then intersects with the root locus $g_{k,k,2}$ on the real axis at the real breakaway point $q_{k,2k-2,2}$. This manner of intersection between $g_{k,k+1,2}$ and the additional root loci is also referred to as the first type of intersection with the additional root loci.

Therefore, in the first kind of non-standard linked root loci, there exists at least one instance of the first type of intersection with the additional root loci.

\begin{definition}
When $k>1$ and a root locus $g_{k,1,2}$ emitted from a real zero $t_{k,1,2}$ of $p(s)$ leaves real axis and enters $\mathbb{C}^{*}$, $1\le k<n$. Beginning from $g_{k,1,2}$, consider root loci $g_{k,j_1,2}$ (for $j_1=1,\cdots,k$, $0<2j_1-2<2k-2$, $0<2j_1-1<2k-1$) emitted from distinct real poles $t_{k,j_1,2}$ of (2.2). At the end point of the structure, there are a pair of root loci $g_{k,k+1,2}$ and $\overline{g_{k,k+1,2}}$ are emitted from a pair of non-real zeros $s_{k,k+1,2}$ and $\overline{s_{k,k+1,2}}$. These root loci intersect at real breakaway points $Q_{k,2}$, $q_{k,2j_1-2,2}$ $q_{k,2j_1-1,2}$ on the real axis. These intersecting root loci ($g_{k,1,2}$, $g_{k,j_1,2}$, $g_{k,k+1,2}$ and $\overline{g_{k,k+1,2}}$) are connected by their real breakaway points $Q_{k,2}$, $q_{k,2j_1-1,2}$, $q_{k,2j_1-2,2}$. There must be the additional root loci that join into the original standard linked root loci. There is at least one instance of the first type of intersection with the additional root loci. The structure formed by these root loci ($g_{k,1,2}$, $g_{k,j_1,2}$, $g_{k,k+1,2}$ and $\overline{g_{k,k+1,2}}$ and the additional root loci) is called \emph{the first kind of non-standard linked root loci of (2.2)}.
\end{definition}
When $k=1$ and a root locus $g_{1,1,2}$ emitted from a real zero $t_{1,1,2}$ of $p(s)$ leaves real axis and enters $\mathbb{C}^{*}$, $k=1$. Beginning from $g_{1,1,2}$, at the end point of the structure, there are a pair of root loci $g_{1,2,2}$ and $\overline{g_{1,2,2}}$ are emitted from a pair of non-real zeros $s_{1,2,2}$ and $\overline{s_{1,2,2}}$. These root loci intersect at real breakaway points $Q_{1,2}$, $q_{1,1,2}$ on the real axis. These intersecting root loci ($g_{1,1,2}$, $g_{1,2,2}$ and $\overline{g_{1,2,2}}$) are connected by their real breakaway points $Q_{1,2}$, $q_{1,1,2}$. There must be the additional root loci that join into the original standard linked root loci. There is at least one instance of the first type of intersection with the additional root loci. The structure formed by these root loci ($g_{1,1,2}$, $g_{1,2,2}$ and $\overline{g_{1,2,2}}$ and the additional root loci) is called \emph{the first kind of non-standard linked root loci of (2.2)}.

In this paper, for all "the additional root loci", we do not need to provide their detailed results. We only express their existence. For all  non-standard linked root loci, we only need to prove that they all end at the root loci of non-real zeros, and the root loci of non-real zeros can all intersect at the real breakaway point $Q_{k,j_3}$, $q_{k,2j_1-1,j_3}$, $q_{k,2j_1-2,j_3}$.  $j_3=2,3,4$. In this paper, we do not need to give the complete structure of the non-standard linked root loci.

Let $g_{2,1,2}$ be the root locus emitted from a real zero $t_{2,1,2}$ of $p$, and let $z_{2,1,2}$ be the real zero of (2.2) adjacent to $t_{2,1,2}$. The root locus $g_{2,1,2}$ would leave the real axis in the interval $L_{2,1,2}$ between $t_{2,1,2}$ and $z_{2,1,2}$ to enter $\mathbb{C}^{*}$. However, if no other root locus intersects $g_{2,1,2}$ in $L_{2,1,2}$, then $g_{2,1,2}$ must end at $z_{2,1,2}$ and remain on the real axis. Thus, for $g_{2,1,2}$ to leave the real axis and enter  $\mathbb{C}^{*}$, another root locus must intersect $g_{2,1,2}$ within $L_{2,1,2}$. We denote this intersecting root locus by $g_{2,2,2}$.

Suppose the root locus $g_{2,2,2}$ departs from another real zero $t_{2,2,2}$ of $p$. Let $z_{2,2,2}$ be the real zero of (2.2) adjacent to $t_{2,2,2}$. $g_{2,2,2}$ leaves the real axis from the interval $L_{2,2,2}$ between $t_{2,2,2}$ and $z_{2,2,2}$ to enter $\mathbb{C}^{*}$. If no other root loci intersect $g_{2,2,2}$ within $L_{2,2,2}$, then $g_{2,2,2}$ emitted from $t_{2,2,2}$ must end at $z_{2,2,2}$ and remain on the real axis. Therefore, for $g_{2,2,2}$ to leave the real axis and enter  $\mathbb{C}^{*}$, another root locus must intersect $g_{2,2,2}$ within $L_{2,2,2}$. We denote this intersecting root locus by $g_{2,3,2}$. $g_{2,3,2}$ and $\overline{g_{2,3,2}}$ are emitted from a pair of the non-real zeros $s_{2,3,2}$ and $\overline{s_{2,3,2}}$ of $p(s)$.

Here, at least one of the two cases must occur. Case 1. $g_{2,2,2}$ intersects with the additional root loci at real breakaway points on the real axis. The additional root loci do not contain $g_{2,1,2}$. After $g_{2,2,2}$  intersects with the additional root loci on the real axis, $g_{2,2,2}$  then intersects with $g_{2,1,2}$ at a real breakaway point.

Case 2. $g_{2,3,2}$  intersects with the additional root loci at real breakaway points on the real axis. The additional root loci do not contain $g_{2,2,2}$. After $g_{2,3,2}$  intersects with the additional root loci on the real axis, $g_{2,3,2}$  then intersects with $g_{2,2,2}$ at a real breakaway point.

\begin{lemma}
The root loci $g_{2,1,2}$, $g_{2,2,2}$, and possibly other root loci intersect at a real breakaway point $Q_{2,2}$. At $Q_{2,2}$, $g_{2,1,2}$ leaves the real axis and enters  $\mathbb{C}^{*}$. $g_{2,3,2}$ and $\overline{g_{2,3,2}}$ are emitted from a pair of non-real zeros $s_{2,3,2}$ and $\overline{s_{2,3,2}}$ of $p(s)$. At least one of the two cases must occur.  Then the root loci $g_{2,1,2}$, $g_{2,2,2}$, $g_{2,3,2}$, $\overline{g_{2,3,2}}$  and the additional root loci form the first kind of non-standard linked root loci. The terminus of this structure consists of $g_{2,3,2}$ or $\overline{g_{2,3,2}}$.
\end{lemma}
\begin{proof}
$g_{2,3,2}$ and $\overline{g_{2,3,2}}$ intersect at a real breakaway point $q_{2,3,2}$ within $L_{2,2,2}$. From $q_{2,3,2}$,  $g_{2,3,2}$ and $\overline{g_{2,3,2}}$ extend  oppositely on the real axis. Let $g_{2,3,2}$ extend to $t_{2,2,2}$. $\overline{g_{2,3,2}}$ maybe ends at $z_{2,2,2}$. $\overline{g_{2,3,2}}$ maybe ends at the other zero of (2.2).

If $g_{2,3,2}$  doesn't intersect with the additional root loci on the real axis, then $g_{2,2,2}$ and $g_{2,3,2}$  extend toward each other and intersect at a real breakaway point $q_{2,2,2}$ within $L_{2,2,2}$.

If $g_{2,3,2}$  intersect with the additional root loci on the real axis, the additional root loci do not contain $g_{2,2,2}$. After $g_{2,3,2}$ intersects with the additional root loci on the real axis, then $g_{2,2,2}$ and $g_{2,3,2}$  extend toward each other and intersect at a real breakaway point $q_{2,2,2}$ within $L_{2,2,2}$.

At $q_{2,2,2}$, $g_{2,2,2}$ and $g_{2,3,2}$ leave the real axis and enter $\mathbb{C}^{*}$.  Due to symmetry about the real axis, $g_{2,2,2}$ and $g_{2,3,2}$  which have left the real axis must intersect again at a real breakaway point $q_{2,1,2}$ in $L_{2,1,2}$. Therefore, $g_{2,2,2}$ and $g_{2,3,2}$ extend in  $\mathbb{C}^{*}$ from $q_{2,2,2}$ to $q_{2,1,2}$.

From $q_{2,1,2}$,  $g_{2,2,2}$ and $g_{2,3,2}$ extend  oppositely on the real axis. Let $g_{2,2,2}$ extend to $t_{2,1,2}$. $g_{2,3,2}$ maybe ends at $z_{2,1,2}$. $g_{2,3,2}$  maybe ends at the other zero of (2.2).

If $g_{2,2,2}$  doesn't intersect with the additional root loci on the real axis, then $g_{2,1,2}$ and the root locus $g_{2,2,2}$  extend toward each other and intersect, along with other possible root loci, at $Q_{2,2}$  in $L_{2,1,2}$.

If $g_{2,2,2}$  intersects with the additional root loci on the real axis, the additional root loci do not contain $g_{2,1,2}$. After $g_{2,2,2}$  intersects with the additional root loci on the real axis, then $g_{2,1,2}$ and the root locus $g_{2,2,2}$  extend toward each other and intersect, along with other possible root loci, at $Q_{2,2}$  in $L_{2,1,2}$.

Therefore, $g_{2,1,2}$, $g_{2,2,2}$, $g_{2,3,2}$, $\overline{g_{2,3,2}}$  and the additional root loci form the first kind of non-standard linked root loci. $g_{2,1,2}$ and $g_{2,2,2}$  leave the real axis and enter  $\mathbb{C}^{*}$.

The same as the case in Lemma 2.15, the real breakaway points may coincide. This overlapping situation has no impact on the proof or the results of this lemma. Since we are unable to provide the situation of "the additional root loci", the exact situation of the real breakaway points overlapping cannot be accurately given either. We only need to know that such a situation exists.
\end{proof}

Let $g_{k,j_1,2}$ be the root locus emitted from a real zero $t_{k,j_1,2}$ of $p$. Let $z_{k,j_1,2}$ be the real zero of (2.2) adjacent to $t_{k,j_1,2}$. $g_{k,j_1,2}$ leaves the real axis at a point in the interval $L_{k,j_1,2}$ between $t_{k,j_1,2}$ and $z_{k,j_1,2}$.  The root loci from the non-real zeros $s_{k,k+1,2}$ and $\overline{s_{k,k+1,2}}$ are $g_{k,k+1,2}$ and $\overline{g_{k,k+1,2}}$.

$g_{k,j_1+1,2}$ and $g_{k,j_1+2,2}$ intersect at a real breakaway point $q_{k,2j_1-1,2}$ within $L_{k,j_1,2}$. From $q_{k,2j_1-1,2}$,  $g_{k,j_1+1,2}$ and $g_{k,j_1+2,2}$ extend  oppositely on the real axis. Let $g_{k,j_1+1,2}$  extend to $t_{k,j_1,2}$.

Here, at least one of the $k$ cases must occur. Case $j_1$. $j_1=1,\cdots,k-1$. $g_{k,j_1+1,2}$ intersects with the additional root loci at real breakaway points on the real axis. The additional root loci do not contain $g_{k,j_1,2}$. After $g_{k,j_1+1,2}$ intersects with the additional root loci on the real axis, $g_{k,j_1+1,2}$ then intersects with $g_{k,j_1,2}$ at a real breakaway point.

$g_{k,k+1,2}$ and $\overline{g_{k,k+1,2}}$ intersect at a real breakaway point $q_{k,2k-1,2}$ within $L_{k,k,2}$. From $q_{k,2k-1,2}$,  $g_{k,k+1,2}$ and $\overline{g_{k,k+1,2}}$ extend  oppositely on the real axis. Let $g_{k,k+1,2}$  extend to $t_{k,k,2}$.

Case $k$. $g_{k,k+1,2}$  intersects with the additional root loci at real breakaway points on the real axis. The additional root loci do not contain $g_{k,k,2}$. After $g_{k,k+1,2}$ intersects with the additional root loci on the real axis, $g_{k,k+1,2}$ then intersects with $g_{k,k,2}$ at a real breakaway point.
\begin{lemma}
When a root locus $g_{k,1,2}$ emitted from a real zero of $p(s)$ leaves the real axis and enters $\mathbb{C}^{*}$, then, there must exist a pair of non-real zeros and  $k-1$ other real zeros.  At least one of the $k$ cases must occur.  These root loci $g_{k,k+1,2}$, $\overline{g_{k,k+1,2}}$ and $g_{k,j_1,2}$ ((for $1\le k<n$, $j_1=1,\cdots,k$, $0<2j_1-2<2k-2$, $0<2j_1-1<2k-1$) intersect at real breakaway points $Q_{k,2}$, $q_{k,2j_1-2,2}$,  $q_{k,2j_1-1,2}$ on the real axis. At the end point of the structure, there are a pair of root loci $g_{k,k+1,2}$ and $\overline{g_{k,k+1,2}}$ are emitted from a pair of non-real zeros $s_{k,k+1,2}$ and $\overline{s_{k,k+1,2}}$.
\end{lemma}

\begin{proof}
In Lemma 2.22, there are four main root loci in the first kind of non-standard linked root loci, and two intervals between zeros and poles of (2.2).

$g_{k,1,2}$, $g_{k,2,2}$, and possibly other root loci intersect at a real breakaway point $Q_{k,2}$  in $L_{k,1,2}$. At $Q_{k,2}$, $g_{k,1,2}$  and $g_{k,2,2}$ leave the real axis and enter  $\mathbb{C}^{*}$. The situation of $g_{k,1,2}$ and $g_{k,2,2}$ is identical to the situation given in Lemma 2.22. At the front of Definition 2.20, we have given the result of $k=1$. When $k=2$,  Lemma 2.22 give results. For general $k>2$, as following.

$g_{k,j_1+1,2}$ and $g_{k,j_1+2,2}$ intersect at the real breakaway point $q_{k,2j_1,2}$ in $L_{k,j_1+1,2}$. From $q_{k,2j_1,2}$, $g_{k,j_1+1,2}$ and $g_{k,j_1+2,2}$ leaves the real axis and enters   $\mathbb{C}^{*}$. Due to symmetry about the real axis, $g_{k,j_1+1,2}$ and $g_{k,j_1+2,2}$ intersect at  the real breakaway point $q_{k,2j_1-1,2}$  in $L_{k,j_1,2}$.

$g_{k,j_1+1,2}$ and $g_{k,j_1+2,2}$ extend in  $\mathbb{C}^{*}$ from $q_{k,2j_1,2}$ to $q_{k,2j_1-1,2}$. $t_{k,j_1,2}$ emits $g_{k,j_1,2}$ to interval $L_{k,j_1,2}$. From $q_{k,2j_1-1,2}$,  $g_{k,j_1+1,2}$ and $g_{k,j_1+2,2}$ extend  oppositely on the real axis. Let $g_{k,j_1+1,2}$  extend to $t_{k,j_1,2}$. $g_{k,j_1+2,2}$  maybe ends at zero  $z_{k,j_1,2}$. $g_{k,j_1+2,2}$  maybe ends at the other zero of (2.2).

If $g_{k,j_1+1,2}$ does not intersect with  the additional root loci at real breakaway points on the real axis, both $g_{k,j_1,2}$ and $g_{k,j_1+1,2}$ extend towards each other, so $g_{k,j_1+1,2}$  intersects with $g_{k,j_1,2}$ at a real breakaway point $q_{k,2j_1-2,2}$ within $L_{k,j_1,2}$. At point $q_{k,2j_1-2,2}$, $g_{k,j_1,2}$ and $g_{k,j_1+1,2}$  leave the real axis and enters  $\mathbb{C}^{*}$.

If $g_{k,j_1+1,2}$  intersects with the additional root loci at real breakaway points on the real axis, the additional root loci do not contain $g_{k,j_1,2}$. After $g_{k,j_1+1,2}$  intersects with the additional root loci on the real axis,  both $g_{k,j_1,2}$ and $g_{k,j_1+1,2}$ extend towards each other, so $g_{k,j_1+1,2}$  intersects with $g_{k,j_1,2}$ at  the real breakaway point $q_{k,2j_1-2,2}$  in $L_{k,j_1,2}$. At point $q_{k,2j_1-2,2}$, $g_{k,j_1,2}$ and $g_{k,j_1+1,2}$  leave the real axis and enters $\mathbb{C}^{*}$.

Due to symmetry about the real axis, $g_{k,j_1,2}$ and $g_{k,j_1+1,2}$ which have left the real axis must intersect again at a real breakaway point $q_{k,2j_1-3,2}$  in $L_{k,j_1-1,2}$. $g_{k,j_1,2}$ and $g_{k,j_1+1,2}$ extend in  $\mathbb{C}^{*}$ from $q_{k,2j_1-2,2}$ to $q_{k,2j_1-3,2}$. $j_1=2,3,\dots,k-2$.

When the root loci intersecting in $L_{k,j_2,2}$ originates from a real zero $t_{k,j_2,2}$. $j_2>k-2$. The root locus $g_{k,j_2,2}$ emitted from $t_{k,j_2,2}$ leaves interval $L_{k,j_2,2}$. However, if no other root locus intersects in $L_{k,j_2,2}$, then $g_{k,j_2,2}$ must end at $z_{k,j_2,2}$ and remain on the real axis. There must be a root locus $g_{k,j_2+1,2}$ intersecting in the interval $L_{k,j_2,2}$. If $g_{k,j_2+1,2}$ comes from a real zero $t_{k,j_2+1,2}$ in another interval $L_{k,j_2+1,2}$. If such a process arises indefinitely, the non-standard linked root loci cannot end and must continue to extend downward.

The breakaway points of the root loci are the intersection points of the root loci. By Theorem 1.28, the number of finite breakaway points of (2.2) is finite. The number of the intersection points of the root loci of (2.2) is finite. The number of the intersection root loci is finite. Therefore, the intersection of root loci of (2.2) must end. There must exist an intersection of root loci emitted from non-real zeros on the real axis, and root loci emitted from non-real zeros intersect in an interval between a zero and a pole. Therefore, such linked root loci cannot continue infinitely.

If the root locus emitted from a non-real zero $s_{k,k+1,2}$ intersects the interval $L_{k,k,2}$, the first kind of non-standard linked root loci end and cannot extend further downward. The first kind of non-standard linked root loci also end at the interval where the root loci emitted from a non-real zero of $p(s)$ intersect. Conversely, if no root locus emitted from a non-real zero intersects within the interval, the first kind of non-standard linked root loci cannot end either.

Therefore, based on the results we obtained earlier, we have thus proved that there must exist a pair of non-real zeros $s_{k,k+1,2}$ and $\overline{s_{k,k+1,2}}$. The root loci $g_{k,k+1,2}$ and $\overline{g_{k,k+1,2}}$ emitted from $s_{k,k+1,2}$ and $\overline{s_{k,k+1,2}}$ intersect within a certain interval $L_{k,k,2}$.
\end{proof}

\begin{lemma}
When a root locus $g_{k,1,2}$ emitted from a real zero of $p(s)$ leaves the real axis and enters $\mathbb{C}^{*}$, then, there must exist a pair of non-real zeros and $k-1$ other real zeros.  At least one of the $k$ cases must occur.  These root loci $g_{k,k+1,2}$, $\overline{g_{k,k+1,2}}$ and $g_{k,j_1}$ (for $1\le k<n$, $j_1=1,\cdots,k$, $0<2j_1-2<2k-2$, $0<2j_1-1<2k-1$) intersect at real breakaway points $Q_{k,2}$, $q_{k,2j_1-2,2}$, $q_{k,2j_1-1,2}$ on the real axis. These real breakaway points connect $g_{k,1,2}$, $g_{k,j_1,2}$, $g_{k,k+1,2}$, and $\overline{g_{k,k+1,2}}$ together, forming a first kind of non-standard linked root locus. The terminus of this structure consists of $g_{k,k+1,2}$ and $\overline{g_{k,k+1,2}}$.
\end{lemma}

\begin{proof}
According to the condition of this lemma, by Lemma 2.23, at the end point of the structure, there are a pair of root loci $g_{k,k+1,2}$ and $\overline{g_{k,k+1,2}}$ are emitted from a pair of non-real zeros $s_{k,k+1,2}$ and $\overline{s_{k,k+1,2}}$.

The non-real zeros $s_{k,k+1,2}$ and $\overline{s_{k,k+1,2}}$ are in  $\mathbb{C}^{*}$. $s_{k,k+1,2}$ and $\overline{s_{k,k+1,2}}$ emit the root locus $g_{k,k+1,2}$ and $\overline{g_{k,k+1,2}}$ in  $\mathbb{C}^{*}$ separately. $g_{k,k+1,2}$ and $\overline{g_{k,k+1,2}}$ intersect at a real breakaway point $q_{k,2k-1,2}$ within $L_{k,k,2}$. $g_{k,k+1,2}$ and $\overline{g_{k,k+1,2}}$ depart from $q_{k,2k-1,2}$. By Theorem 1.28, $g_{k,k+1,2}$ and $\overline{g_{k,k+1,2}}$ cannot extend in the same direction, and they must extend in opposite directions along the real axis. Let $g_{k,k+1,2}$  extend to $t_{k,k,2}$. $\overline{g_{k,k+1,2}}$ maybe ends at  zero  $z_{k,k,2}$. $\overline{g_{k,k+1,2}}$  maybe ends at the other zero of (2.2).

If $g_{k,k+1,2}$ does not intersect with the additional root loci at the real breakaway point on the real axis, both $g_{k,k,2}$ and $g_{k,k+1,2}$  extend towards each other, so $g_{k,k+1,2}$ intersects with $g_{k,k,2}$ at a real breakaway point $q_{k,2k-2,2}$ within $L_{k,k,2}$.

However, the following situation may also occur. $g_{k,k+1,2}$ intersects with the additional root loci at real breakaway points on the real axis. The additional root loci do not contain $g_{k,k,2}$. After $g_{k,k+1,2}$ intersects with the additional root loci on the real axis, both $g_{k,k,2}$ and $g_{k,k+1,2}$  extend towards each other, so $g_{k,k+1,2}$  intersects with $g_{k,k,2}$ at a real breakaway point $q_{k,2k-2,2}$ within $L_{k,k,2}$.

At point $q_{k,2k-2,2}$, $g_{k,k,2}$ and $g_{k,k+1,2}$  leaves the real axis and enters  $\mathbb{C}^{*}$.

Due to symmetry about the real axis, $g_{k,k+1,2}$  and $g_{k,k,2}$ which have left the real axis must intersect again at a real breakaway point $q_{k,2k-3,2}$ within $L_{k,k-1,2}$. Therefore, $g_{k,k,2}$ and $g_{k,k+1,2}$  extend in $\mathbb{C}^{*}$ from $q_{k,2k-2,2}$ to $q_{k,2k-3,2}$.

$g_{k,k,2}$ and $g_{k,k+1,2}$ depart from $q_{k,2k-3,2}$. They must extend in opposite directions along the real axis.  Let $g_{k,k,2}$  extend to $t_{k,k-1,2}$. $g_{k,k+1,2}$ maybe ends at $z_{k,k-1,2}$. $g_{k,k+1,2}$ maybe ends at the other zero of (2.2).

If $g_{k,k,2}$  does not intersect with the additional root loci at the real breakaway point on the real axis, both $g_{k,k-1,2}$ and $g_{k,k,2}$  extend towards each other, so $g_{k,k,2}$  intersects with $g_{k,k-1,2}$ at a real breakaway point $q_{k,2k-4,2}$ within $L_{k,k-1,2}$.

However, the following situation may also occur. $g_{k,k,2}$  intersects with the additional root loci at real breakaway points on the real axis. The additional root loci do not contain $g_{k,k-1,2}$. After $g_{k,k,2}$ intersects with the additional root loci on the real axis, both $g_{k,k-1,2}$ and $g_{k,k,2}$  extend towards each other, so $g_{k,k,2}$  intersects with $g_{k,k-1,2}$ at a real breakaway point $q_{k,2k-4,2}$ within $L_{k,k-1,2}$.

At point $q_{k,2k-4,2}$, $g_{k,k-1,2}$ and $g_{k,k,2}$  leave the real axis and enter  $\mathbb{C}^{*}$.

Due to symmetry about the real axis, $g_{k,k-1,2}$ and $g_{k,k,2}$ which have left the real axis must intersect again at a real breakaway point $q_{k,2k-5,2}$ within $L_{k,k-2,2}$. Therefore, $g_{k,k-1,2}$ and $g_{k,k,2}$ extend in $\mathbb{C}^{*}$ from $q_{k,2k-4,2}$ to $q_{k,2k-5,2}$. Let $g_{k,k-1,2}$ extend to $t_{k,k-2,2}$. $g_{k,k,2}$ maybe ends at $z_{k,k-2,2}$. $g_{k,k,2}$ maybe ends at the other zero of (2.2). $g_{k,k-1,2}$ intersects with $g_{k,k-2,2}$.

When such a situation appears, starting from the root locus $g_{k,1,2}$ emitted from the real zero $t_{k,1,2}$ of $p(s)$, the first kind of non-standard linked root loci terminate and do not extend further. The terminus of the first kind of non-standard linked root loci is composed of the root loci $g_{k,k+1,2}$ and $\overline{g_{k,k+1,2}}$.
The root loci $g_{k,1,2}$, $g_{k,2,2}$, $\cdots$, $g_{k,k+1,2}$, and $\overline{g_{k,k+1,2}}$  and the additional root loci form the first kind of non-standard linked root loci.

In Lemma 2.24, there are $k+2$ main root loci in the first kind of non-standard linked root loci, and $k$ intervals between zeros and poles of (2.2).

According to the previously proved result of the first non-standard linked root loci, the root locus $g_{k,1,2}$ at the beginning of the first kind of non-standard linked root loci and the root loci $g_{k,k+1,2}$ and $\overline{g_{k,k+1,2}}$ at the end of the first kind of non-standard linked root loci, by the root loci $g_{k,j_1,2}$ of real zeros in the middle, must be connected in the first kind of non-standard linked root loci. The beginning of the first kind of non-standard linked root loci is a root locus $g_{k,1,2}$ emitted from a real zero $t_{k,1,2}$, and $g_{k,1,2}$ leaves the real axis and enters  $\mathbb{C}^{*}$. The end of the first kind of non-standard linked root loci must be a pair of root loci $g_{k,k+1,2}$ and $\overline{g_{k,k+1,2}}$. The root loci $g_{k,k+1,2}$ and $\overline{g_{k,k+1,2}}$ intersect on the real axis, and they are emitted from the non-real zeros $s_{k,k+1,2}$ and $\overline{s_{k,k+1,2}}$.

The same as the case in Lemma 2.17, the real breakaway points may coincide. This overlapping situation has no impact on the proof or the results of this lemma. Since we are unable to provide the situation of "the additional root loci," the exact situation of the real breakaway points overlapping cannot be accurately given either. We only need to know that such a situation exists.
\end{proof}
To make it clear for us to explain the concept of two main lines of the first kind of non-standard linked root loci below, it starts from the terminal root locus and follows all the way back to the beginning root locus.

The non-real zeros $s_{k,k+1,2}$ and $\overline{s_{k,k+1,2}}$ emit the root loci $g_{k,k+1,2}$ and $\overline{g_{k,k+1,2}}$. The root loci $g_{k,k+1,2}$ and $\overline{g_{k,k+1,2}}$  first intersect at the real breakaway point $q_{k,2k-1,2}$ in the interval $L_{k,k,2}$. $g_{k,k+1,2}$ and $\overline{g_{k,k+1,2}}$ departs from $q_{k,2k-1,2}$. $g_{k,k+1,2}$ and $\overline{g_{k,k+1,2}}$ extend in opposite directions along the real axis. $g_{k,k+1,2}$ extends towards $t_{k,k,2}$. $\overline{g_{k,k+1,2}}$ may end at the real zero $z_{k,k,2}$. $\overline{g_{k,k+1,2}}$  maybe leaves the real axis and ends the other zero of (2.2).

If $g_{k,k+1,2}$  does not intersect with the additional root loci at the real breakaway point on the real axis,
$g_{k,k+1,2}$  intersects with $g_{k,k,2}$ at a real breakaway point $q_{k,2k-2,2}$ within $L_{k,k,2}$. At point $q_{k,2k-2,2}$, $g_{k,k,2}$ and $g_{k,k+1,2}$ leave the real axis and enters  $\mathbb{C}^{*}$.

However, the following situation may also occur. $g_{k,k+1,2}$  intersects with the additional root loci at real breakaway points on the real axis.  The additional root loci do not contain $g_{k,k,2}$.

In this situation, there is the first type of intersection with the additional root loci. After $g_{k,k+1,2}$  intersects with the additional root loci on the real axis, $g_{k,k+1,2}$  intersects with $g_{k,k,2}$ at a real breakaway point $q_{k,2k-2,2}$ within $L_{k,k,2}$. At point $q_{k,2k-2,2}$, $g_{k,k,2}$ and $g_{k,k+1,2}$ leave the real axis and enters  $\mathbb{C}^{*}$.

Due to symmetry about the real axis, $g_{k,k,2}$ and $g_{k,k+1,2}$ which have left the real axis must intersect again at a real breakaway point $q_{k,2k-3,2}$ within $L_{k,k-1,2}$. $g_{k,k,2}$ and $g_{k,k+1,2}$  extend in  $\mathbb{C}^{*}$ from $q_{k,2k-2,2}$ to $q_{k,2k-3,2}$.

$g_{k,k,2}$ and $g_{k,k+1,2}$  depart from $q_{k,2k-3,2}$.  $g_{k,k,2}$ and $g_{k,k+1,2}$  must extend in opposite directions along the real axis. $g_{k,k,2}$ extends towards $t_{k,k-1,2}$. $g_{k,k+1,2}$  may end at the real zero $z_{k,k-1,2}$. $g_{k,k+1,2}$  maybe leaves the real axis and ends the other zero of (2.2).

If $g_{k,k,2}$  does not intersect with the additional root loci at the real breakaway point on the real axis,
$g_{k,k,2}$  intersects with $g_{k,k-1,2}$ at a real breakaway point $q_{k,2k-4,2}$ within $L_{k,k-1,2}$. At point $q_{k,2k-4,2}$, $g_{k,k-1,2}$ and $g_{k,k,2}$  leave the real axis and enters  $\mathbb{C}^{*}$.

However, the following situation may also occur. $g_{k,k,2}$ intersects with the additional root loci at real breakaway points on the real axis.  The additional root loci do not contain $g_{k,k-1,2}$.

In this situation, there is the first type of intersection with the additional root loci. After $g_{k,k,2}$  intersects with the additional root loci on the real axis, $g_{k,k,2}$  intersects with $g_{k,k-1,2}$ at a real breakaway point $q_{k,2k-4,2}$ within $L_{k,k-1,2}$. At point $q_{k,2k-4,2}$, $g_{k,k-1,2}$ and $g_{k,k,2}$  leave the real axis and enters  $\mathbb{C}^{*}$.

Due to symmetry about the real axis, $g_{k,k-1,2}$ and $g_{k,k,2}$  which have left the real axis must intersect again at a real breakaway point $q_{k,2k-5,2}$ within $L_{k,k-2,2}$. $g_{k,k-1,2}$ and $g_{k,k,2}$  extend in  $\mathbb{C}^{*}$ from $q_{k,2k-4,2}$ to $q_{k,2k-5,2}$.

$g_{k,j_1+1,2}$ and $g_{k,j_1+2,2}$ intersect at the real breakaway point $q_{k,2j_1,2}$ in $L_{k,j_1+1,2}$. From $q_{k,2j_1,2}$, $g_{k,j_1+1,2}$ and $g_{k,j_1+2,2}$ leave the real axis and enters $\mathbb{C}^{*}$. $g_{k,j_1+1,2}$ and $g_{k,j_1+2,2}$ intersect at  the real breakaway point $q_{k,2j_1-1,2}$  in $L_{k,j_1,2}$. $g_{k,j_1+1,2}$ and $g_{k,j_1+2,2}$ extend in  $\mathbb{C}^{*}$ from $q_{k,2j_1,2}$ to $q_{k,2j_1-1,2}$.

From $q_{k,2j_1-1,2}$, $g_{k,j_1+1,2}$ and $g_{k,j_1+2,2}$ extend in opposite directions along the real axis. Let $g_{k,j_1+1,2}$  extends to $t_{k,j_1,2}$. $g_{k,j_1+2,2}$  maybe ends at the real zero $z_{k,j_1,2}$. $g_{k,j_1+2,2}$  maybe leaves the real axis and ends the other zero of (2.2).

If $g_{k,j_1+1,2}$  does not intersect with the additional root loci at real breakaway points on the real axis, $g_{k,j_1+1,2}$ intersects with $g_{k,j_1,2}$  at a real breakaway point $q_{k,2j_1-2,2}$ within $L_{k,j_1,2}$. At point $q_{k,j_1-2,2}$, $g_{k,j_1,2}$ and $g_{k,j_1+1,2}$ leave the real axis and enters $\mathbb{C}^{*}$.

The following situation may occur. $g_{k,j_1+1,2}$  intersects with the additional root loci at real breakaway points on the real axis.  The additional root loci do not contain $g_{k,j_1,2}$. After $g_{k,j_1+1,2}$ intersects with the additional root loci on the real axis, $g_{k,j_1+1,2}$  intersects with $g_{k,j_1,2}$ at  the real breakaway point $q_{k,2j_1-2,2}$  in $L_{k,j_1,2}$. At point $q_{k,2j_1-2,2}$, $g_{k,j_1,2}$ and $g_{k,j_1+1,2}$  leave the real axis and enters $\mathbb{C}^{*}$.

Due to symmetry about the real axis, $g_{k,j_1,2}$ and $g_{k,j_1+1,2}$  intersect at  the real breakaway point $q_{k,2j_1-3,2}$  in $L_{k,j_1-1,2}$. $g_{k,j_1,2}$ and $g_{k,j_1+1,2}$ extend in  $\mathbb{C}^{*}$ from $q_{k,2j_1-2,2}$ to $q_{k,2j_1-3,2}$.
 $j_1=1,2\cdots,k-2$.

$g_{k,1,2}$ and the root locus $g_{k,2,2}$ intersect, along with other possible root loci, at $Q_{k,2}$ in $L_{k,1,2}$.

The first kind of non-standard linked root loci consist of two main lines with multiple branches. Except for the terminal root loci emitted from the non-real poles of (2.2), the rest main root loci  originate from the real poles of (2.2).

In $\mathbb{C}^{*}$, the section of the root locus $g_{k,k+1,2}$ from the non-real zero $s_{k,k+1,2}$ to the real breakaway point $q_{k,2k-1,2}$ constitutes the terminus part of two main lines. In $\mathbb{C}^{*}$, the section of the root locus $\overline{g_{k,k+1,2}}$ from the non-real zero $\overline{s_{k,k+1,2}}$ to $q_{k,2k-1,2}$ constitutes the terminus part of two main lines.

If $g_{k,k+1,2}$ does not intersect with the additional root loci on the real axis, then the interval on the real axis from $q_{k,2k-1,2}$ to the real breakaway point $q_{k,2k-2,2}$ constitutes a part of two main lines.

If $g_{k,k+1,2}$ intersects with the additional root loci on the real axis, then, between $q_{k,2k-1,2}$ and $q_{k,2k-2,2}$, $g_{k,k+1,2}$ intersects with the additional root loci, the part of $g_{k,k+1,2}$  is distributed on the real axis and the other part is symmetrically in $\mathbb{C}^{*}$.

$g_{k,k+1,2}$  extends from $q_{k,2k-1,2}$ to the real breakaway point $q_{k,2k-2,2}$. The part of $g_{k,k+1,2}$ distributed in this interval constitutes a part of two main lines. When $g_{k,k+1,2}$ is distributed in $\mathbb{C}^{*}$, the $g_{k,k+1,2}$ and the part which is symmetric about the real axis are both a part of two main lines.

Here, for both cases, we introduce two concepts. If $g_{k,k+1,2}$ does not intersect with the additional root loci on the real axis, the interval on the real axis from $q_{k,2k-1,2}$ to the real breakaway point $q_{k,2k-2,2}$ is a part of $g_{k,k+1,2}$. We call the segment of the root locus of $g_{k,k+1,2}$ from $q_{k,2k-1,2}$ to the real breakaway point $q_{k,2k-2,2}$ as the root locus segment on the real axis.

If $g_{k,k+1,2}$ intersects with the additional root loci on the real axis, then in the interval on the real axis from $q_{k,2k-1,2}$ to the real breakaway point $q_{k,2k-2,2}$,  $g_{k,k+1,2}$  is distributed on the real axis. $g_{k,k+1,2}$ and the additional root locus are symmetrically in $\mathbb{C}^{*}$. We call the segment of the root locus of $g_{k,k+1,2}$  from $q_{k,2k-1,2}$ to the real breakaway point $q_{k,2k-2,2}$ as the root locus segment on the real axis and outside the real axis.

In the first kind of non-standard linked root loci, there are $k$ intervals between adjacent real zeros and real poles of (2.2). Within each of the interval $L_{k,j_1,2}$ ($j_1=1,\ldots,k$) between the adjacent real zero $z_{k,j_1,2}$ and the real pole $t_{k,j_1,2}$ of (2.2), there exists a segment extending from the real breakaway point $q_{k,2j_1-2,2}$ to the real breakaway point $q_{k,2j_1-1,2}$. For the first kind of non-standard linked root loci, the distribution of root loci on this segment presents one of two possible cases. In contrast, for the standard linked root loci, the distribution of root loci on this segment is uniquely determined. For the first kind of non-standard linked root loci, we provide two specific definitions for the two possible cases of the distribution of root loci on this segment.

\begin{definition}
In $L_{k,j_1,2}$, if $g_{k,j_1+1,2}$  does not intersect with the additional root loci on the real axis, the interval on the real axis from $q_{k,2j_1-1,2}$ to $q_{k,2j_1-2,2}$ is a part of $g_{k,j_1+1,2}$. We refer to the segment of the root locus of $g_{k,j_1+1,2}$ from $q_{k,2j_1-1,2}$ to $q_{k,2j_1-2,2}$ as the root locus segment on the real axis. $1\le k<n$, $j_1=1,\cdots,k-1$.
\end{definition}
In $L_{k,k,2}$, if $g_{k,k+1,2}$ does not intersect with the additional root loci on the real axis, the interval on the real axis from $q_{k,2k-1,2}$ to $q_{k,2k-2,2}$ is a part of $g_{k,k+1,2}$. We refer to the segment of the root locus of $g_{k,k+1,2}$  from $q_{k,2k-1,2}$ to $q_{k,2k-2,2}$ as the root locus segment on the real axis.

\begin{definition}
In $L_{k,j_1,2}$, if $g_{k,j_1+1,2}$  intersects with the additional root loci on the real axis,  between $q_{k,2j_1-1,2}$ and $q_{k,2j_1-2,2}$, the part of $g_{k,j_1+1,2}$  is distributed both on the real axis and the other part is symmetrically to the additional root locus in $\mathbb{C}^{*}$. We refer to the segment of the root locus of $g_{k,j_1+1,2}$  from $q_{k,2j_1-1,2}$ to $q_{k,2j_1-2,2}$ as the root locus segment on the real axis and outside the real axis. $1\le k<n$, $j_1=1,\cdots,k-1$.
\end{definition}
In $L_{k,k,2}$, if $g_{k,k+1,2}$ intersects with the additional root loci on the real axis, between $q_{k,2k-1,2}$ and $q_{k,2k-2,2}$, the part of $g_{k,k+1,2}$  is distributed both on the real axis and the other part is symmetrically to the additional toot locus in $\mathbb{C}^{*}$. We refer to the segment of the root locus of $g_{k,k+1,2}$   from $q_{k,2k-1,2}$ to $q_{k,2k-2,2}$ as the root locus segment on the real axis and outside the real axis.

Therefore, the interval on the real axis from $q_{k,2k-1,2}$ to the real breakaway point $q_{k,2k-2,2}$ constitutes a part of two main lines. This part is the root locus segment on the real axis or the root locus segment on the real axis and outside the real axis.

In $\mathbb{C} ^{*}$, the section of the root locus $g_{k,k,2}$ from $q_{k,2k-2,2}$ to the real breakaway point $q_{k,2k-3,2}$ constitutes a part of two main lines. In $\mathbb{C} ^{*}$, the section of the root locus $g_{k,k+1,2}$ from $q_{k,2k-2,2}$ to $q_{k,2k-3,2}$ constitutes a part of two main lines.

The interval on the real axis from $q_{k,2k-3,2}$ to the real breakaway point $q_{k,2k-4,2}$ constitutes a part of two main lines. This part is the root locus segment on the real axis or the root locus segment on the real axis and outside the real axis.

In $\mathbb{C} ^{*}$, the section of the root locus $g_{k,k-1,2}$ from $q_{k,2k-4,2}$ to the real breakaway point $q_{k,2k-5,2}$ constitutes a part of two main lines. In $\mathbb{C} ^{*}$, the section of the root locus $g_{k,k,2}$ from $q_{k,2k-4,2}$ to $q_{k,2k-5,2}$ constitutes a part of two main lines.

In $\mathbb{C} ^{*}$, the section of the root locus $g_{k,j_1+1,2}$ from the real breakaway point $q_{k,2j_1,2}$ to the real breakaway point $q_{k,2j_1-1,2}$ constitutes a part of two main lines. In $\mathbb{C} ^{*}$, the section of the root locus $g_{k,j_1+2,2}$ from the real breakaway point $q_{k,2j_1,2}$ to the real breakaway point $q_{k,2j_1-1,2}$ constitutes a part of two main lines.

The interval on the real axis from $q_{k,2j_1-1,2}$ to $q_{k,2j_1-2,2}$ constitutes a part of two main lines.  This part is the root locus segment on the real axis or the root locus segment on the real axis and outside the real axis.

In $\mathbb{C} ^{*}$, the section of the root locus $g_{k,j_1,2}$ from $q_{k,2j_1-2,2}$ to the real breakaway point $q_{k,2j_1-3,2}$ constitutes a part of two main lines. In $\mathbb{C} ^{*}$, the section of the root locus $g_{k,j_1+1,2}$ from $q_{k,2j_1-2,2}$ to $q_{k,2j_1-3,2}$ constitutes a part of two main lines. $j_1=1,2,\cdots,k-2$.

In $\mathbb{C} ^{*}$, the section of the root locus $g_{k,2,2}$ from $q_{k,2,2}$ to the real breakaway point $q_{k,1,2}$ constitutes a part of two main lines. In $\mathbb{C} ^{*}$, the section of the root locus $g_{k,3,2}$ from $q_{k,2,2}$ to $q_{k,1,2}$ constitutes a part of two main lines.

The interval on the real axis from $q_{k,1,2}$ to the real breakaway point $Q_{k,2}$ constitutes a part of two main lines. This part is the root locus segment on the real axis or the root locus segment on the real axis and outside the real axis.

All the parts given previously, when connected together, form two main lines of the first kind of non-standard linked root loci. In the first kind of non-standard linked root loci, there are only two main lines. Part of the two main lines is distributed on the real axis. The other part of two main lines distributed in $\mathbb{C}^{*}$, and they are symmetric about the real axis.

The section of $\overline{g_{k,k+1,2}}$ from $q_{k,2k-1,2}$ to $z_{k,k,2}$ constitutes a branch. The section of $g_{k,k,2}$ from $t_{k,k,2}$ to $q_{k,2k-2,2}$ constitutes a branch. And so on $\cdots$. In the first kind of non-standard linked root loci, there are multiple branches.

In Lemma 2.22, in $\mathbb{C} ^{*}$, the section of the root locus $g_{2,3,2}$ from the non-real zero $s_{2,3,2}$ to the real breakaway point $q_{2,3,2}$ constitutes the terminus part of two main lines. In $\mathbb{C} ^{*}$, the section of the root locus $\overline{g_{2,3,2}}$ from the non-real zero $\overline{s_{2,3,2}}$ to $q_{2,3,2}$ constitutes the terminus part of two main lines.

The interval on the real axis from $q_{2,3,2}$ to the real breakaway point $q_{2,2,2}$ constitutes a part of two main lines. This part is the root locus segment on the real axis or the root locus segment on the real axis and outside the real axis.

In $\mathbb{C} ^{*}$, the section of the root locus $g_{2,2,2}$ from $q_{2,2,2}$ to the real breakaway point $q_{2,1,2}$ constitutes a part of two main lines. In $\mathbb{C} ^{*}$, the section of the root locus $g_{2,3,2}$ from $q_{2,2,2}$ to $q_{2,1,2}$ constitutes a part of two main lines.

The interval on the real axis from $q_{2,1,2}$ to the real breakaway point $Q_{2,2}$ constitutes a part of two main lines. This part is the root locus segment on the real axis or the root locus segment on the real axis and outside the real axis.

All the parts given previously, when connected together, form two main lines of the first kind of non-standard linked root loci. In the first kind of non-standard linked root loci, there are only two main lines. Part of the two main lines is distributed on the real axis. The other part of two main lines distributed in $\mathbb{C}^{*}$, and they are symmetric about the real axis.

The section of $\overline{g_{2,3,2}}$ from $q_{2,3,2}$ to $z_{2,2,2}$ constitutes a branch. The section of $g_{2,2,2}$ from $t_{2,2,2}$ to $q_{2,2,2}$ constitutes a branch. The section of $g_{2,3,2}$ from $q_{2,1,2}$ to $z_{2,1,2}$ constitutes a branch. In the first kind of non-standard linked root loci, there are four branches.

Therefore, according to  the above explanations of the two main lines of first kind of non-standard linked root loci given above, we can define the two main lines of the first kind of non-standard linked root loci of (2.2) as follows.
\begin{definition}
The non-real zeros $s_{k,k+1,2}$ and $\overline{s_{k,k+1,2}}$ emit the root loci $g_{k,k+1,2}$ and $\overline{g_{k,k+1,2}}$. In the first kind of non-standard linked root loci, beginning at the non-real zeros $s_{k,k+1,2}$ and $\overline{s_{k,k+1,2}}$, the two lines of $g_{k,k+1,2}$ or $\overline{g_{k,k+1,2}}$ that extend all the way to $Q_{k,2}$ constitute the two main lines of the first kind of non-standard linked root loci of (2.2).
\end{definition}

\begin{lemma}
In the first kind of non-standard linked root loci, a root locus $g_{k,k+1,2}$ or $\overline{g_{k,k+1,2}}$ originating from a non-real zero $s_{k,k+1,2}$ or $\overline{s_{k,k+1,2}}$ can be distributed on two main lines of the first kind of non-standard linked root loci. The root locus $g_{k,k+1,2}$ or $\overline{g_{k,k+1,2}}$ intersects at the real breakaway point $Q_{k,2}$. The root locus $g_{k,k+1,2}$ or $\overline{g_{k,k+1,2}}$ leaves from $Q_{k,2}$ then enters $\mathbb{C}^{*}$.
\end{lemma}

\begin{proof}
In the explanation of two main lines, for the root loci that eventually reach the real zeros in (2.2) in the first kind of non-standard linked root loci, one of their distribution as follows. The root locus $\overline{g_{k,k+1,2}}$ ends at $z_{k,k,2}$.  The root locus $g_{k,l,2}$ ends at $z_{k,l-2,2}$. $l=3,\cdots, k+1$. Based on this distributed way, neither the root locus $g_{k,k+1,2}$ nor $\overline{g_{k,k+1,2}}$ directly intersect at $Q_{k,2}$.

The intersecting root loci are referred to as "common root loci". They possess "common" properties. Beginning at the non-real zeros $s_{k,k+1,2}$ and $\overline{s_{k,k+1,2}}$, the two main lines of the first kind of non-standard linked root loci both extend towards the breakaway point $Q_{k,2}$. Therefore, the root locus emitted from the non-real zero $s_{k,k+1,2}$ or $\overline{s_{k,k+1,2}}$ can be distributed on two main lines of the first kind of non-standard linked root loci. That is, among the root loci intersecting at $Q_{k,2}$, one is $g_{k,k+1,2}$ or $\overline{g_{k,k+1,2}}$.

The root locus $g_{k,k+1,2}$ from a non-real zero $s_{k,k+1,2}$, distributed on two main lines of the first kind of non-standard linked root loci in the following way, and intersect at $Q_{k,2}$.

The root locus $\overline{g_{k,k+1,2}}$ ends at $z_{k,k,2}$. The root locus $g_{k,k,2}$ ends at $z_{k,k-1,2}$. The root locus $g_{k,l_1,2}$  ends at $z_{k,l_1-1,2}$. $l_1=3,\cdots, k-1$. $g_{k,3,2}$ ends at $z_{k,2,2}$. $g_{k,2,2}$ ends at $z_{k,1,2}$. $g_{k,k+1,2}$ must intersect at $Q_{k,2}$.

The root locus $g_{k,k+1,2}$ originates from the non-real zero $s_{k,k+1,2}$ and extends to $q_{k,2k-1,2}$. At $q_{k,2k-1}$, $g_{k,k+1,2}$ enters the real axis. $g_{k,k+1,2}$ extends along the root locus segment on the real axis or the root locus segment on the real axis and outside the real axis from $q_{k,2k-1,2}$ to $q_{k,2k-2,2}$. Then $g_{k,k+1,2}$ enters $\mathbb{C}^{*}$ from $q_{k,2k-2,2}$ and extends to $q_{k,2k-3,2}$ in $\mathbb{C}^{*}$. And so on $\cdots$. At $q_{k,3,2}$, $g_{k,k+1,2}$ re-enters the real axis. $g_{k,k+1,2}$ extends along the root locus segment on the real axis or the root locus segment on the real axis and outside the real axis from $q_{k,3,2}$ to $q_{k,2,2}$. Then $g_{k,k+1,2}$ enters $\mathbb{C}^{*}$ from $q_{k,2,2}$ and extends to $q_{k,1,2}$ in $\mathbb{C}^{*}$. At $q_{k,1,2}$, $g_{k,k+1,2}$ re-enters the real axis. $g_{k,k+1,2}$ extends along the root locus segment on the real axis or the root locus segment on the real axis and  outside the real axis from $q_{k,1,2}$ to $Q_{k,2}$. Then $g_{k,k+1,2}$ enters $\mathbb{C}^{*}$ from $Q_{k,2}$.

The root locus $g_{k,k+1,2}$ ends at $z_{k,k,2}$. The root locus $g_{k,k,2}$ ends at $z_{k,k-1,2}$. The root locus $g_{k,l_1,2}$ ends at $z_{k,l_1-1,2}$. $l_1=3,\cdots, k-1$. $g_{k,3,2}$ ends at $z_{k,2,2}$. $g_{k,2,2}$ ends at $z_{k,1,2}$. Under this case, $\overline{g_{k,k+1,2}}$ must intersect at $Q_{k,2}$.

The root locus $\overline{g_{k,k+1,2}}$ originates from the non-real zero $\overline{s_{k,k+1,2}}$ and extends to $q_{k,2k-1,2}$. At $q_{k,2k-1,2}$, $\overline{g_{k,k+1,2}}$ enters the real axis. $\overline{g_{k,k+1,2}}$ extends along the root locus segment on the real axis or the root locus segment on the real axis and outside the real axis from $q_{k,2k-1,2}$ to $q_{k,2k-2,2}$. Then $\overline{g_{k,k+1,2}}$ enters $\mathbb{C}^{*}$ from $q_{k,2k-2,2}$ and extends to $q_{k,2k-3,2}$ in $\mathbb{C}^{*}$. And so on $\cdots$. At $q_{k,3,2}$, $\overline{g_{k,k+1,2}}$ re-enters the real axis. $\overline{g_{k,k+1,2}}$ extends along the root locus segment on the real axis or the root locus segment on the real axis and outside the real axis from $q_{k,3,2}$ to $q_{k,2,2}$. Then $\overline{g_{k,k+1,2}}$ enters $\mathbb{C}^{*}$ from $q_{k,2,2}$ and extends to $q_{k,1,2}$ in $\mathbb{C}^{*}$. At $q_{k,1,2}$, $\overline{g_{k,k+1,2}}$ re-enters the real axis. $\overline{g_{k,k+1,2}}$ extends along the root locus segment on the real axis or the root locus segment on the real axis and outside the real axis from $q_{k,1,2}$ to $Q_{k,2}$. Then $\overline{g_{k,k+1,2}}$ enters $\mathbb{C}^{*}$ from $Q_{k,2}$.

Therefore, in the first kind of non-standard linked root loci,  a root locus $g_{k,k+1,2}$ or $\overline{g_{k,k+1,2}}$ originating from a non-real zero $s_{k,k+1,2}$ or $\overline{s_{k,k+1,2}}$ can be distributed on two main lines of the first kind of non-standard linked root loci. The root locus $g_{k,k+1,2}$ or $\overline{g_{k,k+1,2}}$  intersects at the real breakaway point $Q_{k,2}$. The root locus $g_{k,k+1,2}$ or $\overline{g_{k,k+1,2}}$  leaves from $Q_{k,2}$ then enters $\mathbb{C}^{*}$.
\end{proof}
If no root loci intersect in the interval $L_{k,1,2}$, $g_{k,1,2}$ can only remain on the real axis. The root locus $g_{k,1,2}$ of the real zero leaves the real axis and enters $\mathbb{C}^{*}$, which is only a result.  The reason for this result is that $g_{k,k+1,2}$ or $\overline{g_{k,k+1,2}}$ intersects on the real axis, generating the first kind of non-standard linked root loci. The root loci $g_{k,k+1,2}$ and $\overline{g_{k,k+1,2}}$ of the non-real zeros $s_{k,k+1,2}$ and $\overline{s_{k,k+1,2}}$ intersect on the real axis, generating the first kind of non-standard linked root loci. This is the fundamental reason.

On the two main lines of each of the first kind of non-standard linked root loci, the direction of extension is from the non-real zeros $s_{k,k+1,2}$ or $\overline{s_{k,k+1,2}}$ originating, all the way extending to the real breakaway point $Q_{k,2}$. The two main lines can be the root loci $g_{k,k+1,2}$ or $\overline{g_{k,k+1,2}}$ of the non-real zeros $s_{k,k+1,2}$ or $\overline{s_{k,k+1,2}}$. That is, $g_{k,k+1,2}$ or $\overline{g_{k,k+1,2}}$ successively intersects with the root loci of real zeros on the real axis, and lets the two main lines as its root loci, extends all the way to the breakaway point $Q_{k,2}$.

The root locus $g_{1,1,3}$ is emitted from a real zero $t_{1,1,3}$ of $p(s)$. The root loci $g_{1,2,3}$ and $\overline{g_{1,2,3}}$ are separately emitted from the non-real zero $s_{1,2,3}$ and $\overline{s_{1,2,3}}$. The root loci of (2.2) of $2q\pi$ and $2q\pi+\pi$ degrees are symmetric about the real axis. The conjugate non-real zeros and poles of (2.2) are symmetric about the real axis. $g_{1,2,3}$ and $\overline{g_{1,2,3}}$ are symmetric about the real axis. After $\overline{g_{1,2,3}}$ and $g_{1,2,3}$ intersect at $q_{1,1,3}$, from the real breakaway point $q_{1,1,3}$, $\overline{g_{1,2,3}}$ and $g_{1,2,3}$ extend oppositely on the real axis.  Let $g_{1,2,3}$ extend to $t_{1,1,3}$.

Unlike the standard linked root loci and the first kind of non-standard linked root loci, firstly, $g_{1,1,3}$ intersects with the additional root loci at real breakaway points on the real axis. The additional root loci do not include $g_{1,2,3}$. After $g_{1,1,3}$ intersects with the additional root loci on the real axis, $g_{1,1,3}$  then intersects  $g_{1,2,3}$, along with other possible root loci, at the real breakaway point $Q_{1,3}$.
Unlike the standard linked root loci, there are the additional root loci that join into the original standard linked root loci.

Therefore, $g_{1,1,3}$ together with $g_{1,2,3}$ and $\overline{g_{1,2,3}}$ form a connected structure. We call them the second kind of non-standard linked root loci.

$g_{k,1,3}$ originates from the real zero $t_{k,1,3}$ of $p(s)$. $g_{k,2,3}$ originates from the real zero $t_{k,2,3}$ of $p(s)$. If  $g_{k,1,3}$ doesn't intersect with the additional root loci at real breakaway points on the real axis. On the real axis, $g_{k,2,3}$ intersects $g_{k,1,3}$, along with other possible root loci, at $Q_{k,3}$. From $Q_{k,3}$, $g_{k,1,3}$ and $g_{k,2,3}$ leave the real axis and enter $\mathbb{C}^{*}$. $1\le k<n$.

Unlike the standard linked root loci and the first kind of non-standard linked root loci, it is possible that $g_{k,1,3}$ intersects with the additional root loci at real breakaway points on the real axis. The additional root loci do not include $g_{k,2,3}$. After $g_{k,1,3}$ intersects with the additional root loci on the real axis, $g_{k,1,3}$  then intersects $g_{k,2,3}$, along with other possible root loci,  at the real breakaway point $Q_{k,3}$. From $Q_{k,3}$, $g_{k,1,3}$ and $g_{k,2,3}$ leave the real axis and enters $\mathbb{C}^{*}$.

By Lemma 2.3, there is at least one real zero of (2.2) between $t_{k,1,3}$ and $t_{k,2,3}$ to separate them. Therefore, on the real axis, $g_{k,2,3}$ cannot directly intersect with $g_{k,1,3}$. $g_{k,2,3}$ must leave the real axis and enter $\mathbb{C}^{*}$.

If  $g_{k,2,3}$ doesn't intersect with the additional root loci at real breakaway points on the real axis.  $g_{k,2,3}$ leaves the real axis and enters $\mathbb{C}^{*}$. By Lemma 2.7, this result requires that the third root locus, denoted as $g_{k,3,3}$, intersects $g_{k,2,3}$ at $q_{k,2,3}$ on the real axis. Suppose $g_{k,3,3}$ originates from the real zero $t_{k,3,3}$ of $p$. From $q_{k,2,3}$, $g_{k,2,3}$ and $g_{k,3,3}$ leave the real axis and enter $\mathbb{C}^{*}$.

If $g_{k,2,3}$ intersects with the additional root loci at real breakaway points on the real axis. The additional root loci do not include $g_{k,3,3}$. After $g_{k,2,3}$ intersects with the additional root loci on the real axis, $g_{k,2,3}$ leaves the real axis and enters $\mathbb{C}^{*}$. By Lemma 2.7, this result requires that the third root locus, denoted as $g_{k,3,3}$, intersects $g_{k,2,3}$ at $q_{k,2,3}$ on the real axis. Suppose $g_{k,3,3}$ originates from the real zero $t_{k,3,3}$ of $p$. From $q_{k,2,3}$, $g_{k,2,3}$ and $g_{k,3,3}$ leave the real axis and enter $\mathbb{C}^{*}$.

Because $g_{k,2,3}$ needs to intersect with $g_{k,1,3}$ at $Q_{k,3}$, $g_{k,2,3}$ and $g_{k,3,3}$ leave the real axis and enter $\mathbb{C}^{*}$. Due to symmetry about the real axis, $g_{k,2,3}$ and $g_{k,3,3}$ which have left the real axis must intersect again at a real breakaway point $q_{k,1,3}$. From $q_{k,1,3}$, $g_{k,2,3}$ and $g_{k,3,3}$ extend in opposite directions along the real axis. Let $g_{k,2,3}$  extends toward $t_{k,1,3}$. Then $g_{k,1,3}$ and the root locus $g_{k,2,3}$ extend toward each other and intersect, along with any other possible root loci, at $Q_{k,3}$.

By Lemma 2.3, there must be at least one real zero of (2.2) between $t_{k,2,3}$ and $t_{k,3,3}$ to separate them. Therefore, on the real axis, $g_{k,3,3}$ cannot directly intersect $g_{k,2,3}$. $g_{k,3,3}$ must leave the real axis and enter $\mathbb{C}^{*}$.

If the intersecting root loci continue to originate from real zeros, such a process must be carried on.

The intersection points of root loci are breakaway points of (2.2). By Theorem 1.28, the number of breakaway points of (2.2) is finite, implying the number of intersecting root loci is finite. Hence, this process cannot continue indefinitely. However, based on prior results, such an intersection chain can only terminate when a root locus $g_{k,k+1,3}$ from a non-real zero $s_{k,k+1,3}$ of $p(s)$ intersects a root locus $g_{k,k,3}$ from a real zero at $q_{k,2k-2,3}$ on the real axis. $g_{k,k,3}$ originates from the real zero $t_{k,k,3}$ of $p(s)$.

The root locus $g_{k,k+1,3}$ is emitted from the conjugate non-real zero $s_{k,k+1,3}$.
Because the root loci of (2.2) of $2q\pi$ and $2q\pi+\pi$ degrees are symmetrical about the real axis. The conjugate non-real zeros and poles of (2.2) are also symmetrical about the real axis. The root locus $\overline{g_{k,k+1,3}}$ is emitted from the conjugate non-real zero $\overline{s_{k,k+1,3}}$. Therefore, $g_{k,k+1,3}$ and $\overline{g_{k,k+1,3}}$ are symmetrical about the real axis.

If $g_{k,k+1,3}$ and $g_{k,k,3}$ intersect on the real axis, then $\overline{g_{k,k+1,3}}$ must intersect with $g_{k,k+1,3}$ first at $q_{k,2k-1,3}$. After $\overline{g_{k,k+1,3}}$ intersects with $g_{k,k+1,3}$, on the real axis, from $q_{k,2k-1,3}$, they extend in opposite directions. Let $g_{k,k+1,3}$  extend toward $t_{k,k,3}$.

If $g_{k,k,3}$ does not intersect with the additional root loci on the real axis, $g_{k,k+1,3}$  then intersects with $g_{k,k,3}$ at a real breakaway point $q_{k,2k-2,3}$. From $q_{k,2k-2,3}$, $g_{k,k,3}$ and $g_{k,k+1,3}$ leave the real axis and enter $\mathbb{C}^{*}$.

If $g_{k,k,3}$ intersects with the additional root loci at the real breakaway points on the real axis. The additional root loci do not include $g_{k,k+1,3}$. After $g_{k,k,3}$ intersects with the additional root loci on the real axis, $g_{k,k+1,3}$ then intersects with $g_{k,k,3}$ at a real breakaway point $q_{k,2k-2,3}$. From $q_{k,2k-2,3}$, $g_{k,k,3}$ and $g_{k,k+1,3}$ leave the real axis and enter $\mathbb{C}^{*}$.

In the process of forming the non-standard linked root loci, a root locus emitted from a real zero intersects with the additional root loci and then intersects again with a root locus emitted from a real zero. Alternatively, a root locus emitted from a real zero intersects with the additional root loci and then intersects with a root locus emitted from a non-real zero. Such a scenario occurs at least once. There are the additional root loci that join into the original standard linked root loci. Therefore, $g_{k,1,3}$, $g_{k,2,3}$, $\cdots$, together with $g_{k,k+1,3}$ and $\overline{g_{k,k+1,3}}$ form a connected structure. We call them the second kind of non-standard linked root loci.

In the  second kind of non-standard linked root loci, $k$ root loci emitted from the real poles of (2.2) connect with a pair of root loci emitted from the non-real zeros. The additional root loci join the standard linked root loci.

The root locus emitted from a real pole of (2.2) may intersect with the root locus emitted from the next real pole of (2.2). The root locus emitted from a real pole of (2.2) may intersect with the additional root loci, then the root locus emitted from a real pole of (2.2) intersects with the root locus emitted from the next real pole of (2.2). Or, the root locus emitted from a real pole of (2.2) intersects with the root locus emitted from a pair of non-real poles of (2.2). The root locus emitted from a real pole of (2.2) may intersect with the additional root loci, then the root locus emitted from a real pole of (2.2) intersects with the root locus emitted from a pair of non-real poles of (2.2). These intersection of the root loci occurs within the intervals between real poles and real zeros. The rule for forming the second kind of non-standard linked root loci is straightforward.

In the process of forming the second kind of non-standard linked root loci, consider a root locus $g_{k,j_1,3}$ emitted from the real zero $t_{k,j_1,3}$, and two root loci $g_{k,j_1+1,3}$ and $g_{k,j_1+2,3}$ emitted from the real zeros ($t_{k,j_1+1,3}$, $t_{k,j_1+2,3}$). These two root loci $g_{k,j_1+1,3}$ and $g_{k,j_1+2,3}$ intersect at the real breakaway point $q_{k,2j_1,3}$. From $q_{k,2j_1,3}$, the two root loci $g_{k,j_1+1,3}$ and $g_{k,j_1+2,3}$ leave the real axis and enter $\mathbb{C}^{*}$. Subsequently, the root loci $g_{k,j_1+1,3}$ and $g_{k,j_1+2,3}$ intersect on the real axis at the real breakaway point $q_{k,2j_1-1,3}$, and then extend further along the real axis in opposite directions. Let $g_{k,j_1+1,3}$ extend toward $t_{k,j_1,3}$. The root locus $g_{k,j_1,3}$ intersects with the additional root loci on the real axis. After this intersection, $g_{k,j_1,3}$ then intersects with the root locus $g_{k,j_1+1,3}$ on the real axis at the real breakaway point $q_{k,2j_1-2,3}$. This specific manner of intersection between $g_{k,j_1,3}$ and the additional root loci is formally termed. We refer to this intersection between $g_{k,j_1,3}$ and the additional root loci as the second type of intersection with the additional root loci. Here,  $j_1= 1, \cdots, k-1$.

In the process of forming the second kind of non-standard linked root loci, consider a root locus $g_{k,k,3}$ emitted from the real zero $t_{k,k,3}$, and two root loci $g_{k,k+1,3}$ and $\overline{g_{k,k+1,3}}$ emitted from the non-real zeros ($s_{k,k+1,3}$ or $\overline{s_{k,k+1,3}}$). The root loci $g_{k,k+1,3}$ and $\overline{g_{k,k+1,3}}$ intersect on the real axis at the real breakaway point $q_{k,2k-1,3}$, and then extend further along the real axis in opposite directions. Let $g_{k,k+1,3}$ extend toward $t_{k,k,3}$. The root locus $g_{k,k,3}$ intersects with the additional root loci on the real axis. After this intersection, $g_{k,k,3}$ then intersects with the root locus $g_{k,k+1,3}$ on the real axis at the real breakaway point $q_{k,2k-2,3}$. This specific manner of intersection between $g_{k,k,3}$ and the additional root loci is formally termed. We refer to this intersection between $g_{k,k,3}$ and the additional root loci as the second type of intersection with the additional root loci.

\begin{definition}
The root locus $g_{k,j_1,3}$ intersects with the additional root loci on the real axis. After $g_{k,j_1,3}$ intersects with the additional root loci on the real axis, it subsequently intersects with $g_{k,j_1+1,3}$ on the real axis at the real breakaway point $q_{k,2j_1-2,3}$. This pattern of intersection with the additional root loci is defined as the second type of intersection with the additional root loci, where $1\le k<n$, $j_1 = 1, \cdots, k-1$.
\end{definition}
The root locus $g_{k,k,3}$ intersects with the additional root loci on the real axis. After $g_{k,k,3}$ intersects with the additional root loci on the real axis, it subsequently intersects with $g_{k,k+1,3}$ on the real axis at the real breakaway point $q_{k,2k-2,3}$. This pattern of intersection with the additional root loci is also defined as the second type of intersection with the additional root loci.

Therefore, in the second kind of non-standard linked root loci, there exists at least one instance of the second type of intersection with the additional root loci.

\begin{definition}
When $k>1$ and a root locus $g_{k,1,3}$ emitted from a real zero $t_{k,1,3}$ of $p(s)$ leaves real axis and enters $\mathbb{C}^{*}$. Beginning from $g_{k,1,3}$, consider root loci $g_{k,j_1,3}$ (for $1\le k<n$, $j_1=1,\cdots,k$, $0<2j_1-2<2k-2$, $0<2j_1-1<2k-1$) emitted from distinct real poles $t_{k,j_1,3}$ of (2.2). At the end point of the structure, there are a pair of root loci $g_{k,k+1,3}$ and $\overline{g_{k,k+1,3}}$ are emitted from a pair of non-real zeros $s_{k,k+1,3}$ and $\overline{s_{k,k+1,3}}$. These root loci intersect at real breakaway points $Q_{k,3}$, $q_{k,2j_1-2,3}$ $q_{k,2j_1-1,3}$ on the real axis. These intersecting root loci ($g_{k,1,3}$, $g_{k,j_1,3}$, $g_{k,k+1,3}$ and $\overline{g_{k,k+1,3}}$) are connected by their real breakaway points $Q_{k,3}$, $q_{k,2j_1-1,3}$, $q_{k,2j_1-2,3}$. There must be the additional root loci that join into the original standard linked root loci. There is at least one instance of the second type of intersection with the additional root loci. The structure formed by these root loci ($g_{k,1,3}$, $g_{k,j_1,3}$, $g_{k,k+1,3}$ and $\overline{g_{k,k+1,3}}$ and the additional root loci) is called \emph{the second kind of non-standard linked root loci of (2.2)}.
\end{definition}
When $k=1$ and a root locus $g_{1,1,3}$ emitted from a real zero $t_{1,1,3}$ of $p(s)$ leaves real axis and enters $\mathbb{C}^{*}$, $k=1$. Beginning from $g_{1,1,3}$, at the end point of the structure, there are a pair of root loci $g_{1,2,3}$ and $\overline{g_{1,2,3}}$ are emitted from a pair of non-real zeros $s_{1,2,3}$ and $\overline{s_{1,2,3}}$. These root loci intersect at real breakaway points $Q_{1,3}$, $q_{1,1,3}$ on the real axis. These intersecting root loci ($g_{1,1,3}$, $g_{1,2,3}$ and $\overline{g_{1,2,3}}$) are connected by their real breakaway points $Q_{1,3}$, $q_{1,1,3}$. There must be the additional root loci that join into the original standard linked root loci. There is at least one instance of the second type of intersection with the additional root loci. The structure formed by these root loci ($g_{1,1,3}$, $g_{1,2,3}$ and $\overline{g_{1,2,3}}$ and the additional root loci) is called \emph{the second kind of non-standard linked root loci of (2.2)}.

Let $g_{2,1,3}$ be the root locus emitted from a real zero $t_{2,1,3}$ of $p$, and let $z_{2,1,3}$ be the real zero of (2.2) adjacent to $t_{2,1,3}$. The root locus $g_{2,1,3}$ would leave the real axis in the interval $L_{2,1,3}$ between $t_{2,1,3}$ and $z_{2,1,3}$ to enter $\mathbb{C}^{*}$. However, if no other root locus intersects $g_{2,1,3}$ in $L_{2,1,3}$, then $g_{2,1,3}$ must end at $z_{2,1,3}$ and remain on the real axis. Thus, for $g_{2,1,3}$ to leave the real axis and enter  $\mathbb{C}^{*}$, another root locus must intersect $g_{2,1,3}$ within $L_{2,1,3}$. We denote this intersecting root locus by $g_{2,2,3}$.

Suppose the root locus $g_{2,2,3}$ departs from another real zero $t_{2,2,3}$ of $p$. Let $z_{2,2,3}$ be the real zero of (2.2) adjacent to $t_{2,2,3}$. $g_{2,2,3}$ leaves the real axis from the interval $L_{2,2,3}$ between $t_{2,2,3}$ and $z_{2,2,3}$ to enter $\mathbb{C}^{*}$. If no other root loci intersect $g_{2,2,3}$ within $L_{2,2,3}$, then $g_{2,2,3}$ emitted from $t_{2,2,3}$ must end at $z_{2,2,3}$ and remain on the real axis. Therefore, for $g_{2,2,3}$ to leave the real axis and enter  $\mathbb{C}^{*}$, another root locus must intersect $g_{2,2,3}$ within $L_{2,2,3}$. We denote this intersecting root locus by $g_{2,3,3}$. $g_{2,3,3}$ and $\overline{g_{2,3,3}}$ are emitted from a pair of the non-real zeros $s_{2,3,3}$ and $\overline{s_{2,3,3}}$ of $p(s)$.

Here, at least one of the two cases must occur. Case 1. $g_{2,1,3}$ intersects with the additional root loci at real breakaway points on the real axis. The additional root loci do not contain $g_{2,2,3}$. After $g_{2,1,3}$ intersects with the additional root loci on the real axis, $g_{2,1,3}$ then intersects with $g_{2,2,3}$ at a real breakaway point.

Case 2. $g_{2,2,3}$  intersects with the additional root loci at real breakaway points on the real axis. The additional root loci do not contain $g_{2,3,3}$. After $g_{2,2,3}$ intersects with the additional root loci on the real axis, $g_{2,2,3}$  then intersects with $g_{2,3,3}$ at a real breakaway point.
\begin{lemma}
The root loci $g_{2,1,3}$, $g_{2,2,3}$, and possibly other root loci intersect at a real breakaway point $Q_{2,3}$. At $Q_{2,3}$, $g_{2,1,3}$ leaves the real axis and enters  $\mathbb{C}^{*}$. $g_{2,3,3}$ and $\overline{g_{2,3,3}}$ are emitted from a pair of the non-real zeros $s_{2,3,3}$ and $\overline{s_{2,3,3}}$ of $p(s)$. At least one of the two cases must occur.  Then the root loci $g_{2,1,3}$, $g_{2,2,3}$, $g_{2,3,3}$, $\overline{g_{2,3,3}}$  and the additional root loci form the second kind of non-standard linked root loci. The terminus of this structure consists of $g_{2,3,3}$ or $\overline{g_{2,3,3}}$.
\end{lemma}
\begin{proof}
$g_{2,3,3}$ and $\overline{g_{2,3,3}}$ intersect at a real breakaway point $q_{2,3,3}$ within $L_{2,2,3}$. From $q_{2,3,3}$, $g_{2,3,3}$ and $\overline{g_{2,3,3}}$  extend oppositely on the real axis. Suppose $g_{2,3,3}$ extends toward $t_{2,2,3}$. $\overline{g_{2,3,3}}$ maybe ends at $z_{2,2,3}$. $\overline{g_{2,3,3}}$ maybe ends at the other zero of (2.2). If  $g_{2,2,3}$ doesn't intersect with the additional root loci at real breakaway points on the real axis, then $g_{2,2,3}$ and $g_{2,3,3}$ extend toward each other and intersect at a real breakaway point $q_{2,2,3}$ within $L_{2,2,3}$.

If $g_{2,2,3}$ intersects with the additional root loci at the real breakaway points on the real axis, the additional root loci do not contain $g_{2,3,3}$. After $g_{2,2,3}$ intersects with the additional root loci on the real axis, then $g_{2,2,3}$ and $g_{2,3,3}$ extend toward each other and intersect at a real breakaway point $q_{2,2,3}$ within $L_{2,2,3}$.

At $q_{2,2,3}$, $g_{2,2,3}$ and $g_{2,3,3}$ leave the real axis. $g_{2,2,3}$ and $g_{2,3,3}$  enter  $\mathbb{C}^{*}$.
Due to symmetry about the real axis, $g_{2,2,3}$ and $g_{2,3,3}$ which have left the real axis must intersect again at a real breakaway point $q_{2,1,3}$ in $L_{2,1,3}$. Therefore, $g_{2,2,3}$ and $g_{2,3,3}$ extend in  $\mathbb{C}^{*}$ from $q_{2,2,3}$ to $q_{2,1,3}$. From $q_{2,1,3}$, $g_{2,2,3}$ and $g_{2,3,3}$  extend oppositely on the real axis. Suppose $g_{2,2,3}$ extends toward $t_{2,1,3}$. $g_{2,3,3}$ maybe ends at $z_{2,1,3}$. $g_{2,3,3}$ maybe ends at the other zero of (2.2).

If  $g_{2,1,3}$ doesn't intersect with the additional root loci at real breakaway points on the real axis, then $g_{2,1,3}$ and the root locus $g_{2,2,3}$ extend toward each other and intersect, along with  other possible root loci, at $Q_{2,3}$ in $L_{2,1,3}$.

If $g_{2,1,3}$ intersects with the additional root loci at the real breakaway points on the real axis, the additional root loci do not contain $g_{2,2,3}$. After $g_{2,1,3}$ intersects with the additional root loci on the real axis, then $g_{2,1,3}$ and the root locus $g_{2,2,3}$ extend toward each other and intersect, along with  other possible root loci, at $Q_{2,3}$ in $L_{2,1,3}$.

Therefore, $g_{2,1,3}$, $g_{2,2,3}$, $g_{2,3,3}$, $\overline{g_{2,3,3}}$  and the additional root loci form the second kind of non-standard linked root loci. $g_{2,1,3}$ and $g_{2,2,3}$ leave the real axis and enter  $\mathbb{C}^{*}$.

The same as the case in Lemma 2.15, the real breakaway points may coincide. This overlapping situation has no impact on the proof or the results of this lemma. Since we are unable to provide the situation of "the additional root loci", the exact situation of the real breakaway points overlapping cannot be accurately given either. We only need to know that such a situation exists.
\end{proof}

Let $g_{k,j_1,3}$ be the root locus emitted from a real zero $t_{k,j_1,3}$ of $p$. Let $z_{k,j_1,3}$ be the real zero of (2.2) adjacent to $t_{k,j_1,3}$. $g_{k,j_1,3}$ leaves the real axis at a point in the interval $L_{k,j_1,3}$ between $t_{k,j_1,3}$ and $z_{k,j_1,3}$.  The root loci from the non-real zeros $s_{k,k+1,3}$ and $\overline{s_{k,k+1,3}}$ are $g_{k,k+1,3}$ and $\overline{g_{k,k+1,3}}$.

$g_{k,j_1+1,3}$ and $g_{k,j_1+2,3}$ which have left the real axis must intersect again at a real breakaway point $q_{k,2j_1-1,3}$ in $L_{k,j_1,3}$. From $q_{k,2j_1-1,3}$, $g_{k,j_1+1,3}$ and $g_{k,j_1+2,3}$  extend oppositely on the real axis. Suppose $g_{k,j_1+1,3}$ extends toward $t_{k,j_1,3}$.

Here, at least one of the $k$ cases must occur. Case $j_1$. $j_1=1,\cdots,k-1$. $g_{k,j_1,3}$  intersects with the additional root loci at real breakaway points on the real axis. The additional root loci do not contain $g_{k,j_1+1,3}$. After $g_{k,j_1,3}$ intersects with the additional root loci on the real axis,  then $g_{k,j_1,3}$ intersects with $g_{k,j_1+1,3}$  at a real breakaway point.

$g_{k,k+1,3}$ and $\overline{g_{k,k+1,3}}$ intersect  at a real breakaway point $q_{k,2k-1,3}$ in $L_{k,k,3}$. From $q_{k,2k-1,3}$, $g_{k,k+1,3}$ and $\overline{g_{k,k+1,3}}$  extend oppositely on the real axis. Suppose $g_{k,k+1,3}$ extends toward $t_{k,k,3}$.

Case $k$. $g_{k,k,3}$ intersects with the additional root loci at real breakaway points on the real axis. The additional root loci do not contain $g_{k,k+1,3}$. After  $g_{k,k,3}$ intersects with the additional root loci on the real axis, $g_{k,k,3}$ then intersects with  $g_{k,k+1,3}$ at a real breakaway point.

\begin{lemma}
When a root locus $g_{k,1,3}$ emitted from a real zero of $p(s)$ leaves the real axis and enters $\mathbb{C}^{*}$, then, there must exist a pair of non-real zeros and $k-1$ other real zeros.  At least one of the $k$ cases must occur.  These root loci $g_{k,k+1,3}$ and $\overline{g_{k,k+1,3}}$, $g_{k,j_1,3}$  (for $1\le k<n$, $j_1=1,\cdots,k$, $0<2j_1-2<2k-2$, $0<2j_1-1<2k-1$) intersect at real breakaway points $Q_{k,3}$, $q_{k,2j_1-1,3}$, $q_{k,2j_1-2,3}$  on the real axis. At the end point of the structure, there are a pair of root loci $g_{k,k+1,3}$ and $\overline{g_{k,k+1,3}}$ are emitted from a pair of non-real zeros $s_{k,k+1,3}$ and $\overline{s_{k,k+1,3}}$.
\end{lemma}

\begin{proof}
In Lemma 2.31, there are four main root loci in the second kind of non-standard linked root loci, and two intervals between zeros and poles of (2.2).

$g_{k,1,3}$, $g_{k,2,3}$, and possibly other root loci intersect at a real breakaway point $Q_{k,3}$  in $L_{k,1,3}$. At $Q_{k,3}$, $g_{k,1,3}$  and $g_{k,2,3}$ leave the real axis and enter  $\mathbb{C}^{*}$. The situation of $g_{k,1,3}$ and $g_{k,2,3}$ is identical to the situation given in Lemma 2.31. At the front of Definition 2.29, we have given the result of $k=1$. When $k=2$, Lemma 2.31 give results. For general $k>2$, as following.

$g_{k,j_1+1,3}$ and $g_{k,j_1+2,3}$ intersect at the real breakaway point $q_{k,2j_1,3}$ in $L_{k,j_1+1,3}$. From $q_{k,2j_1,3}$, $g_{k,j_1+1,3}$ and $g_{k,j_1+2,3}$ leave the real axis and enter   $\mathbb{C}^{*}$. Due to symmetry about the real axis, $g_{k,j_1+1,3}$ and $g_{k,j_1+2,3}$ intersect at  the real breakaway point $q_{k,2j_1-1,3}$  in $L_{k,j_1,3}$.

$g_{k,j_1+1,3}$ and $g_{k,j_1+2,3}$ extend in  $\mathbb{C}^{*}$ from $q_{k,2j_1,3}$ to $q_{k,2j_1-1,3}$. $t_{k,j_1,3}$ emits $g_{k,j_1,3}$ to interval $L_{k,j_1,3}$. From $q_{k,j_1-1,3}$, $g_{k,j_1+1,3}$ and $g_{k,j_1+2,3}$  extend oppositely on the real axis. Suppose $g_{k,j_1+1,3}$ extends toward $t_{k,j_1,3}$. $g_{k,j_1+2,3}$ maybe ends at zero  $z_{k,j_1,3}$. $g_{k,j_1+2,3}$ maybe ends at the other zero of (2.2).

If $g_{k,j_1,3}$  does not intersect with the additional root loci at real breakaway points on the real axis.  Both $g_{k,j_1,3}$ and $g_{k,j_1+1,3}$ extend towards each other, so $g_{k,j_1+1,3}$ intersects with $g_{k,j_1,3}$ at a real breakaway point $q_{k,2j_1-2,3}$ within $L_{k,j_1,3}$. At point $q_{k,2j_1-2,3}$, $g_{k,j_1,3}$ and $g_{k,j_1+1,3}$ leaves the real axis and enters  $\mathbb{C}^{*}$.

The following situation may occur. $g_{k,j_1,3}$ intersects with the additional root loci at real breakaway points on the real axis. The additional root loci do not contain $g_{k,j_1+1,3}$. After $g_{k,j_1,3}$ intersects with the additional root loci on the real axis, both $g_{k,j_1,3}$ and $g_{k,j_1+1,3}$ extend towards each other, so $g_{k,j_1+1,3}$ intersects with $g_{k,j_1,3}$ at a real breakaway point $q_{k,2j_1-2,3}$ within $L_{k,j_1,3}$. At point $q_{k,2j_1-2,3}$, $g_{k,j_1,3}$ and $g_{k,j_1+1,3}$ leaves the real axis and enters  $\mathbb{C}^{*}$.

Due to symmetry about the real axis, $g_{k,j_1,3}$ and $g_{k,j_1+1,3}$ which have left the real axis must intersect again at a real breakaway point $q_{k,2j_1-3,3}$  in $L_{k,j_1-1,3}$. $g_{k,j_1,3}$ and $g_{k,j_1+1,3}$ extend in  $\mathbb{C}^{*}$ from $q_{k,2j_1-2,3}$ to $q_{k,2j_1-3,3}$. $j_1=3,\dots,k-2$.

When the root loci intersecting in $L_{k,j_2,3}$ originates from a real zero $t_{k,j_2,3}$. $j_2>k-2$. The root locus $g_{k,j_2,3}$ emitted from $t_{k,j_2,3}$ leaves interval $L_{k,j_2,3}$. However, if no other root locus intersects in $L_{k,j_2,3}$, then $g_{k,j_2,3}$ must end at $z_{k,j_2,3}$ and remain on the real axis. There must be a root locus $g_{k,j_2+1,3}$ intersecting in the interval $L_{k,j_2,3}$. If $g_{k,j_2+1,3}$ comes from a real zero $t_{k,j_2+1,3}$ in another interval $L_{k,j_2+1,3}$. If such a process arises indefinitely, the non-standard linked root loci cannot end and must continue to extend downward.

The breakaway points of the root loci are the intersection points of the root loci. By Theorem 1.28, the number of finite breakaway points of (2.2) is finite. The number of the intersection points of the root loci of (2.2) is finite. The number of the intersection root loci is finite. Therefore, the intersection of root loci of (2.2) must end. There must exist an intersection of root loci emitted from non-real zeros on the real axis, and root loci emitted from non-real zeros intersect in an interval between a zero and a pole. Therefore, such linked root loci cannot continue infinitely.

If the root locus emitted from a non-real zero $s_{k,k+1,3}$ intersects the interval $L_{k,k,3}$, the second kind of non-standard linked root loci end and cannot extend further downward. The second kind of non-standard linked root loci also end at the interval where the root loci emitted from a non-real zero of $p(s)$ intersect. Conversely, if no root locus emitted from a non-real zero intersects within the interval, the second kind of non-standard linked root loci cannot end either.

Therefore, based on the results we obtained earlier, we have thus proved that there must exist a pair of non-real zeros $s_{k,k+1,3}$ and $\overline{s_{k,k+1,3}}$. The root loci $g_{k,k+1,3}$ and $\overline{g_{k,k+1,3}}$ emitted from $s_{k,k+1,3}$ and $\overline{s_{k,k+1,3}}$ intersect within a certain interval $L_{k,k,3}$.
\end{proof}

\begin{lemma}
When a root locus $g_{k,1,3}$ emitted from a real zero of $p(s)$ leaves the real axis and enters $\mathbb{C}^{*}$, then, there must exist a pair of non-real zeros and $k-1$ other real zeros.  At least one of the $k$ cases must occur.   These root loci $g_{k,k+1,3}$, $\overline{g_{k,k+1,3}}$  and $g_{k,j_1,3}$  (for $1\le k<n$, $j_1=1,\cdots,k$, $0<2j_1-2<2k-2$, $0<2j_1-1<2k-1$) intersect at real breakaway points $Q_{k,3}$, $q_{k,2j_1-1}$, $q_{k,2j_1-2}$ on the real axis. These real breakaway points connect $g_{k,1,3}$, $g_{k,j_1,3}$, $g_{k,k+1,3}$, and $\overline{g_{k,k+1,3}}$ together, forming a second kind of non-standard linked root locus. The terminus of this structure consists of $g_{k,k+1,3}$ and $\overline{g_{k,k+1,3}}$.
\end{lemma}

\begin{proof}
According to the condition of this lemma, by Lemma 2.32, at the end point of the structure, there are a pair of root loci $g_{k,k+1,3}$ and $\overline{g_{k,k+1,3}}$ are emitted from a pair of non-real zeros  $s_{k,k+1,3}$ and $\overline{s_{k,k+1,3}}$.

The non-real zeros $s_{k,k+1,3}$ and $\overline{s_{k,k+1,3}}$ are in  $\mathbb{C}^{*}$. $s_{k,k+1,3}$ and $\overline{s_{k,k+1,3}}$ emit the root locus $g_{k,k+1,3}$ and $\overline{g_{k,k+1,3}}$ in  $\mathbb{C}^{*}$ separately. $g_{k,k+1,3}$ and $\overline{g_{k,k+1,3}}$ intersect at a real breakaway point $q_{k,2k-1,3}$ within $L_{k,k,3}$. $g_{k,k+1,3}$ and $\overline{g_{k,k+1,3}}$ depart from $q_{k,2k-1,3}$. By Theorem 1.28, $g_{k,k+1,3}$ and $\overline{g_{k,k+1,3}}$ cannot extend in the same direction, and they must extend in opposite directions along the real axis. Suppose $g_{k,k+1,3}$ extends towards $t_{k,k,3}$. $\overline{g_{k,k+1,3}}$ maybe ends at  zero  $z_{k,k,3}$. $\overline{g_{k,k+1,3}}$ maybe ends at the other zero of (2.2).

If $g_{k,k,3}$  does not intersect with the additional root loci at the real breakaway point on the real axis, both $g_{k,k,3}$ and $g_{k,k+1,3}$ extend towards each other, so $g_{k,k+1,3}$ intersects with $g_{k,k,3}$ at a real breakaway point $q_{k,2k-2,3}$ within $L_{k,k,3}$.

However, the following situation may also occur. $g_{k,k,3}$  intersects with the additional root loci at real breakaway points on the real axis. The additional root loci do not contain $g_{k,k+1,3}$. After $g_{k,k,3}$ intersects with the additional root loci on the real axis, both $g_{k,k,3}$ and $g_{k,k+1,3}$ extend towards each other, so $g_{k,k+1,3}$ intersects with $g_{k,k,3}$ at a real breakaway point $q_{k,2k-2,3}$ within $L_{k,k,3}$.

At point $q_{k,2k-2,3}$, $g_{k,k,3}$ and  $g_{k,k+1,3}$ leave the real axis and enter  $\mathbb{C}^{*}$.

Due to symmetry about the real axis, $g_{k,k+1,3}$ and $g_{k,k,3}$ which have left the real axis must intersect again at a real breakaway point $q_{k,2k-3,3}$ within $L_{k,k-1,3}$. Therefore, $g_{k,k,3}$ and $g_{k,k+1,3}$ extend in $\mathbb{C}^{*}$ from $q_{k,2k-2,3}$ to $q_{k,2k-3,3}$.

$g_{k,k,3}$ and $g_{k,k+1,3}$ depart from $q_{k,2k-3,3}$. They must extend in opposite directions along the real axis. Suppose $g_{k,k,3}$ extends towards $t_{k,k-1,3}$. $g_{k,k+1}$ maybe ends at  zero  $z_{k,k-1,3}$. $g_{k,k+1,3}$ maybe ends at the other zero of (2.2).

If $g_{k,k-1,3}$  does not intersect with the additional root loci at the real breakaway point on the real axis, both $g_{k,k-1,3}$ and $g_{k,k,3}$ extend towards each other, so $g_{k,k,3}$ intersects with $g_{k,k-1,3}$ at a real breakaway point $q_{k,2k-4,3}$ within $L_{k,k-1,3}$.

However, the following situation may also occur. $g_{k,k-1,3}$  intersects with the additional root loci at real breakaway points on the real axis. The additional root loci do not contain $g_{k,k,3}$. After $g_{k,k-1,3}$ intersects with the additional root loci on the real axis, both $g_{k,k-1}$ and $g_{k,k,3}$ extend towards each other, so $g_{k,k-1,3}$ intersects with $g_{k,k,3}$ at a real breakaway point $q_{k,2k-4,3}$ within $L_{k,k-1,3}$.

At point $q_{k,2k-4,3}$, $g_{k,k-1}$ and  $g_{k,k}$ leave the real axis and enter  $\mathbb{C}^{*}$.

Due to symmetry about the real axis, $g_{k,k,3}$ and $g_{k,k-1,3}$ which have left the real axis must intersect again at a real breakaway point $q_{k,2k-5,3}$ within $L_{k,k-2,3}$. Therefore, $g_{k,k,3}$ and $g_{k,k-1k,3}$ extend in $\mathbb{C}^{*}$ from $q_{k,2k-4,3}$ to $q_{k,2k-5,3}$. Let $g_{k,k-1,3}$ extend to $t_{k,k-2,3}$. $g_{k,k,3}$ maybe ends at $z_{k,k-2,3}$. $g_{k,k,3}$ maybe ends at the other zero of (2.2).  $g_{k,k-1,3}$ intersects with $g_{k,k-2,3}$.

When such a situation appears, starting from the root locus $g_{k,1,3}$ emitted from the real zero $t_{k,1,3}$ of $p(s)$, the second kind of non-standard linked root loci terminate and do not extend further. The terminus of the second kind of non-standard linked root loci is composed of the root loci $g_{k,k+1,3}$ and $\overline{g_{k,k+1,3}}$.
The root loci $g_{k,1,3}$, $g_{k,2,3}$, $\cdots$, $g_{k,k+1,3}$, and $\overline{g_{k,k+1,3}}$  and the additional root loci form the second kind of  non-standard linked root loci.

According to the previously proved result of the second kind of non-standard linked root loci, the root locus $g_{k,1,3}$ at the beginning of the second kind of non-standard linked root loci and the root loci $g_{k,k+1,3}$ and $\overline{g_{k,k+1,3}}$ at the end of the second kind of non-standard linked root loci, by the root loci $g_{k,j_1,3}$ of real zeros in the middle, must be connected in the second kind of non-standard linked root loci. The beginning of the second kind of non-standard linked root loci is a root locus $g_{k,1,3}$ emitted from a real zero $t_{k,1,3}$, and $g_{k,1,3}$ leaves the real axis and enters  $\mathbb{C}^{*}$. The end of the second kind of non-standard linked root loci must be a pair of root loci $g_{k,k+1,3}$ and $\overline{g_{k,k+1,3}}$. The root loci $g_{k,k+1,3}$ and $\overline{g_{k,k+1,3}}$ intersect on the real axis, and they are emitted from the non-real zeros $s_{k,k+1,3}$ and $\overline{s_{k,k+1,3}}$.

In Lemma 2.33, there are $k+2$ main root loci in the second kind of non-standard linked root loci, and $k$ intervals between zeros and poles of (2.2).

The same as the case in Lemma 2.17, the real breakaway points may coincide. This overlapping situation has no impact on the proof or the results of this lemma. Since we are unable to provide the situation of "the additional root loci", the exact situation of the real breakaway points overlapping cannot be accurately given either. We only need to know that such a situation exists.
\end{proof}

To make it clear for us to explain the concept of two main lines of the second kind of non-standard linked root loci below, it starts from the terminal root locus and follows all the way back to the beginning root locus.

The root loci $g_{k,k+1,3}$ and $\overline{g_{k,k+1,3}}$ emitted from the non-real zero $s_{k,k+1,3}$ and $\overline{s_{k,k+1,3}}$ first intersect at the real breakaway point $q_{k,2k-1,3}$ in $L_{k,k,3}$. $g_{k,k+1,3}$ and $\overline{g_{k,k+1,3}}$ depart from $q_{k,2k-1}$. $g_{k,k+1,3}$ and $\overline{g_{k,k+1,3}}$ extend in opposite directions along the real axis. Let $g_{k,k+1,3}$ extend towards $t_{k,k,3}$. $\overline{g_{k,k+1,3}}$ may end at the real zero $z_{k,k,3}$. $\overline{g_{k,k+1,3}}$ maybe leaves the real axis and ends the other zero of (2.2).

If $g_{k,k,3}$ does not intersect with the additional root loci at real breakaway points on the real axis,   $g_{k,k+1,3}$ intersects with $g_{k,k,3}$ at a real breakaway point $q_{k,2k-2,3}$ in $L_{k,k,3}$. At point $q_{k,2k-2,3}$, $g_{k,k,3}$ and $g_{k,k+1,3}$ leave the real axis and enter $\mathbb{C}^{*}$.

However, the following situation may also occur. $g_{k,k,3}$  intersects with the additional root loci at real breakaway points on the real axis.  The additional root loci do not contain $g_{k,k+1,3}$.

In this situation, there is the second type of intersection with the additional root loci.  After $g_{k,k,3}$ intersects with the additional root loci on the real axis,  $g_{k,k+1,3}$ intersects with $g_{k,k,3}$ at a real breakaway point $q_{k,2k-2,3}$ in $L_{k,k,3}$. At point $q_{k,2k-2,3}$, $g_{k,k,3}$ and $g_{k,k+1,3}$ leave the real axis and enter $\mathbb{C}^{*}$.

Due to symmetry about the real axis, $g_{k,k,3}$ and $g_{k,k+1,3}$ which have left the real axis must intersect again at a real breakaway point $q_{k,2k-3,3}$ in $L_{k,k-1,3}$. $g_{k,k,3}$ and $g_{k,k+1,3}$ extend in  $\mathbb{C}^{*}$ from $q_{k,2k-2,3}$ to $q_{k,2k-3,3}$.

$g_{k,k,3}$ and $g_{k,k+1,3}$ depart from $q_{k,2k-3,3}$.  $g_{k,k,3}$ and $g_{k,k+1,3}$  must extend in opposite directions along the real axis. Let  $g_{k,k,3}$  extend towards $t_{k,k-1,3}$. $g_{k,k+1,3}$ may end at the real zero $z_{k,k-1,3}$. $g_{k,k+1,3}$ maybe leaves the real axis and ends the other zero of (2.2).

If $g_{k,k-1,3}$ does not intersect with the additional root loci, $g_{k,k,3}$ intersects with $g_{k,k-1,3}$ at a real breakaway point $q_{k,2k-4,3}$ in $L_{k,k-1,3}$. At point $q_{k,2k-4,3}$, $g_{k,k,3}$ and $g_{k,k-1,3}$ leave the real axis and enter $\mathbb{C}^{*}$.

However, the following situation may also occur. $g_{k,k-1,3}$  intersects with the additional root loci at real breakaway points on the real axis.  The additional root loci do not contain $g_{k,k,3}$. In this situation, there is the second type of intersection with the additional root loci.  After $g_{k,k-1,3}$ intersects with the additional root loci on the real axis,  $g_{k,k,3}$ intersects with $g_{k,k-1,3}$ at a real breakaway point $q_{k,2k-4,3}$ in $L_{k,k-1,3}$. At point $q_{k,2k-4,3}$, $g_{k,k-1,3}$ and $g_{k,k,3}$ leave the real axis and enter $\mathbb{C}^{*}$.

Due to symmetry about the real axis, $g_{k,k-1,3}$ and $g_{k,k,3}$ which have left the real axis must intersect again at a real breakaway point $q_{k,2k-5,3}$ in $L_{k,k-2,3}$. $g_{k,k-1,3}$ and $g_{k,k,3}$ extend in  $\mathbb{C}^{*}$ from $q_{k,2k-4,3}$ to $q_{k,2k-5,3}$.

$g_{k,j_1+1,3}$ and $g_{k,j_1+2,3}$ intersect at the real breakaway point $q_{k,2j_1,3}$ in $L_{k,j_1+1,3}$. From $q_{k,2j_1,3}$, $g_{k,j_1+1,3}$ and $g_{k,j_1+2,3}$ leave the real axis and enter   $\mathbb{C}^{*}$.  $g_{k,j_1+1,3}$ and $g_{k,j_1+2,3}$ intersect at  the real breakaway point $q_{k,2j_1-1,3}$  in $L_{k,j_1,3}$. $g_{k,j_1+1,3}$ and $g_{k,j_1+2,3}$ extend in  $\mathbb{C}^{*}$ from $q_{k,2j_1,3}$ to $q_{k,2j_1-1,3}$. $g_{k,j_1+1}$ and $g_{k,j_1+2,3}$ depart from $q_{k,2j_1-1,3}$. $g_{k,j_1+1}$ and $g_{k,j_1+2}$  must extend in opposite directions along the real axis. Suppose $g_{k,j_1+1,3}$ extends towards $t_{k,j_1,3}$.
$g_{k,j_1+2,3}$ maybe ends at zero  $z_{k,j_1,3}$. $g_{k,j_1+2,3}$ maybe ends at the other zero of (2.2).

If $g_{k,j_1,3}$  does not intersect with the additional root loci at real breakaway points on the real axis.  $g_{k,j_1+1,3}$ intersects with $g_{k,j_1,3}$ at a real breakaway point $q_{k,2j_1-2,3}$ within $L_{k,j_1,3}$. At point $q_{k,2j_1-2,3}$, $g_{k,j_1,3}$ and $g_{k,j_1+1,3}$ leaves the real axis and enters  $\mathbb{C}^{*}$.

If $g_{k,j_1,3}$  intersects with the additional root loci at real breakaway points on the real axis. The additional root loci do not contain $g_{k,j_1+1,3}$. After $g_{k,j_1,3}$ intersects with the additional root loci on the real axis, $g_{k,j_1+1,3}$ intersects with $g_{k,j_1,3}$ at a real breakaway point $q_{k,2j_1-2,3}$ within $L_{k,j_1,3}$. At point $q_{k,2j_1-2,3}$, $g_{k,j_1,3}$ and $g_{k,j_1+1,3}$ leaves the real axis and enters  $\mathbb{C}^{*}$.

Due to symmetry about the real axis, the root loci $g_{k,j_1,3}$ and $g_{k,j_1+1,3}$ intersect at the real breakaway point $q_{k,2j_1-3,3}$  in $L_{k,j_1-1,3}$. $g_{k,j_1,3}$ and $g_{k,j_1+1}$ extend in $\mathbb{C}^{*}$ from $q_{k,2j_1-2,3}$ to $q_{k,2j_1-3,3}$. $j_1=1,2,\cdots,k-2$.

$g_{k,1,3}$ and the root locus $g_{k,2,3}$ intersect, along with other possible root loci, at $Q_{k,3}$  in $L_{k,1,3}$.

The second kind of non-standard linked root loci consist of two main lines with multiple branches. Except for the terminal root loci emitted from the non-real poles of (2.2), the rest main root loci  originate from the real poles of (2.2).

In $\mathbb{C} ^{*}$, the section of the root locus $g_{k,k+1,3}$ from the non-real zero $s_{k,k+1,3}$ to the real breakaway point $q_{k,2k-1,3}$ constitutes the terminus part of two main lines. In $\mathbb{C} ^{*}$, the section of the root locus $\overline{g_{k,k+1,3}}$ from the non-real zero $\overline{s_{k,k+1,3}}$ to $q_{k,2k-1,3}$ constitutes the terminus part of two main lines.

The interval on the real axis from $q_{k,2k-1,3}$ to the real breakaway point $q_{k,2k-2,3}$ constitutes a part of two main lines. In $\mathbb{C} ^{*}$, the section of the root locus $g_{k,k,3}$ from $q_{k,2k-2,3}$ to the real breakaway point $q_{k,2k-3,3}$ constitutes a part of two main lines. In $\mathbb{C} ^{*}$, the section of the root locus $g_{k,k+1,3}$ from $q_{k,2k-2,3}$ to $q_{k,2k-3,3}$ constitutes a part of two main lines.

The interval on the real axis from $q_{k,2k-3,3}$ to the real breakaway point $q_{k,2k-4,3}$ constitutes a part of two main lines. In $\mathbb{C} ^{*}$, the section of the root locus $g_{k,k-1,3}$ from $q_{k,2k-4,3}$ to the real breakaway point $q_{k,2k-5,3}$ constitutes a part of two main lines. In $\mathbb{C} ^{*}$, the section of the root locus $g_{k,k,3}$ from $q_{k,2k-4}$ to $q_{k,2k-5,3}$ constitutes a part of two main lines.

In $\mathbb{C} ^{*}$, the section of the root locus $g_{k,j_1+1,3}$ from the real breakaway point $q_{k,2j_1,3}$ to the real breakaway point $q_{k,2j_1-1,3}$ constitutes a part of two main lines. In $\mathbb{C} ^{*}$, the section of the root locus $g_{k,j_1+2,3}$ from the real breakaway point $q_{k,2j_1,3}$ to the real breakaway point $q_{k,2j_1-1,3}$ constitutes a part of two main lines. The interval on the real axis from $q_{k,2j_1-1,3}$ to $q_{k,2j_1-2,3}$ constitutes a part of two main lines.  In $\mathbb{C} ^{*}$, the section of the root locus $g_{k,j_1,3}$ from $q_{k,2j_1-2,3}$ to the real breakaway point $q_{k,2j_1-3,3}$ constitutes a part of two main lines. In $\mathbb{C} ^{*}$, the section of the root locus $g_{k,j_1+1,3}$ from $q_{k,2j_1-2,3}$ to $q_{k,2j_1-3,3}$ constitutes a part of two main lines. $j_1=1,2,\cdots,k-2$.

In $\mathbb{C} ^{*}$, the section of the root locus $g_{k,2,3}$ from $q_{k,2,3}$ to the real breakaway point $q_{k,1,3}$ constitutes a part of two main lines. In $\mathbb{C} ^{*}$, the section of the root locus $g_{k,3,3}$ from $q_{k,2,3}$ to $q_{k,1,3}$ constitutes a part of two main lines. The interval on the real axis from $q_{k,1,3}$ to the real breakaway point $Q_{k,3}$ constitutes a part of two main lines.

All the parts given previously, when connected together, form two main lines of the second kind of non-standard linked root loci. In the second kind of non-standard linked root loci, there are only two main lines. Part of the two main lines is distributed on the real axis. The other part of two main lines distributed in $\mathbb{C}^{*}$, and they are symmetric about the real axis.

The section of $\overline{g_{k,k+1,3}}$ from $q_{k,2k-1,3}$ to $z_{k,k,3}$ constitutes a branch. The section of $g_{k,k,3}$ from $t_{k,k,3}$ to $q_{k,2k-2,3}$ constitutes a branch. And so on $\cdots$. In the second non-standard linked root loci, there are multiple branches.

In Lemma 2.31, in $\mathbb{C} ^{*}$, the section of the root locus $g_{2,3,3}$ from the non-real zero $s_{2,3,3}$ to the real breakaway point $q_{2,3,3}$ constitutes the terminus part of two main lines. In $\mathbb{C} ^{*}$, the section of the root locus $\overline{g_{2,3,3}}$ from the non-real zero $\overline{s_{2,3,3}}$ to $q_{2,3,3}$ constitutes the terminus part of two main lines.

The interval on the real axis from $q_{2,3,3}$ to the real breakaway point $q_{2,2,3}$ constitutes a part of two main lines. In $\mathbb{C} ^{*}$, the section of the root locus $g_{2,2,3}$ from $q_{2,2,3}$ to the real breakaway point $q_{2,1,3}$ constitutes a part of two main lines. In $\mathbb{C} ^{*}$, the section of the root locus $g_{2,3,3}$ from $q_{2,2,3}$ to $q_{2,1,3}$ constitutes a part of two main lines.

The interval on the real axis from $q_{2,1,3}$ to the real breakaway point $Q_{2,3}$ constitutes a part of two main lines.

All the parts given previously, when connected together, form two main lines of the second kind of non-standard linked root loci. In the second kind of non-standard linked root loci, there are only two main lines. Part of the two main lines is distributed on the real axis. The other part of two main lines distributed in $\mathbb{C}^{*}$, and they are symmetric about the real axis.

The section of $\overline{g_{2,3,3}}$ from $q_{2,3,3}$ to $z_{2,2,3}$ constitutes a branch. The section of $g_{2,2,3}$ from $t_{2,2,3}$ to $q_{2,2,3}$ constitutes a branch. The section of $g_{2,3,3}$ from $q_{2,1,3}$ to $z_{2,1,3}$ constitutes a branch. In the second kind of non-standard linked root loci, there are four branches.

Therefore, based on the above explanations of the two main lines of the second kind of non-standard linked root loci, we can define the two main lines of the second kind of non-standard linked root loci of (2.2) as follows.

\begin{definition}
The non-real zeros $s_{k,k+1,3}$ and $\overline{s_{k,k+1,3}}$ emit root loci $g_{k,k+1,3}$ and $\overline{g_{k,k+1,3}}$. In the second kind of non-standard linked root loci, beginning at the non-real zeros $s_{k,k+1,3}$ and $\overline{s_{k,k+1,3}}$, the two lines which $g_{k,k+1,3}$ and $\overline{g_{k,k+1,3}}$ can extend all the way to $Q_{k,3}$ are the two main lines of the second kind of non-standard linked root loci of (2.2).
\end{definition}

\begin{lemma}
In the second kind of non-standard linked root loci, a root locus $g_{k,k+1,3}$ or $\overline{g_{k,k+1,3}}$ originating from a non-real zero $s_{k,k+1,3}$ or $\overline{s_{k,k+1,3}}$ can be distributed on two main lines of the second kind of non-standard linked root loci. The root locus $g_{k,k+1,3}$ or $\overline{g_{k,k+1,3}}$  intersects at the real breakaway point $Q_{k,3}$. The root locus $g_{k,k+1,3}$ or $\overline{g_{k,k+1,3}}$  leaves from $Q_{k,3}$ then enters $\mathbb{C}^{*}$.
\end{lemma}

\begin{proof}
In the explanation of two main lines, for the root loci that eventually reach the real zeros in (2.2) in the second kind of non-standard linked root loci, one of their distribution as follows. The root locus $\overline{g_{k,k+1,3}}$ ends at $z_{k,k,3}$. The root locus $g_{k,l,3}$ ends at $z_{k,l-2,3}$. $l=3,\cdots, k+1$. Based on this distributed way, neither the root locus $g_{k,k+1,3}$ nor $\overline{g_{k,k+1,3}}$ directly intersect at $Q_{k,3}$.

The intersecting root loci are referred to as "common root loci". They possess "common" properties. Beginning at the non-real zeros $s_{k,k+1,3}$ and $\overline{s_{k,k+1,3}}$, the two main lines of the second kind of non-standard linked root loci both extend towards the breakaway point $Q_{k,3}$. Therefore, the root locus emitted from the non-real zero $s_{k,k+1,3}$ or $\overline{s_{k,k+1,3}}$ can be distributed on two main lines of the second kind of non-standard linked root loci. That is, among the root loci intersecting at $Q_{k,3}$, one is $g_{k,k+1,3}$ or $\overline{g_{k,k+1,3}}$.

The root locus $g_{k,k+1,3}$ from a non-real zero $s_{k,k+1,3}$, distributed on two main lines of the second kind of non-standard linked root loci in the following way, and intersect at $Q_{k,3}$.

The root locus $\overline{g_{k,k+1,3}}$ ends at $z_{k,k,3}$. The root locus $g_{k,k,3}$ ends at $z_{k,k-1,3}$. The root locus $g_{k,l_1,3}$  ends at $z_{k,l_1-1,3}$. $l_1=3,\cdots, k-1$. $g_{k,3,3}$ ends at $z_{k,2,3}$. $g_{k,2,3}$ ends at $z_{k,1,3}$. Under this case, $g_{k,k+1,3}$ must intersect at $Q_{k,3}$.

The root locus $g_{k,k+1,3}$ originates from the non-real zero $s_{k,k+1,3}$ and extends to $q_{k,2k-1,3}$. At $q_{k,2k-1,3}$, $g_{k,k+1,3}$ enters the real axis. On the real axis, $g_{k,k+1,3}$ extends from $q_{k,2k-1,3}$ to $q_{k,2k-2,3}$. Then $g_{k,k+1,3}$ enters $\mathbb{C}^{*}$ from $q_{k,2k-2,3}$ and extends to $q_{k,2k-3,3}$ in $\mathbb{C}^{*}$. And so on $\cdots$. At $q_{k,3,3}$, $g_{k,k+1,3}$ re-enters the real axis. On the real axis, $g_{k,k+1,3}$ extends from $q_{k,3,3}$ to $q_{k,2,3}$. Then $g_{k,k+1,3}$ enters $\mathbb{C}^{*}$ from $q_{k,2,3}$ and extends to $q_{k,1,3}$ in $\mathbb{C}^{*}$. At $q_{k,1,3}$, $g_{k,k+1,3}$ re-enters the real axis. On the real axis, $g_{k,k+1,3}$ extends from $q_{k,1,3}$ to $Q_{k,3}$. Then $g_{k,k+1,3}$ enters $\mathbb{C}^{*}$ from $Q_{k,3}$.

The root locus $g_{k,k+1,3}$ ends at $z_{k,k,3}$. The root locus $g_{k,k,3}$ ends at $z_{k,k-1,3}$. The root locus $g_{k,l_1,3}$ ends at $z_{k,l_1-1,3}$. $l_1=3,\cdots, k-1$. $g_{k,3,3}$ ends at $z_{k,2,3}$. $g_{k,2,3}$ ends at $z_{k,1,3}$. Under this case, $\overline{g_{k,k+1,3}}$ must intersect at $Q_{k,3}$.

The root locus $\overline{g_{k,k+1,3}}$ originates from the non-real zero $\overline{s_{k,k+1,3}}$ and extends to $q_{k,2k-1,3}$. At $q_{k,2k-1,3}$, $\overline{g_{k,k+1,3}}$ enters the real axis. On the real axis, $\overline{g_{k,k+1,3}}$ extends from $q_{k,2k-1,3}$ to $q_{k,2k-2,3}$. Then $\overline{g_{k,k+1,3}}$ enters $\mathbb{C}^{*}$ from $q_{k,2k-2,3}$ and extends to $q_{k,2k-3,3}$ in $\mathbb{C}^{*}$. And so on $\cdots$. At $q_{k,3,3}$, $\overline{g_{k,k+1,3}}$ re-enters the real axis. On the real axis, $\overline{g_{k,k+1,3}}$ extends from $q_{k,3,3}$ to $q_{k,2,3}$. Then $\overline{g_{k,k+1,3}}$ enters $\mathbb{C}^{*}$ from $q_{k,2,3}$ and extends to $q_{k,1,3}$ in $\mathbb{C}^{*}$. At $q_{k,1,3}$, $\overline{g_{k,k+1,3}}$ re-enters the real axis. On the real axis, $\overline{g_{k,k+1,3}}$ extends from $q_{k,1,3}$ to $Q_{k,3}$. Then $\overline{g_{k,k+1,3}}$ enters $\mathbb{C}^{*}$ from $Q_{k,3}$.

Therefore, in the second kind of non-standard linked root loci,  a root locus $g_{k,k+1,3}$ or $\overline{g_{k,k+1,3}}$ originating from a non-real zero $s_{k,k+1,3}$ or $\overline{s_{k,k+1,3}}$ can be distributed on two main lines of the second kind of non-standard linked root loci. The root locus $g_{k,k+1,3}$ or $\overline{g_{k,k+1,3}}$ intersects at the real breakaway point $Q_{k,3}$. The root locus $g_{k,k+1,3}$ or $\overline{g_{k,k+1,3}}$ leaves from $Q_{k,3}$ then enters $\mathbb{C}^{*}$.
\end{proof}
If no root loci intersect in the interval $L_{k,1,3}$, $g_{k,1,3}$ can only remain on the real axis.
The root locus $g_{k,1,3}$ of the real zero of $p(s)$ leaves the real axis and enters $\mathbb{C}^{*}$, which is only a result. The reason for this result is that $g_{k,k+1,3}$ or $\overline{g_{k,k+1,3}}$ intersects on the real axis, generating the second kind of non-standard linked root loci. The root locus $g_{k,k+1,3}$ or $\overline{g_{k,k+1,3}}$ of the non-real zeros $s_{k,k+1,3}$ or $\overline{s_{k,k+1,3}}$ intersects on the real axis, generating the second kind of non-standard linked root loci. This is the fundamental reason.

On the two main lines of each of the second kind of non-standard linked root loci, the direction of extension is from the non-real zeros $s_{k,k+1,3}$ or $\overline{s_{k,k+1,3}}$ originating, all the way extending to the real breakaway point $Q_{k,3}$. The two main lines can be the root loci $g_{k,k+1,3}$ or $\overline{g_{k,k+1,3}}$ of the non-real zeros $s_{k,k+1,3}$ or $\overline{s_{k,k+1,3}}$. That is, $g_{k,k+1,3}$ or $\overline{g_{k,k+1,3}}$ successively intersects with the root loci of real zeros on the real axis, and lets the two main lines as its root loci, extends all the way to the breakaway point $Q_{k,3}$.

The first kind of non-standard linked root loci must satisfy four conditions.

The first condition: Two root loci $g_{k,j_1+1,2}$ and $g_{k,j_1+2,2}$ emitted from the real zeros ($t_{k,j_1+1,2}$, $t_{k,j_1+2,2}$). Two root loci $g_{k,k+1,2}$ and $\overline{g_{k,k+1,2}}$ emitted from the non-real zeros ($s_{k,k+1,2}$, $\overline{s_{k,k+1,2}}$). The real zero $t_{k,j_1,2}$ emits the root locus $g_{k,j_1,2}$. The real zero $t_{k,k,2}$ emits the root locus $g_{k,k,2}$.

The second condition: The two root loci $g_{k,j_1+1,2}$ and $g_{k,j_1+2,2}$ must intersect at the real breakaway point $q_{k,2j_1-1,2}$ on the real axis. From $q_{k,2j_1-1,2}$, the two root loci $g_{k,j_1+1,2}$ and $g_{k,j_1+2,2}$ extend further along the real axis in opposite directions. The two root loci $g_{k,k+1,2}$ and $\overline{g_{k,k+1,2}}$ must intersect at the real breakaway point $q_{k,2k-1,2}$ on the real axis. From $q_{k,2k-1,2}$, the two root loci $g_{k,k+1,2}$ and $\overline{g_{k,k+1,2}}$ extend further along the real axis in opposite directions. The preceding sentence states a condition. The following sentence describes a property shared by all pairs of root loci intersecting on the real axis, rather than constituting an additional condition.

The third condition: The root locus $g_{k,j_1+1,2}$ further intersects with the additional root loci on the real axis. The root locus $g_{k,k+1,2}$ further intersects with the additional root loci on the real axis.

The fourth condition: After $g_{k,j_1+1,2}$ intersects with the additional root loci on the real axis, it then intersects with $g_{k,j_1,2}$ on the real axis at the real breakaway point $q_{k,2j_1-2,2}$. After $g_{k,k+1,2}$ intersects with the additional root loci on the real axis, it then intersects with $g_{k,k,2}$ on the real axis at the real breakaway point $q_{k,2k-2,2}$.

The root loci satisfying these four conditions constitute the first kind of non-standard linked root loci.

The second kind of non-standard linked root loci must satisfy three conditions.

The first condition: Two root loci $g_{k,j_1+1,3}$ and $g_{k,j_1+2,3}$ emitted from the real zeros ($t_{k,j_1+1,3}$, $t_{k,j_1+2,3}$). Two root loci $g_{k,k+1,3}$ and $\overline{g_{k,k+1,3}}$ emitted from the non-real zeros ($s_{k,k+1,3}$ or $\overline{s_{k,k+1,3}}$). The real zero $t_{k,j_1,3}$ emits a root locus $g_{k,j_1,3}$.
The real zero $t_{k,k,3}$ emits a root locus $g_{k,k,3}$.

The second condition: The root locus $g_{k,j_1,3}$ intersects with the additional root loci on the real axis.

The third condition: After $g_{k,j_1,3}$ intersects with the additional root loci on the real axis, it then intersects with $g_{k,j_1+1,3}$ on the real axis at the real breakaway point $q_{k,2j_1-2,3}$. After $g_{k,k,3}$ intersects with the additional root loci on the real axis, it then intersects with the root locus $g_{k,k+1,3}$ on the real axis at the real breakaway point $q_{k,2k-2,3}$.

The root loci satisfying these three conditions constitute the second kind of non-standard linked root loci.

There are distinct differences between the first type of intersection with the additional root loci and the second type of intersection with the additional root loci.

The first difference: For the first type of intersection with the additional root loci, the root locus $g_{k,j_1+1,2}$ originates from one endpoint of another interval. Alternatively, for the first type, the root locus $g_{k,k+1,2}$ originates from a non-real zero. For the second type of intersection with the additional root loci, the root locus $g_{k,j_{2},3}$ originates from an endpoint within its own interval. Alternatively, the root locus $g_{k,k,3}$ originates from the endpoint of the final interval on the real axis within the second kind of non-standard linked root loci.

The second difference: For the first type of intersection with additional root loci, initially $g_{k,j_1+1,2}$ intersects with $g_{k,j_1+2,2}$ at a real breakaway point $q_{k,2j_1-1,2}$. Subsequently, $g_{k,j_1+1,2}$ intersects with the additional root loci on the real axis. After this intersection, $g_{k,j_1+1,2}$ then intersects with $g_{k,j_1,2}$ on the real axis at the real breakaway point $q_{k,2j_1-2,2}$. Similarly for the first type, initially $g_{k,k+1,2}$ intersects with $\overline{g_{k,k+1,2}}$ at a real breakaway point $q_{k,2k-1,2}$. Subsequently, $g_{k,k+1,2}$ intersects with the additional root loci on the real axis. After this intersection, $g_{k,k+1,2}$ then intersects with $g_{k,k,2}$ on the real axis at the real breakaway point $q_{k,2k-2,2}$.

For the second type of intersection with additional root loci, the root locus $g_{k,j_1,3}$, emitted from an interval endpoint, first intersects with the additional root loci on the real axis. After this intersection, $g_{k,j_1,3}$ intersects with $g_{k,j_1+1,3}$ at the real breakaway point $q_{k,2j_1-2,3}$. Similarly for the second type, the root locus $g_{k,k,3}$, emitted from an interval endpoint, first intersects with the additional root loci on the real axis. After this intersection, $g_{k,k,3}$ intersects with $g_{k,k+1,3}$ at the real breakaway point $q_{k,2k-2,3}$.

The specific differences between the first and second kinds of non-standard linked root loci, as detailed above, can also be expressed more succinctly as follows:

In the first kind of non-standard linked root loci, only the first type of intersection with the additional root loci occurs; the second type of intersection with the additional root loci does not exist. In the second kind of non-standard linked root loci, only the second type of intersection with the additional root loci occurs; the first type of intersection with the additional root loci does not exist.

The two main lines of the second non-standard linked root loci are identical to the two main lines of the standard linked root loci. This outcome arises for a specific reason.

In the standard linked root loci, the real zero $t_{k,j_1,1}$ emits the root locus $g_{k,j_1,1}$. Two root loci $g_{k,j_1+1,1}$ and $g_{k,j_1+2,1}$ emitted from the real zeros ($t_{k,j_1+1,1}$, $t_{k,j_1+2,1}$). These two root loci $g_{k,j_1+1,1}$ and $g_{k,j_1+2,1}$ leave the real axis and enter $\mathbb{C}^{*}$ from the breakaway point $q_{k,2j_1,1}$. They intersect at the real breakaway point $q_{k,2j_1-1,1}$. From $q_{k,2j_1-1,1}$, the two root loci $g_{k,j_1+1}$ and $g_{k,j_1+2,1}$ extend in opposite directions. The root locus $g_{k,j_1+1,1}$ intersects with the root locus $g_{k,j_1,1}$ at the real breakaway point $q_{k,2j_1-2,1}$. From $q_{k,2j_1-2,1}$, the root loci $g_{k,j_1+1,1}$ and $g_{k,j_1,1}$ leave the real axis and enter $\mathbb{C}^{*}$. In the standard linked root loci, the real zero $t_{k,k,1}$ emits the root locus $g_{k,k,1}$. Two root loci $g_{k,k+1,1}$ and $\overline{g_{k,k+1,1}}$ emitted from the non-real zeros ($s_{k,k+1,1}$, $\overline{s_{k,k+1,1}}$). These two root loci intersect at the real breakaway point $q_{k,2k-1,1}$. From $q_{k,2k-1,1}$, the two root loci $g_{k,k+1,1}$ and $\overline{g_{k,k+1,1}}$ extend in opposite directions. The root locus $g_{k,k+1,1}$ intersects with the root locus $g_{k,k,1}$ at the real breakaway point $q_{k,2k-2,1}$. From $q_{k,2k-2,1}$, the root loci $g_{k,k+1,1}$ and $g_{k,k,1}$ leave the real axis and enter $\mathbb{C}^{*}$.

Based on the result concerning the two main lines of the standard linked root loci, two main lines of the standard linked root loci are as follows: In $\mathbb{C}^{*}$, from $q_{k,2j_1,1}$ to $q_{k,2j_1-1,1}$, they are segments of the two root loci $g_{k,j_1+1,1}$ and $g_{k,j_1+2,1}$. On the real axis, from $q_{k,2j_1-1,1}$ to $q_{k,2j_1-2,1}$, this interval is a segment of the root locus $g_{k,j_1+1,1}$. In $\mathbb{C}^{*}$, from $q_{k,2j_1-2,1}$ to $q_{k,2j_1-3,1}$, they are segments of the root loci $g_{k,j_1+1,1}$ and $g_{k,j_1,1}$. In $\mathbb{C}^{*}$, from $s_{k,k+1,1}$ or $\overline{s_{k,k+1,1}}$ to $q_{k,2k-1,1}$, they are segments of the two root loci $g_{k,k+1,1}$ or $\overline{g_{k,k+1,1}}$. On the real axis, from $q_{k,2k-1,1}$ to $q_{k,2k-2,1}$, this interval is a segment of the root locus $g_{k,k+1,1}$. In $\mathbb{C}^{*}$, from $q_{k,2k-2}$ to $q_{k,2k-3,1}$, they are segments of the root loci $g_{k,k+1,1}$ and $g_{k,k,1}$.

In the second kind of non-standard linked root loci, the real zero $t_{k,j_1,3}$ emits the root locus $g_{k,j_1,3}$. Two root loci $g_{k,j_1+1,3}$ and $g_{k,j_1+2,3}$ emitted from the real zeros ($t_{k,j_1+1,3}$, $t_{k,j_1+2,3}$). These two root loci $g_{k,j_1+1,3}$ and $g_{k,j_1+2,3}$ leave the real axis and enter $\mathbb{C}^{*}$ from the breakaway point $q_{k,2j_1,3}$. They intersect at the real breakaway point $q_{k,2j_1-1,3}$. From $q_{k,2j_1-1,3}$, the two root loci $g_{k,j_1+1,3}$ and $g_{k,j_1+2,3}$ extend in opposite directions. The root locus $g_{k,j_1+1,3}$ intersects with the root locus $g_{k,j_1,3}$ at the real breakaway point $q_{k,2j_1-2,3}$. From $q_{k,2j_1-2,3}$, the root loci $g_{k,j_1+1,3}$ and $g_{k,j_1,3}$ leave the real axis and enter $\mathbb{C}^{*}$.
In the second kind of non-standard linked root loci, the real zero $t_{k,k,3}$ emits the root locus $g_{k,k,3}$. Two root loci $g_{k,k+1,3}$ and $\overline{g_{k,k+1,3}}$ emitted from the non-real zeros ($s_{k,k+1,3}$, $\overline{s_{k,k+1,3}}$). These two root loci intersect at the real breakaway point $q_{k,2k-1,3}$. From $q_{k,2k-1,3}$, the two root loci $g_{k,k+1,3}$ and $\overline{g_{k,k+1,3}}$ extend in opposite directions. The root locus $g_{k,k+1,3}$ intersects with the root locus $g_{k,k,3}$ at the real breakaway point $q_{k,2k-2,3}$. From $q_{k,2k-2,3}$, the root loci $g_{k,k+1,3}$ and $\Gamma_{k,k,3}$ leave the real axis and enter $\mathbb{C}^{*}$.

Based on the result concerning the two main lines of the second kind of non-standard linked root loci,  two main lines of the second kind of non-standard linked root loci are as follows: In $\mathbb{C}^{*}$, from $q_{k,2j_1,3}$ to $q_{k,2j_1-1,3}$, they are segments of the two root loci $g_{k,j_1+1,3}$ and $g_{k,j_1+2,3}$. On the real axis, from $q_{k,2j_1-1,3}$ to $q_{k,2j_1-2,3}$, this interval is a segment of the root locus $g_{k,j_1+1,3}$. In $\mathbb{C}^{*}$, from $q_{k,2j_1-2,3}$ to $q_{k,2j_1-3,2}$, they are segments of the root loci $g_{k,j_1+1,3}$ and $g_{k,j_1,3}$. In $\mathbb{C}^{*}$, from $s_{k,k+1,3}$ or $\overline{s_{k,k+1,3}}$ to $q_{k,2k-1,3}$, they are segments of the two root loci $g_{k,k+1,3}$ or $\overline{g_{k,k+1,3}}$. On the real axis, from $q_{k,2k-1,3}$ to $q_{k,2k-2,3}$, this interval is a segment of the root locus $g_{k,k+1,3}$. In $\mathbb{C}^{*}$, from $q_{k,2k-2,3}$ to $q_{k,2k-3,3}$, they are segments of the root loci $g_{k,k+1,3}$ and $g_{k,k,3}$.

In the standard linked root loci, the root locus $g_{k,j_1,1}$ emitted from the real zero $t_{k,j_1,1}$ and extends to the real breakaway point $q_{k,2j_1-2,1}$. According to this direction of extension, this segment of the root locus constitutes a branch of the standard linked root loci. It cannot lie on the two main lines of the standard linked root loci.

In the second kind of non-standard linked root loci, the root locus $g_{k,j_1,3}$ emitted from the real zero $t_{k,j_1,3}$ and extends to the real breakaway point $q_{k,2j_1-2,3}$. During its extension, before it reaches $q_{k,2j_1-2,3}$, it may intersect with the additional root loci on the real axis. Although the root locus $g_{k,j_1,3}$ may intersect with the additional root loci on the real axis, the direction of its extension remains from $t_{k,j_1,3}$ to the real breakaway point $q_{k,2j_1-2,3}$. According to this direction of extension, this segment of the root locus constitutes a branch of the second kind of non-standard linked root loci. It cannot lie on the two main lines of the second kind of non-standard linked root loci.

Thus, this demonstrates that the two main lines of the second kind of non-standard linked root loci are identical to the two main lines of the standard linked root loci.

The root locus $g_{1,1,4}$ is emitted from a real zero $t_{1,1,4}$ of $p(s)$. The root loci $g_{1,2,4}$ and $\overline{g_{1,2,4}}$ are separately emitted from the non-real zero $s_{1,2,4}$ and $\overline{s_{1,2,4}}$. $g_{1,1,4}$ intersects with the additional root loci at real breakaway points on the real axis. The additional root loci do not include $g_{1,2,4}$.

The root loci of (2.2) of $2q\pi$ and $2q\pi+\pi$ degrees are symmetric about the real axis. The conjugate non-real zeros and poles of (2.2) are symmetric about the real axis. The root locus $\overline{g_{1,2,4}}$ emitted from the conjugate zero $\overline{s_{1,2,4}}$, so $g_{1,2,4}$ and $\overline{g_{1,2,4}}$ are symmetric about the real axis. After $\overline{g_{1,2,4}}$ and $g_{1,2,4}$ intersect at $q_{1,1,4}$, from the real breakaway point $q_{1,1,4}$, $\overline{g_{1,2,4}}$ and $g_{1,2,4}$ extend oppositely on the real axis. Let $g_{1,2,4}$ extend to $t_{1,1,4}$. $g_{1,2,4}$ intersects with the additional root loci at real breakaway points on the real axis. The additional root loci do not include $g_{1,1,4}$. After $g_{1,2,4}$  intersects with the additional root loci on the real axis, and after $g_{1,1,4}$ intersects with the additional root loci on the real axis, $g_{1,1,4}$,  $g_{1,2,4}$ then intersects  $g_{1,1,4}$, along with other possible root loci, at the real breakaway point $Q_{1,4}$.  There exist the first type and the second type of intersections with the additional root loci that join into the original standard linked root loci.

In the above paragraph, two "the additional root loci" are mentioned. The two "the additional root loci" have no any relationship, and they cannot intersect or connect to each other.  They represent two categories of additional root loci distributed at different locations.

Therefore, $g_{1,1,4}$ together with $g_{1,2,4}$ and $\overline{g_{1,2,4}}$ form a connected structure. We call them the third kind of non-standard linked root loci.

$g_{k,1,4}$ originates from the real zero $t_{k,1,4}$ of $p(s)$. $g_{k,2,4}$ originates from the real zero $t_{k,2,4}$ of $p(s)$.  If  $g_{k,1,4}$ doesn't intersect with the additional root loci at real breakaway points on the real axis, and if  $g_{k,2,4}$  doesn't intersect with the additional root loci at real breakaway points on the real axis,  $g_{k,2,4}$ intersects  $g_{k,1,4}$, along with other possible root loci, at $Q_{k,4}$ on the real axis. From $Q_{k,4}$, $g_{k,1,4}$ and $g_{k,2,4}$ leave the real axis and enters $\mathbb{C}^{*}$.  By Lemma 2.3, there is at least one real zero of (2.2) between $t_{k,1,4}$ and $t_{k,2,4}$ to separate them. Therefore, on the real axis, $g_{k,2,4}$ cannot directly intersect with $g_{k,1,4}$. $g_{k,2,4}$ must leave the real axis and enter $\mathbb{C}^{*}$. $1\le k<n$.

$g_{k,2,4}$ leaves the real axis and enters $\mathbb{C}^{*}$. By Lemma 2.7, this result requires that the third root locus, denoted as $g_{k,3,4}$, intersects $g_{k,2,4}$ at $q_{k,2,4}$ on the real axis. Suppose $g_{k,3,4}$ originates from the real zero $t_{k,3,4}$ of $p$.  From $q_{k,2,4}$, $g_{k,2,4}$ and $g_{k,3,4}$ leave the real axis and enter $\mathbb{C}^{*}$.

Because $g_{k,2,4}$ needs to intersect with $g_{k,1,4}$ at $Q_{k,4}$, $g_{k,2,4}$ and $g_{k,3,4}$ leave the real axis and enter $\mathbb{C}^{*}$. Due to symmetry about the real axis, $g_{k,2,4}$ and $g_{k,3,4}$ which have left the real axis must intersect again at a real breakaway point $q_{k,1,4}$. From $q_{k,1,4}$, $g_{k,2,4}$ and $g_{k,3,4}$ extend in opposite directions along the real axis. Suppose $g_{k,2,4}$ extends toward $t_{k,1,4}$. Then $g_{k,1,4}$ and the root locus $g_{k,2,4}$ extend toward each other and intersect, along with other possible root loci, at $Q_{k,4}$.

If $g_{k,1,4}$ intersects with the additional root loci at real breakaway points on the real axis, and if  $g_{k,2,4}$ doesn't intersect with the additional root loci at real breakaway points on the real axis. The additional root loci do not include $g_{k,2,4}$. After $g_{k,1,4}$ intersects with the additional root loci on the real axis, $g_{k,2,4}$ intersects  $g_{k,1,4}$ along with other possible root loci, at $Q_{k,4}$ on the real axis. From $Q_{k,4}$, $g_{k,1,4}$ and $g_{k,2,4}$ leave the real axis and enters $\mathbb{C}^{*}$.  By Lemma 2.3, there is at least one real zero of (2.2) between $t_{k,1,4}$ and $t_{k,2,4}$ to separate them. Therefore, on the real axis, $g_{k,2,4}$ cannot directly intersect with $g_{k,1,4}$. $g_{k,2,4}$ must leave the real axis and enter $\mathbb{C}^{*}$.

$g_{k,2,4}$ leaves the real axis and enters $\mathbb{C}^{*}$. By Lemma 2.7, this result requires that the third root locus, denoted as $g_{k,3,4}$, intersects $g_{k,2,4}$ at $q_{k,2,4}$ on the real axis. Suppose $g_{k,3,4}$ originates from the real zero $t_{k,3,4}$ of $p$.  From $q_{k,2,4}$, $g_{k,2,4}$ and $g_{k,3,4}$ leave the real axis and enter $\mathbb{C}^{*}$.

Because $g_{k,2,4}$ needs to intersect with $g_{k,1,4}$ at $Q_{k,4}$, $g_{k,2,4}$ and $g_{k,3,4}$ leave the real axis and enter $\mathbb{C}^{*}$. Due to symmetry about the real axis, $g_{k,2,4}$ and $g_{k,3,4}$ which have left the real axis must intersect again at a real breakaway point $q_{k,1,4}$. From $q_{k,1,4}$, $g_{k,2,4}$ and $g_{k,3,4}$ extend in opposite directions along the real axis.  Suppose  $g_{k,2,4}$  extends toward $t_{k,1,4}$. Then $g_{k,1,4}$ and the root locus $g_{k,2,4}$ extend toward each other and intersect, along with other possible root loci, at $Q_{k,4}$.

If $g_{k,1,4}$ intersects with the additional root loci at real breakaway points on the real axis, the additional root loci do not include $g_{k,2,4}$. If $g_{k,2,4}$  intersects with the additional root loci at the real breakaway points on the real axis, the additional root loci do not include $g_{k,1,4}$. After $g_{k,1,4}$ intersects with the additional root loci on the real axis,
and after $g_{k,2,4}$ intersects with the additional root loci on the real axis, repeated the previous proof, we get: $g_{k,2,4}$ and $g_{k,3,4}$ extend from $q_{k,2,4}$ to $q_{k,1,4}$ in $\mathbb{C}^{*}$. $g_{k,1,4}$ and the root locus $g_{k,2,4}$ extend toward each other and intersect, along with other possible root loci, at $Q_{k,4}$.

If $g_{k,1,4}$ doesn't intersect with the additional root loci at real breakaway points on the real axis, if $g_{k,2,4}$  intersects with the additional root loci at the real breakaway points on the real axis, the additional root loci do not include $g_{k,1,4}$. After $g_{k,2,4}$ intersects with the additional root loci on the real axis, repeated the previous proof, we get: $g_{k,2,4}$ and $g_{k,3,4}$ extend from $q_{k,2,4}$ to $q_{k,1,4}$ in $\mathbb{C}^{*}$. $g_{k,1,4}$ and the root locus $g_{k,2,4}$ extend toward each other and intersect, along with other possible root loci, at $Q_{k,4}$.

By Lemma 2.3, there must be at least one real zero of (2.2) between $t_{k,2,4}$ and $t_{k,3,4}$ to separate them. Therefore, on the real axis, $g_{k,3,4}$ cannot directly intersect $g_{k,2,4}$. $g_{k,3,4}$ must leave the real axis and enter $\mathbb{C}^{*}$.

If the intersecting root loci continue to originate from real zeros, such a process must be carried on.

The intersection points of root loci are breakaway points of (2.2). By Theorem 1.28, the number of breakaway points of (2.2) is finite, implying the number of intersecting root loci is finite. Hence, this process cannot continue indefinitely. However, based on prior results, such an intersection chain can only terminate when a root locus  from a non-real zero  of $p(s)$ intersects a root locus from a real zero at the real breakaway point on the real axis.

The root locus $g_{k,k+1,4}$ is emitted from the conjugate non-real zero $s_{k,k+1,4}$.
Because the root loci of (2.2) of $2q\pi$ and $2q\pi+\pi$ degrees are symmetrical about the real axis. The conjugate non-real zeros and poles of (2.2) are also symmetrical about the real axis. The root locus $\overline{g_{k,k+1,4}}$ is emitted from the conjugate non-real zero $\overline{s_{k,k+1,4}}$. Therefore, $g_{k,k+1,4}$ and $\overline{g_{k,k+1,4}}$ are symmetrical about the real axis. $g_{k,k,4}$ originates from the real zero $t_{k,k,4}$ of $p(s)$.

If $g_{k,k+1,4}$ and $g_{k,k,4}$ intersect on the real axis, then $\overline{g_{k,k+1,4}}$ must intersect with $g_{k,k+1,4}$ first at $q_{k,2k-1,4}$. After $\overline{g_{k,k+1,4}}$ intersects with $g_{k,k+1,4}$, on the real axis, they extend in opposite directions after $q_{k,2k-1,4}$. Let $g_{k,k+1,4}$ extend to $t_{k,k,4}$.

If $g_{k,k,4}$ doesn't intersect with the additional root loci at real breakaway points on the real axis, and if  $g_{k,k+1,4}$ doesn't intersect with the additional root loci at real breakaway points on the real axis, on the real axis, $g_{k,k+1,4}$  intersects with $g_{k,k,4}$ at $q_{k,2k-2,4}$. From $q_{k,2k-2,4}$, $g_{k,k,4}$ and $g_{k,k+1,4}$  leave the real axis and enter $\mathbb{C}^{*}$.

If $g_{k,k,4}$ intersects with the additional root loci at real breakaway points on the real axis, the additional root loci do not include $g_{k,k+1,4}$. If  $g_{k,k+1,4}$  doesn't intersect with the additional root loci at real breakaway points on the real axis, after $g_{k,k,4}$ intersects with the additional root loci at real breakaway points on the real axis, $g_{k,k+1,4}$  intersects with $g_{k,k,4}$ at $q_{k,2k-2,4}$. From $q_{k,2k-2,4}$, $g_{k,k,4}$ and $g_{k,k+1,4}$  leave the real axis and enter $\mathbb{C}^{*}$.

If $g_{k,k,4}$ intersects with the additional root loci at real breakaway points on the real axis, the additional root loci do not include $g_{k,k+1,4}$. If  $g_{k,k+1,4}$   intersects with the additional root loci at real breakaway points on the real axis, the additional root loci do not include $g_{k,k,4}$. After $g_{k,k,4}$ intersects with the additional root loci at real breakaway points on the real axis,  and after $g_{k,k+1,4}$  intersects with the additional root loci at real breakaway points on the real axis, $g_{k,k+1,4}$  intersects with $g_{k,k,4}$ at $q_{k,2k-2,4}$. From $q_{k,2k-2,4}$, $g_{k,k,4}$ and $g_{k,k+1,4}$  leave the real axis and enter $\mathbb{C}^{*}$.

If $g_{k,k,4}$ doesn't intersect with the additional root loci at real breakaway points on the real axis, and if  $g_{k,k+1,4}$   intersects with the additional root loci at real breakaway points on the real axis, the additional root loci do not include $g_{k,k,4}$. After $g_{k,k+1,4}$  intersects with the additional root loci at real breakaway points on the real axis, $g_{k,k+1,4}$  intersects with $g_{k,k,4}$ at $q_{k,2k-2,4}$. From $q_{k,2k-2,4}$, $g_{k,k,4}$ and $g_{k,k+1,4}$ leave the real axis and enter $\mathbb{C}^{*}$.

During the formation of the third kind of non-standard linked root loci, a root locus emitted from a real zero intersects with the additional root loci and then intersects again with a root locus emitted from a real zero. Alternatively, a root locus emitted from a non-real zero intersects with the additional root loci and then intersects with a root locus emitted from a real zero. There exists at least one instance of the first type of intersection with the additional root loci.

A root locus emitted from a real zero intersects with the additional root loci and then intersects again with a root locus emitted from a real zero. Alternatively, a root locus emitted from a real zero intersects with the additional root loci and then intersects with a root locus emitted from a pair of non-real zero.
There exists at least one instance of the second type of intersection with the additional root loci.
These are the additional root loci that join into the original standard linked root loci.

Therefore, $g_{k,1,4}$, $g_{k,2,4}$, $\cdots$, together with $g_{k,k+1,4}$ and $\overline{g_{k,k+1,4}}$ form a connected structure. We call them the third kind of non-standard linked root loci.

In the third kind of non-standard linked root loci, $k$ root loci emitted from the real poles of (2.2) connect with a pair of root loci emitted from the non-real zeros. The additional root loci join the linked root loci.

The root locus emitted from a real pole of (2.2) may intersect with the root locus emitted from the next real pole of (2.2). The root locus emitted from a real pole of (2.2) may intersect with the additional root loci, then the root locus emitted from a real pole of (2.2) intersects with the root locus emitted from the next real pole of (2.2). The root locus emitted from a pair of non-real pole of (2.2) may intersect with the root locus emitted from the real pole of (2.2). The root locus emitted from a pair of non-real pole of (2.2) may intersect with the additional root loci, then the root locus emitted from a pair of non-real pole of (2.2) intersects with the root locus emitted from the real pole of (2.2). These intersection of the root loci occurs within the intervals between real poles and real zeros.
The rule for forming the third non-standard linked root loci is straightforward.
\begin{definition}
When $k>1$ and a root locus $g_{k,1,4}$ emitted from a real zero $t_{k,1,4}$ of $p(s)$ leaves real axis and enters $\mathbb{C}^{*}$, beginning from $g_{k,1,4}$, consider root loci $g_{k,j_1,4}$ (for $1\le k<n$, $j_1=1,\cdots,k$, $0<2j_1-2<2k-2$, $0<2j_1-1<2k-1$) emitted from distinct real poles $t_{k,j_1,4}$ of (2.2). At the end point of the structure, there are a pair of root loci $g_{k,k+1,4}$ and $\overline{g_{k,k+1,4}}$ are emitted from a pair of non-real zeros $s_{k,k+1,4}$ and $\overline{s_{k,k+1,4}}$. These root loci intersect at real breakaway points $Q_{k,4}$, $q_{k,2j_1-2,4}$ $q_{k,2j_1-1,4}$ on the real axis. These intersecting root loci ($g_{k,1,4}$, $g_{k,j_1,4}$, $g_{k,k+1,4}$ and $\overline{g_{k,k+1,4}}$) are connected by their real breakaway points $Q_{k,4}$, $q_{k,2j_1-1,4}$, $q_{k,2j_1-2,4}$. There must be the additional root loci that join into the original standard linked root loci. There exists at least one instance of the first type of intersection with the additional root loci and one instance of the second type of intersection with the additional root loci. The structure formed by these root loci ($g_{k,1,4}$, $g_{k,j_1,4}$, $g_{k,k+1,4}$ and $\overline{g_{k,k+1,4}}$ and the additional root loci) is called \emph{the third kind of non-standard linked root loci of (2.2)}.
\end{definition}
When $k=1$ and a root locus $g_{1,1,4}$ emitted from a real zero $t_{1,1,4}$ of $p(s)$ leaves real axis and enters $\mathbb{C}^{*}$, $k=1$. Beginning from $g_{1,1,4}$, at the end point of the structure, there are a pair of root loci $g_{1,2,4}$ and $\overline{g_{1,2,4}}$ are emitted from a pair of non-real zeros $s_{1,2,4}$ and $\overline{s_{1,2,4}}$. These root loci intersect at real breakaway points $Q_{1,4}$, $q_{1,1,4}$ on the real axis. These intersecting root loci ($g_{1,1,4}$, $g_{1,2,4}$ and $\overline{g_{1,2,4}}$) are connected by their real breakaway points $Q_{1,4}$, $q_{1,1,4}$. There must be the additional root loci that join into the original standard linked root loci. There exists at least one instance of the first type of intersection with the additional root loci and one instance of the second type of intersection with the additional root loci. The structure formed by these root loci ($g_{1,1,4}$, $g_{1,2,4}$ and $\overline{g_{1,2,4}}$ and the additional root loci) is called \emph{the third kind of non-standard linked root loci of (2.2)}.

Let $g_{2,1,4}$ be the root locus emitted from a real zero $t_{2,1,4}$ of $p$, and let $z_{2,1,4}$ be the real zero of (2.2) adjacent to $t_{2,1,4}$. The root locus $g_{2,1,4}$ would leave the real axis in the interval $L_{2,1,4}$ between $t_{2,1,4}$ and $z_{2,1,4}$ to enter $\mathbb{C}^{*}$. However, if no other root locus intersects $g_{2,1,4}$ in $L_{2,1,4}$, then $g_{2,1,4}$ must end at $z_{2,1,4}$ and remain on the real axis. Thus, for $g_{2,1,4}$ to leave the real axis and enter  $\mathbb{C}^{*}$, another root locus must intersect $g_{2,1,4}$ within $L_{2,1,4}$. We denote this intersecting root locus by $g_{2,2,4}$.

Suppose the root locus $g_{2,2,4}$ departs from another real zero $t_{2,2,4}$ of $p$. Let $z_{2,2,4}$ be the real zero of (2.2) adjacent to $t_{2,2,4}$. $g_{2,2,4}$ leaves the real axis from the interval $L_{19,2}$ between $t_{2,2,4}$ and $z_{2,2,4}$ to enter $\mathbb{C}^{*}$. If no other root loci intersect $g_{2,2,4}$ within $L_{2,2,4}$, then $g_{2,2,4}$ emitted from $t_{2,2,4}$ must end at $z_{2,2,4}$ and remain on the real axis. Therefore, for $g_{2,2,4}$ to leave the real axis and enter  $\mathbb{C}^{*}$, another root locus must intersect $g_{2,2,4}$ within $L_{2,2,4}$. We denote this intersecting root locus by $g_{2,3,4}$. $g_{2,3,4}$ and $\overline{g_{2,3,4}}$ are emitted from a pair of non-real zeros $s_{2,3,4}$ and $\overline{s_{2,3,4}}$ of $p(s)$.

Regarding the first type of intersection with the additional root loci under the conditions of Lemma 2.37, at least one of the two cases must occur. Case 1. $g_{2,2,4}$  intersects with the additional root loci at real breakaway points on the real axis. The additional root loci do not contain $g_{2,1,4}$. After $g_{2,2,4}$ intersects with the additional root loci on the real axis, $g_{2,2,4}$ then intersects with $g_{2,1,4}$ at a real breakaway point.

Case 2. $g_{2,3,4}$ intersects with the additional root loci at real breakaway points on the real axis. The additional root loci do not contain $g_{2,2,4}$. After $g_{2,3,4}$ intersects with the additional root loci on the real axis, $g_{2,3,4}$  then intersects with $g_{2,2,4}$ at a real breakaway point.

Regarding the second type of intersection with the additional root loci under the conditions of Lemma 2.37, at least one of the two cases must occur. Case 1. $g_{2,1,4}$  intersects with the additional root loci at real breakaway points on the real axis. The additional root loci do not contain $g_{2,2,4}$. After $g_{2,1,4}$  intersects with the additional root loci on the real axis, $g_{2,1,4}$  then intersects with  $g_{2,2,4}$ at a real breakaway point.

Case 2. $g_{2,2,4}$  intersects with the additional root loci at real breakaway points on the real axis. The additional root loci do not contain $g_{2,3,4}$. After $g_{2,2,4}$  intersects with the additional root loci on the real axis, $g_{2,2,4}$ then intersects with $g_{2,3,4}$  at a real breakaway point.

\begin{lemma}
The root loci $g_{2,1,4}$, $g_{2,2,4}$, and possibly other root loci intersect at a real breakaway point $Q_{2,4}$. At $Q_{2,4}$, $g_{2,1,4}$ leaves the real axis and enters  $\mathbb{C}^{*}$. $g_{2,3,4}$ and $\overline{g_{2,3,4}}$ are emitted from a pair of non-real zeros $s_{2,3,4}$ and $\overline{s_{2,3,4}}$ of $p(s)$. Regarding the first type of intersection with the additional root loci under the conditions of this lemma, at least one of the two cases must occur. Regarding the second type of intersection with the additional root loci under the conditions of this lemma, at least one of the two cases must occur. Then the root loci $g_{2,1,4}$, $g_{2,2,4}$, $g_{2,3,4}$, $\overline{g_{2,3,4}}$  and the additional root loci form the third kind of non-standard linked root loci. The terminus of this structure consists of $g_{2,3,4}$ or $\overline{g_{2,3,4}}$.
\end{lemma}
\begin{proof}
$g_{2,3,4}$ and $\overline{g_{2,3,4}}$ intersect at a real breakaway point $q_{2,3,4}$ within $L_{2,2,4}$. From $q_{2,3,4}$, $g_{2,3,4}$ and $\overline{g_{2,3,4}}$  extend oppositely on the real axis, suppose $g_{2,3,4}$ extends toward $t_{2,2,4}$. $\overline{g_{2,3,4}}$ maybe ends at $z_{2,2,4}$. $\overline{g_{2,3,4}}$ maybe ends at the other zero of (2.2).

If $g_{2,3,4}$  doesn't intersect with the additional root loci on the real axis, and if $g_{2,2,4}$  doesn't intersect with the additional root loci on the real axis, then $g_{2,2,4}$ and $g_{2,3,4}$ extend toward each other and intersect at a real breakaway point $q_{2,2,4}$ within $L_{2,2,4}$.

If $g_{2,3,4}$ doesn't intersect with the additional root loci on the real axis, and if $g_{2,2,4}$  intersects with the additional root loci on the real axis, the additional root loci do not contain $g_{2,3,4}$. After $g_{2,2,4}$  intersects with the additional root loci on the real axis, then $g_{2,2,4}$ and $g_{2,3,4}$ extend toward each other and intersect at a real breakaway point $q_{2,2,4}$ within $L_{2,2,4}$.

If $g_{2,3,4}$   intersects with the additional root loci on the real axis, the additional root loci do not contain $g_{2,2,4}$. If $g_{2,2,4}$  intersects with the additional root loci on the real axis, the additional root loci do not contain $g_{2,3,4}$.  After $g_{2,3,4}$  intersects with the additional root loci on the real axis, and after $g_{2,2,4}$ intersects with the additional root loci on the real axis, then $g_{2,2,4}$ and $g_{2,3,4}$ extend toward each other and intersect at a real breakaway point $q_{2,2,4}$ within $L_{2,2,4}$.

If $g_{2,3,4}$  intersects with the additional root loci on the real axis, the additional root loci do not contain $g_{2,2,4}$. If $g_{2,2,4}$ doesn't intersect with the additional root loci on the real axis, after $g_{2,3,4}$ intersects with the additional root loci on the real axis, then $g_{2,2,4}$ and $g_{2,3,4}$ extend toward each other and intersect at a real breakaway point $q_{2,2,4}$ within $L_{2,2,4}$.

At $q_{2,2,4}$, $g_{2,2,4}$ and $g_{2,3,4}$ leave the real axis. $g_{2,2,4}$ and $g_{2,3,4}$  enter  $\mathbb{C}^{*}$.  Due to symmetry about the real axis, $g_{2,2,4}$ and $g_{2,3,4}$ which have left the real axis must intersect again at a real breakaway point $q_{2,1,4}$ in $L_{2,1,4}$. Therefore, $g_{2,2,4}$ and $g_{2,3,4}$ extend in  $\mathbb{C}^{*}$ from $q_{2,2,4}$ to $q_{2,1,4}$. From $q_{2,1,4}$, $g_{2,2,4}$ and $g_{2,3,4}$ extend oppositely on the real axis, suppose $g_{2,2,4}$ extends toward $t_{2,1,4}$. $g_{2,3,3}$ maybe ends at $z_{2,1,4}$. $g_{2,3,4}$ maybe ends at the other zero of (2.2).

If $g_{2,2,4}$ doesn't intersect with the additional root loci on the real axis, and if $g_{2,1,4}$  doesn't intersect with the additional root loci on the real axis, then $g_{2,1,4}$ and the root locus $g_{2,2,4}$ extend toward each other and intersect, along with other possible root loci, at $Q_{2,4}$  in $L_{2,1,3}$.

If $g_{2,2,4}$ doesn't intersect with the additional root loci on the real axis, and if $g_{2,1,4}$  intersects with the additional root loci on the real axis, the additional root loci do not contain $g_{2,2,4}$. After $g_{2,1,4}$  intersects with the additional root loci on the real axis, then $g_{2,1,4}$ and the root locus $g_{2,2,4}$ extend toward each other and intersect, along with other possible root loci, at $Q_{2,4}$  in $L_{2,1,3}$.

If $g_{2,2,4}$ intersects with the additional root loci on the real axis, the additional root loci do not contain $g_{2,1,4}$. If $g_{2,1,4}$  intersects with the additional root loci on the real axis, the additional root loci do not contain $g_{2,2,4}$.  After $g_{2,2,4}$  intersects with the additional root loci on the real axis, and after $g_{2,1,4}$ intersects with the additional root loci on the real axis, then $g_{2,1,4}$ and the root locus $g_{2,2,4}$ extend toward each other and intersect, along with other possible root loci, at $Q_{2,4}$  in $L_{2,1,3}$.

If $g_{2,2,4}$  intersects with the additional root loci on the real axis, the additional root loci do not contain $g_{2,1,4}$. If $g_{2,1,4}$ doesn't intersect with the additional root loci on the real axis, after $g_{2,2,4}$  intersects with the additional root loci on the real axis, then $g_{2,1,4}$ and the root locus $g_{2,2,4}$ extend toward each other and intersect, along with other possible root loci, at $Q_{2,4}$  in $L_{2,1,3}$.

Therefore, $g_{2,1,4}$, $g_{2,2,4}$, $g_{2,3,4}$, $\overline{g_{2,3,4}}$  and the additional root loci form the third kind of non-standard linked root loci. $g_{2,1,4}$ and $g_{2,2,4}$  leave the real axis and enter  $\mathbb{C}^{*}$.

The same as the case in Lemma 2.15, the real breakaway points may coincide. This overlapping situation has no impact on the proof or the results of this lemma. Since we are unable to provide the situation of "the additional root loci", the exact situation of the real breakaway points overlapping cannot be accurately given either. We only need to know that such a situation exists.
\end{proof}
Let $g_{k,j_1,4}$ be the root locus emitted from a real zero $t_{k,j_1,4}$ of $p$. Let $z_{k,j_1,4}$ be the real zero of (2.2) adjacent to $t_{k,j_1,4}$. $g_{k,j_1,4}$ leaves the real axis at a point in the interval $L_{k,j_1,4}$ between $t_{k,j_1,4}$ and $z_{k,j_1,4}$.  The root loci from the non-real zeros $s_{k,k+1,4}$ and $\overline{s_{k,k+1,4}}$ are $g_{k,k+1,4}$ and $\overline{g_{k,k+1,4}}$.

$g_{k,j_1+1,4}$ and $g_{k,j_1+2,4}$ intersect at a real breakaway point $q_{k,2j_1-1,4}$ within $L_{k,j_1,4}$. From $q_{k,2j_1-1,4}$,  $g_{k,j_1+1,4}$ and $g_{k,j_1+2,4}$ extend  oppositely on the real axis. Let $g_{k,j_1+1}$  extend to $t_{k,j_1,4}$.

Regarding the first type of intersection with the additional root loci under the conditions of Lemma 2.38, at least one of the $k$ cases must occur. Case $j_1$. $j_1=1,\cdots,k-1$. $g_{k,j_1+1,4}$ intersects with the additional root loci at real breakaway points on the real axis. The additional root loci do not contain $g_{k,j_1,4}$. After $g_{k,j_1+1,4}$ intersects with the additional root loci on the real axis, $g_{k,j_1+1,4}$ then intersects with $g_{k,j_1,4}$ at a real breakaway point.

$g_{k,k+1,4}$ and $\overline{g_{k,k+1,4}}$ intersect at a real breakaway point $q_{k,2k-1,4}$ within $L_{k,k,4}$. From $q_{k,2k-1,4}$,  $g_{k,k+1,4}$ and $\overline{g_{k,k+1,4}}$ extend  oppositely on the real axis. Let $g_{k,k+1,4}$  extend to $t_{k,k,4}$.

Case $k$. $g_{k,k+1,4}$  intersects with the additional root loci at real breakaway points on the real axis. The additional root loci do not contain $g_{k,k,4}$. After $g_{k,k+1,4}$ intersects with the additional root loci on the real axis, $g_{k,k+1,4}$ then intersects with $g_{k,k,4}$ at a real breakaway point.

Regarding the second type of intersection with the additional root loci under the conditions of Lemma 2.38, at least one of the $k$ cases must occur.  Case $j_1$. $j_1=1,\cdots,k-1$. $g_{k,j_1,4}$  intersects with the additional root loci at real breakaway points on the real axis. The additional root loci do not contain $g_{k,j_1+1,4}$. After $g_{k,j_1,4}$ intersects with the additional root loci on the real axis,  then $g_{k,j_1,4}$ intersects with $g_{k,j_1+1,4}$  at a real breakaway point.

Case $k$. $g_{k,k,4}$ intersects with the additional root loci at real breakaway points on the real axis. The additional root loci do not contain $g_{k,k+1,4}$. After  $g_{k,k,4}$ intersects with the additional root loci on the real axis, $g_{k,k,4}$ then intersects with $g_{k,k+1,4}$ at a real breakaway point.

\begin{lemma}
When a root locus $g_{k,1,4}$ emitted from a real zero of $p(s)$ leaves the real axis and enters $\mathbb{C}^{*}$, then, there must exist a pair of non-real zeros and $k-1$ other real zeros. Regarding the first type of intersection with the additional root loci under the conditions of this lemma,  at least one of the $k$ cases must occur. Regarding the second type of intersection with the additional root loci under the conditions of this lemma,  at least one of the $k$ cases must occur. These root loci $g_{k,k+1,4}$, $\overline{g_{k,k+1,4}}$ and $g_{k,j_1,4}$ (for $1\le k<n$, $j_1=1,\cdots,k$, $0<2j_1-2<2k-2$, $0<2j_1-1<2k-1$) intersect at real breakaway points $Q_{k,4}$, $q_{k,2j_1-1,4}$,  $q_{k,2j_1-2,4}$  on the real axis. At the end point of the structure, there are a pair of root loci $g_{k,k+1,4}$ and $\overline{g_{k,k+1,4}}$ are emitted from a pair of non-real zeros $s_{k,k+1,4}$ and $\overline{s_{k,k+1,4}}$.
\end{lemma}

\begin{proof}
In Lemma 2.37, there are four main root loci in the third kind of non-standard linked root loci, and two intervals between zeros and poles of (2.2).

$g_{k,1,4}$, $g_{k,2,4}$, and possibly other root loci intersect at a real breakaway point $Q_{k,4}$  in $L_{k,1,4}$. At $Q_{k,4}$, $g_{k,1,4}$ and $g_{k,2,4}$ leave the real axis and enter  $\mathbb{C}^{*}$. The situation of $g_{k,1,4}$ and $g_{k,2,4}$ is identical to the situation given in Lemma 2.37. At the front of Definition 2.36, we have given the result of $k=1$. When $k=2$, Lemma 2.37 give results. For general $k>2$, as following.

$g_{k,j_1+1,4}$ and $g_{k,j_1+2,4}$ intersect at the real breakaway point $q_{k,2j_1,4}$ in $L_{k,j_1+1,4}$. From $q_{k,2j_1,4}$, $g_{k,j_1+1,4}$ and $g_{k,j_1+2,4}$ leaves the real axis and enters   $\mathbb{C}^{*}$. Due to symmetry about the real axis, $g_{k,j_1+1,4}$ and $g_{k,j_1+2,4}$ intersect at  the real breakaway point $q_{k,2j_1-1,4}$  in $L_{k,j_1,4}$.

$g_{k,j_1+1,4}$ and $g_{k,j_1+2,4}$ extend in  $\mathbb{C}^{*}$ from $q_{k,2j_1,4}$ to $q_{k,2j_1-1,4}$. $t_{k,j_1,4}$ emits $g_{k,j_1,4}$ to interval $L_{k,j_1,4}$. From $q_{k,j_1-1,4}$, $g_{k,j_1+1,4}$ and $g_{k,j_1+2,4}$  extend oppositely on the real axis. Suppose $g_{k,j_1+1,4}$ extends toward $t_{k,j_1,4}$. $g_{k,j_1+2,4}$ maybe ends at zero  $z_{k,j_1,4}$. $g_{k,j_1+2,4}$ maybe ends at the other zero of (2.2).

If $g_{k,j_1+1,4}$  does not intersect with the additional root loci at real breakaway points on the real axis, and if $g_{k,j_1,4}$ does not intersect with the additional root loci at real breakaway points on the real axis, both $g_{k,j_1,4}$ and $g_{k,j_1+1,4}$ extend towards each other, so $g_{k,j_1+1,4}$ intersects with $g_{k,j_1,4}$ at a real breakaway point $q_{k,2j_1-2,4}$ within $L_{k,j_1,4}$.

If $g_{k,j_1+1,4}$  does not intersect with the additional root loci at real breakaway points on the real axis, and if $g_{k,j_1,4}$  intersects with the additional root loci at real breakaway points on the real axis, the additional root loci do not contain $g_{k,j_1+1,4}$. After $g_{k,j_1,4}$ intersects with the additional root loci on the real axis, both $g_{k,j_1,4}$ and $g_{k,j_1+1,4}$ extend towards each other, so $g_{k,j_1+1,4}$ intersects with $g_{k,j_1,4}$ at a real breakaway point $q_{k,2j_1-2,4}$ within $L_{k,j_1,4}$.

If $g_{k,j_1+1,4}$ intersects with the additional root loci at real breakaway points on the real axis, the additional root loci do not contain $g_{k,j_1,4}$. If $g_{k,j_1,4}$  intersects with the additional root loci at real breakaway points on the real axis, the additional root loci do not contain $g_{k,j_1+1,4}$. After $g_{k,j_1+1,4}$  intersects with the additional root loci on the real axis, and after $g_{k,j_1,4}$ intersects with the additional root loci on the real axis, $g_{k,j_1+1,4}$  intersects with $g_{k,j_1,4}$ at  the real breakaway point $q_{k,2j_1-2,4}$  in $L_{k,j_1,4}$.

If $g_{k,j_1+1,4}$ intersects with the additional root loci at real breakaway points on the real axis, the additional root loci do not contain $g_{k,j_1,4}$. If $g_{k,j_1,4}$ does not intersect with the additional root loci at real breakaway points on the real axis, after $g_{k,j_1+1,4}$ intersects with the additional root loci on the real axis, $g_{k,j_1+1,4}$  intersects with $g_{k,j_1,4}$ at  the real breakaway point $q_{k,2j_1-2,4}$  in $L_{k,j_1,4}$.

At point $q_{k,2j_1-2,4}$, $g_{k,j_1,4}$ and $g_{k,j_1+1,4}$ leave the real axis and enters $\mathbb{C}^{*}$.
Due to symmetry about the real axis, $g_{k,j_1,4}$ and $g_{k,j_1+1,4}$ which have left the real axis must intersect again at a real breakaway point $q_{k,2j_1-3,4}$  in $L_{k,j_1-1,4}$. $g_{k,j_1,4}$ and $g_{k,j_1+1,4}$ extend in  $\mathbb{C}^{*}$ from $q_{k,2j_1-2,4}$ to $q_{k,2j_1-3,4}$. $j_1=2,3,\dots,k-2$.

When the root loci intersecting in $L_{k,j_2,4}$ originates from a real zero $t_{k,j_2,4}$. $j_2>k-2$. The root locus $g_{k,j_2,4}$ emitted from $t_{k,j_2,4}$ leaves interval $L_{k,j_2,4}$. However, if no other root locus intersects in $L_{k,j_2,4}$, then $g_{k,j_2,4}$ must end at $z_{k,j_2,4}$ and remain on the real axis. There must be a root locus $g_{k,j_2+1,4}$ intersecting in the interval $L_{k,j_2,4}$. If $g_{k,j_2+1,4}$ comes from a real zero $t_{k,j_2+1,4}$ in another interval $L_{k,j_2+1,4}$. If such a process arises indefinitely, the non-standard linked root loci cannot end and must continue to extend downward.

The breakaway points of the root loci are the intersection points of the root loci. By Theorem 1.28, the number of finite breakaway points of (2.2) is finite. The number of the intersection points of the root loci of (2.2) is finite. The number of the intersection root loci is finite. Therefore, the intersection of root loci of (2.2) must end. There must exist an intersection of root loci emitted from non-real zeros on the real axis, and root loci emitted from non-real zeros intersect in an interval between a zero and a pole. Therefore, such linked root loci cannot continue infinitely.

If the root locus emitted from a non-real zero $s_{k,k+1,4}$ intersects the interval $L_{k,k,4}$, the third kind of non-standard linked root loci end and cannot extend further downward. The third kind of non-standard linked root loci also end at the interval where the root loci emitted from a non-real zero of $p(s)$ intersect. Conversely, if no root locus emitted from a non-real zero intersects within the interval, the third kind of non-standard linked root loci cannot end either.

Therefore, based on the results we obtained earlier, we have thus proved that there must exist a pair of non-real zeros $s_{k,k+1,4}$ and $\overline{s_{k,k+1,4}}$. The root loci $g_{k,k+1,4}$ and $\overline{g_{k,k+1,4}}$ emitted from $s_{k,k+1,4}$ and $\overline{s_{k,k+1,4}}$ intersect within a certain interval $L_{k,k,4}$.
\end{proof}

\begin{lemma}
When a root locus $g_{k,1,4}$ emitted from a real zero of $p(s)$ leaves the real axis and enters $\mathbb{C}^{*}$, then, there must exist a pair of non-real zeros and $k-1$ other real zeros. Regarding the first type of intersection with the additional root loci under the conditions of this lemma,  at least one of the $k$ cases must occur.  Regarding the second type of intersection with the additional root loci under the conditions of this lemma, at least one of the $k$ cases must occur. These root loci $g_{k,k+1,4}$, $\overline{g_{k,k+1,4}}$ and $g_{k,j_1,4}$ (for $1\le k<n$, $j_1=1,\cdots,k$, $0<2j_1-2<2k-2$, $0<2j_1-1<2k-1$) intersect at real breakaway points $Q_{k,4}$, $q_{k,2j_1-1,4}$, $q_{k,2j_1-2,4}$  on the real axis. These real breakaway points connect $g_{k,j_1,4}$, $g_{k,k+1,4}$, and $\overline{g_{k,k+1,4}}$ together, forming a third kind of non-standard linked root locus. The terminus of this structure consists of $g_{k,k+1,4}$ and $\overline{g_{k,k+1,4}}$.
\end{lemma}

\begin{proof}
According to the condition of this lemma, by Lemma 2.38, at the end point of the structure, there are a pair of root loci $g_{k,k+1,4}$ and $\overline{g_{k,k+1,4}}$ are emitted from a pair of non-real zeros.

The non-real zeros $s_{k,k+1,4}$ and $\overline{s_{k,k+1,4}}$ are in  $\mathbb{C}^{*}$. $s_{k,k+1,4}$ and $\overline{s_{k,k+1,4}}$ emit the root locus $g_{k,k+1,4}$ and $\overline{g_{k,k+1,4}}$ in  $\mathbb{C}^{*}$ separately. $g_{k,k+1,4}$ and $\overline{g_{k,k+1,4}}$ intersect at a real breakaway point $q_{k,2k-1,4}$ within $L_{k,k,4}$. $g_{k,k+1,4}$ and $\overline{g_{k,k+1,4}}$ depart from $q_{k,2k-1,4}$. By Theorem 1.28, $g_{k,k+1,4}$ and $\overline{g_{k,k+1,4}}$ cannot extend in the same direction, and they must extend in opposite directions along the real axis.  Suppose $g_{k,k+1,4}$ extends towards $t_{k,k,4}$. $\overline{g_{k,k+1,4}}$ maybe ends at  zero  $z_{k,k,4}$. $\overline{g_{k,k+1,4}}$ maybe ends at the other zero of (2.2).

If $g_{k,k+1,4}$ does not intersect with the additional root loci at real breakaway points on the real axis, and if $g_{k,k,4}$ does not intersect with the additional root loci at real breakaway points on the real axis, both $g_{k,k,4}$ and $g_{k,k+1,4}$ extend towards each other, so $g_{k,k+1,4}$ intersects with $g_{k,k,4}$ at a real breakaway point $q_{k,2k-2,4}$ within $L_{k,k,4}$.

If $g_{k,k+1,4}$ does not intersect with the additional root loci at real breakaway points on the real axis, and if $g_{k,k,4}$  intersects with the additional root loci at real breakaway points on the real axis, the additional root loci do not contain $g_{k,k+1,4}$. After $g_{k,k,4}$ intersects with the additional root loci on the real axis, both $g_{k,k,4}$ and $g_{k,k+1,4}$ extend towards each other, so $g_{k,k+1,4}$ intersects with $g_{k,k,4}$ at a real breakaway point $q_{k,2k-2,4}$ within $L_{k,k,4}$.

If $g_{k,k+1,4}$ intersects with the additional root loci at real breakaway points on the real axis, the additional root loci do not contain $g_{k,k,4}$. If $g_{k,k,4}$  intersects with the additional root loci at real breakaway points on the real axis, the additional root loci do not contain $g_{k,k+1,4}$. After $g_{k,k+1,4}$ intersects with the additional root loci on the real axis, and after $g_{k,k,4}$ intersects with the additional root loci on the real axis, $g_{k,k+1,4}$ intersects with $g_{k,k,4}$ at  the real breakaway point $q_{k,2k-2,4}$  in $L_{k,k,4}$.

If $g_{k,k+1,4}$ intersects with the additional root loci at real breakaway points on the real axis, the additional root loci do not contain $g_{k,k,4}$. If $g_{k,k,4}$ does not intersect with the additional root loci at real breakaway points on the real axis, after $g_{k,k+1,4}$ intersects with the additional root loci on the real axis, $g_{k,k+1,4}$ intersects with $g_{k,k,4}$ at  the real breakaway point $q_{k,2k-2,4}$  in $L_{k,k,4}$.

At point $q_{k,2k-2,4}$, $g_{k,k,4}$ and $g_{k,k+1,4}$ leave the real axis and enters $\mathbb{C}^{*}$.

Due to symmetry about the real axis, $g_{k,k+1,4}$ and $g_{k,k,4}$ which have left the real axis must intersect again at a real breakaway point $q_{k,2k-3,4}$  within $L_{k,k-1,4}$. Therefore, $g_{k,k,4}$ and $g_{k,k+1,4}$ extend in $\mathbb{C}^{*}$ from $q_{k,2k-2,4}$ to $q_{k,2k-3,4}$.

$g_{k,k,4}$ and $g_{k,k+1,4}$ depart from $q_{k,2k-3,4}$. They must extend in opposite directions along the real axis. Suppose $g_{k,k,4}$ extends towards $t_{k,k-1,4}$. $g_{k,k+1,4}$ maybe ends at $z_{k,k-1,4}$. $g_{k,k+1,4}$ maybe ends at the other zero of (2.2).

If $g_{k,k,4}$ does not intersect with the additional root loci at real breakaway points on the real axis, and if $g_{k,k-1,4}$ does not intersect with the additional root loci at real breakaway points on the real axis, both $g_{k,k-1,4}$ and $g_{k,k,4}$ extend towards each other, so $g_{k,k,4}$ intersects with $g_{k,k-1,4}$ at a real breakaway point $q_{k,2k-4,4}$ within $L_{k,k-1,4}$.

If $g_{k,k,4}$ does not intersect with the additional root loci at real breakaway points on the real axis, and if $g_{k,k-1,4}$  intersects with the additional root loci at real breakaway points on the real axis, the additional root loci do not contain $g_{k,k,4}$. After $g_{k,k-1,4}$ intersects with the additional root loci on the real axis, both $g_{k,k-1,4}$ and $g_{k,k,4}$ extend towards each other, so $g_{k,k,4}$ intersects with $g_{k,k-1,4}$ at a real breakaway point $q_{k,2k-4,4}$ within $L_{k,k-1,4}$.

If $g_{k,k,4}$ intersects with the additional root loci at real breakaway points on the real axis, the additional root loci do not contain $g_{k,k-1,4}$. If $g_{k,k-1,4}$  intersects with the additional root loci at real breakaway points on the real axis, the additional root loci do not contain $g_{k,k,4}$. After $g_{k,k,4}$ intersects with the additional root loci on the real axis, and after $g_{k,k-1,4}$ intersects with the additional root loci on the real axis, $g_{k,k,4}$ intersects with $g_{k,k-1,4}$ at  the real breakaway point $q_{k,2k-4,4}$  in $L_{k,k-1,4}$.

If $g_{k,k,4}$ intersects with the additional root loci at real breakaway points on the real axis, the additional root loci do not contain $g_{k,k-1,4}$. If $g_{k,k-1,4}$ does not intersect with the additional root loci at real breakaway points on the real axis, after $g_{k,k,4}$ intersects with the additional root loci on the real axis, $g_{k,k,4}$  intersects with $g_{k,k-1,4}$ at  the real breakaway point $q_{k,2k-4,4}$  in $L_{k,k-1,4}$.

At point $q_{k,2k-4,4}$, $g_{k,k-1,4}$ and $g_{k,k,4}$ leave the real axis and enter  $\mathbb{C}^{*}$.

Due to symmetry about the real axis, $g_{k,k-1,4}$ and $g_{k,k,4}$ which have left the real axis must intersect again at a real breakaway point $q_{k,2k-5,4}$. Therefore, $g_{k,k-1,4}$ and $g_{k,k,4}$ extend in $\mathbb{C}^{*}$ from $q_{k,2k-4,4}$ to $q_{k,2k-5,4}$. Let $g_{k,k-1,4}$ extend to $t_{k,k-2,4}$. $g_{k,k,4}$ maybe ends at $z_{k,k-2,4}$. $g_{k,k,4}$ maybe ends at the other zero of (2.2).  $g_{k,k-1,4}$ intersects with $g_{k,k-2,4}$.

When such a situation appears, starting from the root locus $g_{k,1,4}$ emitted from the real zero $t_{k,1,4}$ of $p(s)$, the third kind of non-standard linked root loci terminate and do not extend further. The terminus of the third kind of non-standard linked root loci is composed of the root loci $g_{k,k+1,4}$ and $\overline{g_{k,k+1,4}}$.
The root loci $g_{k,1,4}$, $g_{k,2,4}$, $\cdots$, $g_{k,k+1,4}$, and $\overline{g_{k,k+1,4}}$  and the additional root loci form the third kind of non-standard linked root loci.

According to the previously proved result of the third kind of non-standard linked root loci, the root locus $g_{k,1,4}$ at the beginning of the third kind of non-standard linked root loci and the root loci $g_{k,k+1,4}$ and $\overline{g_{k,k+1,4}}$ at the end of the third kind of non-standard linked root loci, by the root loci $g_{k,j_1,4}$ of real zeros in the middle, must be connected in the third kind of non-standard linked root loci. The beginning of the third kind of non-standard linked root loci is a root locus $g_{k,1,4}$ emitted from a real zero $t_{k,1,4}$, and $g_{k,1,4}$ leaves the real axis and enters  $\mathbb{C}^{*}$. The end of the third kind of non-standard linked root loci must be a pair of root loci $g_{k,k+1,4}$ and $\overline{g_{k,k+1,4}}$. The root loci $g_{k,k+1,4}$ and $\overline{g_{k,k+1,4}}$ intersect on the real axis, and they are emitted from the non-real zeros $s_{k,k+1,4}$ and $\overline{s_{k,k+1,4}}$.

In Lemma 2.39, there are $k+2$ main root loci in the third kind of non-standard linked root loci, and $k$ intervals between zeros and poles of (2.2).

The same as the case in Lemma 2.17, the real breakaway points may coincide. This overlapping situation has no impact on the proof or the results of this lemma. Since we are unable to provide the situation of "the additional root loci," the exact situation of the real breakaway points overlapping cannot be accurately given either. We only need to know that such a situation exists.
\end{proof}

To make it clear for us to explain the concept of two main lines of the third kind of non-standard linked root loci below, it starts from the terminal root locus and follows all the way back to the beginning root locus.

The non-real zeros $s_{k,k+1,4}$ and $\overline{s_{k,k+1,4}}$ are in  $\mathbb{C}^{*}$. $s_{k,k+1,4}$ and $\overline{s_{k,k+1,4}}$ emit the root locus $g_{k,k+1,4}$ and $\overline{g_{k,k+1,4}}$ in  $\mathbb{C}^{*}$ separately. $g_{k,k+1,4}$ and $\overline{g_{k,k+1,4}}$ intersect at a real breakaway point $q_{k,2k-1,4}$ within $L_{k,k,4}$.

$g_{k,k+1,4}$ and $\overline{g_{k,k+1,4}}$ depart from $q_{k,2k-1,4}$. $g_{k,k+1,4}$ and $\overline{g_{k,k+1,4}}$ extend in opposite directions along the real axis. Let $g_{k,k+1,4}$  extend to $t_{k,k,4}$. $\overline{g_{k,k+1,4}}$ maybe ends at  zero  $z_{k,k,4}$. $\overline{g_{k,k+1,4}}$ maybe ends at the other zero of (2.2).

If $g_{k,k+1,4}$  does not intersect with the additional root loci at real breakaway points on the real axis, and if $g_{k,k,4}$ does not intersect with the additional root loci at real breakaway points on the real axis, $g_{k,k+1,4}$ intersects with $g_{k,k,4}$ at a real breakaway point $q_{k,2k-2,4}$ within $L_{k,k,4}$. At point $q_{k,2k-2,4}$, $g_{k,k,4}$ and $g_{k,k+1,4}$ leave the real axis and enter $\mathbb{C}^{*}$. In this situation, there doesn't exist an instance of the first type of intersection with the additional root loci and an instance of the second type of intersection with the additional root loci.

If $g_{k,k+1,4}$  does not intersect with the additional root loci at real breakaway points on the real axis, and if $g_{k,k,4}$  intersects with the additional root loci at real breakaway points on the real axis, the additional root loci do not contain $g_{k,k+1,4}$. After $g_{k,k,4}$ intersects with the additional root loci on the real axis, $g_{k,k+1,4}$ intersects with $g_{k,k,4}$ at a real breakaway point $q_{k,2k-2,4}$ within $L_{k,k,4}$. At point $q_{k,2k-2,4}$, $g_{k,k,4}$ and $g_{k,k+1,4}$ leave the real axis and enter $\mathbb{C}^{*}$. In this situation, there exists  one instance of the second type of intersection with the additional root loci.

If $g_{k,k+1,4}$  intersects with the additional root loci at real breakaway points on the real axis, the additional root loci do not contain $g_{k,k,4}$. If $g_{k,k,4}$ intersects with the additional root loci at real breakaway points on the real axis, the additional root loci do not contain $g_{k,k+1,4}$. After $g_{k,k+1,4}$  intersects with the additional root loci on the real axis, and after $g_{k,k,4}$ intersects with the additional root loci on the real axis, $g_{k,k+1,4}$ intersects with $g_{k,k,4}$ at a real breakaway point $q_{k,2k-2,4}$ within $L_{k,k,4}$. At point $q_{k,2k-2,4}$, $g_{k,k,4}$ and $g_{k,k+1,4}$ leave the real axis and enter $\mathbb{C}^{*}$. In this situation, there exists at least one instance of the first type of intersection with the additional root loci and one instance of the second type of intersection with the additional root loci.

If $g_{k,k+1,4}$  intersects with the additional root loci at real breakaway points on the real axis, the additional root loci do not contain $g_{k,k,4}$. If $g_{k,k,4}$ does not intersect with the additional root loci at real breakaway points on the real axis, after $g_{k,k+1,4}$ intersects with the additional root loci on the real axis, $g_{k,k+1,4}$ intersects with $g_{k,k,4}$ at a real breakaway point $q_{k,2k-2,4}$ within $L_{k,k,4}$. At point $q_{k,2k-2,4}$, $g_{k,k,4}$ and $g_{k,k+1,4}$ leave the real axis and enter $\mathbb{C}^{*}$. In this situation, there exists  one instance of the first type of intersection with the additional root loci.

Due to symmetry about the real axis, $g_{k,k+1,4}$ and $g_{k,k,4}$ which have left the real axis must intersect again at a real breakaway point $q_{k,2k-3,4}$ within $L_{k,k-1,4}$. Therefore, $g_{k,k,4}$ and $g_{k,k+1,4}$ extend in $\mathbb{C}^{*}$ from $q_{k,2k-2,4}$ to $q_{k,2k-3,4}$.
$g_{k,k,4}$ and $g_{k,k+1,4}$ depart from $q_{k,2k-3,4}$.  They must extend in opposite directions along the real axis. Let $g_{k,k,4}$ extend to $t_{k,k-1,4}$. $g_{k,k+1,4}$ maybe ends at $z_{k,k-1,4}$. $g_{k,k+1,4}$ maybe ends at the other zero of (2.2).

If $g_{k,k,4}$  does not intersect with the additional root loci at real breakaway points on the real axis, and if $g_{k,k-1,4}$ does not intersect with the additional root loci at real breakaway points on the real axis, $g_{k,k,4}$ intersects with $g_{k,k-1,4}$ at  the real breakaway point $q_{k,2k-4,4}$  in $L_{k,k-1,4}$. At point $q_{k,2k-4,4}$, $g_{k,k-1,4}$ and $g_{k,k,4}$ leave the real axis and enter  $\mathbb{C}^{*}$. In this situation, there doesn't exist an instance of the first type of intersection with the additional root loci and an instance of the second type of intersection with the additional root loci.

If $g_{k,k,4}$ does not intersect with the additional root loci at real breakaway points on the real axis, and if  $g_{k,k-1,4}$  intersects with the additional root loci at real breakaway points on the real axis, the additional root loci do not contain $g_{k,k,4}$. After $g_{k,k-1,4}$ intersects with the additional root loci on the real axis, $g_{k,k,4}$ intersects with $g_{k,k-1,4}$ at  the real breakaway point $q_{k,2k-4,4}$  in $L_{k,k-1,4}$. At point $q_{k,2k-4,4}$, $g_{k,k-1,4}$ and $g_{k,k,4}$ leave the real axis and enter  $\mathbb{C}^{*}$. In this situation, there exists  one instance of the second type of intersection with the additional root loci.

If $g_{k,k,4}$ intersects with the additional root loci at real breakaway points on the real axis, the additional root loci do not contain $g_{k,k-1,4}$. If $g_{k,k-1,4}$  intersects with the additional root loci at real breakaway points on the real axis, the additional root loci do not contain $g_{k,k,4}$. After $g_{k,k,4}$ intersects with the additional root loci on the real axis, and after $g_{k,k-1,4}$ intersects with the additional root loci on the real axis, $g_{k,k,4}$ intersects with $g_{k,k-1,4}$ at  the real breakaway point $q_{k,2k-4,4}$  in $L_{k,k-1,4}$. At point $q_{k,2k-4,4}$, $g_{k,k-1,4}$ and $g_{k,k,4}$ leave the real axis and enter  $\mathbb{C}^{*}$. In this situation, there exists at least one instance of the first type of intersection with the additional root loci and one instance of the second type of intersection with the additional root loci.

If $g_{k,k,4}$ intersects with the additional root loci at real breakaway points on the real axis, the additional root loci do not contain $g_{k,k-1,4}$. If $g_{k,k-1,4}$ does not intersect with the additional root loci at real breakaway points on the real axis, after $g_{k,k,4}$ intersects with the additional root loci on the real axis, $g_{k,k,4}$ intersects with $g_{k,k-1,4}$ at  the real breakaway point $q_{k,2k-4,4}$  in $L_{k,k-1,4}$. At point $q_{k,2k-4,4}$, $g_{k,k-1,4}$ and $g_{k,k,4}$ leave the real axis and enter  $\mathbb{C}^{*}$. In this situation, there exists  one instance of the first type of intersection with the additional root loci.

Due to symmetry about the real axis, $g_{k,k-1,4}$ and $g_{k,k,4}$ which have left the real axis must intersect again at a real breakaway point $q_{k,2k-5,4}$ within $L_{k,k-2,4}$. Therefore, $g_{k,k-1,4}$ and $g_{k,k,4}$ extend in $\mathbb{C}^{*}$ from $q_{k,2k-4,4}$ to $q_{k,2k-5,4}$.

$g_{k,j_1+1,4}$ and $g_{k,j_1+2,4}$ intersect at the real breakaway point $q_{k,2j_1,4}$ in $L_{k,j_1+1,4}$. From $q_{k,2j_1,4}$, $g_{k,j_1+1,4}$ and $g_{k,j_1+2,4}$ leave the real axis and enters $\mathbb{C}^{*}$. $g_{k,j_1+1,4}$ and $g_{k,j_1+2,4}$ intersect at  the real breakaway point $q_{k,2j_1-1,4}$  in $L_{k,j_1,4}$. $g_{k,j_1+1,4}$ and $g_{k,j_1+2,4}$ extend in  $\mathbb{C}^{*}$ from $q_{k,2j_1,4}$ to $q_{k,2j_1-1,4}$. From $q_{k,2j_1-1,4}$, $g_{k,j_1+1,4}$ and $g_{k,j_1+2,4}$ must extend in opposite directions along the real axis. Let $g_{k,j_1+1,4}$ extend to $t_{k,j_1,4}$. $g_{k,j_1+2,4}$ maybe ends at zero  $z_{k,j_1,4}$. $g_{k,j_1+2,4}$ maybe ends at the other zero of (2.2).

If $g_{k,j_1+1,4}$  does not intersect with the additional root loci at real breakaway points on the real axis, and if $g_{k,j_1,4}$ does not intersect with the additional root loci at real breakaway points on the real axis, $g_{k,j_1+1,4}$  intersects with $g_{k,j_1,4}$ at  the real breakaway point $q_{k,2j_1-2,4}$  in $L_{k,j_1,4}$. At point $q_{k,2j_1-2,4}$, $g_{k,j_1,4}$ and $g_{k,j_1+1,4}$ leave the real axis and enters $\mathbb{C}^{*}$. In this situation, there doesn't exist an instance of the first type of intersection with the additional root loci and an instance of the second type of intersection with the additional root loci.

If $g_{k,j_1+1,4}$  does not intersect with the additional root loci at real breakaway points on the real axis, and if $g_{k,j_1,4}$  intersects with the additional root loci at real breakaway points on the real axis, the additional root loci do not contain $g_{k,j_1+1,4}$. After $g_{k,j_1,4}$ intersects with the additional root loci on the real axis, $g_{k,j_1+1,4}$  intersects with $g_{k,j_1,4}$ at  the real breakaway point $q_{k,2j_1-2,4}$  in $L_{k,j_1,4}$. At point $q_{k,2j_1-2,4}$, $g_{k,j_1,4}$ and $g_{k,j_1+1,4}$ leave the real axis and enters $\mathbb{C}^{*}$. In this situation, there exists  one instance of the second type of intersection with the additional root loci.

If $g_{k,j_1+1,4}$ intersects with the additional root loci at real breakaway points on the real axis, the additional root loci do not contain $g_{k,j_1,4}$. If $g_{k,j_1,4}$  intersects with the additional root loci at real breakaway points on the real axis, the additional root loci do not contain $g_{k,j_1+1,4}$. After $g_{k,j_1+1,4}$ intersects with the additional root loci on the real axis, and after $g_{k,j_1,4}$ intersects with the additional root loci on the real axis, $g_{k,j_1+1,4}$  intersects with $g_{k,j_1,4}$ at  the real breakaway point $q_{k,2j_1-2,4}$  in $L_{k,j_1,4}$. At point $q_{k,2j_1-2,4}$, $g_{k,j_1,4}$ and $g_{k,j_1+1,4}$ leave the real axis and enters $\mathbb{C}^{*}$. In this situation, there exists at least one instance of the first type of intersection with the additional root loci and one instance of the second type of intersection with the additional root loci.

If $g_{k,j_1+1,4}$ intersects with the additional root loci at real breakaway points on the real axis, the additional root loci do not contain $g_{k,j_1,4}$. If $g_{k,j_1,4}$ does not intersect with the additional root loci at real breakaway points on the real axis, after $g_{k,j_1+1,4}$ intersects with the additional root loci on the real axis, $g_{k,j_1+1,4}$  intersects with $g_{k,j_1,4}$ at  the real breakaway point $q_{k,2j_1-2,4}$  in $L_{k,j_1,4}$. At point $q_{k,2j_1-2,4}$, $g_{k,j_1,4}$ and $g_{k,j_1+1,4}$ leave the real axis and enters $\mathbb{C}^{*}$. In this situation, there exists  one instance of the first type of intersection with the additional root loci.

Due to symmetry about the real axis, $g_{k,j_1,4}$ and $g_{k,j_1+1,4}$ which have left the real axis must intersect again at a real breakaway point $q_{k,2j_1-3,4}$  in $L_{k,j_1-1,4}$. $g_{k,j_1,4}$ and $g_{k,j_1+1,4}$ extend in  $\mathbb{C}^{*}$ from $q_{k,2j_1-2,4}$ to $q_{k,2j_1-3,4}$. $j_1=1,2,\cdots,k-2$.

$g_{k,1,4}$ and the root locus $g_{k,2,4}$ intersect, along with other possible root loci, at $Q_{k,4}$  in $L_{k,1,4}$.

The third kind of non-standard linked root loci consist of two main lines with multiple branches. Except for the terminal root loci emitted from the non-real poles of (2.2), the rest main root loci  originate from the real poles of (2.2).

In  $\mathbb{C}^{*}$, the section of the root locus $g_{k,k+1,4}$ from the non-real zero $s_{k,k+1,4}$ to the real breakaway point $q_{k,2k-1,4}$ constitutes the terminus part of two main lines. In  $\mathbb{C}^{*}$, the section of the root locus $\overline{g_{k,k+1,4}}$ from the non-real zero $\overline{s_{k,k+1,4}}$ to $q_{k,2k-1,4}$ constitutes the terminus part of two main lines.

If $g_{k,k+1,4}$  does not intersect with the additional root loci on the real axis, then the interval on the real axis from $q_{k,2k-1,4}$ to the real breakaway point $q_{k,2k-2,4}$ constitutes a part of two main lines.

If $g_{k,k+1,4}$  intersects with the additional root loci on the real axis, then, between $q_{k,2k-1,4}$ to the real breakaway point $q_{k,2k-2,4}$, $g_{k,k+1,4}$ intersects with the additional root loci, the part of $g_{k,k+1,4}$ is distributed on the real axis and the other part is symmetrically in $\mathbb{C}^{*}$.

$g_{k,k+1,4}$  extends from $q_{k,2k-1,4}$ to the real breakaway point $q_{k,2k-2,4}$. The part of $g_{k,k+1,4}$  distributed in this interval constitutes a part of two main lines. When $g_{k,k+1,4}$ is distributed in $\mathbb{C}^{*}$, the $g_{k,k+1,4}$ and the part which is symmetric about the real axis are both a part of two main lines.

Here, for both cases, we introduce two concepts. If $g_{k,k+1,4}$  does not intersect with the additional root loci on the real axis, the interval on the real axis from $q_{k,2k-1,4}$ to the real breakaway point $q_{k,2k-2,4}$ is a part of $g_{k,k+1,4}$. We call the segment of the root locus of $g_{k,k+1,4}$  from $q_{k,2k-1,4}$ to the real breakaway point $q_{k,2k-2,4}$ as the root locus segment on the real axis.

If $g_{k,k+1,4}$  intersects with the additional root loci on the real axis, then in the interval on the real axis from $q_{k,2k-1,4}$ to the real breakaway point $q_{k,2k-2,4}$, $g_{k,k+1,4}$ is distributed both on the real axis and symmetrically in $\mathbb{C}^{*}$. We call the segment of the root locus of $g_{k,k+1,4}$  from $q_{k,2k-1,4}$ to the real breakaway point $q_{k,2k-2,4}$ as the root locus segment on the real axis and outside the real axis.

In the third kind of non-standard linked root loci, there are $k$ intervals between adjacent real zeros and real poles of (2.2). It is also established that within each of the interval $L_{k,j_1,4}$ ($j_1=1,\cdots,k$) between the adjacent real zero $z_{k,j_1,4}$ and real pole $t_{k,j_1,4}$ of (2.2), there exists a segment extending from the real breakaway point $q_{k,2j_1-2,4}$ to the real breakaway point $q_{k,2j_1-1,4}$. For the third kind of non-standard linked root loci, the distribution of root loci on this segment presents one of two possible cases. In contrast, for the standard linked root loci, the distribution of root loci on this segment is uniquely determined. For the third kind of non-standard linked root loci, we provide two specific definitions for the two possible cases of the distribution of root loci on this segment.

\begin{definition}
In $L_{k,j_1,4}$, if $g_{k,j_1+1,4}$  does not intersect with the additional root loci on the real axis, the interval on the real axis from $q_{k,2j_1-1,4}$ to $q_{k,2j_1-2,4}$ is a part of $g_{k,j_1+1,4}$ . We refer to the segment of the root locus of $g_{k,j_1+1,4}$  from $q_{k,2j_1-1,4}$ to $q_{k,2j_1-2,4}$ as the root locus segment on the real axis. $j_1=1,2,\cdots,k-1$.
\end{definition}
In $L_{k,k,4}$, if $g_{k,k+1,4}$ does not intersect with the additional root loci on the real axis, the interval on the real axis from $q_{k,2k-1,4}$ to $q_{k,2k-2,4}$ is a part of $g_{k,k+1,4}$. We refer to the segment of the root locus of $g_{k,k+1,4}$  from $q_{k,2k-1,4}$ to $q_{k,2k-2,4}$ as the root locus segment on the real axis.

\begin{definition}
In $L_{k,j_1,4}$, if $g_{k,j_1+1,4}$   intersects with the additional root loci on the real axis,  between $q_{k,2j_1-1,4}$ and $q_{k,2j_1-2,4}$, the part of $g_{k,j_1+1,4}$  is distributed both on the real axis and the other part is symmetrically to the additional root locus in $\mathbb{C}^{*}$. We refer to the segment of the root locus of $g_{k,j_1+1,4}$   from $q_{k,2j_1-1,4}$ to $q_{k,2j_1-2,4}$ as the root locus segment on the real axis and outside the real axis. $j_1=1,2,\cdots,k-1$.
\end{definition}
In $L_{k,k,4}$, if $g_{k,k+1,4}$ intersects with the additional root loci on the real axis, between $q_{k,2k-1,4}$ and $q_{k,2k-2,4}$, the part of $g_{k,k+1,4}$  is distributed both on the real axis and the other part is symmetrically to the additional toot locus in $\mathbb{C}^{*}$. We refer to the segment of the root locus of $g_{k,k+1,4}$ from $q_{k,2k-1,4}$ to $q_{k,2k-2,4}$ as the root locus segment on the real axis and outside the real axis.

Therefore, the interval on the real axis from $q_{k,2k-1,4}$ to the real breakaway point $q_{k,2k-2,4}$ constitutes a part of two main lines. This part is the root locus segment on the real axis or the root locus segment on the real axis and outside the real axis.

In $\mathbb{C} ^{*}$, the section of the root locus $g_{k,k,4}$ from $q_{k,2k-2,4}$ to the real breakaway point $q_{k,2k-3,4}$ constitutes a part of two main lines. In $\mathbb{C} ^{*}$, the section of the root locus $g_{k,k+1,4}$ from $q_{k,2k-2,4}$ to $q_{k,2k-3,4}$ constitutes a part of two main lines.

The interval on the real axis from $q_{k,2k-3,4}$ to the real breakaway point $q_{k,2k-4,5}$ constitutes a part of two main lines. This part is the root locus segment on the real axis or the root locus segment on the real axis and outside the real axis.

In $\mathbb{C} ^{*}$, the section of the root locus $g_{k,k-1,4}$ from $q_{k,2k-4,4}$ to the real breakaway point $q_{k,2k-5,4}$ constitutes a part of two main lines. In $\mathbb{C} ^{*}$, the section of the root locus $g_{k,k,4}$ from $q_{k,2k-4,4}$ to $q_{k,2k-5,4}$ constitutes a part of two main lines.

In $\mathbb{C} ^{*}$, the section of the root locus $g_{k,j_1+1,4}$  from $q_{k,2j_1,4}$ to the real breakaway point $q_{k,2j_1-1,4}$ constitutes a part of two main lines. In $\mathbb{C} ^{*}$, the section of the root locus $g_{k,j_1+2,4}$ from $q_{k,2j_1,4}$ to $q_{k,2j_1-1,4}$ constitutes a part of two main lines.

The interval on the real axis from $q_{k,2j_1-1,4}$ to the real breakaway point $q_{k,2j_1-2,4}$ constitutes a part of two main lines. This part is the root locus segment on the real axis or the root locus segment on the real axis and outside the real axis.

In $\mathbb{C} ^{*}$, the section of the root locus $g_{k,j_1,4}$ from $q_{k,2j_1-2,4}$ to the real breakaway point $q_{k,2j_1-3,4}$ constitutes a part of two main lines. In $\mathbb{C} ^{*}$, the section of the root locus $g_{k,j_1+1,4}$ from $q_{k,2j_1-2,4}$ to $q_{k,2j_1-3,4}$ constitutes a part of two main lines. $j_1=1,\cdots,k-2$.

In $\mathbb{C} ^{*}$, the section of the root locus $g_{k,2,4}$ from $q_{k,2,4}$ to the real breakaway point $q_{k,1,4}$ constitutes a part of two main lines. In $\mathbb{C} ^{*}$, the section of the root locus $g_{k,3,4}$ from $q_{k,2,4}$ to $q_{k,1,4}$ constitutes a part of two main lines.

The interval on the real axis from $q_{k,1,4}$ to the real breakaway point $Q_{k,4}$ constitutes a part of two main lines. This part is the root locus segment on the real axis or the root locus segment on the real axis and outside the real axis.

All the parts given previously, when connected together, form two main lines of the third kind of non-standard linked root loci. In the third kind of non-standard linked root loci, there are only two main lines. Part of the two main lines is distributed on the real axis. The other part of two main lines distributed in $\mathbb{C}^{*}$, and they are symmetric about the real axis.

The section of $\overline{g_{k,k+1,4}}$ from $q_{k,2k-1,4}$ to $z_{k,k,4}$ constitutes a branch. The section of $g_{k,k,4}$ from $t_{k,k,4}$ to $q_{k,2k-2,4}$ constitutes a branch. And so on $\cdots$. In the third kind of non-standard linked root loci, there are multiple branches.

In Lemma 2.37, in $\mathbb{C} ^{*}$, the section of the root locus $g_{2,3,4}$ from the non-real zero $s_{2,3,4}$ to the real breakaway point $q_{2,3,4}$ constitutes the terminus part of two main lines. In $\mathbb{C} ^{*}$, the section of the root locus $\overline{g_{2,3,4}}$ from the non-real zero $\overline{s_{2,3,4}}$ to $q_{2,3,4}$ constitutes the terminus part of two main lines.

The interval on the real axis from $q_{2,3,4}$ to the real breakaway point $q_{2,2,4}$ constitutes a part of two main lines. This part is the root locus segment on the real axis or the root locus segment on the real axis and outside the real axis.

In $\mathbb{C} ^{*}$, the section of the root locus $g_{2,2,4}$ from $q_{2,2,4}$ to the real breakaway point $q_{2,1,4}$ constitutes a part of two main lines. In $\mathbb{C} ^{*}$, the section of the root locus $g_{2,3,4}$ from $q_{2,2,4}$ to $q_{2,1,4}$ constitutes a part of two main lines.

The interval on the real axis from $q_{2,1,4}$ to the real breakaway point $Q_{2,4}$ constitutes a part of two main lines. This part is the root locus segment on the real axis or the root locus segment on the real axis and outside the real axis.

All the parts given previously, when connected together, form two main lines of the third kind of non-standard linked root loci. In the third kind of non-standard linked root loci, there are only two main lines. Part of the two main lines is distributed on the real axis. The other part of two main lines distributed in $\mathbb{C}^{*}$, and they are symmetric about the real axis.

The section of $\overline{g_{2,3,4}}$ from $q_{2,3,4}$ to $z_{2,2,4}$ constitutes a branch. The section of $g_{2,2,4}$ from $t_{2,2,4}$ to $q_{2,2,4}$ constitutes a branch. The section of $g_{2,3,4}$ from $q_{2,1,4}$ to $z_{2,1,4}$ constitutes a branch. In the third kind of non-standard linked root loci, there are four branches.

Therefore, according to  the above explanations of the two main lines of third kind of non-standard linked root loci given above, we can define the two main lines of the third kind of non-standard linked root loci of (2.2) as follows.
\begin{definition}
The non-real zeros $s_{k,k+1,4}$ and $\overline{s_{k,k+1,4}}$ emit the root loci $g_{k,k+1,4}$ and $\overline{g_{k,k+1,4}}$. In the third kind of non-standard linked root loci, beginning at the non-real zeros  $s_{k,k+1,4}$ and $\overline{s_{k,k+1,4}}$, the two lines of $g_{k,k+1,4}$ and $\overline{g_{k,k+1,4}}$ that extend all the way to $Q_{k,4}$ constitute the two main lines of the third kind of non-standard linked root loci of (2.2).
\end{definition}

\begin{lemma}
In the third kind of non-standard linked root loci, a root locus $g_{k,k+1,4}$ or $\overline{g_{k,k+1,4}}$ originating from a non-real zero $s_{k,k+1,4}$ or $\overline{s_{k,k+1,4}}$ can be distributed on two main lines of the third kind of non-standard linked root loci. The root locus $g_{k,k+1,4}$ or $\overline{g_{k,k+1,4}}$ intersects at the real breakaway point $Q_{k,4}$. The root locus $g_{k,k+1,4}$ or $\overline{g_{k,k+1,4}}$ leaves from $Q_{k,4}$ then enters $\mathbb{C}^{*}$.
\end{lemma}

\begin{proof}
In the explanation of two main lines, for the root loci that eventually reach the real zeros in (2.2) in the third kind of non-standard linked root loci, one of their distribution as follows. The root locus $\overline{g_{k,k+1}}$ ends at $z_{k,k,4}$.  The root locus $g_{k,l,4}$ ends at $z_{k,l-2,4}$. $l=3,\cdots, k+1$. Based on this distributed way, neither the root locus $g_{k,k+1,4}$ nor $\overline{g_{k,k+1,4}}$ directly intersect at $Q_{k,4}$.

The intersecting root loci are referred to as "common root loci". They possess "common" properties. Beginning at the non-real zeros $s_{k,k+1,4}$ and $\overline{s_{k,k+1,4}}$, the two main lines of the third kind of non-standard linked root loci both extend towards the breakaway point $Q_{k,4}$. Therefore, the root locus emitted from the non-real zero $s_{k,k+1,4}$ or $\overline{s_{k,k+1,4}}$ can be distributed on two main lines of the third kind of non-standard linked root loci. That is, among the root loci intersecting at $Q_{k,4}$, one is $g_{k,k+1,4}$ or $\overline{g_{k,k+1,4}}$.

The root locus $g_{k,k+1,4}$ from a non-real zero $s_{k,k+1,4}$, distributed on two main lines of the third kind of non-standard linked root loci in the following way, and intersect at $Q_{k,4}$.

The root locus $\overline{g_{k,k+1,4}}$ ends at $z_{k,k,4}$. The root locus $g_{k,k,4}$ ends at $z_{k,k-1,4}$. The root locus $g_{k,l_1,4}$ ends at $z_{k,l_1-1,4}$. $l_1=3,\cdots, k-1$. $g_{k,3,4}$ ends at $z_{k,2,4}$. $g_{k,2,4}$ ends at $z_{k,1,4}$. $g_{k,k+1,4}$ must intersect at $Q_{k,4}$.

The root locus $g_{k,k+1,4}$ originates from the non-real zero $s_{k,k+1,4}$ and extends to $q_{k,2k-1,4}$. At $q_{k,2k-1,4}$, $g_{k,k+1,4}$ enters the real axis. $g_{k,k+1,4}$ extends along the root locus segment on the real axis or the root locus segment on the real axis and outside the real axis from $q_{k,2k-1,4}$ to $q_{k,2k-2,4}$. Then $g_{k,k+1,4}$ enters $\mathbb{C}^{*}$ from $q_{k,2k-2,4}$ and extends to $q_{k,2k-3,4}$ in $\mathbb{C}^{*}$. And so on $\cdots$. At $q_{k,3,4}$, $g_{k,k+1,4}$ re-enters the real axis. $g_{k,k+1,4}$ extends along the root locus segment on the real axis or the root locus segment on the real axis and outside the real axis from $q_{k,3,4}$ to $q_{k,2,4}$. Then $g_{k,k+1,4}$ enters $\mathbb{C}^{*}$ from $q_{k,2,4}$ and extends to $q_{k,1,4}$ in $\mathbb{C}^{*}$. At $q_{k,1,4}$, $g_{k,k+1,4}$ re-enters the real axis. $g_{k,k+1,4}$ extends along the root locus segment on the real axis or the root locus segment on the real axis and  outside the real axis from $q_{k,1,4}$ to $Q_{k,4}$. Then $g_{k,k+1,4}$ enters $\mathbb{C}^{*}$ from $Q_{k,4}$.

The root locus $g_{k,k+1,4}$ ends at $z_{k,k,4}$. The root locus $g_{k,k,4}$ ends at $z_{k,k-1,4}$. The root locus $g_{k,l_1,4}$ ends at $z_{k,l_1-1,4}$. $l_1=3,\cdots, k-1$. $g_{k,3,4}$ ends at $z_{k,2,4}$. $g_{k,2,4}$ ends at $z_{k,1,4}$. Under this case, $\overline{g_{k,k+1,4}}$ must intersect at $Q_{k,4}$.

The root locus $\overline{g_{k,k+1,4}}$ originates from the non-real zero $\overline{s_{k,k+1,4}}$ and extends to $q_{k,2k-1,4}$. At $q_{k,2k-1,4}$, $\overline{g_{k,k+1,4}}$ enters the real axis. $\overline{g_{k,k+1,4}}$ extends along the root locus segment on the real axis or the root locus segment on the real axis and outside the real axis from $q_{k,2k-1,4}$ to $q_{k,2k-2,4}$. Then $\overline{g_{k,k+1,4}}$ enters $\mathbb{C}^{*}$ from $q_{k,2k-2,4}$ and extends to $q_{k,2k-3,4}$ in $\mathbb{C}^{*}$. And so on $\cdots$. At $q_{k,3,4}$, $\overline{g_{k,k+1,4}}$ re-enters the real axis. $\overline{g_{k,k+1,4}}$ extends along the root locus segment on the real axis or the root locus segment on the real axis and outside the real axis from $q_{k,3,4}$ to $q_{k,2,4}$. Then $\overline{g_{k,k+1,4}}$ enters $\mathbb{C}^{*}$ from $q_{k,2,4}$ and extends to $q_{k,1,4}$ in $\mathbb{C}^{*}$. At $q_{k,1,4}$, $\overline{g_{k,k+1,4}}$ re-enters the real axis. $\overline{g_{k,k+1,4}}$ extends along the root locus segment on the real axis or the root locus segment on the real axis and outside the real axis from $q_{k,1,4}$ to $Q_{k,4}$. Then $\overline{g_{k,k+1,4}}$ enters $\mathbb{C}^{*}$ from $Q_{k,4}$.

Therefore, in the third kind of non-standard linked root loci,  a root locus $g_{k,k+1,4}$ or $\overline{g_{k,k+1,4}}$ originating from a non-real zero $s_{k,k+1,4}$ or $\overline{s_{k,k+1,4}}$ can be distributed on two main lines of the third non-standard linked root loci. The root locus $g_{k,k+1,4}$ or $\overline{g_{k,k+1,4}}$  intersects at the real breakaway point $Q_{k,4}$. The root locus $g_{k,k+1,4}$ or $\overline{g_{k,k+1,4}}$  leaves from $Q_{k,4}$ then enters $\mathbb{C}^{*}$.
\end{proof}

The third kind of non-standard linked root loci must satisfy six conditions.

The first condition: Two root loci $g_{k,j_1+1,4}$ and $g_{k,j_1+2,4}$ emitted from the real zeros ($t_{k,j_1+1,4}$, $t_{k,j_1+2,4}$). The real zero $t_{k,j_1,4}$ emits a  root locus $g_{k,j_1,4}$. Two root loci $g_{k,k+1,4}$ and $\overline{g_{k,k+1,4}}$ emitted from the non-real zeros ($s_{k,k+1,4}$ or $\overline{s_{k,k+1,4}}$). The real zero $t_{k,k,4}$ emits a root locus $g_{k,k,4}$.

The second condition: The two root loci $g_{k,j_1+1,4}$ and $g_{k,j_1+2,4}$ must intersect at the real breakaway point $q_{k,2j_1-1,4}$ on the real axis. From $q_{k,2j_1-1,4}$, the two root loci $g_{k,j_1+1,4}$ and $g_{k,j_1+2,4}$ extend further along the real axis in opposite directions. The two root loci $g_{k,k+1,4}$ and $\overline{g_{k,k+1,4}}$ must intersect at the real breakaway point $q_{k,2k-1,4}$ on the real axis. From $q_{k,2k-1,4}$, the two root loci $g_{k,k+1,4}$ and $\overline{g_{k,k+1,4}}$ extend further along the real axis in opposite directions. The preceding sentence states a condition, while the following sentence describes a property shared by all pairs of root loci intersecting on the real axis, rather than constituting an additional condition.

The third condition: The root locus $g_{k,j_1+1,4}$ further intersects with the additional root loci on the real axis. The root locus $g_{k,k+1,4}$ further intersects with the additional root loci on the real axis.

The fourth condition: After $g_{k,j_1+1,4}$ intersects with the additional root loci on the real axis, it then intersects with $g_{k,j_1,4}$ on the real axis at the real breakaway point $q_{k,2j_1-2,4}$. After $g_{k,k+1,4}$ intersects with the additional root loci on the real axis, it then intersects with $g_{k,k,4}$ on the real axis at the real breakaway point $q_{k,2k-2,4}$.

The fifth condition: The real zero $t_{k,j_1,4}$ emits a root locus $g_{k,j_1,4}$. This root locus $g_{k,j_1,4}$ intersects with the additional root loci on the real axis.

The sixth condition: After $g_{k,j_1,4}$ intersects with the additional root loci on the real axis, it then intersects with $g_{k,j_1+1,4}$ on the real axis at the real breakaway point $q_{k,2j_1-2}$. After $g_{k,k,4}$ intersects with the additional root loci on the real axis, it then intersects with the root locus $g_{k,k+1,4}$ on the real axis at the real breakaway point $q_{k,2k-2,4}$.

The root loci satisfying these six conditions constitute the third kind of non-standard linked root loci.

After the root locus $g_{k,j_1+1,4}$ intersects with the additional root loci on the real axis, it may then intersect with a root locus $g_{k,j_1,4}$ that does not intersect with the additional root loci on the real axis. This is one scenario. Alternatively, the root locus $g_{k,j_1+1,4}$ may not intersect with the additional root loci on the real axis but may intersect with a root locus $g_{k,j_1,4}$ that does intersect with the additional root loci on the real axis. This is another scenario. It is also possible that after $g_{k,j_1+1,4}$ intersects with the additional root loci on the real axis, it then intersects with a root locus $g_{k,j_1,4}$ that also intersects with the additional root loci on the real axis. This is yet another scenario. Finally, the root locus $g_{k,j_1+1,4}$ may not intersect with the additional root loci on the real axis and may intersect with a root locus $g_{k,j_1,4}$ that also does not intersect with the additional root loci on the real axis. This constitutes a further scenario.

In the third kind of non-standard linked root loci, at least one instance of the first type of intersection with the additional root loci and at least one instance of the second type of intersection with the additional root loci must exist. This is precisely what distinguishes the third kind of non-standard linked root loci from the first and second kinds.

The two main lines of the first kind of non-standard linked root loci are identical to the two main lines of the third kind of non-standard linked root loci. This outcome arises for a specific reason.

In the first kind of non-standard linked root loci, the real zero $t_{k,j_1,2}$ emits the root locus $g_{k,j_1,2}$. Two root loci $g_{k,j_1+1,2}$ and $g_{k,j_1+2,2}$ emitted from the real zeros ($t_{k,j_1+1,2}$, $t_{k,j_1+2,2}$). These two root loci $g_{k,j_1+1,2}$ and $g_{k,j_1+2,2}$ leave the real axis and enter $\mathbb{C}^{*}$ from the breakaway point $q_{k,2j_1,2}$. They intersect at the real breakaway point $q_{k,2j_1-1,2}$. From $q_{k,2j_1-1,2}$, the two root loci $g_{k,j_1+1,2}$ and $g_{k,j_1+2,2}$ extend in opposite directions. If the root locus $g_{k,j_1+1,2}$ intersects with the additional root loci on the real axis, then after this intersection, $g_{k,j_1+1,2}$ subsequently intersects with the root locus $g_{k,j_1,2}$ at the real breakaway point $q_{k,2j_1-2,2}$. From $q_{k,2j_1-2,2}$, the root loci $g_{k,j_1+1,2}$ and $g_{k,j_1,2}$ leave the real axis and enter $\mathbb{C}^{*}$.

The real zero $t_{k,k,2}$ emits the root locus $g_{k,k,2}$. Two root loci $g_{k,k+1,2}$ and $\overline{g_{k,k+1,2}}$ emitted from the non-real zeros ($s_{k,k+1,2}$, $\overline{s_{k,k+1,2}}$). These two root loci intersect at the real breakaway point $q_{k,2k-1,2}$. From $q_{k,2k-1,2}$, the two root loci $g_{k,k+1,2}$ and $\overline{g_{k,k+1,2}}$ extend in opposite directions. If the root locus $g_{k,k+1,2}$ intersects with the additional root loci on the real axis, then after this intersection, $g_{k,k+1,2}$ subsequently intersects with the root locus $g_{k,k,2}$ at the real breakaway point $q_{k,2k-2,2}$. From $q_{k,2k-2,2}$, the root loci $g_{k,k+1,2}$ and $g_{k,k,2}$ leave the real axis and enter $\mathbb{C}^{*}$.

If the root locus $g_{k,j_1+1,2}$ does not intersect with the additional root loci on the real axis, then $g_{k,j_1+1,2}$ directly intersects with the root locus $g_{k,j_1,2}$ at the real breakaway point $q_{k,2j_1-2,2}$. From $q_{k,2j_1-2,2}$, the root loci $g_{k,j_1+1,2}$ and $g_{k,j_1,2}$ leave the real axis and enter $\mathbb{C}^{*}$.

If the root locus $g_{k,k+1,2}$ does not intersect with the additional root loci on the real axis, then $g_{k,k+1,2}$ directly intersects with the root locus $g_{k,k,2}$ at the real breakaway point $q_{k,2k-2,2}$. From $q_{k,2k-2,2}$, the root loci $g_{k,k+1,2}$ and $g_{k,k,2}$ leave the real axis and enter $\mathbb{C}^{*}$.

Based on the earlier result concerning the two main lines of the first kind of non-standard linked root loci, two main lines of the first kind of non-standard linked root loci are as follows: In $\mathbb{C}^{*}$, from $q_{k,2j_1,2}$ to $q_{k,2j_1-1,2}$, two main lines of the first kind of non-standard linked root loci are segments of the two root loci $g_{k,j_1+1,2}$ and $g_{k,j_1+2,2}$. The interval on the real axis from $q_{k,2j_1-1,2}$ to $q_{k,2j_1-2,2}$ constitutes a part of the two main lines. This part may be a root locus segment on the real axis or a root locus segment both on and off the real axis. This interval is a segment of the root locus $g_{k,j_1+1,2}$ or of the additional root loci. In $\mathbb{C}^{*}$, from $q_{k,2j_1-2,2}$ to $q_{k,2j_1-3,2}$, they are segments of the root loci $g_{k,j_1+1,2}$ and $g_{k,j_1,2}$.

In $\mathbb{C}^{*}$, from $s_{k,k+1,2}$ or $\overline{s_{k,k+1,2}}$ to $q_{k,2k-1,2}$, they are segments of the two root loci $g_{k,k+1,2}$ and $\overline{g_{k,k+1,2}}$. The interval on the real axis from $q_{k,2k-1,2}$ to $q_{k,2k-2,2}$ constitutes a part of the two main lines. This part may be a root locus segment on the real axis or a root locus segment both on and off the real axis. This interval is a segment of the root locus $g_{k,k+1,2}$ or of the additional root loci. In $\mathbb{C}^{*}$, from $q_{k,2k-2,2}$ to $q_{k,2k-3,2}$, they are segments of the root loci $g_{k,k+1,2}$ and $g_{k,k,2}$.

In the third kind of non-standard linked root loci, the real zero $t_{k,j_1,4}$ emits the root locus $g_{k,j_1,4}$. Two root loci $g_{k,j_1+1,4}$ and $g_{k,j_1+2,4}$ emitted from the real zeros ($t_{k,j_1+1,4}$, $t_{k,j_1+2,4}$). These two root loci $g_{k,j_1+1,4}$ and $g_{k,j_1+2,4}$ leave the real axis and enter $\mathbb{C}^{*}$ from the breakaway point $q_{k,2j_1,4}$. They intersect at the real breakaway point $q_{k,2j_1-1,4}$. From $q_{k,2j_1-1,4}$, the two root loci $g_{k,j_1+1,4}$ and $g_{k,j_1+2,4}$ extend in opposite directions. If the root locus $g_{k,j_1+1,4}$ intersects with the additional root loci on the real axis, then after this intersection, $g_{k,j_1+1,4}$ subsequently intersects with the root locus $g_{k,j_1,4}$ at the real breakaway point $q_{k,2j_1-2,4}$. From $q_{k,2j_1-2,4}$, the root loci $g_{k,j_1+1,4}$ and $g_{k,j_1,4}$ leave the real axis and enter $\mathbb{C}^{*}$.

The real zero $t_{k,k,4}$ emits the root locus $g_{k,k,4}$. Two root loci $g_{k,k+1,4}$ and $\overline{g_{k,k+1,4}}$ emitted from the non-real zeros ($s_{k,k+1,4}$, $\overline{s_{k,k+1,4}}$). These two root loci intersect at the real breakaway point $q_{k,2k-1,4}$. From $q_{k,2k-1,4}$, the two root loci $g_{k,k+1,4}$ and $\overline{g_{k,k+1,4}}$ extend in opposite directions. If the root locus $g_{k,k+1,4}$ intersects with the additional root loci on the real axis, then after this intersection, $g_{k,k+1,4}$ subsequently intersects with the root locus $g_{k,k,4}$ at the real breakaway point $q_{k,2k-2,4}$. From $q_{k,2k-2,4}$, the root loci $g_{k,k+1,4}$ and $g_{k,k,4}$ leave the real axis and enter $\mathbb{C}^{*}$.

If the root locus $g_{k,j_1+1,4}$ does not intersect with the additional root loci on the real axis, then $g_{k,j_1+1,4}$ directly intersects with the root locus $g_{k,j_1,4}$ at the real breakaway point $q_{k,2j_1-2,4}$. From $q_{k,2j_1-2,4}$, the root loci $g_{k,j_1+1,4}$ and $g_{k,j_1,4}$ leave the real axis and enter $\mathbb{C}^{*}$.

If the root locus $g_{k,k+1,4}$ does not intersect with the additional root loci on the real axis, then $g_{k,k+1,4}$ directly intersects with the root locus $g_{k,k,4}$ at the real breakaway point $q_{20,2k-2}$. From $q_{k,2k-2,4}$, the root loci $g_{k,k+1,4}$ and $g_{k,k,4}$ leave the real axis and enter $\mathbb{C}^{*}$.

Based on the earlier results concerning the two main lines of the third kind of non-standard linked root loci, the two main lines of the third kind of non-standard linked root loci are as follows: In $\mathbb{C}^{*}$, from $q_{k,2j_1,4}$ to $q_{k,2j_1-1,4}$, the two main lines of the third kind of non-standard linked root loci are segments of the two root loci $g_{k,j_1+1,4}$ and $g_{k,j_1+2,4}$. The interval on the real axis from $q_{k,2j_1-1,4}$ to $q_{k,2j_1-2,4}$ constitutes a part of the two main lines. This part may be a root locus segment on the real axis or a root locus segment both on and off the real axis. This interval is a segment of the root locus $g_{k,j_1+1,4}$ or of the additional root loci. In $\mathbb{C}^{*}$, from $q_{k,2j_1-2,4}$ to $q_{k,2j_1-3,4}$, they are segments of the root loci $g_{k,j_1+1,4}$ and $g_{k,j_1,4}$.

In $\mathbb{C}^{*}$, from $s_{k,k+1,4}$ or $\overline{s_{k,k+1,4}}$ to $q_{k,2k-1,4}$, they are segments of the two root loci $g_{k,k+1,4}$ or $\overline{g_{k,k+1,4}}$. The interval on the real axis from $q_{k,2k-1,4}$ to $q_{k,2k-2,4}$ constitutes a part of the two main lines. This part is a root locus segment on the real axis or a root locus segment both on and off the real axis. This interval is a segment of the root locus $g_{k,k+1,4}$ or of the additional root loci. In $\mathbb{C}^{*}$, from $q_{k,2k-2,4}$ to $q_{k,2k-3,4}$, they are segments of the root loci $g_{k,k+1,4}$ and $g_{k,k,4}$.

In the first kind of non-standard linked root loci, the root locus $g_{k,j_1,2}$ emitted from the real zero $t_{k,j_1,2}$. $g_{k,j_1,2}$ starts from $t_{k,j_1,2}$ and extends to the real breakaway point $q_{k,2j_1-2,2}$. According to this direction of extension, this portion constitutes a branch of the first kind of non-standard linked root loci. It cannot lie on the two main lines of the first kind of non-standard linked root loci.

In the first kind of non-standard linked root loci, the root locus $g_{k,k,2}$ emitted from the real zero $t_{k,k,2}$. $g_{k,k,2}$ starts from $t_{k,k,2}$ and extends to the real breakaway point $q_{k,2k-2,2}$. According to this direction of extension, this portion constitutes a branch of the first kind of non-standard linked root loci. It cannot lie on the two main lines of the first kind of non-standard linked root loci.

In the third kind of non-standard linked root loci, the root locus $g_{k,j_1,4}$ emitted from the real zero $t_{k,j_1,4}$. $g_{k,j_1,4}$ starts from $t_{k,j_1,4}$ and extends to the real breakaway point $q_{k,2j_1-2,4}$. During its extension, it may intersect with the additional root loci on the real axis. Although the root locus $g_{k,j_1,4}$ may intersect with the additional root loci on the real axis, the direction of its extension remains from $t_{k,j_1,4}$ to the real breakaway point $q_{k,2j_{2}-2,4}$. According to this direction of extension, this portion constitutes a branch of the third kind of non-standard linked root loci. It cannot lie on the two main lines of the third kind of non-standard linked root loci.

In the third kind of non-standard linked root loci, the root locus $g_{k,k,4}$ emitted from the real zero $t_{k,k,4}$. It starts from $t_{k,k,4}$ and extends to the real breakaway point $q_{k,2k-2,4}$. During its extension, it may intersect with the additional root loci on the real axis. Although the root locus $g_{k,k,4}$ may intersect with the additional root loci on the real axis, the direction of its extension remains from $t_{k,k,4}$ to the real breakaway point $q_{k,2k-2,4}$. According to this direction of extension, this portion constitutes a branch of the third kind of non-standard linked root loci. It cannot lie on the two main lines of the third kind of non-standard linked root loci.

Therefore, this demonstrates that the two main lines of the first kind of non-standard linked root loci are identical to the two main lines of the third kind of non-standard linked root loci.

Based on the previously presented results for linked root loci, there can only be these four scenarios for generating linked root loci.

The first scenario: In the linked root loci, there exist neither the first type nor the second type of intersection with the additional root loci. The linked root loci are the standard linked root loci.

The second scenario: In the linked root loci, only the first type of intersection with the additional root loci exists; the second type does not exist. The linked root loci are the first kind of non-standard linked root loci.

The third scenario: In the linked root loci, only the second type of intersection with the additional root loci exists; the first type does not exist. The linked root loci are the second kind of non-standard linked root loci.

The fourth scenario: In the linked root loci, both the first type and the second type of intersection with the additional root loci exist. The linked root loci are the third kind of non-standard linked root loci.

\begin{definition}
Standard linked root loci, the first kind of non-standard linked root loci, the second kind of non-standard linked root loci, and the third kind of non-standard linked root loci together constitute the linked root loci.
\end{definition}

\begin{definition}
The non-real zeros $s_{k,k+1,j_3}$ and $\overline{s_{k,k+1,j_3}}$ emit root loci $g_{k,k+1,j_3}$ and $\overline{g_{k,k+1,j_3}}$. In the linked root loci, beginning at the non-real zeros $s_{k,k+1,j_3}$ and $\overline{s_{k,k+1,j_3}}$, the two lines which $g_{k,k+1,j_3}$ and $\overline{g_{k,k+1,j_3}}$ can extend all the way to $Q_{k,j_3}$ are the two main lines of the linked root loci of (2.2). $j_3=1,2,3,4$.
\end{definition}
For each of these four scenarios, two lemmas are provided. They are respectively: Lemma 2.17 and Lemma 2.19; Lemma 2.24 and Lemma 2.28; Lemma 2.33 and Lemma 2.35; Lemma 2.39 and Lemma 2.43.  Lemma 2.17, Lemma 2.24, Lemma 2.33, and Lemma 2.39 can be combined into Lemma 2.46. Lemma 2.19, Lemma 2.28, Lemma 2.35, and Lemma 2.43 can be combined into Lemma 2.47. This yields the two merged lemmas.

\begin{lemma}
When a root locus $g_{k,1,j_3}$ emitted from a real zero of $p(s)$ leaves the real axis and enters $\mathbb{C}^{*}$, then, there must exist a pair of non-real zeros and $k-1$ other real zeros. The root loci from the non-real zeros are $g_{k,k+1,j_3}$ and $\overline{g_{k,k+1,j_3}}$, and root loci from the other real zeros are $g_{k,j_1,j_3}$ (for $1\le k<n$, $j_1=1,\cdots,k$, $0<2j_1-2<2k-2$, $0<2j_1-1<2k-1$, $j_3=1,2,3,4$). These root loci $g_{k,j_1,j_3}$, $g_{k,k+1,j_3}$ and $\overline{g_{k,k+1,j_3}}$  intersect at real breakaway points $Q_{k,j_3}$, $q_{k,2j_1-2,j_3}$, $q_{k,2j_1-1,j_3}$ on the real axis. These real breakaway points $Q_{k,j_3}$, $q_{k,2j_1-2,j_3}$, $q_{k,2j_1-1,j_3}$ connect $g_{k,j_1,j_3}$, $g_{k,k+1,j_3}$, and $\overline{g_{k,k+1,j_3}}$ together, forming a  linked root locus. The terminus of this structure consists of $g_{k,k+1,j_3}$ and $\overline{g_{k,k+1,j_3}}$.
\end{lemma}

\begin{lemma}
In the  linked root loci, a root locus $g_{k,k+1,j_3}$ or $\overline{g_{k,k+1,j_3}}$ originating from a non-real zero $s_{k,k+1,j_3}$ or $\overline{s_{k,k+1,j_3}}$ can be distributed on two main lines of the linked root loci. The root locus $g_{k,k+1,j_3}$ or $\overline{g_{k,k+1,j_3}}$ intersects at the real breakaway point $Q_{k,j_3}$. The root locus $g_{k,k+1,j_3}$ or $\overline{g_{k,k+1,j_3}}$ leaves from $Q_{k,j_3}$ then enters $\mathbb{C}^{*}$. (for $1\le k<n$, $j_1=1,\cdots,k$, $0<2j_1-2<2k-2$, $0<2j_1-1<2k-1$, $j_3=1,2,3,4$).
\end{lemma}

When all root loci intersecting at a real breakaway point of (2.2) originate directly from real zeros of $p$, the real breakaway points are generated in two distinct ways, classified as follows:

\begin{enumerate}
\item \textbf{First category of real breakaway points:} At least two root loci intersect at a real breakaway point $Q_{a}$ in an interval. One root locus  comes from a real zero of $p(s)$ at an endpoint of the interval.

    (1). If this root locus comes from a real zero of $p(s)$ at an endpoint of the interval doesn't intersect with the other root loci on the real axis and it remains on the real axis before intersecting the root loci originated from real zeros at $Q_{a}$.  This root locus extends along the real axis and intersect with the root loci from real zeros of $p(s)$, generating a real breakaway point $Q_{a}$.

    (2). If this root locus comes from a real zero of $p(s)$ at an endpoint of the interval intersects with the other root loci on the real axis, this situation is like as the second type of the additional root loci intersection, after this root locus intersects with the other root loci, this root locus extends along the real axis and intersect with the root loci from real zeros of $p(s)$, generating a real breakaway point $Q_{a}$.

\item \textbf{Second category of real breakaway points:} All root loci generating the real breakaway point $Q_{a}$ originate from real zeros of $p(s)$ but leave the real axis and enter $\mathbb{C}^{*}$. Among these root loci which originate from real zeros of $p(s)$, like first category of real breakaway points, there possibly exist the situation like as the second type of the additional root loci intersection. The root loci that  originate from the real zero of $p(s)$ must leave the real axis, and  enter $\mathbb{C}^{*}$. The root loci that originate from the real zeros of $p(s)$ then intersect at $Q_{a}$ on the real axis.
\end{enumerate}

Here, the other root loci are not the additional root loci. For the first category of real breakaway points, a root locus emitted from an interval endpoint may first intersect with the other root loci and subsequently intersect at $Q_{a}$. In fact, for both subcases, regardless of which subcase the real breakaway point $Q_{a}$ belongs to, it can be proven to be associated with at least one pair of non-real zeros of the polynomial $p(s)$.

Since $p$ has at least one real zero, Lemma 2.3 implies that each real pole of (2.2)
is separated by at least one real zero of $p^{'}$. Consequently, no two real poles of (2.2)
are adjacent. Therefore, on the real axis, intervals exist between adjacent real zeros
and real poles of (2.2), and may exist between adjacent real zeros, but not between
adjacent real poles.

\begin{lemma}
Suppose the real polynomial $p$ has at least one real zero. Consider first-category of real breakaway points. If all root loci intersecting at a real breakaway point $Q_{a}$ in (2.2) originate directly from real zeros of $p$, then, at least one pair of the non-real zeros of $p$ emits a $2q\pi$ or $2q\pi+\pi$ degree root locus, this root locus intersects at $Q_{a}$ with other root loci on the real axis. The intersection point $Q_{a}$ of the root loci is a first kind of zero of $(p^{'}/p)^{'}$.
\end{lemma}
\begin{proof}
Let $t_{k,0,j_3}$ be a real pole of (2.2) with adjacent real zero $z_{k,0,j_3}$ of (2.2), forming interval $L_{k,0,j_3}$. The pole  $t_{k,0,j_3}$ emits root locus $R_{k,0,j_3}$. Let $R_{k,i_3,j_3}$ be the $r_2$ branches of root loci emitted from the other real zeros $t_{k,i_3,j_3}$ of $p$, where $i_3=1,2,\cdots, r_2$. $r_2\ge 2$ is a positive integer. $R_{k,i_3,j_3}$ intersects with $R_{k,0,j_3}$ at the real breakaway point $Q_{a}$ in $L_{k,0,j_3}$.

Before $R_{k,0,j_3}$ intersects with the root loci $R_{k,i_3,j_3}$ at $Q_{a}$, if this root locus $R_{k,0,j_3}$ doesn't intersect with the other root loci on the real axis, $R_{k,0,j_3}$ does not leave the real axis nor enter $\mathbb{C}^{*}$. This root locus $R_{k,0,j_3}$ extends along the real axis and intersect with the root loci $R_{k,i_3,j_3}$ from real zeros of $p(s)$, generating a real breakaway point $Q_{a}$. $Q_{a}$ is a first-kind zero of $(p^{'}(s)/p(s))^{'}$ and $r_2+1$ branches of root loci intersect there.

Before $R_{k,0,j_3}$ intersects with the root loci $R_{k,i_3,j_3}$ at $Q_{a}$, if this root locus $R_{k,0,j_3}$  intersects with the other root loci on the real axis, this situation is like as the second type of the other root loci intersection. After this root locus $R_{k,0,j_3}$ intersects with the other root loci, this root locus $R_{k,0,j_3}$ extends along the real axis and intersect with the root loci $R_{k,i_3,j_3}$ from real zeros of $p(s)$, generating a real breakaway point $Q_{a}$. $Q_{a}$ is a first-kind zero of $(p^{'}(s)/p(s))^{'}$ and $r_2+1$ branches of root loci intersect there.

The root locus $R_{k,i_3,j_3}$ comes from another real zero of $p(s)$. Before $R_{k,i_3,j_3}$ intersects with other root loci at $Q_a$, $R_{k,i_3,j_3}$ leaves the real axis and enters $\mathbb{C}^{*}$.

Now consider root locus $g_{k,1,j_3}=R_{k,1,j_3}$ directly emitted from real zero $t_{k,1,j_3}$ of $p$. Let $z_{k,1,j_3}$ be the real zero of (2.2) adjacent to real pole $t_{k,1,j_3}$, forming interval $L_{k,1,j_3}$. The root locus $g_{k,1,j_3}$ leaves the interval $L_{k,1,j_3}$ and enters $\mathbb{C}^{*}$. Then, the root locus $g_{k,1,j_3}$ with the root loci $R_{k,i_3,j_3}$ and $R_{k,0,j_3}$ intersect at $Q_{a}$ in the interval $L_{k,1,j_3}$.

Therefore, root locus $g_{k,1,j_3}$ satisfies all conditions of Lemma 2.46.  Applying Lemma 2.46,  there exist root loci intersecting with $g_{k,1,j_3}$ and causing it to leave the real axis. These root loci are successively emitted from $k-1$ real zeros of $p$. The root loci emitted from these real zeros form the linked root loci. The root loci from the non-real zeros are $g_{k,k+1,j_3}$ and $\overline{g_{k,k+1,j_3}}$, and those from the real zeros are $g_{k,j_1,j_3}$ (for $j_1=2,3,\cdots,k$). These root loci intersect at real breakaway points $q_{k,2j_1-1,j_3}$, $q_{k,2j_1-2,j_3}$  on the real axis.

Therefore, $g_{k,1,j_3}$, $g_{k,j_1,j_3}$, $g_{k,k+1,j_3}$, $\overline{g_{k,k+1,j_3}}$ form the linked root loci.

The root locus $g_{k,1,j_3}$ is the beginning of the linked root loci. $g_{k,k+1,j_3}$ and $\overline{g_{k,k+1,j_3}}$  are as terminating of the linked root loci. By Lemma 2.47, within the linked root loci between $g_{k,1,j_3}$ and $g_{k,k+1,j_3}$ and $\overline{g_{k,k+1,j_3}}$, the root locus originating from the non-real zeros $s_{k,k+1,j_3}$ and $\overline{s_{k,k+1,j_3}}$ can be distributed on two main lines of the linked root loci. $g_{k,k+1,j_3}$ or $\overline{g_{k,k+1,j_3}}$  originating from the non-real zeros $s_{k,k+1,j_3}$ and $\overline{s_{k,k+1,j_3}}$ intersects at the real breakaway point $Q_{k,j_3}$ in $L_{k,1,j_3}$. $g_{k,k+1,j_3}$ or $\overline{g_{k,k+1,j_3}}$  then enters $\mathbb{C}^{*}$ from $Q_{k,j_3}$. $g_{k,k+1,j_3}$ or $\overline{g_{k,k+1,j_3}}$  can enter the real axis within the interval $L_{k,0,j_3}$ and intersect at $Q_a$. $g_{k,k+1,j_3}$ or $\overline{g_{k,k+1,j_3}}$ may not enter the real axis and directly intersect at $Q_{a}$ in $\mathbb{C}^{*}$.

Therefore, $g_{k,k+1,j_3}$ or $\overline{g_{k,k+1,j_3}}$ intersects at $Q_{a}$. So, at least one pair of the non-real zeros $s_{k,k+1,j_3}$ or $\overline{s_{k,k+1,j_3}}$ of $p$ emits $2q\pi$ or $2q\pi+\pi$ degree root locus $g_{k,k+1,j_3}$ or $\overline{g_{k,k+1,j_3}}$, this root locus $g_{k,k+1,j_3}$ or $\overline{g_{k,k+1,j_3}}$  intersects at $Q_{a}$ with other root loci $R_{k,i_3,j_3}$ on the real axis. The intersection point $Q_{a}$ of the root loci is the first kind of zero of $(p^{'}/p)^{'}$.
\end{proof}

\begin{lemma}
Suppose the real polynomial $p$ has at least one real zero. Consider second-category of real breakaway points. If all root loci intersecting at a real breakaway point $Q_{a}$ in (2.2) originate directly from real zeros of $p$, then, at least one pair of the non-real zeros of $p$ emits a $2q\pi$ or $2q\pi+\pi$ degree root locus, this root locus intersects at $Q_{a}$ with other root loci on the real axis. The intersection point $Q_{a}$ of the root loci is a first kind of zero of $(p^{'}/p)^{'}$.
\end{lemma}
\begin{proof}
Let $t_{k,0,j_3}$ be a real pole of (2.2) with adjacent real zero $z_{k,0,j_3}$ of (2.2), forming interval $L_{k,0,j_3}$. The pole $t_{k,0,j_3}$ emits root locus $R_{k,0,j_3}$. Let $R_{k,i_4,j_3}$ be the $r_3$ branches of root loci emitted from the other real zeros $t_{k,i_4,j_3}$ of $p$, where $i_4=1,2,\cdots, r_3$. $r_3\ge 2$ is a positive integer. $R_{k,i_4,j_3}$ intersects with $R_{k,0,j_3}$ at the real breakaway point $Q_{a}$ in an interval $L_{k,0,j_3}$. Before $R_{k,0,j_3}$ intersects with $R_{k,i_4,j_3}$ at $Q_{a}$, $R_{k,0,j_3}$ must leave the real axis and enter $\mathbb{C}^{*}$.  Before $R_{k,i_4,j_3}$ intersects with the root loci $R_{k,0,j_3}$ at $Q_a$, $R_{k,i_4,j_3}$ leaves the real axis and enters $\mathbb{C}^{*}$.

Unlike the first case, in the second case, $Q_{a}$ may lie in one of several intervals. This problem doesn't affect the proof here, it is not a problem that requires special attention. We only need to know that $Q_{a}$ exists within a certain interval $L_{k,0,j_3}$ on the real axis.

Before $R_{k,1,j_3}$ intersects with the root loci $R_{k,i_4,j_3}$ and $R_{k,0,j_3}$ at $Q_{a}$, if this root locus $R_{k,1,j_3}$ doesn't intersect with the other root loci on the real axis, this root locus $R_{k,1,j_3}$ intersect with the root loci $R_{k,i_4,j_3}$ and $R_{k,0,j_3}$ from real zeros of $p(s)$, generating a real breakaway point $Q_{a}$. $Q_{a}$ is a first-kind zero of $(p^{'}/p)^{'}$ and $r_3+1$ branches of root loci intersect there.

Before $R_{k,1,j_3}$ intersects with the root loci $R_{k,i_4,j_3}$ and $R_{k,0,j_3}$ at $Q_{a}$, if this root locus $R_{k,1,j_3}$ intersects with the other root loci on the real axis, this situation is like as the second type of the other root loci intersection. After this root locus $R_{k,1,j_3}$ intersects with the other root loci, this root locus $R_{k,1,j_3}$ intersect with the root loci $R_{k,i_4,j_3}$ and $R_{k,0,j_3}$ from real zeros of $p(s)$, generating a real breakaway point $Q_{a}$. $Q_{a}$ is a first-kind zero of $(p^{'}/p)^{'}$ and $r_3+1$ branches of root loci intersect there.

Now consider root locus $R_{k,1,j_3}=g_{k,1,j_3}$ directly emitted from real zero $t_{k,1,j_3}$ of $p$. Let $z_{k,1,j_3}$ be the real zero of (2.2) adjacent to real pole $t_{k,1,j_3}$, forming interval $L_{k,1,j_3}$. The root locus $g_{k,1,j_3}$ leaves the interval $L_{k,1,j_3}$ and enters $\mathbb{C}^{*}$. Then, the root locus $g_{k,1,j_3}$ intersects at $Q_{a}$ in the interval $L_{k,0,j_3}$.

Therefore, root locus $g_{k,1,j_3}$ satisfies all conditions of Lemma 2.46.  Applying Lemma 2.46,  there exist root loci intersecting with $g_{k,1,j_3}$ and causing it to leave the real axis. These root loci are successively emitted from $k-1$ real zeros of $p$. The root loci emitted from these real zeros form the linked root loci. The root loci from the non-real zeros are $g_{k,k+1,j_3}$ and $\overline{g_{k,k+1,j_3}}$, and those from the real zeros are $g_{k,j_1,j_3}$ (for $j_1=2,3,\cdots,k$). These root loci intersect at real breakaway points $q_{k,2j_1-1,j_3}$, $q_{k,2j_1-2,j_3}$  on the real axis.

Therefore, $g_{k,1,j_3}$, $g_{k,j_1,j_3}$, $g_{k,k+1,j_3}$, $\overline{g_{k,k+1,j_3}}$ form the linked root loci.

The root locus $g_{k,1,j_3}$ is the beginning of the linked root loci. $g_{k,k+1,j_3}$ and $\overline{g_{k,k+1,j_3}}$  are as terminating of the linked root loci. By Lemma 2.47, within the linked root loci between $g_{k,1,j_3}$ and $g_{k,k+1,j_3}$ and $\overline{g_{k,k+1,j_3}}$, the root locus $g_{k,k+1,j_3}$ or $\overline{g_{k,k+1,j_3}}$ originating from the non-real zeros $s_{k,k+1,j_3}$ and $\overline{s_{k,k+1,j_3}}$ can be distributed on two main lines of the linked root loci. $g_{k,k+1,j_3}$ or $\overline{g_{k,k+1,j_3}}$  originating from the non-real zeros $s_{k,k+1,j_3}$ and $\overline{s_{k,k+1,j_3}}$ intersects at the real breakaway point $Q_{k,j_3}$ in $L_{k,1,j_3}$. $g_{k,k+1,j_3}$ or $\overline{g_{k,k+1,j_3}}$  then enters $\mathbb{C}^{*}$ from $Q_{k,j_3}$. $g_{k,k+1,j_3}$ or $\overline{g_{k,k+1,j_3}}$  can enter the real axis within the interval $L_{k,0,j_3}$ and intersect at $Q_{a}$. $g_{k,k+1,j_3}$ or $\overline{g_{k,k+1,j_3}}$ may not enter the real axis and directly intersect at $Q_{a}$ in $\mathbb{C}^{*}$.

Therefore, $g_{k,k+1,j_3}$ or $\overline{g_{k,k+1,j_3}}$ intersects at $Q_{a}$. So, at least one pair of the non-real zeros $s_{k,k+1,j_3}$ or $\overline{s_{k,k+1,j_3}}$ of $p$ emits $2q\pi$ or $2q\pi+\pi$ degree root locus $g_{k,k+1,j_3}$ or $\overline{g_{k,k+1,j_3}}$, this root locus $g_{k,k+1,j_3}$ or $\overline{g_{k,k+1,j_3}}$  intersects at $Q_{a}$ with the root loci $R_{k,i_4,j_3}$ on the real axis. The intersection point $Q_{a}$ of the root loci is the first kind of zero of $(p^{'}/p)^{'}$.
\end{proof}

Combining Lemmas 2.48 and 2.49 yields Lemma 2.50.

\begin{lemma}
Suppose the real polynomial $p$ has at least one real zero. If all root loci intersecting at a real breakaway point $Q_{a}$ in (2.2) originate directly from real zeros of $p$, then, at least one pair of the non-real zeros of $p$ emits a $2q\pi$ or $2q\pi+\pi$ degree root locus, this root locus intersects at $Q_{a}$ with other root loci on the real axis. On the real axis, the intersection point $Q_{a}$ of the root loci is a first kind of zero of $(p^{'}/p)^{'}$.
\end{lemma}

For root loci intersecting at a real breakaway point $Q_{a}$ in (2.2), if at least one of them originates directly from a non-real zero $s_a$ or $\overline{s_a}$ of $p$, under the condition, by Lemma 2.11, we have: at least one pair of the non-real zeros of $p$ emits a $2q\pi$ or $2q\pi+\pi$ degree root locus, and this root locus intersects with other root loci on the real axis. On the real axis, the intersection point $Q_{a}$ of the root loci is a first kind of zero of $(p^{'}/p)^{'}$.

If all root loci intersecting at a real breakaway point $Q_{a}$ in (2.2) originate directly from the real zeros of $p$, under the condition,  by Lemma 2.50, we have: at least one pair of the non-real zeros of $p$ emits a $2q\pi$ or $2q\pi+\pi$ degree root locus, and this root locus intersects with other root loci on the real axis. On the real axis, the intersection point $Q_{a}$ of the root loci is a first kind of zero of $(p^{'}/p)^{'}$.

Therefore, combining Lemmas 2.11 and 2.50 yields Lemma 2.51.

\begin{lemma}
Suppose the real polynomial $p$ has at least one real zero. When the root loci intersect at the real breakaway point $Q_{a}$. Then, at least one pair of the non-real zeros of $p$ emits a $2q\pi$ or $2q\pi+\pi$ degree root locus, this root locus intersects with other root loci at $Q_{a}$ on the real axis. On the real axis, the intersection point $Q_{a}$ of the root loci is a first kind of zero of $(p^{'}/p)^{'}$.
\end{lemma}

The real breakaway point $Q_{a}$ is any real breakaway point. This real breakaway point can be produced by the intersection of root loci from non-real zeros. It can also be entirely generated by the intersection of root loci from real zeros. However, Lemma 2.46 proves that the root loci from real zeros and those of non-real zeros form linked root loci. Lemma 2.47  proves that the root locus from a non-real zero can be distributed on two main lines of the linked root loci. Therefore, this type of real breakaway point is also produced by the intersection of root loci from non-real zeros with other root loci. Thus, by Theorem 2.6 and Lemma 2.51, we have:

\begin{lemma}
Suppose the real polynomial $p$ has at least one real zero. Only the root locus emitted from a non-real zero of the real polynomial $p$ intersects on the real axis. A first kind of zero of $(p^{'}/p)^{'}$ can be generated. A $2q\pi$ or $2q\pi+\pi$ degree root locus emitted from at least one pair of the non-real zeros of $p$ intersects with other root loci on the real axis. On the real axis, the intersection point of the root loci is a first kind of zero of $(p^{'}/p)^{'}$.
\end{lemma}
Lemma 2.9 proves the result which the real polynomial $p$ has no any real zeros. Lemma 2.52 proves the result which the real polynomial $p$ has at least one real zero.
Lemmas 2.9 and 2.52 together imply Theorem 0.3.

The proof of Theorem 0.4.
\begin{proof}
In Theorem 0.3, at least one pair of the non-real zeros of $p$ emits $2q\pi$ or $2q\pi+\pi$ degree root locus, and this root locus intersects with other root loci on the real axis. Through the $2q\pi$ or $2q\pi+\pi$ degree root locus emitted from the pair of the non-real zeros of $p$, each first kind of zero of $(p^{'}/p)^{'}$ is connected to a pair of the non-real zeros of the real polynomial $p$. When we prove Theorem 0.3, we use three lemmas. They are Lemma 2.9, Lemma 2.11, and Lemma 2.52, respectively.

In Lemma 2.9, only the root loci of non-real zeros generate first kind of zero of $(p^{'}/p)^{'}$. So, through the $2q\pi$ or $2q\pi+\pi$ degree root locus emitted from a pair of the non-real zeros of $p$, each first kind of zero of $(p^{'}/p)^{'}$ is connected to a pair of the non-real zeros of $p$.

In Lemma 2.11, for root loci intersecting at a real breakaway point $Q_a$ in (2.2), if at least one of them originates directly from a non-real zero $s_a$ or $\overline{s_a}$ of $p$, thus, through the $2q\pi$ or $2q\pi+\pi$ degree root locus emitted from a pair of the non-real zeros $s_a$ or $\overline{s_a}$ of $p$, the first kind of zero $Q_{a}$ of $(p^{'}/p)^{'}$ is connected to a pair of the non-real zeros $s_a$ or $\overline{s_a}$ of $p$.

In Lemma 2.52, for root loci intersecting at a real breakaway point $Q_{a}$ in (2.2), if all of them originate directly from real zeros $s_{k,j_1,j_3}$ of $p$. At least root locus from real zero of $p$ forms a linked root loci. Through two main lines of the linked root loci, the first kind of zero $Q_{a}$ of $(p^{'}/p)^{'}$ is connected to a pair of the non-real zeros $s_{k,k+1,j_3}$ or $\overline{s_{k,k+1,j_3}}$ of $p$.

The first kind of zero $Q_{a}$ of $(p^{'}/p)^{'}$ is any real Fourier critical point. Therefore, we prove: through at least one $2q\pi$ or $2q\pi+\pi$ degree root locus of the non-real zeros of $p$, each first kind of zero of $(p^{'}/p)^{'}$ is connected to a pair of the non-real zeros of $p$.
\end{proof}

In fact, Lemma 2.52 has a dual nature. Its condition requires that all root loci intersecting at a real breakaway point $Q_{a}$ in (2.2) originate directly from real zeros of $p$. However, the lemma concludes that the $2q\pi$ or $2q\pi+\pi$ degree root locus emitted from at least one pair of non-real zeros of $p$ intersects with other root loci at $Q_{a}$ on the real axis. While the root loci intersecting at $Q_{a}$ can originate exclusively from real zeros, root loci emitted from non-real zeros can also intersect with other root loci at $Q_{a}$. Both scenarios are reasonable. Crucially, if terminal root loci from non-real zeros do not intersect on the real axis, then the scenario where all root loci originate from real zeros and intersect at $Q_{a}$ cannot exist. Therefore, we only show Lemma 2.52.

We can give another result which the real Fourier critical point is generated:

Now suppose the real polynomial $p$ has at least one real zero. If all root loci intersecting at a real breakaway point $Q_{a}$ in (2.2) originate directly from real zeros of $p$, then at least one root locus starting from a real zero of $p$ becomes the beginning root locus of the linked root loci, while the terminal root locus is emitted from a non-real zero. Under the condition that the root locus from a non-real zero intersects the real axis, the $2q\pi$ or $2q\pi+\pi$ degree root locus emitted from non-real zeros of $p$ intersects at $Q_{a}$ on the real axis. Here, the intersection point $Q_{a}$ is a first-kind zero of $(p^{'}/p)^{'}$ on the real axis.

Under this condition, the root locus emitted from real zeros can generate real Fourier critical points. An alternative result on generating real Fourier critical points $Q_{a}$ presents this principle more comprehensively. Both Lemma 2.52 and this alternative result are accurate and reasonable. Lemma 2.52 is fully consistent with Lemma 2.9, whereas the alternative result is not. Therefore, we present only Lemma 2.52 and omit the alternative result on generating real Fourier critical points $Q_{a}$.

\begin{example}
The real polynomial $p_{e1}(s)$ is given by $p_{e1}(s)=s^{20}+s^{19}+s^{18}+s^{17}+s^{16}+s^{15}+s^{14}+s^{13}+s^{12}+s^{11}+s^{10}+s^9+s^8+s^7+s^6+s^5+s^4+s^3+s^2+s+1$.
Its derivative is: $p_{e1}^{'}(s)=20s^{19}+19s^{18}+18s^{17}+17s^{16}+16s^{15}+15s^{14}+14s^{13}+13s^{12}+12s^{11}+11s^{10}+10s^9+9s^8+8s^7+7s^6+6s^5+5s^4+4s^3+3s^2+2s+1$.
\end{example}

The real polynomial $p_{e1}(s)$ has twenty conjugate non-real zeros $s_{e1,1,2}=-0.989 \pm 0.149i$, $s_{e1,3,4}=-0.901 \pm 0.434i$, $s_{e1,5,6}=-0.733 \pm 0.68i$, $s_{e1,7,8}=-0.5 \pm 0.866i$, $s_{e1,9,10}=-0.233 \pm 0.975i$, $s_{e1,11,12}=0.0747 \pm 0.997i$, $s_{e1,13,14}=0.365 \pm 0.931i$, $s_{e1,15,16}=0.623 \pm 0.782i$. $s_{e1,17,18}=0.826 \pm 0.563i$, $s_{e1,19,20}=0.956 \pm 0.295i$. The real polynomial $p_{e1}^{'}(s)$ has eighteen conjugate non-real zeros $s^ {'}_{e1,1,2}=-0.795\pm 0.252i$, $s^{'}_{e1,3,4}=-0.683\pm 0.482i$. $s^{'}_{e1,5,6}=-0.507\pm 0.668i$, $s^{'}_{e1,7,8}=-0.282\pm 0.794i$. $s^{'}_{e1,9,10}=-0.0282\pm 0.847i$, $s^{'}_{e1,11,12}=0.232\pm 0.822i$. $s^{'}_{e1,13,14}=0.476\pm 0.721i$, $s^{'}_{e1,15,16}=0.685\pm 0.551i$. $s^{'}_{e1,17,18}=0.844\pm 0.322i$. In addition, it has a real zero $s^{'}_{e1,19}=-0.834$.

We establish the root locus equation of the rational function $\frac{p_{e1}^{'}(s)}{p_{e1}(s)}$: $K\frac{p_{e1}^{'}(s)}{p_{e1}(s)}=\pm 1$.

The interval $(-\infty, s^{'}_{e1,19})$ is the $2q\pi+\pi$ degree root locus of the root locus equation of this example. The interval $(-\infty, s^{'}_{e1,19})$ is the $2q\pi+\pi$ degree root locus between the two zeros of the root locus equation of this example. Since both the negative infinite zero $s=-\infty$ and the finite zero $s^{'}_{e1,19}$ need to receive the $2q\pi+\pi$ degree root locus, the two conjugate non-real zeros $s_{e1,1,2}$ of the polynomial $p_{e1}(s)$ each emits a $2q\pi+\pi$ degree root locus, which enters the real axis from the breakaway point $b_{e1,1}=-1.15$ in the interval. The two $2q\pi+\pi$ degree root loci finally end at the infinite zero $s=-\infty$ and the finite zero $s^{'}_{e1,19}$.

The interval $(s^{'}_{e1,19}, +\infty)$ is the $2q\pi$ degree root locus of the root locus equation of this example. The interval $(s^{'}_{e1,19}, +\infty)$ is the $2q\pi$ degree root locus between the two zeros of the root locus equation of this example. Since both the finite zero $s^{'}_{e1,19}$ and the infinite zero $+\infty$ need to receive the $2q\pi$ degree root locus, the two conjugate non-real zeros $s_{e1,19,20}$ of the polynomial $p_{e1}(s)$ emit the $2q\pi$ degree root locus to enter the real axis from the breakaway point $b_{e1,2}=1.25$ inside the interval. The two $2q\pi$ degree root loci finally end at the zero $s^{'}_{e1,19}$ and the infinite zero $+\infty$.

The zeros $s_{e1,1,2}$ of the polynomial $p_{e1}(s)$, each emits a $2q\pi+\pi$ degree root locus on the real axis in an infinite interval $(-\infty, s^{'}_{e1,19})$ that intersects. The zeros $s_ {e1,19,20}$ of the polynomial $p_{e1}(s)$, each emits a $2q\pi$ degree root locus intersecting on the real axis in another infinite interval $(s^{'}_{e1,19}, +\infty)$. These results generate the two real Fourier breakaway critical points $b_{e1,1}=-1.15$, $b_{e1,2}=1.25$ of the rational function. Through the $2q\pi$ degree and $2q\pi+\pi$ degree root loci, the two real Fourier breakaway critical points are connected to the four conjugate non-real zeros of the polynomial $p_{e1}(s)$.

Using MATLAB, we sketch the $2q\pi+\pi$ degree and the $2q\pi$ degree root locus of the root locus equation in this example. Figure 1 shows the $2q\pi+\pi$ degree and the $2q\pi$ degree root loci in this example.

\begin{example}
The real polynomial $p_{e2}(s)$ is given by $p_{e2}(s)=s^{21}+s^{20}+s^{19}+s^{18}+s^{17}+s^{16}+s^{15}+s^{14}+s^{13}+s^{12}+s^{11}+s^{10}+s^9+s^8+s^7+s^6+s^5+s^4+s^3+s^2+s+1$. Its derivative is: $p_{e2}^{'}(s)=21s^{20}+20s^{19}+19s^{18}+18s^{17}+17s^{16}+16s^{15}+15s^{14}+14s^{13}+13s^{12}+12s^{11}+11s^{10}+10s^9+9s^8+8s^7+7s^6+6s^5+5s^4+4s^3+3s^2+2s+1$. \end{example}

The real polynomial $p_{e2}(s)$ has twenty conjugate non-real zeros $s_{e2,1,2}=-0.959 \pm 0.282i$, $s_{e2,3,4}=-0.841 \pm 0.541i$, $s_{e2,5,6}=-0.655 \pm 0.756i$, $s_{e2,7,8}=-0.415 \pm 0.91i$, $s_{e2,9,10}=-0.142 \pm 0.99i$, $s_{e2,11,12}=0.142 \pm 0.99i$, $s_{e2,13,14}=0.415 \pm 0.91i$, $s_{e2,15,16}=0.655 \pm 0.756i$, $s_{e2,17,18}=0.841 \pm 0.541i$, $s_{e2,19,20}=0.959 \pm 0.282i$, and $p_{e2}(s)$ also has a real zero $s_{e2,21}=-1$. The real polynomial $p_{e2}^{'}(s)$ has twenty conjugate non-real zeros $s^{'}_{e2,1,2}=-0.83\pm 0.123i$, $s^{'}_{e2,3,4}=-0.76\pm 0.358i$. $s^{'}_{e2,5,6}=-0.626\pm 0.563i$, $s^{'}_{e2,7,8}=-0.438\pm 0.722i$, $s^{'}_{e2,9,10}=-0.211\pm 0.822i$. $s^{'}_{e2,11,12}=0.0354\pm 0.853i$. $s^{'}_{e2,13,14}=0.283\pm 0.813i$. $s^{'}_{e2,15,16}=0.512\pm 0.703i$. $s^{'}_{e2,17,18}=0.706\pm 0.523i$. $s^{'}_{e2,19,20}=0.853\pm 0.309i$.

We establish the root locus equation of the rational function $\frac{p_{e2}^{'}(s)}{p_{e2}(s)}$: $K\frac{p_{e2}^{'}(s)}{p_{e2}(s)}=\pm 1$.

The two conjugate non-real zeros $s_{e2,19,20}$ of the polynomial $p_{e2}(s)$ emit $2q\pi$ degree root loci respectively. These two root loci first intersect at the breakaway point $b_{e2,2}=1.24$ inside the interval $(s_{e2,21},+\infty)$ on the real axis. One of the root loci then intersects again with the $2q\pi$ degree root loci emitted by real zero $s_{e2,21}$ of the polynomial $p_{e2}(s)$ at the breakaway point $b_{e2,1}=0.688$. In this way, two real Fourier breakaway critical points $b_{e2,1}$ and $b_{e2,2}$ are produced. Through the $2q\pi$ degree root locus, the two real Fourier breakaway critical points ($b_{e2,1}$, $b_ {e2,2}$) are connected to the two conjugate non-real zeros of the polynomial $p_{e2}(s)$.

The interval $(-\infty, s_{e2,21})$ is the $2q\pi+\pi$ degree root locus of the root locus equation of this example. The interval $(s_{e2,21}, +\infty)$ is the $2q\pi$ degree root locus of the root locus equation of this example. Using MATLAB, we sketch the $2q\pi$ degree and the $2q\pi+\pi$ degree root locus of root locus equation of this example. We find that there are two breakaway points in the infinite interval $(s_{e2,21},+\infty)$. That is, two real Fourier breakaway critical points $b_{e2,1}$ and $b_{e2,2}$.

Using MATLAB, we computationally prove that the polynomial $p_{e2}(s)$ can only produce two real Fourier breakaway critical points. No other real Fourier breakaway critical point is produced. Figure 2 shows the $2q\pi+\pi$ degree and the $2q\pi$ degree root loci in this example.

Figure 1 and Figure 2 both show the poles (asterisks) and zeros (circles) of the root locus equation of two examples. The curves connecting the two types of points in the left root locus diagram are the root loci of $2q\pi+\pi$ degree of two examples, while the curves connecting the two types of points in the right root locus diagram are the root loci of $2q\pi$ degree.

Example 2.23 demonstrates that a linked root loci must exist. In the right root locus diagram of this example, the $2q\pi$ degree root loci from the non-real zeros $s_{e2,19,20}=0.959 \pm 0.282i$ intersect at the real breakaway point $b_{e2,2}=1.24$. One of these two root loci then intersects with the $2q\pi$ degree root locus emitted from the real zero $s_{e2,21}=-1$ at the real breakaway point $b_{e2,1}=0.688$. The root locus emitted from $s_{e2,21}=-1$ leaves the real axis and finally ends at one of the points $s^{'}_{e2,19,20}=0.853\pm 0.309i$. Together, the two $2q\pi$ degree root loci from the non-real zeros $s_{e2,19,20}=0.959 \pm 0.282i$ and the $2q\pi$ degree root locus from the real zero $s_{e2,21}=-1$ form the simplest linked root loci.

Both figures in Example 2.57 illustrate Lemma 2.9. Since $p_{e1}(s)$ has no real zeros, real Fourier critical points can only arise from non-real zeros.

While the left figure of Example 2.58 does not exemplify the results of this paper, the right figure of Example 2.58 exemplifies Lemma 2.11. In this right figure of Example 2.58, the root loci emitted from non-real zeros generate the real Fourier critical points. $b_{e2,2}$ is produced by the intersection of root loci from two conjugate non-real zeros $s_{e2,19,20}=0.959 \pm 0.282i$, and $b_{e2,1}$ by the intersection of root loci from non-real zeros $s_{e2,19,20}=0.959 \pm 0.282i$ and root locus from real zero $s_{e2,21}=-1$.

The right figure of Example 2.58 also exemplifies Lemma 2.13. Here, the simplest linked root loci are formed by a root locus emitted from the real zero $s_{e2,21}=-1$ of $p_{e2}(s)$, two root loci emitted from the conjugate non-real zeros $s_{e2,19,20}=0.959 \pm 0.282i$, and the two real breakaway points $b_{e2,2}$ and $b_{e2,1}$.

Moreover, the examples used belong to a highly specific category of Littlewood polynomials. These polynomials constitute a special class of rational functions whose root loci possess distinctive properties.

\medskip
Conflict of interest: On behalf of all the authors, the corresponding author states that there are no conflicts of
interest.

Declaration of interest: Authors declare that they have no conflict of interest.

\bibliographystyle{unsrt}

\end{document}